\pdfoutput=1 		% For the arXiv
\documentclass[leqno]{article}
\usepackage[normal]{phaine_style}
\usepackage[margin=1.5in]{geometry}
\usepackage{xparse}
\include{standard_macros}
\renewcommand{\orientedtimes}{\mathbin{\vec{\times}}}

\newcommand{\fpqc}{\ensuremath{\textup{fpqc}}}

\newcommand{\mult}{\mathrm{mult}}

\newcommand{\ab}{\mathrm{ab}}
\newcommand{\const}{\mathrm{const}}
\renewcommand{\lex}{\mathrm{lex}}
\newcommand{\acc}{\mathrm{acc}}
\newcommand{\Funlexacc}{\Fun^{\lex,\acc}}

\renewcommand{\ni}{\smallni}
\renewcommand{\Map}{\mathrm{Map}}

\newcommand{\Max}{\mathrm{Max}}
\newcommand{\Cb}{\upC_{\mathrm{b}}}

\newcommand{\Perfbf}{\categ{Perf}}
\newcommand{\lis}{\mathrm{lis}}
\newcommand{\Dlis}{\Dup_{\lis}}

\newcommand{\LCH}{\categ{LCH}}
\newcommand{\qTopGrp}{\categ{qTopGrp}}

\newcommand{\h}{\mathrm{h}}

\renewcommand{\Fun}{\mathrm{Fun}}
\renewcommand{\cts}{\mathrm{cts}}
\renewcommand{\Gal}{\mathrm{Gal}}
\renewcommand{\Functs}{\Fun^{\cts}}
\newcommand{\Funcond}{\Fun^{\cond}}

\renewcommand{\cocart}{\mathrm{cocart}}
\newcommand{\Funcocart}{\Fun^{\cocart}}

\newcommand{\Ani}{\categ{Ani}}
\renewcommand{\Spc}{\Ani}
\DeclareMathOperator{\Cond}{Cond}
\newcommand{\ICond}{\categ{Cond}}

\newcommand{\CondAni}{\Cond(\Ani)}
\newcommand{\CondSet}{\Cond(\Set)}

\newcommand{\CondCat}{\Cond(\Catinfty)}

\newcommand{\CondGrp}{\Cond(\Grp)}
\newcommand{\ICat}{\mathrm{Cat}}

\newcommand{\Comp}{\categ{Comp}}

\DeclareMathOperator{\RFib}{RFib}

\DeclareMathOperator{\ult}{ult}

\DeclareMathOperator{\FinSet}{\Set_{\fin}}
\DeclareMathOperator{\FinGrp}{\Grp_{\fin}}

\DeclareMathOperator{\ProFin}{\Pro(\FinSet)}
\DeclareMathOperator{\ProFinGrp}{\Pro(\FinGrp)}
\newcommand{\Aniult}{\Ani{}^{\ult}}
\newcommand{\Setult}{\Set{}^{\ult}}

\DeclareMathOperator{\wLoc}{wLoc}
\DeclareMathOperator{\im}{im}

\newcommand{\Extr}{\categ{Extr}}
\newcommand{\Extrop}{\Extr^{\op}}

\DeclareMathOperator{\wc}{wc}
\newcommand{\Affwc}{\Aff{}^{\kern0.15em\wc}}
\newcommand{\Affwcop}{\Aff{}^{\kern0.15em\wc,\op}}

\DeclareMathOperator{\FEt}{F\acute{E}t}

\DeclareMathOperator{\sep}{sep}
\DeclareMathOperator{\qs}{{qs}}

\newcommand{\ZZhat}{\widehat{\ZZ}}

\renewcommand{\top}{{\operatorname{top}}} 

\newcommand{\nc}{\operatorname{nc}}
\newcommand{\tnc}{\operatorname{tnc}}

\newcommand{\aff}{\ensuremath{\textup{aff}}}

\newcommand{\blank}{-}

\newcommand{\pmring}{\xspace{pm-ring}\xspace}
\newcommand{\pmrings}{\xspace{pm-rings}\xspace}

\newcommand{\Sigmacompletion}{\xspace{$\Sigma$-com\-ple\-tion}\xspace}

\newcommand{\wcontractible}{\xspace{w-con\-tract\-i\-ble}\xspace}
\newcommand{\wcontractibles}{\xspace{w-con\-tract\-i\-bles}\xspace}

\newcommand{\RTop}{\categ{RTop}}

\newcommand{\cond}{\mathrm{cond}}

\newcommand{\Bcond}{\Bup^{\cond}}
\newcommand{\BcondGal}{\Bup^{\cond}\mathrm{Gal}}

\newcommand{\protrun}{\trun_{<\infty}}

\newcommand{\Sigmacomp}{_{\Sigma}^{\wedge}}

\newcommand{\TopGrp}{\mathrm{Grp}(\Top)}

\renewcommand{\top}{\mathrm{top}}

\newcommand{\prodisccompl}{^{\wedge}_{\disc}}
\newcommand{\prodisccomp}{^{\wedge}_{\disc}}

\newcommand{\Action}[1]{#1\text{-}\categ{Set}}

\newcommand*{\triple}[2][.05ex]{%
  \mathrel{\vcenter{\offinterlineskip%
  \hbox{$#2$}\vskip#1\hbox{$#2$}\vskip#1\hbox{$#2$}}}}
\newcommand*{\triplerightarrow}{\triple{\rightarrow}}

\newcommand{\NoohiGrp}{\Grp^{\Noohi}}

\newcommand{\Ring}{\mathbf{Ring}}

\DeclareMathOperator{\Freepf}{\widehat{F}r}
\DeclareMathOperator{\Free}{Fr}
\newcommand{\Freetop}{\Free^{\top}}
\newcommand{\Freecond}{\Free^{\cond}}
\newcommand{\freeprod}{*}
\newcommand{\freeprodcond}{\freeprod^{\cond}} %DeclareMathOperator didn't look well with subscripts
\newcommand{\freeprodpre}{\freeprod^{\pre}}
\newcommand{\freeprodtop}{\freeprod^{\top}}

\newcommand{\profincomp}{_{\uppi}^{\wedge}}

\newcommand{\piet}{\uppi^{\et}}
\newcommand{\piproet}{\uppi^{\proet}}
\DeclareMathOperator{\BGal}{BGal}

\newcommand{\Shapeprotrun}{\Pi_{<\infty}}

\newcommand{\Piet}{\Pi_{\infty}^{\et}}
\newcommand{\Pietprotrun}{\Pi_{<\infty}^{\et}}
\newcommand{\Pietprofin}{\widehat{\Pi}_{\infty}^{\et}}

\newcommand{\Picond}{\Pi_{\infty}^{\cond}}
\newcommand{\CondShape}[1]{\Picond(#1)}
\newcommand{\picond}{\uppi^{\cond}}
\newcommand{\picondqs}{\uppi^{\cond,\qs}}

\newcommand{\pizerocond}{\picond_0}

\newcommand{\pione}{\uppi_1}
\newcommand{\pioneproet}{\piproet_1}
\newcommand{\pionecond}{\picond_1}
\newcommand{\pionecondqs}{\picondqs_1}
\newcommand{\pioneet}{\uppi_1^{\et}}
\newcommand{\Pionecond}{\Pi_1^{\mathrm{cond}}}
\newcommand{\Pioneetprofin}{\widehat{\Pi}_{1}^{\et}}

\newcommand{\qcqs}{\mathrm{qcqs}}
\renewcommand{\qc}{\mathrm{qc}}

\newcommand{\Zpos}[1]{{#1}_{\zar}^{\leq}}

\DeclareMathOperator{\subdiv}{sd}
\DeclareMathOperator{\prozar}{prozar}

\DeclareMathOperator{\Hens}{Hens}
\DeclareMathOperator{\HenszarX}{\Hens^{\zar}_X}

\newcommand{\Noohi}{{\operatorname{Noohi}}}

\newcommand{\from}{\colon}

\DeclareMathOperator{\characteristic}{char}

\newcommand{\upC}{\mathrm{C}}

\renewcommand{\Setfin}{\Set_{\fin}}

\newcommand{\ProSetfin}{\Pro(\Setfin)}

\newcommand{\ProAni}{\Pro(\Ani)}
\newcommand{\Anifin}{\Ani_{\uppi}}
\newcommand{\Anitrun}{\Ani_{<\infty}}
\newcommand{\ProAnifin}{\Pro(\Anifin)}
\newcommand{\ProAnitrun}{\Pro(\Anitrun)}

\DeclareMathOperator{\Cocart}{Cocart}
\DeclareMathOperator{\LFib}{LFib}
\newcommand{\MSpec}{\mathrm{MSpec}}
\renewcommand{\Spec}{\mathrm{Spec}}

\newcommand{\QQell}{\QQ_{\el}}
\newcommand{\ZZell}{\ZZ_{\el}}

\newcommand{\Xkbar}{X_{\kbar}}
\newcommand{\Xsbar}{X_{\sbar}}

\NewDocumentCommand{\multgrp}{o}{\mathbb{G}\IfValueTF{#1}{_{\mup, #1}}{_{\mup}}}

\DeclareMathOperator{\cl}{cl}

\newcommand{\proethyp}{_{\proet}^{\hyp}}

\newcommand{\Xproethyp}{X\proethyp}
\newcommand{\Yproethyp}{Y\proethyp}

\newcommand{\Zaraff}[1]{\mathrm{Zar}{}_{#1}^{\kern0.1em\aff}}
\newcommand{\Et}[1]{\mathrm{\acute{E}t}_{#1}}
\newcommand{\Etaff}[1]{\mathrm{\acute{E}t}{}_{#1}^{\aff}}
\newcommand{\ProEt}[1]{\mathrm{Pro\acute{E}t}_{#1}}
\newcommand{\ProEtwc}[1]{\mathrm{Pro\acute{E}t}{}_{#1}^{\wc}}
\newcommand{\ProEtaff}[1]{\mathrm{Pro\acute{E}t}{}_{#1}^{\kern0.1em\aff}}

\newcommand{\ProZaraff}[1]{\mathrm{ProZar}{}_{#1}^{\kern0.1em\aff}}

\title{\Large The condensed homotopy type of a scheme}

\author{\normalsize Peter J. Haine \and\normalsize Tim Holzschuh \and\normalsize Marcin Lara \and\normalsize Catrin Mair \and\normalsize Louis Martini \and\normalsize Sebastian Wolf}

\date{\normalsize \textit{with an appendix by} Bogdan Zavyalov \\ 
\hfill \\ 
\normalsize \today}

\begin{document}

\maketitle

%-------------------------------------------------------------------%
%-------------------------------------------------------------------%
%  Abstract                                                         %
%-------------------------------------------------------------------%
%-------------------------------------------------------------------%

\begin{abstract} 
    We study a condensed version of the étale homotopy type of a scheme, which refines both the usual étale homotopy type of Friedlander--Artin--Mazur and the proétale fundamental group of Bhatt--Scholze.
    In the first part of this paper, we prove that this \emph{condensed homotopy type} satisfies descent along integral morphisms and that the expected fiber sequences hold.
    We also provide explicit computations, for example, for rings of continuous functions.
    A key ingredient in many of our arguments is a description of the condensed homotopy type using the \emph{Galois category} of a scheme introduced by Barwick--Glasman--Haine.
    
    In the second part, we focus on the fundamental group of the condensed homotopy type in more detail.
    We show that, unexpectedly, the fundamental group of the condensed homotopy type of the affine line $\AA^1_{\CC} $ over the complex numbers is nontrivial.
    Nonetheless, its Noohi completion recovers the proétale fundamental group of Bhatt--Scholze.
    Moreover, we show that a mild correction---passing to the \emph{quasiseparated quotient}---fixes most of this group's quirks. 
    Surprisingly, this quotient is often a topological group.
\end{abstract}

\tableofcontents

\newpage

%-------------------------------------------------------------------%
%-------------------------------------------------------------------%
%  Introduction                                                     %
%-------------------------------------------------------------------%
%-------------------------------------------------------------------%

\section{Introduction}

%-------------------------------------------------------------------%
%  Motivation                                                       %
%-------------------------------------------------------------------%
  
\subsection{Motivation and overview}

Let $ X $ be a locally topologically noetherian scheme. 
In their work on the proétale topology \cite[\S7]{BhattScholzeProetale}, Bhatt and Scholze defined a refinement of the étale fundamental group called the \textit{proétale} fundamental group \smash{$ \piproet_1(X) $}. 
Its profinite completion recovers the usual étale fundamental group; moreover, the proétale and étale fundamental groups coincide for normal schemes.
While the étale fundamental group classifies local systems with values in profinite rings such as $ \ZZell $, it generally does not classify $ \QQell $-local systems.
The proétale fundamental group fixes this, as it has the better feature that it classifies local systems in a more general class of topological rings.

The (SGA 3) étale fundamental group is the fundamental group of the \textit{étale homotopy type}, a proanima introduced by Artin--Mazur \cite[\S9]{MR0245577} and Friedlander \cite[\S4]{MR676809}.
The étale homotopy type classifies derived $ \ZZell $-local systems, and has a number of important applications.
For example, Friedlander's \cite{MR366929} and Sullivan's \cite{MR0442930} proofs of the Adams Conjecture, Feng's proof \cite{MR4122428} of Tate's 1966 conjecture on the Artin--Tate pairing \cite{MR1610977}, and applications to anabelian geometry \cites{MR3248993}{MR3549624}.

Motivated by the utility of the proétale fundamental group and the étale homotopy type, one desires a common refinement of the two to a `homotopy type' that classifies derived $ \QQell $-local systems and refines the key properties of the étale homotopy type.
The main goal of this article is to use the theory of condensed mathematics introduced by Clausen--Scholze \cites{Scholze:condensednotes} to investigate such a refinement.  

This article is not the first to \emph{introduce} a condensed refinement of the étale homotopy type; one definition has been given by Barwick--Glasman--Haine via exodromy \cite[13.8.10]{Exodromy}, and another one, following a suggestion by Bhatt--Scholze \cite[Remark 4.2.9]{BhattScholzeProetale}, was given by Hemo--Richarz--Scholbach \cite[Appendix A]{hemo2023constructible_sheaves_schemes}.
But beyond a few basic properties, little more was known about these refinements.
Hence, the primary aim of this article is to undertake a thorough investigation of them. 

The definition given in \cite{hemo2023constructible_sheaves_schemes} proceeds as follows.
For a qcqs scheme $X$, pick a proétale hypercover $X_\bullet \to X$ by \wcontractible schemes.  
Then for every $n \in \NN$, the set of connected components $\uppi_0(X_n)$ is naturally a profinite set.  
Define the \emph{condensed homotopy type} of $X$ to be the colimit  
\begin{equation*}
    \textstyle \Picond(X) \colonequals \colim_{\Deltaop} \uppi_0(X_\bullet) \in \CondAni \comma
\end{equation*}  
computed in the \category $\CondAni$ of condensed anima.
The idea is that the condensed homotopy type should be `trivial' (meaning having no higher homotopy groups) on \wcontractible affines, and on general schemes, defined via proétale hyperdescent.
More formally, $ \Picond $ is the unique hypercomplete proétale cosheaf whose value on \wcontractible affines is $ \uppi_0 $.

This definition is convenient for some formal manipulations but often too inexplicit to directly compute in concrete examples.
To remedy this, one of the main tools that we use relies on the work of Barwick--Glasman--Haine \cite{Exodromy}.  
They introduced an explicit profinite category $ \Gal(X) $ whose underlying category is the category of points of the étale topos of $ X $; the profinite structure globalizes the topologies on the absolute Galois groups of the residue fields of $ X $.

The pro-category $ \Gal(X) $ can be regarded as a condensed category; the aforementioned condensed refinement of the étale homotopy type proposed by Barwick--Glasman--Haine \cite[13.8.10]{Exodromy} is the \textit{condensed classifying anima} of $ \Gal(X) $, obtained by inverting all morphisms in this condensed category.
Wolf showed that the whole hypercomplete proétale \topos can be recovered from the condensed category $ \Gal(X) $ \cite{MR4574234}.
Using Wolf's theorem, we prove in \cref{prop:Picond_is_BGal} that this proposed definition agrees with the other proposal mentioned above:
\begin{equation*}
    \Picond(X) \simeq \Bcond \Gal(X) \period
\end{equation*}  

Before explaining our main results in detail, we now turn to briefly summarizing the contents of this article.
This article consists of two parts. 
In the first part, we show that, in many respects, the condensed homotopy type behaves as one would expect from a refinement of the étale homotopy type.  
Among other results, we show that an analogue of the \textit{fundamental fiber sequence} holds and that the condensed homotopy type satisfies \textit{integral descent}; see \cref{intro_thm:proetale_fundamental_fiber_sequence,intro_thm:integral-descent-proetale-topos} below.  
We also provide explicit computations of the condensed homotopy type, for example for rings of continuous functions $\upC(T,\CC)$, where $T$ is a compact Hausdorff space (see \cref{intro_thm:etale_homotopy_type_of_rings_of_continuous_functions}).  

In the second part of this article, we focus on the \emph{condensed fundamental group}.  
Every geometric point $\xbar \to X$ defines a point of the condensed anima $\Picond(X)$, giving rise to condensed groups  
\begin{equation*}
    \picond_n(X,\xbar) \colonequals \uppi_n(\Picond(X),\xbar) \period
\end{equation*}  
Computing these groups is generally difficult, and the results can be wild and unexpected.  
For instance, we prove in \cref{cor:pionecond-of-A1-is-nontrivial} that the fundamental group of the affine line over the complex numbers is \emph{nontrivial}:  
\begin{equation*}
    \picond_1(\AA_{\CC}^1,\xbar) \neq 1 \period
\end{equation*}  
While this departs from the classical situation, we show that the \emph{Noohi completion} of $\picond_1(X,\xbar)$ recovers the proétale fundamental group of Bhatt--Scholze; see \cref{thm:recovering_BS_fundamental_group}.  
In fact, we prove that already the \emph{quasiseparated quotient} \smash{$\pionecondqs(X,\xbar)$}, a milder completion similar to the Hausdorff quotient of topological groups, behaves computationally as expected.
Also, surprisingly, in many situations the quasiseparated quotient turns out to be a topological group.
See \cref{intro_thm:normal-implies-profinite}, the van Kampen formula (\cref{intro_thm:vanKampen-for-cond-qs}), and the Künneth formula (\cref{intro_thm:Kunneth_formula_for_quasiseparated_fundamental_groups}).
Studying \smash{$\pionecondqs$} is another major theme of the second part of this article.

%-------------------------------------------------------------------%
%  Results about the condensed homotopy type                        %
%-------------------------------------------------------------------%

\subsection{Results about the condensed homotopy type}\label{intro_subsec:general_results}

We now turn to explaining the results that we prove in the first part of this paper in detail.
The first is a condensed version of the `fundamental exact sequence' for the étale fundamental group.

\begin{theorem}[(fundamental fiber sequence, \Cref{cor:proetale_fundamental_fiber_sequence})]\label{intro_thm:proetale_fundamental_fiber_sequence}
	Let $ f \colon \fromto{X}{S} $ be a morphism between qcqs schemes, and let $ \fromto{\sbar}{S} $ be a geometric point of $ S $.
	If $ \dim(S) = 0 $, then the naturally null sequence 
	\begin{equation*}
		\begin{tikzcd}[sep=1.5em]
			\Picond(\Xsbar) \arrow[r] & \Picond(X) \arrow[r] & \Picond(S) 
		\end{tikzcd}
	\end{equation*}
	is a fiber sequence in the \category $ \CondAni $.
\end{theorem}

Second, using a profinite version of Quillen's Theorem B, we prove the following analogue of a result of Friedlander \cite[Theorem~3.7]{MR352099}.

\begin{theorem}[(\Cref{thm:smooth_fiber_sequ})]\label{intro_thm:smooth_fiber_sequ}
    Let $ f \colon X \to S $ be a smooth and proper morphism between qcqs schemes and let $\sbar \to S $ be a geometric point.
    Let $ \Sigma $ be a nonempty set of primes invertible on $ S $.
    Then the induced map
    \begin{equation*}
        \Picond(X_{\sbar}) \to \fib_{\sbar}(\Picond(X) \to \Picond(S))
    \end{equation*}
    becomes an equivalence after completion at $ \Sigma $.
\end{theorem}

Third, we show that the hypercomplete proétale \topos and the condensed homotopy type have descent along hypercovers by integral surjections:

\begin{theorem}[(integral hyperdescent, \Cref{cor:integral-hyperdescent-main-result})]\label{intro_thm:integral-descent-proetale-topos}
    The functor \smash{$ X \mapsto \Xproethyp $} sending a qcqs scheme $ X $ to its hypercomplete proétale \topos satisfies integral hyperdescent.
    As a consequence, if $X_{\bullet} \twoheadrightarrow X$ is an integral hypercover, then the natural map of condensed anima
    \begin{equation*}
        \textstyle \colim_{\Deltaop} \Picond(X_{\bullet}) \to \Picond(X)
    \end{equation*}
    is an equivalence.
\end{theorem}

The description of $\Picond(X)$ via exodromy is a crucial ingredient in our proof of \Cref{intro_thm:integral-descent-proetale-topos}; it follows rather quickly from the fact that, for an integral morphism of schemes $f \colon X \to Y$, the functor
$\Gal(f)$ is a right fibration of condensed \categories.
See \cref{lem:integral-morphism-is-right-fibration-underlying}.

Finally, we give a complete computation of the condensed and étale homotopy types of rings of continuous functions to the complex numbers:

\begin{theorem}[(\Cref{cor:condensed_homotopy_type_of_rings_of_continuous_functions})]\label{intro_thm:etale_homotopy_type_of_rings_of_continuous_functions}
    Let $ T $ be a compact Hausdorff space and consider the ring $ \upC(T,\CC) $ of continuous functions to the complex numbers.
    Then there is a natural equivalence of condensed anima
    \begin{equation*}
        \Picond(\Spec(\upC(T,\CC))) \equivalent T \period
    \end{equation*}
    (Here, the right-hand side denotes the condensed set represented by $ T $.)
   
    As a consequence, up to protruncation, the étale homotopy type of $ \Spec(\upC(T,\CC)) $ is equivalent to the shape of the topological space $ T $.
    In particular, if $ T $ admits a CW structure, then, up to protruncation, the étale homotopy type of $ \Spec(\upC(T,\CC)) $ recovers the underlying anima of $ T $.
\end{theorem}

\begin{remark}
    The computation of the protruncated étale homotopy type of rings of continuous functions seems new.
    We also do not know of a direct computation that does not pass through the condensed homotopy type.
\end{remark}

%-------------------------------------------------------------------%
%  Results about the condensed fundamental group                    %
%-------------------------------------------------------------------%

\subsection{Results about the condensed fundamental group}\label{intro_subsec:condensed_fundamental_group}

We now turn to our results about the condensed fundamental group.
But first, let us remark that we also obtain a reasonably explicit description of the condensed set of connected components of $ \Picond(X) $.

\begin{theorem}[(\Cref{thm:description_of_pi_0,cor:pi0s_match_for_finitely_many_irr_comps})]\label{intro_thm:description_of_pi_0}
	Let $ X $ be a qcqs scheme.
	Then, for any extremally disconnected profinite set $ S $, we have
	\begin{equation*}
		\picond_{0}(X)(S) = \Map_{\qc}(S,\lvert X \rvert)/\kern-0.2em\sim \comma
	\end{equation*}
    where $\sim$ is the equivalence relation generated by pointwise specializations.

    In particular, if $ X $ has finitely many irreducible components, then $ \picond_{0}(X) $ coincides with the usual profinite set $ \uppi_{0}(X) $ of connected components of $ X $.
\end{theorem}

\begin{remark}[(see \Cref{ex:adic_disk})]
    Let $ R $ be a ring with the property that $ |\Spec(R)| $ is homeomorphic to the underlying spectral space of Huber's adic unit disk over $ \QQ_p $.
    Then the condensed set \smash{$\pizerocond(\Spec(R)) $} coincides with the \textit{separated quotient} of the space $ |\Spec(R)| $. 
    This is a compact Hausdorff space, and moreover, it coincides with the Berkovich unit disk, i.e.,
    \begin{equation*}
        \pizerocond(\Spec(R)) \simeq |\DD^{1,\mathrm{Berk}}_{\QQ_p}| \period
    \end{equation*}
    While this example feels rather contrived in the realm of schemes, in a follow-up article we plan to study a similarly defined condensed homotopy type for rigid spaces.
\end{remark}

We now turn to our results about the condensed fundamental group.
As stated earlier, the condensed fundamental group of $ \AA_{\CC}^1 $ is nontrivial:

\begin{theorem}[(\Cref{cor:pionecond-of-A1-is-nontrivial})]\label{intro_thm:pionecond-of-A1-is-nontrivial}
    Let $ \xbar \to \AA^1_{\CC} $ be a geometric point.
    Then the abelianization of the underlying group \smash{$\picond_{1}(\AA_{\CC}^1,\xbar)(\ast)$} is nontrivial.
    As a consequence, $ \picond_{1}(\AA_{\CC}^1,\xbar) \neq 1 $.
\end{theorem}

\noindent One way to remedy this lies in the relationship between the condensed and proétale fundamental groups.
The proétale fundamental group has the property that it is a \textit{Noohi group} in the sense of \cite[\S7.1]{MR3379634}.
A consequence of \Cref{intro_thm:pionecond-of-A1-is-nontrivial} is that the condensed fundamental group is not generally a Noohi group.
The process of Noohi completion $ G \mapsto G^{\Noohi} $ extends from topological groups to condensed groups, and we prove:

\begin{theorem}[(\Cref{thm:recovering_BS_fundamental_group})]\label{intro_thm:recovering_BS_fundamental_group}
    Let $ X $ be be a qcqs scheme with finitely many irreducible components and $\xbar \to X$ a geometric point.
    Then there is a natural isomorphism
    \begin{equation*}
        \pionecond(X,\xbar)^{\Noohi} \isomorphism \pioneproet(X,\xbar) \period
    \end{equation*}
\end{theorem}

In the case of $ \AA_{\CC}^1 $, we prove that an operation much milder than Noohi completion forces \smash{$ \picond_1(\AA_{\CC}^1) $} to become trivial.
Specifically, Clausen and Scholze introduced a localization $ \goesto{A}{A^{\qs}} $ of the category of condensed sets called the \textit{quasiseparated quotient} \cite[Lecture VI]{Scholze:analyticnotes}, and we show:

\begin{theorem}[(\Cref{thm:normal-implies-profinite})]\label{intro_thm:normal-implies-profinite}
    Let $ X $ be a qcqs geometrically unibranch scheme with finitely many irreducible components, and let $ \xbar \to X $ be a geometric point.
    Then there is a natural isomorphism 
    \begin{equation*}
        \pionecondqs(X,\xbar) \isomorphism \pioneet(X,\xbar) \period
    \end{equation*}
\end{theorem}

As a consequence of \Cref{intro_thm:proetale_fundamental_fiber_sequence,intro_thm:description_of_pi_0}, we deduce a fundamental exact sequence for the quasiseparated quotient of the condensed fundamental group:

\begin{theorem}[(fundamental exact sequence, \Cref{cor:fundamental-fiber-sequence-on-qs-quotients})]\label{intro_thm:fundamental-fiber-sequence-on-qs-quotients}
    Let $ k $ be a field with separable closure $ \kbar $, let $ X $ be a qcqs $ k $-scheme, and fix a geometric point $ \xbar \to X_{\kbar} $.
    If $ X $ is geometrically connected and $ X_{\kbar} $ has finitely many irreducible components, then the sequence
    \begin{equation*}
        \begin{tikzcd}[cramped, sep=small]
            1 \arrow[r] & \pionecondqs(X_{\kbar},\xbar) \arrow[r] & \pionecondqs(X,\xbar) \arrow[r] & \Gal_{k} \arrow[r] & 1
        \end{tikzcd}
    \end{equation*}
    is exact.
\end{theorem}

\Cref{intro_thm:normal-implies-profinite} can be used, together with integral descent (\Cref{intro_thm:integral-descent-proetale-topos}), to show that for many non-normal schemes, the quasiseparated quotient of the condensed fundamental group still admits a description in terms of the étale fundamental group.
Moreover, surprisingly, it is a (Hausdorff) topological group rather than some more complicated condensed group.

\begin{theorem}[{(van Kampen formula for $ \pionecondqs $, special case of \Cref{cor:vanKampen-for-cond-qs})}]\label{intro_thm:vanKampen-for-cond-qs}
    Let $ X $ be a Nagata qcqs scheme and let $ X^{\nu} = \coprod_i X^{\nu}_i $ be the decomposition of its normalization into connected components. 
    After choosing base points and étale paths, one has that
    \begin{equation*}
        \pionecondqs(X,\xbar) \simeq \underline{\big(\freeprodtop_i\piet_1(X^{\nu}_i,\xbar_i)\freeprodtop\ZZ^{*r}\big)/H'} \period
    \end{equation*}
    Here, $ \ZZ^{*r} $ is a free (discrete) group of finite rank, $ \freeprodtop $ denotes the free topological product and $H'$ is an explicit closed normal subgroup.
\end{theorem}

Using the van Kampen and the Künneth formulas for the étale fundamental group, we prove:

\begin{theorem}[{(Künneth formula for $ \pionecondqs $, \Cref{cor:Kunneth_formula_for_quasiseparated_fundamental_groups})}]\label{intro_thm:Kunneth_formula_for_quasiseparated_fundamental_groups}
    Let $ k $ be a separably closed field and let $ X $ and $ Y $ be schemes of finite type over $ k $.
    If $ Y $ is proper or $ \characteristic(k) = 0 $, then the natural homomorphism of condensed groups
    \begin{equation*}
        \pionecondqs(X\times_{k} Y,(\xbar,\ybar)) \to \pionecondqs(X,\xbar) \times \pionecondqs(Y,\ybar)
    \end{equation*}
    is an isomorphism.
\end{theorem}

\noindent In some ways, the group $ \pionecondqs $ is even better-behaved than $\pioneproet$ (see, e.g., \cref{rem:kurosh}). 

%-------------------------------------------------------------------%
%  Related work                                                     %
%-------------------------------------------------------------------%

\subsection{Related work}\label{intro_subsec:related_work}

As mentioned earlier, the first definitions of the condensed homotopy type were given via exodromy by Barwick--Glasman--Haine \cite[13.8.10]{Exodromy}, by Bhatt--Scholze \cite[Remark 4.2.9]{BhattScholzeProetale} and by Hemo--Richarz--Scholbach \cite[Appendix A]{hemo2023constructible_sheaves_schemes}.
Another approach to the condensed homotopy type that mostly uses (simplicial) topological spaces rather than condensed mathematics (along the lines of Artin and Mazur's work) was studied by Meffle \cite{Meffle:proetalehomotopytype}. 

Some results and definitions in this article constitute a part of doctoral theses of the forth \cite{CatrinsThesis} and sixth \cite{Sebastian_Wolf-thesis} named authors.

%-------------------------------------------------------------------%
%  Linear overview                                                  %
%-------------------------------------------------------------------%

\subsection{Linear overview}\label{intro_subsec:linear_overview}

In \cref{sec:preliminaries}, we recall some preliminaries on condensed anima, pro-objects, condensed \categories, and proétale sheaves.

\Cref{part:the_condensed_homotopy_type} is dedicated to proving fundamental results about the condensed homotopy type.
In \cref{sec:condensed_homotopy_type}, we give three definitions of the condensed homotopy type, and prove that they are equivalent. 
We also compute the condensed homotopy type of henselian local rings (\Cref{cor:Picond_of_henselian_local_rings}).
In \cref{sec:connected_componenets_of_the_condensed_homotopy_type}, we prove \Cref{intro_thm:description_of_pi_0}, giving an explicit description of the connected components of the condensed homotopy type.
As an application of this explicit description, we also we compute the condensed homotopy type of rings of continuous functions (\Cref{intro_thm:etale_homotopy_type_of_rings_of_continuous_functions}).

\Cref{sec:fiber_sequences} is dedicated to producing fiber sequences for the condensed homotopy type.
Specifically, we prove the fundamental fiber sequence (\Cref{intro_thm:proetale_fundamental_fiber_sequence}) as well as an analogue of a result of Friedlander relating the condensed homotopy type of the geometric fiber of a smooth proper morphism to the fiber of the induced map on condensed homotopy types (\Cref{intro_thm:smooth_fiber_sequ}).
In \cref{sec:integral-descent}, we prove that the condensed homotopy type satisfies integral hyperdescent (\Cref{intro_thm:integral-descent-proetale-topos}).

In \Cref{part:the_condensed_fundamental_group}, we turn our attention to the condensed fundamental group.
In \cref{sec:quasiseparated_quotient_of_the_condensed_fundamental_group}, we start by showing that \smash{$ \picond_1(\AA_{\CC}^1) $} is nontrivial (\Cref{intro_thm:pionecond-of-A1-is-nontrivial}).
We then study the quasiseparated quotient of the condensed fundamental group.
In particular, we prove \Cref{intro_thm:normal-implies-profinite,intro_thm:fundamental-fiber-sequence-on-qs-quotients,intro_thm:vanKampen-for-cond-qs,intro_thm:Kunneth_formula_for_quasiseparated_fundamental_groups}.
In \cref{sec:Noohi_completion_of_the_condensed_fundamental_group}, we prove that the Noohi completion of the condensed fundamental group recovers the proétale fundamental group (\Cref{intro_thm:recovering_BS_fundamental_group}).

We have three appendices. 
\Cref{appendix:rings_of_continuous_functions_and_Cech-Stone_compactification}, by Bogdan Zavyalov, is on the structure of rings of continuous functions and the relationship between these rings and Čech--Stone compactification.
We need these results for the computation of the condensed homotopy type of rings of continuous functions, however were not able to find any sources that contained all of the results we needed.

In \Cref{appendix:a_profinite_analogue_of_Quillens_theorem_B}, we prove an analogue of Quillen's Theorem B for profinite completions of classifying anima of condensed \categories.
Together with the description of the condensed homotopy type via exodromy, this is the key tool we use to prove \Cref{intro_thm:smooth_fiber_sequ}.

It is well-known that there is an isomorphism between the absolute Galois group of the function field $ \CC(t) $ and the free profinite group on the set $ \CC $.
See, for example \cites{MR162796}{MR1800587}{MR1352283}.
It seems to be folklore that this isomorphism can be chosen to be compatible with decomposition groups; this is crucial for our proof that $ \picond_1(\AA_{\CC}^1) \neq 1 $.
Since we could not find this proven in the literature, and there are some subtleties involved, we have included a proof in \Cref{appendix:Galois_groups_of_function_fields}. 

%-------------------------------------------------------------------%
%  Conventions                                                      %
%-------------------------------------------------------------------%

\subsection{Conventions}\label{intro_subsec:notational_conventions}

\subsubsection*{Set theory}

As usual when working with condensed mathematics, there are some set-theoretic issues one needs to deal with.
We give detailed explanations on how we handle these in \cref{rem:first_set_theory_remark,rem:set_theory_continued,rem:set_theory_doesnt_matter}.

\subsubsection*{Notational conventions}

We use the following standard notation.
\begin{enumerate}
	\item We write $ \Catinfty $ for the large \category of small \categories, and write $ \Ani \subset \Catinfty $ for the full subcategory spanned by the anima (also called \groupoids or spaces).
    
	\item Given a small \category $ \Ccal $, we write $ \PSh(\Ccal) \colonequals \Fun(\Ccal^{\op},\Ani) $ for the \category of presheaves of anima on $ \Ccal $.

    \item Given \atopos $ \Xcal $, we write $ \Xcal^{\hyp} \subset \Xcal $ for the full subcategory spanned by the hypercomplete objects.
    The inclusion is accessible and admits a left exact accessible left adjoint, so that $ \Xcal^{\hyp} $ is also \atopos, called the \textit{hypercompletion} of $ \Xcal $.

    \item Given \asite $ (\Ccal,\tau) $, we write $ \Sh_{\tau}(\Ccal) $ for the \topos of sheaves of anima on $ \Ccal $ with respect to $ \tau $.
    We write $ \Shhyp_{\tau}(\Ccal) \colonequals \Sh_{\tau}(\Ccal)^{\hyp} $.
    The \topos $ \Shhyp_{\tau}(\Ccal) $ can also be identified as the full subcategory of $ \Sh_{\tau}(\Ccal) $ spanned by those sheaves that also satisfy descent for \textit{hypercovers}.
    If the topology $ \tau $ is clear from the context, we may omit it from the notation.
    
    \item Given a scheme $ X $, we write $\Et{X}$ and $\ProEt{X}$ for its \emph{étale} and \emph{proétale site}, respectively.
    Moreover, we write $X_{\et} \colonequals \Sh(\Et{X})$ and $X_{\proet} \colonequals \Sh(\ProEt{X})$ for the \topoi of étale and proétale sheaves of anima on $ X $, respectively. 

    \item For an integer $ n \geq 0 $, we write $ [n] $ for the poset $ \{0 < \cdots < n \} $.

    \item For each integer $ n \geq 0 $, we write $ \DDelta_{\leq n} \subset \DDelta $ for the full subcategory spanned by $ [0] $, \ldots, $ [n] $.
\end{enumerate}

%-------------------------------------------------------------------%
%  Notational conventions                                           %
%-------------------------------------------------------------------%

\subsection{Acknowledgments}\label{intro_subsec:acknowledgments}

First and foremost the authors want to thank Clark Barwick.
Many of the results and ideas in this article were suggested to us or at least inspired by Clark.
He also collaborated on an early stage of this project, and this paper owes him a huge mathematical debt.
We also want to thank Peter Scholze for useful remarks about \smash{$ \pizerocond $} and drawing our attention to the quasiseparated quotient of \smash{$ \pionecond $} as a possibly better-behaved invariant.
We also want to thank Piotr Achinger for asking us about the condensed homotopy type of rings of continuous functions.
We furthermore want to thank Bhargav Bhatt, Denis-Charles Cisinski, Remy van Dobben de Bruyn, Timo Richarz, and Jakob Stix for helpful discussions.

PH gratefully acknowledges support from the NSF Mathematical Sciences Postdoctoral Research Fellowship under Grant \#DMS-2102957.
LM was partially supported by the project Pure Mathematics in Norway, funded by Trond Mohn Foundation and Tromsø Research Foundation.
SW gratefully acknowledges support from the SFB 1085 Higher Invariants in Regensburg, funded by the DFG. 
TH, ML, and CM gratefully acknowledge support by Deutsche For\-schungs\-ge\-mein\-schaft (DFG, German Research Foundation) through the Collaborative Research Centre TRR 326 Geometry and Arithmetic of Uniformized Structures, project number 444845124. 
ML was later supported by the National Science Centre, Poland, grant number 2023/51/D/ST1/02294. 
The last two funding sources have also funded three research stays for our group: in Frankfurt, Kraków, and Sopot. 
We thank the Goethe University and IMPAN for their hospitality. 
For the purpose of Open Access, the authors have applied a CC-BY public copyright license to any Author Accepted Manuscript (AAM) version arising from this submission.

%-------------------------------------------------------------------%
%-------------------------------------------------------------------%
%  Preliminaries                                                    %
%-------------------------------------------------------------------%
%-------------------------------------------------------------------%
  
\section{Preliminaries}\label{sec:preliminaries}

For later use and the convenience of the reader, in this section we record a few definitions and observations on condensed anima (\cref{subsec:recollection_on_condensed_anima}), pro-anima and their relation to condensed anima (\cref{subsec:pro-objects_and_completions}), condensed \categories (\cref{subsec:condensed_categories}), shape theory (\cref{subsec:recollection_on_shape_theory}), and proétale sheaves and \wcontractible objects (\cref{subsec:recollection_on_proetale_sheaves}).

%-------------------------------------------------------------------%
%  Recollection on condensed anima                                  %
%-------------------------------------------------------------------%

\subsection{Recollection on condensed anima}\label{subsec:recollection_on_condensed_anima}

All of the material contained in this subsection is gathered from \cite{pyknoticI} and \cite{Scholze:condensednotes}.

\begin{notation}
    We write $ \Top $ for the category of topological spaces, and $ \Comp \subset \Top $ for the full subcategory spanned by the compact Hausdorff spaces.
    We write $ \upbeta \colon \fromto{\Top}{\Comp} $ for the Čech--Stone compactification functor, i.e., the left adjoint to the inclusion.
    By Stone duality, the category $ \ProSetfin $ of profinite sets embeds fully faithfully into $ \Comp $ with image the full subcategory spanned by the totally disconnected compact Hausdorff spaces.
    We write
    \begin{equation*}
        \Extr \subset \ProSetfin
    \end{equation*}
    for the full subcategory spanned by the \defn{extremally disconnected} profinite sets.
    By a theorem of Gleason \cite{MR0121775}, the projective objects of the category $ \Comp $ are exactly the extremally disconnected profinite sets.
    Moreover, a profinite set is extremally disconnected if and only if it is a retract of the Čech--Stone compactification of a set equipped with the discrete topology.
\end{notation}

\begin{recollection}[(condensed anima)]\label{rec:sifted_colimits_and_homotopy_of_condensed_anima}
    Give the category $ \Comp $ of compact Hausdorff spaces the Grothendieck topology where the covering families are generated by finite jointly surjective families.
    For each compact Hausdorff space $ T $, let $ T^{\updelta} $ denote the underlying set of $ T $ equipped with the discrete topology. 
    By the universal property of Čech--Stone compactification the `identity' map $ \fromto{T^{\updelta}}{T} $ extends to a surjection $ \surjto{\upbeta(T^{\updelta})}{T} $.
    In particular, every compact Hausdorff space admits a surjection from an extremally disconnected profinite set.
    Hence the subcategories
    \begin{equation*}
        \Extr \subset \ProSetfin \subset \Comp 
    \end{equation*}
    are bases for the topology of finite jointly surjective families.
    By \cite[Corollary A.7]{arXiv:2001.00319}, the restriction functors define equivalences of hypercomplete \topoi
    \begin{equation}\label{eq:equivalent_descriptions_of_condensed_anima}
        \Shhyp(\Comp) \equivalence \Shhyp(\ProSetfin) \equivalence \Shhyp(\Extr) \period
    \end{equation}
    The \topos $\CondAni$ of \textit{condensed anima} is any of the equivalent \topoi \eqref{eq:equivalent_descriptions_of_condensed_anima}.
    
    Since every surjection $ T' \twoheadrightarrow T $ of profinite sets with $ T $ extremally disconnected admits a section, a presheaf $ F $ on $ \Extr $ is a hypersheaf if and only if $ F $ carries finite disjoint unions to finite products.
    That is, 
    \begin{equation*}
        \Shhyp(\Extr) \equivalent \Funcross(\Extrop,\Ani) \period
    \end{equation*}
    From this description it follows that sifted colimits in $\CondAni$ can be computed in the presheaf category $ \Fun(\Extr^{\op},\Ani) $.
\end{recollection}

\begin{remark}
     \label{rem:first_set_theory_remark}
     Since the category $ \Comp $ of compact Hausdorff spaces is not a small category, there are some set-theoretic issues in the above discussion.
     We explain how to deal with these issues in \Cref{rem:set_theory_continued}.
\end{remark}

Given the final description of condensed anima, we make the following convenient general definition.

\begin{definition}[(condensed objects)]
    Let $ \Ccal $ be \acategory with finite products.
    The \category of \defn{condensed objects} of $ \Ccal $ is the \category
    \begin{equation*}
        \Cond(\Ccal) \colonequals \Funcross(\Extrop,\Ccal) 
    \end{equation*}
    of finite product-preserving presheaves $ \fromto{\Extrop}{\Ccal} $.
    If $ \Dcal $ is another \category with finite products and $ F \colon \fromto{\Ccal}{\Dcal} $ is a finite product-preserving functor, we write
    \begin{equation*}
        F^{\cond} \colon \fromto{\Cond(\Ccal)}{\Cond(\Dcal)}
    \end{equation*}
    for the functor given by post-composition with $ F $.
\end{definition}

\begin{nul}
    Observe that if $ F \colon \fromto{\Ccal}{\Dcal} $ admits a right adjoint $ G $, then $ G^{\cond} $ is right adjoint to $ F^{\cond} $.
\end{nul}

\begin{recollection}[(homotopy groups of condensed anima)]\label{rec:homotopy_groups_of_condensed_anima}
    The functor $ \uppi_0 \colon \fromto{\Ani}{\Set} $ preserves finite products.
    Moreover, for each integer $ n \geq 1 $, the functor $ \uppi_n \colon \fromto{\Ani_{*}}{\Grp} $ preserves finite products.
    There is a canonical identification
    \begin{equation*}
        \CondAni_{*} = \Cond(\Ani_{*})
    \end{equation*}
    between pointed objects of condensed anima and condensed objects of pointed anima.
    We simply write $ \uppi_0 \colon \fromto{\CondAni}{\CondSet} $ for $ \picond_0 $ and $ \uppi_n \colon \fromto{\CondAni_{*}}{\CondGrp} $ for 
    \begin{equation*}
        \begin{tikzcd}[sep=3em]
            \CondAni_{*} = \Cond(\Ani_{*}) \arrow[r, "\picond_n"] & \CondGrp \period 
        \end{tikzcd}
    \end{equation*}
    Explicitly, given a condensed anima $ A $, the condensed set $ \uppi_0(A) \colon \fromto{\Extrop}{\Set} $ is given by 
    \begin{equation*}
        \uppi_0(A)(S) \colonequals \uppi_0(A(S)) \period 
    \end{equation*}
    Similarly, given a global section $ a \colon \fromto{\ast}{A} $, the condensed group $ \uppi_n(A,a) $ is given by
    \begin{equation*}
        \uppi_n(A, a)(S) \colonequals \uppi_n(A(S),a) \period
    \end{equation*}
\end{recollection}

\begin{recollection}[{\cite[Construction~2.2.12]{pyknoticI}}]
    Write
    \begin{equation*}
        \ev_\ast \colon \CondAni \to \Ani
    \end{equation*}
    for the global sections functor, given by $ A \mapsto A(\pt) $.
    The functor $ \ev_{\pt} $ admits a left adjoint, that we denote by
    \begin{equation*}
        (-)^{\disc} \colon \Ani \to \CondAni
    \end{equation*}
    Furthermore $(-)^{\disc}$ is fully faithful.
    We call the image of $(-)^{\disc}$ the \emph{discrete} condensed anima.
\end{recollection}

\begin{recollection}[(the restricted Yoneda embedding)]\label{rec:the_underline_functor}
   The restricted Yoneda embedding defines a functor
   \begin{align*}
        \Top &\to \CondAni \comma \quad \goesto{T}{\underline{T}} \\ 
   \shortintertext{given by}
        T &\mapsto [S \mapsto \Map_{\Top}(S,T)] \period
   \end{align*}
   Note that this functor factors through $ \CondSet \subset \CondAni $.%
   \footnote{However, note that if $ T $ is not $\Tup_1$, then the the sheaf $\Map_{\Top}(-,T)$ is not generally \emph{accessible} \cite[Warning 2.14 \& Proposition 2.15]{Scholze:condensednotes}.
   So, depending on which way you deal with set-theoretic issues, it is not a condensed set, cf. \cref{rem:set_theory_continued}.
   However, in this paper, we only apply this functor to $\Tup_1$ topological spaces anyways.}
   Also recall that this functor is fully faithful when restricted to the full subcategory of $ \Top $ spanned by the compactly generated topological spaces \cite[Proposition 1.7]{Scholze:condensednotes}.
   Since it rarely leads to confusion, we often omit the underline and simply write $ T $ for $ \underline{T} $.
\end{recollection}

%-------------------------------------------------------------------%
%  Pro-objects and completions                                      %
%-------------------------------------------------------------------%

\subsection{Pro-objects and completions}\label{subsec:pro-objects_and_completions}

We now turn to some recollections about proanima and their relation to condensed anima.

\begin{recollection}[(\pifinite and truncated anima)]
  Let $ A $ be an anima.
  \begin{enumerate}
      \item We say that $ A $ is \defn{truncated} if there exists an integer $ n \geq 0 $ such that for all $ a \in A $ and $ k \geq n $, we have $ \uppi_k(A, a) = 0 $.

      \item We say that $ A $ is \defn{\pifinite} if $ A $ is truncated, $ \uppi_0(A) $ is finite, and for all $ a \in A $ and $ k > 0 $, the group $ \uppi_k(A,a) $ is finite.

      \item We write $\Anifin \subset \Anitrun \subset \Ani$ for the full subcategories of $\Ani$ spanned by the \pifinite and truncated anima, respectively.
  \end{enumerate}
\end{recollection}

\begin{recollection}[(on various completions)]\label{rec:various-completions}
    \hfill
    \begin{enumerate}
        \item Since $\CondAni$ admits cofiltered limits, the inclusions 
        \begin{equation*}
            \Anifin \subset \Anitrun \subset \CondAni
        \end{equation*}
        extend to cofiltered-limit-preserving functors 
        \begin{equation*}
          \ProAnifin \inclusion \ProAnitrun \to \CondAni \period
        \end{equation*}
        Here, the functor $\ProAnitrun \to \CondAni$ is \emph{not} fully faithful.
        However, by \cites[Example~3.3.10]{pyknoticI}[Proposition 0.1]{Haine:closure_properties_of_classes_of_maps}, its restriction to $\ProAnifin$ is fully faithful.

        \item The above chain of functors $\ProAnifin \inclusion \ProAnitrun \to \CondAni$ admits left adjoints 
        \begin{equation*}
          \begin{tikzcd}
            \CondAni \arrow[r, "{(-)\prodisccompl}"'] \arrow[rr, bend left = 2.5em, "{(-)\profincomp}"] & \ProAnitrun \arrow[r, "{(-)\profincomp}"'] & \ProAnifin
          \end{tikzcd}
        \end{equation*}
        that we call the \emph{prodiscretization}, resp., \emph{profinite completion} functors.

        \item Similarly, the inclusions $\FinSet \subset \CondSet$ and $\FinGrp \subset \CondGrp$ induce inclusions $\ProFin \subset \CondSet$ and $\ProFinGrp \subset \CondGrp$ that admit left adjoints
        \begin{equation*}
           \CondSet \to \ProFin \andeq (-)^{\wedge} \from \CondGrp \to \ProFinGrp
        \end{equation*}
        that we refer to as \emph{profinite completion} functors. 
    \end{enumerate}
\end{recollection}

We now explain the effect of profintie completion of condensed anima on $ \uppi_0 $ and $ \uppi_1 $.

\begin{lemma}[{(completions \& $\uppi_0/\uppi_1$)}]\label{lem:pi1-of-completion}
    Let $ A $ be a condensed anima and $a \colon \ast \to A$ a point.
    \begin{enumerate}
        \item\label{item:pi0-of-completion} The map $\uppi_0(A) \to \uppi_0(A\profincomp)$ induced by the unit map $ A \to A\profincomp $ exhibits $\uppi_0(A\profincomp)$ as the profinite completion of $\uppi_0(A)$.
        
        \item\label{item:pi1-of-completion} If $ \uppi_0(A) \in \CondSet $ is discrete, then the unit map $ A \to A\profincomp $ induces an isomorphism of profinite groups
        \begin{equation*}
          \uppi_1(A, a)^{\wedge} \isomorphism \uppi_1(A\profincomp, a) \period
        \end{equation*}
    \end{enumerate}
\end{lemma}

\begin{proof}
    For \eqref{item:pi0-of-completion}, note that since the square of inclusions 
    \begin{equation*}
        \begin{tikzcd}
            \ProSetfin \arrow[r, hooked] \arrow[d, hooked] & \CondSet \arrow[d, hooked] \\
            \ProAnifin \arrow[r, hooked] & \CondAni 
        \end{tikzcd}
    \end{equation*}
    commutes, so does the induced square 
    \begin{equation*}
        \begin{tikzcd}
          \CondAni \arrow[r, "{(-)\profincomp}"] \arrow[d, "{\uppi_0}"'] & \ProAnifin \arrow[d, "{\uppi_0}"] \\
          \CondSet \arrow[r] & \ProFin 
        \end{tikzcd}
    \end{equation*}
    of left adjoints.

    For \eqref{item:pi1-of-completion}, since $\uppi_0(A)$ is a set, we may assume that $\uppi_0(A) = \ast$.
    It suffices to show that, for any finite group $ G $, precomposition induces a bijection
    \begin{equation*}
        \Map_{\CondGrp}(\uppi_1(A, a), G) \isomorphism \Map_{\CondGrp}(\uppi_1(A\profincomp, a), G) = \Map_{\ProFinGrp}(\uppi_1(A\profincomp, a), G) \period
    \end{equation*}
    To see this, note that we have a commutative square
    \begin{equation*}
        \begin{tikzcd}[sep=2.5em]
           \uppi_0 \Map_{\ProAnifin_\ast}(A\profincomp, \Bup G) \arrow[r, "{\uppi_1}", "{\sim}"'] \arrow[d] & \Map_{\ProFinGrp}(\uppi_1(A\profincomp, a), G) \arrow[d] \\
           \uppi_0 \Map_{\CondAni_\ast}(A, \Bup G) \arrow[r, "{\uppi_1}"', "{\sim}"] & \Map_{\CondGrp}(\uppi_1(A, a), G),
        \end{tikzcd}
    \end{equation*}
    where  the vertical maps are those induced by the unit transformation $A \to A\profincomp$.
    Since $\uppi_0(A) = \ast$, by the equivalence of $1$-truncated, pointed connected objects and group objects \HTT{Theorem}{7.2.2.12}, the horizontal maps are bijections.
    It thus suffices to see that the map
    \begin{equation*}
        \Map_{\CondAni_\ast}(A\profincomp, \Bup G) \to \Map_{\CondAni_\ast}(A, \Bup G)
    \end{equation*}
    induces a bijection on $\uppi_0$. 
    But since $ G $ is finite and $\ProAnifin_\ast \inclusion \CondAni_\ast$ is fully faithful, by adjunction it is even an equivalence.
\end{proof}

\begin{remark}
    One cannot drop the assumption that $\uppi_0(A)$ is discrete in \Cref{lem:pi1-of-completion} \eqref{item:pi1-of-completion}.
    Indeed, let $ A $ be the condensed \emph{set} represented by the topological circle $\Sup^1$.
    Then for any $x \in \Sup^1$, we have 
    \begin{equation*}
      \uppi_1(A,x) = \ast \quad \textup{ but } \quad \uppi_1(A\profincomp,x) = \ZZhat \period
    \end{equation*}
\end{remark}

%-------------------------------------------------------------------%
%  Condensed ∞-categories                                           %
%-------------------------------------------------------------------%

\subsection{Condensed \texorpdfstring{$\infty$}{∞}-categories}\label{subsec:condensed_categories}

We now recall some background on internal higher category theory and condensed \categories.
The main point is that it is often useful to use the fact that the \category of condensed \categories is equivalent to the \category of categories internal to condensed anima. 
We refer the reader to \cites[\S3]{arXiv:2103.17141}[\S2]{MR4752519} for more background about internal higher category theory.

\begin{definition}
    Let $ \Bcal $ be \acategory with finite limits.
    A \defn{category internal to $ \Bcal $} is a simplicial object $ F \colon \fromto{\Deltaop}{\Bcal} $ satisfying the following conditions.
    \begin{enumerate}
        \item \emph{Segal condition:} For each integer $ n \geq 2 $, the natural map
        \begin{equation*}
            F([n]) \to F(\{0<1\}) \crosslimits_{F(\{1\})}  F(\{1<2\}) \crosslimits_{F(\{2\})} \cdots \crosslimits_{F(\{n-1\})}  F(\{n-1<n\}) 
        \end{equation*}
        is an equivalence in $ \Bcal $.

        \item \emph{Univalence axiom:} The natural square
        \begin{equation*}
            \begin{tikzcd}
                F([0]) \arrow[r, "\Delta"] \arrow[d] & F([0]) \cross F([0]) \arrow[d] \\ 
                F([3]) \arrow[r] & F(\{0<2\}) \cross F(\{1<3\})
            \end{tikzcd}
        \end{equation*}
        is a pullback square in $ \Bcal $.
        Here, the left vertical map is given by restriction along the unique map $ [3] \to [0] $, the right vertical map is the product of the maps given by restriction along the unique maps $ \{0<2\} \to [0] $ and $ \{1<3\} \to [0] $, and the bottom horizontal map is induced by restriction along the inclusions $ \{0<2\} \inclusion [3] $ and $ \{1<3\} \inclusion [3] $.
    \end{enumerate}
    We write 
    \begin{equation*}
        \ICat(\Bcal) \subset \Fun(\Deltaop,\Bcal) 
    \end{equation*}
    for the full subcategory spanned by the categories internal to $ \Bcal $.
\end{definition}

\begin{remark}
    Elsewhere in the literature, internal categories are also called \textit{complete Segal objects}.
\end{remark}

\begin{nul}\label{nul:Joyal-Tiereny}
	Joyal and Tierney \cite{MR2342834} showed that the nerve construction defines an equivalence 
	\begin{align*}
		\Nerve \colon \Catinfty &\equivalence \ICat(\Ani) \\ 
		C &\mapsto [[n] \mapsto \Map_{\Catinfty}([n],C)]
	\end{align*}
	from the \category of \categories to the \category of categories internal to anima.
    See \cite{MR4865825} for a modern, model-independent proof of this fact.
\end{nul}

\begin{nul}
    The main example that we care about in this paper is the case where $\Bcal = \CondAni$.
    Since the Segal conditions and the sheaf condition are both limit conditions, the canonical equivalence
    \begin{equation*}
    \Fun(\Extr^{\op},\Fun(\Deltaop,\Ani)) \simeq \Fun( \Deltaop, \Fun(\Extr^{\op},\Ani))
    \end{equation*}
    restricts to an equivalence
    \begin{equation*}
        \CondCat \simeq  \ICat(\CondAni) \period
    \end{equation*}
    Therefore, we often implicitly identify $ \CondCat $ with $ \ICat(\CondAni) $.
\end{nul}

We now turn to some specific features of $ \CondCat $.

\begin{definition}[(continuous functors)]
    The \category of condensed \categories is cartesian closed, see \cite[Proposition~3.2.11]{arXiv:2103.17141}.
    For condensed \categories $ \Ccal $ and $ \Dcal $, we denote the internal $ \Hom $ by
    \begin{equation*}
        \Funcond(\Ccal,\Dcal) \period
    \end{equation*}
    Similarly, we write
    \begin{equation*}
        \Functs(\Ccal,\Dcal) \colonequals \Funcond(\Ccal,\Dcal)(\ast)
    \end{equation*}
    for the \category of \defn{continuous functors} $ \fromto{\Ccal}{\Dcal} $.
\end{definition}

\begin{nul}
    Observe that the functor $ (\Ccal,\Dcal) \mapsto \Functs(\Ccal,\Dcal) $ is characterized by the existence of natural equivalences
    \begin{equation*}
        \Map_{\Catinfty}(\Acal,\Functs(\Ccal,\Dcal)) \equivalent \Map_{\CondCat}(\Acal \cross \Ccal, \Dcal)
    \end{equation*}
    for each \category $ \Acal $.
\end{nul}

\begin{nul}
    Explicitly, $ \Functs(\Ccal,\Dcal) $ is given by the end
    \begin{equation*}
        \Functs(\Ccal,\Dcal) \equivalent \int_{S \in \Extr^{\op}} \Fun(\Ccal(S),\Dcal(S)) \comma
    \end{equation*}
    see, for example, \cite[Proposition~2.3]{MR3518559}.
    In particular, the objects in this \category are precisely natural transformations $ \Ccal(-) \to \Dcal(-)$ of functors $\Extr^{\op} \to \Catinfty$.
\end{nul}

Many of the condensed \categories we are interested come from pro-objects:

\begin{observation}\label{obs:categories_internal_to_profinite_anima_as_condensed_categories}
    By taking internal categories on each side, the right adjoint fully faithful embedding $\ProAnifin \to \CondAni$ of \cref{rec:various-completions} induces a fully faithful right adjoint functor
    \begin{equation*}
        \iota \colon \ICat(\ProAnifin) \to \CondCat \period
    \end{equation*}
    Many of the examples of condensed \categories that we care about are in the image of this embedding.
\end{observation}

\noindent For condensed \categories in the image of $\iota$, we can describe their value at Čech--Stone compactifations explicitly:

\begin{proposition}\label{prop:profinitcat_evaluated_at_Stone_chech}
    Consider $\Ccal \in \ICat(\ProAnifin)$ as a condensed \category via $\iota $ and let $ M $ be a set.
    Then the functor
    \begin{equation*}
        \Functs(\upbeta(M), \Ccal) \to \prod_{m \in M} \Ccal(\{m\})
    \end{equation*}
    induced by the inclusions $ \incto{\{m\}}{\upbeta(M)} $ is an equivalence of \categories.
\end{proposition}

\begin{proof}
    It suffices to check that this functor becomes an equivalence after applying the functor $\Map_{\Catinfty}([n],-)$ for every $ n $.
    Since we have a natural chain of equivalences
    \begin{align*}
        \Map_{\Catinfty}([n],\Functs(\upbeta(M), \Ccal)) &\simeq \Map_{\CondCat}(\upbeta(M) \times [n], \Ccal) \\
        &\simeq \Map_{\CondCat}(\upbeta(M), \ev_{[n]}(\Ccal)), 
    \end{align*}
    it suffices to show that the natural map
    \begin{equation*}
        \Map_{\CondCat}(\upbeta(M), \ev_{[n]}(\Ccal)) \to  \prod_{m \in M}  \ev_{[n]}(\Ccal)(\{m\})
    \end{equation*}
    is an equivalence.
    Since $\ev_{[n]}(\Ccal)$ is a profinite anima by assumption and both sides are clearly compatible with limits, we may assume that $\ev_{[n]}(\Ccal) = A $ is a \pifinite anima.

    By \SAG{Lemma}{E.1.6.5}, there exists a Kan complex $A_\bullet$ with values in finite sets such that $\lvert A_\bullet \rvert \simeq A$.
    Since $\upbeta(M)$ is a compact projective object in $\CondAni$, it follows that the natural map
    \begin{equation*}
        \lvert \Map_{\CondAni}(\upbeta(M),A_\bullet) \rvert \to  \Map_{\CondAni}(\upbeta(M),\lvert A_\bullet \rvert)
    \end{equation*}
    is an equivalence.
    Since every $A_n$ is finite, it follows that $\Map_{\CondAni}(\upbeta(M),A_\bullet) \simeq \prod_M A_\bullet$ is an infinite product of Kan complexes.
    Since geometric realizations of Kan complexes commute with arbitrary products,%
    \footnote{This follows from the fact that the homotopy groups of the geometric realization of a Kan complex are computed as its simplicial homotopy groups, and these commute with infinite products.} 
    the natural map
    \begin{equation*}
         \Map_{\CondAni}(\upbeta(M),A) \simeq \lvert \Map_{\CondAni}(\upbeta(M),A_\bullet) \rvert \longrightarrow \prod_M \lvert A_\bullet \rvert  \simeq  \prod_M A
    \end{equation*}
    is an equivalence.
\end{proof}

%-------------------------------------------------------------------%
%  Recollection on shape theory                                     %
%-------------------------------------------------------------------%

\subsection{Recollection on shape theory}\label{subsec:recollection_on_shape_theory}

In this subsection, we recall a bit about shape theory for \topoi.
We do not explicitly need shape theory for most of this paper, but, instead, we work with a relative version of shape theory over the base \topos of condensed anima.
So this subsection serves as motivation for the theory we develop; we also use it to recall some background on shapes of topological spaces and the étale homotopy type.

\begin{recollection}[(protruncation)]
    The inclusion $ \ProAnitrun \subset \ProAni $ admits a left adjoint
    \begin{equation*}
        \protrun \colon \fromto{\ProAni}{\ProAnitrun}
    \end{equation*}
    defined as follows.
    The functor $ \protrun $ is the unique cofiltered-limit-preserving extension of the fully faithful functor $ \incto{\Ani}{\ProAnitrun} $ that sends an anima $ A $ to the cofiltered diagram given by its Postnikov tower $ \{\trun_{\leq n}(A)\}_{n\geq 0} $.
    We refer to $ \protrun $ as the \defn{protruncation} functor.
\end{recollection}

\begin{recollection}
    Let $ \Xcal $ be \atopos.
    We write $ \Gammalowerstar \colonequals \Map_{\Xcal}(1_{\Xcal},-) \colon \Xcal \to \Ani $ for the \defn{global sections} functor.
    The functor $ \Gammalowerstar $ admits a left exact left adjoint $ \Gammaupperstar \colon \Ani \to \Xcal $ referred to as the \defn{constant sheaf} functor.
    The \topos of anima is the terminal object of $ \RTop_{\infty} $, so $ \Gammalowerstar $ is the unique geometric morphism $ \Xcal \to \Ani $.

    While $ \Gammaupperstar $ need not preserve limits in general, the unique cofiltered limit-preserving extension $ \Pro(\Ani) \to \Xcal $ of $ \Gammaupperstar $ preserves all limits and admits a left adjoint 
    \begin{equation*}
        \Gammalowersharp \colon \Xcal \to \Pro(\Ani) \period 
    \end{equation*}
\end{recollection}

\begin{recollection}\label{rec:shape}
    Let $ \Xcal $ be \atopos.
    The \defn{shape} of $ \Xcal $ is the proanima
    \begin{equation*}
        \Shape(\Xcal) \colonequals \Gammalowersharp(1_{\Xcal}) \period
    \end{equation*}
    The assignment $ \Xcal \mapsto \Shape(\Xcal) $ naturally refines to a functor
    \begin{equation*}
        \Shape \colon \RTop_{\infty} \to \ProAni
    \end{equation*}
    that is left adjoint to the unique cofiltered limit-preserving extension of the functor
    \begin{align*}
        \Ani &\to \RTop_{\infty} \\ 
        A &\mapsto \Fun(A,\Ani) \equivalent \Ani_{/A}
    \end{align*}
    with functoriality given by right Kan extension.

    The \textit{protruncated shape} functor is the composite
    \begin{equation*}
        \begin{tikzcd}
            \Shapeprotrun \colon \RTop_{\infty} \arrow[r, "\Shape"] & \ProAni \arrow[r, "\protrun"] & \ProAnitrun \period
        \end{tikzcd}
    \end{equation*} 
    Similarly, the \textit{profinite shape} is defined by composing further with the profinite completion functor
    \begin{equation*}
        \begin{tikzcd}
            \Shapeprofin \colon \RTop_{\infty} \arrow[r, "\Shapeprotrun"] & \ProAnitrun \arrow[r, "{(-)\profincomp}"] & \ProAnifin \period
        \end{tikzcd}
    \end{equation*} 
\end{recollection}

\begin{observation}\label{obs:prodiscrete_completion_as_a_shape}
    The prodiscritization functor $ (-)\prodisccomp \colon \CondAni \to \ProAnitrun $ is the composite of $ \Gammalowersharp \colon \CondAni \to \ProAni $ with the protruncation functor $ \protrun $.
\end{observation}

We now give a useful, alternative description of the shape.

\begin{recollection}\label{rec:pro-objects_as_left_exact_accessible_functors}
    Let $ \Ccal $ be an accessible \category with finite limits (e.g., $ \Ccal = \Ani $).
    Then by \cite[\SAGthm{Definition}{A.8.1.1} \& \SAGthm{Proposition}{A.8.1.6}]{SAG}, there is a natural identification
    \begin{equation*}
        \Pro(\Ccal) \equivalent \Funlexacc(\Ccal,\Ani)^{\op}
    \end{equation*}
    with the opposite of the \category of left exact accessible functors $ \fromto{\Ccal}{\Ani} $. 
    Under these identifications, the protruncation functor $ \protrun \colon \ProAni \to \ProAnitrun $ is identified with the functor 
    \begin{equation*}
        \Funlexacc(\Ani,\Ani)^{\op} \to \Funlexacc(\Anitrun,\Ani)^{\op}
    \end{equation*}
    given by precomposition with the inclusion $ \Anitrun \inclusion \Ani $.

    Given \atopos $ \Xcal $, under this identification of $ \ProAni $, the shape $ \Shape(\Xcal) $ is the left exact accessible functor $ \fromto{\Ani}{\Ani} $ given by the composite
    \begin{equation*}
        \Gamma_{\Xcal,\ast}\Gammaupperstar_{\Xcal} \colon \fromto{\Ani}{\Ani} \period 
    \end{equation*}
    That is, for each anima $ A $, the value of $ \Shape(\Xcal) $ on $ A $ is the global sections of the constant object of $ \Xcal $ with value $ A $.
    Moreover, given a geometric morphism $ \flowerstar \colon \fromto{\Xcal}{\Ycal} $ with unit $ \unit \colon \fromto{\id{\Ycal}}{\flowerstar \fupperstar} $, the induced morphism of proanima $ \fromto{\Shape(\Xcal)}{\Shape(\Ycal)} $ corresponds to the morphism
    \begin{equation*}
        \Gamma_{\Ycal,\ast} \unit \Gammaupperstar_{\Ycal} \colon \Gamma_{\Ycal,\ast} \Gammaupperstar_{\Ycal} \longrightarrow \Gamma_{\Ycal,\ast} \flowerstar \fupperstar \Gammaupperstar_{\Ycal} \equivalent \Gamma_{\Xcal,\ast} \Gammaupperstar_{\Xcal}
    \end{equation*}
    in $ \ProAni^{\op} \subset \Fun(\Ani,\Ani) $.
    We refer the reader to \cites[\HTTsubsec{7.1.6}]{HTT}[\S 2]{MR3763287} for more details.
\end{recollection}

We now explain how the shape of the \topos of sheaves on a locally compact Hausdorff space $ T $ relates to the prodiscretization of the condensed set represented by $ T $ in the sense of \Cref{rec:various-completions}.
To do this, we first need the following lemma.

\begin{lemma}\label{lem:fully_faithful_on_truncated_objects_implies_shape_equivalence}
    Let $ \flowerstar \colon \Xcal \to \Ycal $ be a geometric morphism of \topoi.
    If $ \fupperstar $ is fully faithful when restricted to truncated objects, then $ \Shapetrun(\flowerstar) \colon \Shapetrun(\Xcal) \to \Shapetrun(\Ycal) $ is an equivalence.
\end{lemma}

\begin{proof}
    Note that since $ \fupperstar $ and $ \flowerstar $ are left exact, they preserve truncated objects \HTT{Proposition}{5.5.6.16}.
    Hence the adjunction $ \fupperstar \leftadjoint \flowerstar $ restricts to an adjunction at the level of truncated objects.
    Thus our assumption is that the unit $ \unit \colon \fromto{\id{\Ycal}}{\flowerstar \fupperstar} $ is an equivalence when restricted to truncated objects.
    Under the description of the protruncated shape given in \Cref{rec:pro-objects_as_left_exact_accessible_functors}, we see that we need to show that for each truncated anima $ A $, the map induced by the unit
    \begin{equation*}
        \Gamma_{\Ycal,\ast} \Gammaupperstar_{\Ycal}(A) \longrightarrow \Gamma_{\Ycal,\ast} \flowerstar \fupperstar \Gammaupperstar_{\Ycal}(A) 
    \end{equation*}
    is an equivalence; this follows from our assumption.
\end{proof}

\begin{example}\label{ex:protruncated_shapes_of_X_Xhyp_and_Xpost_agree}
    Let $ \Xcal $ be \atopos.
    There are natural geometric morphisms
    \begin{equation*}
        \Xcal^{\post} \to \Xcal^{\hyp} \inclusion \Xcal \period
    \end{equation*}
    Here, $ \Xcal^{\post} $ is the \textit{Postnikov completion} of $ \Xcal $ in the sense of \SAG{Definition}{A.7.2.5}. 
    By \SAG{Theorem}{A.7.2.4} and \HTT{Lemma}{6.5.2.9}, these geometric morphisms restrict to equivalences on truncated objects.
    Hence they induce equivalences on protruncated shapes.
\end{example}

Now we deal with sheaves on locally compact Hausdorff spaces.

\begin{notation}\label{ntn:shapes_of_topological_spaces}
    For a topological space $ T $, we write $ \Shape(T) \in \ProAni $ for the shape of the \topos $ \Sh(T) $ of sheaves of anima on $ T $.
    We write $ \Shapeprotrun(T) $ for the protruncation of $ \Shape(T) $.
    We write $ \LCH \subset \Top $ for the full subcategory spanned by the locally compact Hausdorff spaces.
\end{notation}

\begin{example}
    If $ T $ is a topological space that admits a CW structure, then $ \Shape(T) $ coincides with the underlying anima of $ T $.
    See \cites[\HAsec{A.4}]{HA}[\S3.2]{arXiv:2010.06473}.
\end{example}

\begin{lemma}\label{lem:prodiscrete_completion_of_LCH_spaces}
    The triangle
    \begin{equation*}
        \begin{tikzcd}
            & \LCH \arrow[dl, hooked'] \arrow[dr, "\Shapeprotrun"] &  \\ 
            \CondAni \arrow[rr, "{(-)\prodisccompl}"'] & & \ProAnitrun
        \end{tikzcd}
    \end{equation*}
    canonically commutes.
\end{lemma}

\begin{proof}
    Let $ T $ be a locally compact Hausdorff space.
    By \cite[Corollary 4.9]{haine2022descentsheavescompacthausdorff}, there is a natural fully faithful left exact left adjoint 
    \begin{equation*}
      \Sh^{\post}(T) \hookrightarrow \CondAni_{/T}
    \end{equation*}
    from the Postnikov completion of the \topos of sheaves on $ T $ to condensed anima sliced over $ T $.
    By \Cref{lem:fully_faithful_on_truncated_objects_implies_shape_equivalence,ex:protruncated_shapes_of_X_Xhyp_and_Xpost_agree}, we deduce that this algebraic morphism induces an equivalence on protruncated shapes 
    \begin{equation*}
      \Shapeprotrun(\CondAni_{/T}) \equivalence \Shapeprotrun(\Sh^{\post}(T)) \equivalent \Shapeprotrun(T) \period
    \end{equation*}

    Note that for any \topos $ \Xcal $ and object $ U \in \Xcal $, the forgetful functor $ \Xcal_{/U} \to \Xcal $ is left adjoint to the pullback functor $ U \cross (-) \colon \Xcal \to \Xcal_{/U} $.
    Hence the shape of $ \Xcal_{/U} $ coincides with the image of $ U $ under $ \Gammalowersharp \colon \Xcal \to \ProAni $.
    Thus by \Cref{obs:prodiscrete_completion_as_a_shape}, the protruncated shape of the slice $ \CondAni_{/T} $ coincides with prodiscretization of the condensed set $ T $.
\end{proof}

\begin{remark}
    \Cref{lem:prodiscrete_completion_of_LCH_spaces} was also (essentially) observed in \cite[Theorem 4.12]{arXiv:2409.01462}.
\end{remark}

%-------------------------------------------------------------------%
%  Recollection on proétale sheaves                                 %
%-------------------------------------------------------------------%

\subsection{Recollection on proétale sheaves}\label{subsec:recollection_on_proetale_sheaves}

We now turn to recalling some background about the proétale topology and proétale sheaves.
The following definition is from \cite{BhattScholzeProetale}:

\begin{definition}\label{def:weakly-etale}
  Let $f \from X \to Y $ be a morphism of schemes.
  \begin{enumerate}
      \item We call $ f \colon X \to Y $ \emph{weakly étale} if both $ f $ and its diagonal $ \Delta_f $ are flat.
      \item We write $ \ProEt{X} $ for the \emph{proétale site of $ X $}, i.e., the site of weakly étale $ X $-schemes equipped with the fpqc topology.
      
      \item We furthermore write $ X_{\proet} \colonequals \Sh(\ProEt{X}) $ for the \defn{proétale \topos of $ X $}.
  \end{enumerate}
\end{definition}

\begin{nul}
    We almost exclusively work with the \textit{hypercomplete} proétale \topos \smash{$ \Xproethyp $}.
\end{nul}

\begin{remark}[(size issues)]\label{rem:set_theory_continued}
    Since the category of weakly étale $ X $-schemes is not small, \Cref{def:weakly-etale} introduces some set-theoretic issues.
    In the end, one can always circumvent these issues and they do not have any serious effect on our results.
    For the more cautious reader, we suggest one of the following two ways of reading this paper:
    \begin{enumerate}
        \item\label{rem:set_theory_continued.1} Fix once and for all two strongly inaccessible cardinals $ \delta < \epsilon $.
        All schemes, spectral spaces, etc. are then assumed to be $\delta$-small and all categorical constructions, such as taking sheaves on a site, are taken with respect to the larger universe determined by $\epsilon$.
        In particular \smash{$ \Xproethyp $} always means hypersheaves of $\epsilon$-small anima on $\delta$-small weakly étale $ X $-schemes, and similarly for the \category of condensed anima $\CondAni$.
        
        \item\label{rem:set_theory_continued.2} If the reader does not  want to work with universes, they may proceed as follows.
        For a scheme $ X $, choose a strong limit cardinal $ \kappa $ such that $ X $ is $ \kappa $-small.
        Write $\ProEt{X,\kappa}$ for the category of $ \kappa $-small weakly étale $ X $-schemes.
        We then define
        \begin{equation*}
            X_{\proet,\kappa}^{\hyp} \colonequals \Shhyp(\ProEt{X,\kappa}) \period
        \end{equation*}
        The assumption that $ \kappa $ is a strong limit cardinal guarantees that there are enough \wcontractibles in $\ProEt{X,\kappa} $, see \cref{def:w-contractible}.
        We then define
        \begin{equation*}
            \Xproethyp \colonequals \textstyle \colim_\kappa X_{\proet,\kappa}^{\hyp}
        \end{equation*}
        and similarly for the category of condensed anima.
        This is also the approach taken by Clausen and Scholze \cite{Scholze:condensednotes}.
        
        However, then some statements about \smash{$ \Xproethyp $} and $\CondAni$, such as \Cref{prop:hypercomplete_proet_sheaves_on_wc}, are no longer true on the nose.
        In such a case, to correct the result, one must make an implicit choice of strong limit cutoff cardinal $ \kappa $, and \smash{$\Xproethyp$} should be understood as \smash{$X_{\proet,\kappa}^{\hyp}$}.
        In the end, a choice of such a $ \kappa $ is harmless and does not affect our results, see \Cref{rem:set_theory_doesnt_matter}.
    \end{enumerate}
    The same discussion applies to the non-hypercomplete proétale \topos $X_{\proet}$.
\end{remark}

We now prove a generalization of \cite[Lemma 5.1.2 \& Corollary 5.1.6]{BhattScholzeProetale}.

\begin{notation}
    For a scheme $ X $, we denote the inclusion $ \Et{X} \to \ProEt{X} $ of the the étale site into the proétale site by $ \nu $.
\end{notation}

\begin{proposition}\label{prop:nuupperstar_fully_faithful_on_truncated_objects}
	Let $ X $ be a qcqs scheme.
	Then the pullback functor $ \nuupperstar \colon X_{\et}^{\hyp} \to \Xproethyp $ is fully faithful when restricted to truncated objects.
\end{proposition}

\begin{notation}
    Let $ X $ be a scheme.
    Write \smash{$ \ProEtaff{X} \subset \ProEt{X} $} for the full subcategory spanned by the affine schemes.
    Note that \smash{$ \ProEtaff{X} $} is a basis for the proétale topology on $ \ProEt{X} $.
    Hence by \cite[Corollary A.7]{arXiv:2001.00319}, restriction along the inclusion defines an equivalence of \categories
    \begin{equation*}
        \Xproethyp = \Shhyp(\ProEt{X}) \equivalence \Shhyp(\ProEtaff{X}) \period
    \end{equation*}
\end{notation}

\begin{proof}[Proof of \Cref{prop:nuupperstar_fully_faithful_on_truncated_objects}]
    First observe that since the left exact pullback functor $\nuupperstar$ preserves $ n $-truncated objects \HTT{Proposition}{5.5.6.16}, the truncated pullback functors are well-defined. 
    We equivalently need to show that the composite
    \begin{equation*}
        \begin{tikzcd}
            X_{\et}^{\hyp} \arrow[r, "\nuupperstar"] & \Xproethyp \arrow[r, "\sim"{yshift=-0.25em}] & \Shhyp(\ProEtaff{X}) 
        \end{tikzcd}
    \end{equation*}
    is fully faithful when restricted to truncated objects.
    To simplify notation, we also denote this composite by $ \nuupperstar $.

    First observe that a presheaf of $ n $-truncated anima $F \colon (\ProEtaff{X})^{\op} \to \Ani_{\leq n} $ is a sheaf if and only if the following conditions hold: 
    \begin{enumerate}
        \item The presheaf $ F $ sends finite disjoint unions of affine schemes proétale over $ X $ to finite products.

        \item For every surjection $ f\colon U \twoheadrightarrow X$ of affine schemes proétale over $ X $ with associated Čech nerve $U_{\bullet}\to X$, the canonical map
        \begin{equation*}
            F(X)\to \lim_{[i] \in \DDelta_{\leq n+1}} F(U_i)
        \end{equation*}        
        is an isomorphism. 
    \end{enumerate} 
    This is just the $ n $-truncation of the sheaf condition as formulated in \SAG{Proposition}{A.3.3.1},%
    \footnote{One easily checks that the category $\ProEtaff{X}\subset \ProEt{X}$ satisfies the conditions stated there.}
    using the fact that totalizations in an $ (n+1) $-category can be calculated as limits over $ \DDelta_{\leq n+1} $ \cite[Proposition A.1]{arXiv:2207.09256}.

    Since the problem is local on $ X $, we immediate reduce to the case where $ X $ is affine. 
    Then, the category $\ProEtaff{X}$ is exactly given by those $U\in \ProEt{X}$ which can be written as a small cofiltered limit $U=\lim_{i \in I} U_i$ 
    of affine schemes $U_i\in \Et{X}$. 
    Now let $n \geq 0 $ be an integer and, let $ F $ be an object of $X_{\et,\leq n}$. 
    The presheaf pullback of $ F $ to the proétale site of $ X $ is given by the formula $U\mapsto \colim_{i\in I^{\op}} F(U_i)$ on all $U\in \ProEtaff{X}$. 
    We wish to show, that this is already a sheaf. 
    For this, we can just copy the proof of \cite[Proposition 7.1.3(2)]{Ultracategories}. 
    The argument there works not only for equalizers, but for all finite limits as they appear in our $ n $-truncated sheaf condition. 
    As $\nu^* F$ restricts to $ F $ on affine étale schemes $\Etaff{X}$, it is clear that we have $\nulowerstar\nuupperstar F = F$ for all $F \in X_{\et,\leq n}$, i.e., the pullback $ \nuupperstar $ is fully faithful when restricted to $ n $-truncated objects. 
    See \cite[Proposition A.5.33]{CatrinsThesis} for more details.
\end{proof}

Now we deduce some consequences for the étale homotopy type.
For this, recall our notation regarding shape theory from \Cref{rec:shape}.

\begin{notation}\label{ntn:etale_homotopy_type}
    Let $ X $ be a scheme. 
    We write
    \begin{equation*}
        \Pietprotrun(X) \colonequals \Shapeprotrun(X_{\et}^{\hyp}) \andeq \Pietprofin(X) \colonequals \Shapeprofin(X_{\et}^{\hyp})
    \end{equation*}
    for the \defn{protruncated étale homotopy type} and the \textit{profinite étale homotopy type} of $ X $, respectively.
\end{notation}

\begin{corollary}\label{cor:protruncated_shapes_of_etale_and_proetale_topoi}
	Let $ X $ be a scheme.
    Then the map
    \begin{equation*}
        \Shapeprotrun(\nulowerstar) \colon \Shapeprotrun(\Xproethyp) \to \Pietprotrun(X)
    \end{equation*}
    is an equivalence.
\end{corollary}

\begin{proof}
	Immediate from \Cref{lem:fully_faithful_on_truncated_objects_implies_shape_equivalence,prop:nuupperstar_fully_faithful_on_truncated_objects}.
\end{proof}

%-------------------------------------------------------------------%
%  Basis of weakly contractible objects                             %
%-------------------------------------------------------------------%

\subsubsection{Basis of weakly contractible objects}

Recall that an object $Y$ of a site $ \Ccal $ is \textit{weakly contractible} if every covering $U \twoheadrightarrow Y$ admits a section.
In the proétale site, weakly contractible qcqs objects are given by \textit{\wcontractible} schemes.

\begin{definition}\label{def:w-contractible}
    A qcqs scheme $ X $ is \textit{\wcontractible} if every weakly étale surjection $U \twoheadrightarrow X$ admits a section.
\end{definition}

For the subsequent characterization of \wcontractibles, recall the following fact on connected components of qcqs schemes.

\begin{lemma}[\stacks{0900}]
    Let $ X $ be a qcqs scheme. 
    Then the set $\uppi_0(X)$ of connected components of $ |X| $, endowed with the quotient topology induced by $ |X| $, is a profinite set.
\end{lemma}

\begin{definition}\label{def:w-strictly-local}
    Let $ X $ be a qcqs scheme.
    We say that $ X $ is \textit{w-local} if the subspace $X_{\cl} \subset |X| $ of closed points is closed and every connected component of $ X $ has a unique closed point.
    We say that $ X $ is \defn{w-strictly local} if $ X $ is \textit{w-local} and every étale surjection $ U \twoheadrightarrow X $ admits a section.
\end{definition}

\begin{remark}
    As observed in \cite[Proposition~3.1]{MR0289501}, since a w-strictly local scheme is a retract of an affine scheme, every w-strictly local scheme is affine.
\end{remark}

\begin{remark}
    By \cite[Lemma~2.2.9]{BhattScholzeProetale}, a qcqs scheme $ X $ is w-strictly local if $ X $ is w-local and the local rings at all closed points are strictly henselian.
\end{remark}

\begin{example}\label{ex:over_sep_fields_everything_is_wstrl}
    Let $\kbar$ be a separably closed field.
    Then any qcqs weakly étale $\kbar$-scheme $ X $ is w-strictly local.
    Indeed, such a scheme is zero dimensional and thus, by Serre's cohomological characterization of affineness, affine.
    By \stacks{092Q}, it is therefore a cofiltered limit of finite disjoint unions of $\Spec(\kbar)$ and hence w-strictly local.
\end{example}

\begin{recollection}[\stacks{0982}]\label{rec:characterisation_weakly_contractible}
    A scheme $ X $ is \textit{\wcontractible} if and only if it is w-strictly local and $\uppi_0(X) \in \ProFin$ is extremally disconnected.
    In particular, \wcontractible schemes are affine.
\end{recollection}

\begin{notation}
    For a scheme $ X $, we write $ \ProEtwc{X} \subset \ProEt{X} $ for the full subcategory spanned by the \wcontractible schemes.
\end{notation}

\begin{recollection}[{\stacks{0990}}]\label{rec:basis-by-weakly-contractible-affines}
    The subcategory $ \ProEtwc{X} \subset \ProEt{X} $ is a basis for the proétale topology.
    But beware that $ \ProEtwc{X} $ is not closed under fiber products in $ \ProEt{X} $.
\end{recollection}

\begin{proposition}\label{prop:hypercomplete_proet_sheaves_on_wc}
    Let $ X $ be a scheme.
    Restriction along the inclusion of sites $ \ProEtwc{X} \subset \ProEt{X}$ defines an equivalence of hypercomplete \topoi
    \begin{equation*}
        \Xproethyp = \Shhyp(\ProEt{X}) \equivalence \Shhyp(\ProEtwc{X}) \period
    \end{equation*}
    Moreover, this \topos can be identified with the \topos of finite product-preserving presheaves 
    \begin{equation*}
      \Fun^{\times}(({\ProEtwc{X}})^{\op}, \Ani) \period
    \end{equation*}
\end{proposition}

\begin{proof}
    This follows from \Cref{rec:basis-by-weakly-contractible-affines} and \cite[Corollary A.7]{arXiv:2001.00319} combined with the defining property of \wcontractible schemes.
    Details are given in \cite[Proposition 2.2.12]{CatrinsThesis}.
\end{proof}

%-------------------------------------------------------------------%
%-------------------------------------------------------------------%
%-------------------------------------------------------------------%
%  The condensed homotopy type                                      %
%-------------------------------------------------------------------%
%-------------------------------------------------------------------%
%-------------------------------------------------------------------%

\newpage
\part{The condensed homotopy type}\label{part:the_condensed_homotopy_type}

%-------------------------------------------------------------------%
%-------------------------------------------------------------------%
%  Three perspectives on the condensed homotopy type                %
%-------------------------------------------------------------------%
%-------------------------------------------------------------------%

\section{Three perspectives on the condensed homotopy type}\label{sec:condensed_homotopy_type}

In this section, we introduce the condensed homotopy type of a scheme $ X $.
As explained in the introduction, we give three definitions, and prove that they are equivalent.
The first, given in \cref{subsec:definition_via_the_relative_shape}, is the relative shape of the hypercomplete proétale \topos \smash{$ \Xproethyp $} over the \topos $ \CondAni $ of condensed anima.
The second, given in \cref{subsec:characterization_as_a_hypercomplete_proetale_cosheaf}, is as the unique hypercomplete proétale cosheaf whose value on a \wcontractible affine $ U $ is the profinite set $ \uppi_0(U) $ of connected components of $ U $.
The last, given in \cref{subsec:definition_via_exodromy}, is as the condensed classifying anima of the Galois category $ \Gal(X) $ introduced by Barwick--Glasman--Haine \cite{Exodromy}.
In \cref{subsec:Henselian_local_rings}, we conclude the section with a sample computation: given a henselian local ring $ R $ with residue field $ \kappa $, we show inclusion of the closed point induces an equivalence
\begin{equation*}
    \equivto{\BGal_{\kappa} \equivalent \Picond(\Spec(\kappa))}{\Picond(\Spec(R))} \period
\end{equation*}

%-------------------------------------------------------------------%
%  Definition via the relative shape                                %
%-------------------------------------------------------------------%

\subsection{Definition via the relative shape}\label{subsec:definition_via_the_relative_shape}

For an \topos $\Xcal$, the idea of shape theory relies on the existence of a canonical colimit preserving functor $\Gammalowersharp \colon \Xcal \to \ProAni$.
We define the condensed homotopy type of a qcqs scheme in the tradition of shape theory but relative to the base $\CondAni$.
To do this, we use the identification
\begin{equation*}
    \Xproethyp \equivalent \Fun^{\times}\left(({\ProEtwc{X}})^{\op}, \Ani\right)
\end{equation*}
of the hypercomplete proétale \topos as the \topos of finite-product preserving presheaves on the site of \wcontractible weakly étale $ X $-schemes (\Cref{prop:hypercomplete_proet_sheaves_on_wc}).

\begin{definition}\label{def:pilowersharp}
	Let $ X $ be a scheme.
	Write
	\begin{equation*}
		\pilowersharp \colon \fromto{\PSh(\ProEtwc{X})}{\CondAni}
	\end{equation*}
	for the colimit-preserving extension of 
	\begin{equation*}
		\uppi_{0} \colon \ProEtwc{X} \to \Extr \inclusion \CondAni 
	\end{equation*}
	along the Yoneda embedding.
\end{definition}

\begin{observation}\label{obs:right-adjoint-of-pilowersharp}
	The functor $ \pilowersharp $ admits a right adjoint 
	\begin{align*}
		\piupperstar\from\CondAni &\to \PSh(\ProEtwc{X}) \\ 
	\intertext{given by the assignment}
		A &\mapsto [W \mapsto A(\uppi_{0}(W))] \period
	\end{align*}
	Note that since the functor $ \uppi_{0} \colon \fromto{\ProEtwc{X}}{\CondAni} $ preserves finite disjoint unions, the right adjoint to $ \pilowersharp $ factors through
	\begin{equation*}
		\Fun^{\times}\left(({\ProEtwc{X}})^{\op}, \Ani\right) \subset \PSh(\ProEtwc{X}) \period
	\end{equation*}
\end{observation}

\begin{notation}
    Given a scheme $ X $, we also write $ \pilowersharp $ for the composite
    \begin{equation*}
        \begin{tikzcd}
            \Xproethyp \arrow[r, "\sim"{yshift=-0.25ex}] & \Fun^{\times}\left(({\ProEtwc{X}})^{\op}, \Ani\right) \arrow[r, "\pilowersharp"] & \CondAni \comma
        \end{tikzcd}
	\end{equation*} 
     where the left-hand functor is the equivalence of \topoi from \Cref{prop:hypercomplete_proet_sheaves_on_wc}.
\end{notation}

Next, we need a generalization of \cite[Lemma 4.2.13]{BhattScholzeProetale}.

\begin{proposition}\label{prop:pilowersharp}
	Let $ X $ be a scheme.
	Then:
	\begin{enumerate}
		\item The functor $ \pilowersharp \colon \fromto{\Xproethyp}{\CondAni} $ is left adjoint to $ \piupperstar \colon \fromto{\CondAni}{\Xproethyp} $.

		\item For each condensed anima $ A $ and \wcontractible affine $ W \in \ProEt{X} $, there is a natural equivalence
 		\begin{equation*}
 			\piupperstar(A)(W) \equivalent A(\uppi_{0}(W)) \period
 		\end{equation*}
	\end{enumerate}
\end{proposition}

\begin{proof}
    As explained in \Cref{obs:right-adjoint-of-pilowersharp}, the functor 
    \begin{equation*}
        \piupperstar \from \CondAni \to \PSh(\ProEtwc{X})
    \end{equation*}
    factors through \smash{$ \Xproethyp $}.
    Hence $ \piupperstar $ remains right adjoint to the restriction of $\pilowersharp$.
    In particular, we have $\piupperstar(A)(U) \equivalent A(\uppi_{0}(U))$.
\end{proof}

\begin{remark}
    The right adjoint $\piupperstar$ is part of a geometric morphism of \topoi
    \begin{equation}\label{gm_proet_cond}
        \begin{tikzcd}[sep=2em]
            \CondAni \arrow[r, shift left=0.5ex, "\piupperstar"] & \Xproethyp \comma \arrow[l, shift left=0.5ex, "\pilowerstar"]
        \end{tikzcd}
    \end{equation} 
    which is induced by the morphism of sites
    \begin{align*}
        \pi\colon \ProSetfin &\longrightarrow \ProEt{X}  \\
        S = \lim_{i \in I} S_i &\longmapsto S\otimes X\colonequals \lim_{i \in I} \coprod_{s\in S_i} X \period
    \end{align*}
    For details, see \cite[Theorem 2.2.13]{CatrinsThesis}.
\end{remark}

Now we are ready for the definition of the condensed homotopy type.

\begin{definition}\label{def:condensed_homotopy_type}
    Let $ X $ be a scheme.
    \begin{enumerate}
        \item The \textit{condensed homotopy type} of $ X $ is the condensed anima 
        \begin{equation*}
            \CondShape{X} \colonequals \pilowersharp(1) \in \CondAni \period
        \end{equation*}
        
        \item The \textit{condensed set of connected components} of $ X $ is the condensed set 
        \begin{equation*}
            \pizerocond(X)\colonequals \uppi_0(\Picond(X))\in \CondSet \period
        \end{equation*}
    \end{enumerate}
\end{definition}

\begin{nul}
    The first part of \Cref{def:condensed_homotopy_type} says that the condensed homotopy type is the relative shape of the \topos \smash{$ \Xproethyp $} over the \topos $ \CondAni $, see \cite[\S4.1]{arXiv:1810.05544} for background on relative shapes.
    Since sending a scheme $ X $ to \smash{$ \pilowerstar \colon \Xproethyp \to \CondAni $} defines a functor
    \begin{equation*}
        \Sch \to (\RTop_{\infty})_{/\CondAni} \comma
    \end{equation*}
    composition with the relative shape over $\CondAni$, therefore defines a functor
    \begin{align}\label{fun:condensed_shape}
        \Picond \colon \Sch\to \CondAni \comma \quad X \mapsto \CondShape{X} \period
    \end{align}
\end{nul}

\begin{warning}
    A consequence of the statement of \cite[Lemma 4.2.13]{BhattScholzeProetale}, is that that for any condensed set $ A $, the formula $ \piupperstar(A)(U) \equivalent A(\uppi_{0}(U)) $ in \Cref{prop:pilowersharp} holds for all qcqs schemes $ U $ of the proétale site of $ X $. 
    However, this is not correct; indeed, if this stronger claim were true, it would follow that for all qcqs schemes $ X $ one has
    \begin{align*} 
        \Map_{\CondSet}(\uppi_{0}(X),A) &\equivalent A(\uppi_{0}(X)) \equivalent \pi^*(A)(X) \\
                                        &\equivalent \Map_{\Xproethyp}(X, \pi^*(A)) \\
                                        &\equivalent \Map_{\CondAni}( \Picond(X), A) \\
                                        &\equivalent \Map_{\CondSet}(\pizerocond(X),A) \period
    \end{align*}
    This would then imply that the condensed set of connected components matches the usual one, i.e., $\uppi_{0}^{\cond}(X) = \uppi_{0}(X)$ in $\CondSet$. 
    As we show in \cref{example:cond_pi0_and_warsaw_circle}, this is not generally the case. 
    However, this is true if $ X $ has finitely many irreducible components, see \Cref{cor:pi0s_match_for_finitely_many_irr_comps}.
    The problem here is that the proof of \cite[Lemma 4.2.13]{BhattScholzeProetale} only works for \wcontractible schemes. 
\end{warning}

The definition tells us the value of the condensed homotopy type on \wcontractible schemes:

\begin{example}\label{ex:Picond_on_w-contractibles}
    Let $ W $ be a \wcontractible scheme.
    Then, by definition, 
    \begin{equation*}
       \Picond(W) = \pilowersharp(1) = \uppi_0(W) \period
    \end{equation*}
    In particular, if $ W $ is the spectrum of a separably closed field, then $ \Picond(W) = \pt $.
\end{example}

\begin{nul}
    One consequence of \Cref{ex:Picond_on_w-contractibles} is that every geometric point $ \xbar \to X $ defines a point  
    \begin{equation*}
        \pt = \Picond(\xbar) \to \Picond(X) 
    \end{equation*}
    of the condensed homotopy type.
    Thus we can take homotopy groups at geometric points:
\end{nul}

\begin{definition}\label{def:condensed_homotopy_groups}
    Let $ X $ be a scheme, let $\xbar \to X $ be a geometric point, and let $ n \geq 1 $.
    The \emph{$ n $-th condensed homotopy group} of $ X $ at $\xbar$ is the condensed group (abelian if $n \geq 2$)
    \begin{equation*}
        \picond_n(X, \xbar) \colonequals \uppi_n(\Picond(X), \xbar) \period
    \end{equation*}
\end{definition}

From the definition, it is easy to see that the condensed homotopy type refines the protruncated and profinite étale homotopy types.
For this result, recall our notation on shapes and étale homotopy types from \cref{subsec:recollection_on_shape_theory,ntn:etale_homotopy_type}.

\begin{lemma}\label{lem:Picond_recovers_Piet}
	Let $ X $ be a scheme.
	Then there are natural equivalences
	\begin{equation*}
		\Picond(X)\prodisccompl \equivalent \Pietprotrun(X) \andeq \Picond(X)\profincomp \equivalent \Pietprofin(X) \period
	\end{equation*}
\end{lemma}

\begin{proof}
    By \Cref{cor:protruncated_shapes_of_etale_and_proetale_topoi}, the protruncated shapes of the (hypercomplete) étale and proétale \topoi agree. 
    This remains true after profinite completion.  
    Thus the claims follow from the claim that the triangle of left adjoints
    \begin{equation*}
        \begin{tikzcd}[row sep=3em, column sep=1.5em]
             &
            \Xproethyp \arrow[dl, "\pilowersharp"'] \arrow[dr, "\Pietprotrun"] & 
            \\ 
            \CondAni  \arrow[rr, "{(-)\prodisccompl}"'] & & \ProAnitrun
        \end{tikzcd}
    \end{equation*}
    commutes. 
    To see this, note that the corresponding diagram of right adjoints commutes by the uniqueness property of the pro-extension \smash{$\ProAni \to \Xproethyp$} of the constant sheaf functor.
\end{proof}

\begin{nul}\label{nul:map-between-condensed-and-etale-homotopy-type}
    The unit of the adjunction $ (-)\prodisccompl \colon \CondAni \rightleftarrows \ProAnitrun $ induces canonical comparison maps 
    \begin{equation*}
        \Picond(X) \to \Pietprotrun(X) \andeq \Picond(X) \to \Pietprofin(X)
    \end{equation*}
    in $\CondAni$.
    In particular, there are canonical comparison homomorphisms
    \begin{equation*}
        \picond_n(X) \to \piet_n(X)
    \end{equation*}
    of the condensed homotopy groups to the (profinite) étale homotopy groups for all $n \geq 0$.
\end{nul}

%-------------------------------------------------------------------%
%  Characterization as a hypercomplete proétale cosheaf             %
%-------------------------------------------------------------------%

\subsection{Characterization as a hypercomplete proétale cosheaf}\label{subsec:characterization_as_a_hypercomplete_proetale_cosheaf}

The goal of this subsection is to prove the following characterization of the condensed homotopy type and derive some consequences for the étale homotopy type.

\begin{notation}
    We write $ \Affwc \subset \Sch $ for the full subcategory spanned by the \wcontractible schemes.
    (Recall from \Cref{rec:characterisation_weakly_contractible} that \wcontractible schemes are affine.)
\end{notation}

\begin{proposition}\label{cor:characterization_of_the_condensed_homotopy_type_as_a_proétale_cosheaf}
    The condensed homotopy type
    \begin{equation*}
        \Picond \colon \fromto{\Sch}{\CondAni}
    \end{equation*}
    is the unique hypercomplete proétale cosheaf whose restriction to \wcontractible schemes is given by the functor
    \begin{equation*}
        \uppi_{0} \colon \Affwc \to \Extr \subset \CondAni \period
    \end{equation*}
\end{proposition}

\begin{proof}
    First notice that since $ \pilowersharp $ preserves colimits, by definition \smash{$ \Picond $} carries proétale hypercoverings to colimit diagrams.
    Moreover, by construction \smash{$ \Picond $} agrees with $ \uppi_0 $ when restricted to \wcontractible schemes (see \Cref{ex:Picond_on_w-contractibles}).
    Thus it suffices to show that every scheme admits a proétale hypercover by \wcontractible schemes.
    Since every scheme admits a Zariski cover by qcqs schemes, we can reduce to the qcqs case.
    In this case, the claim is the content of \stacks{09A1}.
\end{proof}

\begin{remark}[(on set theory)]\label{rem:set_theory_doesnt_matter}
    Let $ X $ be a scheme and $ \kappa $ a strong limit cardinal such that $ X $ is $ \kappa $-small.
    Then there exists a hypercover by \wcontractibles $W_\bullet \to X$ such that $W_n$ is $ \kappa $-small for all $ n $.
    Hence the formula
    \begin{equation*}
        \Picond(X) \simeq \textstyle\colim_{\Deltaop} \uppi_0(W_\bullet)
    \end{equation*}
    shows that for $\kappa < \kappa'$ an implicit choice of cutoff cardinal in \cref{def:condensed_homotopy_type} does not affect the outcome.
    More precisely, under the embedding $\CondAni_\kappa \inclusion \CondAni_{\kappa'}$ one gets carried to the other.
    Equivalently, if one takes the approach to dealing with set theory explained in \Cref{rem:set_theory_continued} \eqref{rem:set_theory_continued.2}, then for all choices of suitable cutoff cardinals the images of the condensed homotopy type in the colimit $\CondAni = \colim_\kappa \CondAni_\kappa$ agree.
    Therefore we can continue to leave choices of cutoff cardinals implicit without getting into trouble.

    If one would try to set up the theory in the setting of \emph{light} condensed anima, one would get a different result in general.
    See also \cref{rem:comparison_with_light}.
\end{remark}

\begin{corollary}\label{cor:characterization_of_etale_homotopy_type_by_proetale_hyperdescent}
    \hfill
    \begin{enumerate}
        \item The protruncated étale homotopy type $ \Pietprotrun \colon \fromto{\Sch}{\ProAnitrun} $ is the unique hypercomplete proétale cosheaf valued in $ \ProAnitrun $ whose restriction to \wcontractible affines coincides with
        \begin{equation*}
            \uppi_{0} \colon \Affwc \to \Extr \inclusion \ProAnitrun \period
        \end{equation*}

        \item The profinite étale homotopy type $ \Pietprofin \colon \fromto{\Sch}{\ProAnifin} $ is the unique hypercomplete proétale cosheaf valued in $ \ProAnifin $ whose restriction to \wcontractible affines coincides with
        \begin{equation*}
            \uppi_{0} \colon \Affwc \to \Extr \inclusion \ProAnifin \period
        \end{equation*}
    \end{enumerate}
\end{corollary}

\begin{proof}
    Since both $ (-)\prodisccompl $ and $(-)\profincomp$ are left adjoints, the composites
    \begin{equation*}
        \begin{tikzcd}[sep=3.5em]
            \Sch \arrow[r, "\Picond"] & \CondAni \arrow[r, "{(-)\prodisccompl}"] & \ProAnitrun
        \end{tikzcd}
    \end{equation*}
    and 
    \begin{equation*}
        \begin{tikzcd}[sep=3.5em]
            \Sch \arrow[r, "\Picond"] & \CondAni \arrow[r, "(-)\profincomp"] & \ProAnifin 
        \end{tikzcd}
    \end{equation*}
    are still hypercomplete proétale cosheaves. 
    Moreover, on \wcontractible affines they both are given by $ \goesto{U}{\uppi_{0}(U)} \in \Extr $. 
    In \Cref{lem:Picond_recovers_Piet}, we have seen that these functors recover the protruncated and profinite étale homotopy types, respectively.
\end{proof}

\begin{remark}\label{rem:agrees-with-HRS}
	It follows immediately from \Cref{cor:characterization_of_the_condensed_homotopy_type_as_a_proétale_cosheaf} that the `condensed shape' defined in \cite[Appendix A]{hemo2023constructible_sheaves_schemes} agrees with our notions.
\end{remark}

In \cite{hemo2023constructible_sheaves_schemes}, Hemo--Richarz--Scholbach prove that $ \Picond(X) $ classifies local systems on $ X $ with coefficients in any condensed ring.
We recall the precise statement here; for this, we need the following definition from \cite{hemo2023constructible_sheaves_schemes}.
In order to state it, recall that we write $ \piupperstar $ for the natural pullback functor $ \fromto{\CondAni}{\Xproethyp} $ of \Cref{obs:right-adjoint-of-pilowersharp}.

\begin{definition}
    Let $\Lambda$ be a condensed ring.
    \begin{enumerate}
        \item We define the condensed \category $\Perfbf_\Lambda$ of \defn{perfect complexes} over $\Lambda$, to be the condensed \category defined by
        \begin{align*}
            \Extr^{\op} \to \Catinfty \comma \qquad S \mapsto \Perf_{\Lambda(S)} \period
        \end{align*}
        Here, $\Perf_{\Lambda(S)}$ is the usual \category of perfect complexes over the ordinary ring $\Lambda(S)$.

        \item Let $ X $ be a qcqs scheme.
        Write $ \Dup(X_{\proet};\Lambda)$ for the derived \category of $ \pi^*\Lambda $-modules on $ X $.
        We define the \category of \defn{lisse} $ \Lambda $-modules $\Dlis(X_{\proet};\Lambda)$ to be the full subcategory of $ \Dup(X_{\proet};\Lambda) $ spanned by the dualizable objects.
    \end{enumerate}
\end{definition}

\begin{proposition}[{\cite[Proposition~A.1]{hemo2023constructible_sheaves_schemes}}]\label{prop:Picond_classifies_lisse_sheaves}
    There is a natural equivalence of \categories
    \begin{equation*}
        \Functs(\Picond(X),\Perfbf_\Lambda) \simeq \Dlis(X_{\proet};\Lambda) \period
    \end{equation*}
\end{proposition}

\begin{remark}
    \Cref{prop:Picond_classifies_lisse_sheaves} is one of the main motivations to study the condensed homotopy type.
    Indeed, the analogous statement for the ususal étale homotopy type $\Piet(X)$ is not even true in for $\Lambda = \QQell$.
    See \cite[Example~7.4.9]{MR3379634} for a concrete counterexample.
\end{remark}

%-------------------------------------------------------------------%
%  The proétale homotopy type via exodromy                          %
%-------------------------------------------------------------------%

\subsection{Definition via exodromy}\label{subsec:definition_via_exodromy}

In this subsection, we explain why the \textit{pyknotic étale homotopy type} defined in \cite[Remark 13.8.10]{Exodromy} agrees with $ \Picond(X) $.
For this, we recall the following definition from \cite{Exodromy} in the general setting of coherent \topoi, but we are most interested in the case of the étale \topos of a scheme.
In order to understand the general definition, the reader may wish to review the theory of coherent \topoi from \cite[\SAGapp{A}]{SAG} or \cite[Chapter 3]{Exodromy}.

\begin{definition}\label{def:Gal}
    Let $ \Xcal $ be a coherent \topos. 
    The \defn{Galois \category} of $ \Xcal $ is the condensed \category $ \Gal(\Xcal) $ defined by the functor
    \begin{align*}
        \ProSetfin^{\op} &\to \Catinfty \\
        S &\mapsto \Fun^{*,\coh}(\Xcal,\Sh(S)) \period
    \end{align*}
    Here, $ \Fun^{*,\coh}(\Xcal,\Sh(S)) $ is the \category of \defn{coherent} algebraic morphisms $ \supperstar \colon \Xcal \to \Sh(S) $ of \topoi, i.e., those left exact left adjoints that send truncated coherent objects of $ \Xcal $ to locally constant constructible sheaves of anima on the topological space $ S $.

    The assignment $ \goesto{\Xcal}{\Gal(\Xcal)} $ defines a functor from the \category of coherent \topoi and coherent geometric morphisms to $ \CondCat $. 
\end{definition}

Now we explain what this definition means more concretely in the two examples we are interested in.

\begin{recollection}
    Let $ X $ be a qcqs scheme.
    Then the \topos $ X_{\et} $ is coherent and by \cite[Lemma 9.5.3 \& Proposition 9.5.4]{Exodromy}, the truncated coherent objects of $ X_{\et} $ are the constructible étale sheaves of anima on $ X $.
\end{recollection}

\begin{notation}
    Let $ X $ be a qcqs scheme.
    We write $ \Gal(X) \colonequals \Gal(X_{\et}) $.
\end{notation}

\begin{recollection}
    Let $ X $ be a qcqs scheme.
    Since the \topos $ X_{\et} $ is $ 1 $-localic, for a profinite set $ S $, the value $ \Gal(X)(S) $ is equivalent to the $ 1 $-category of algebraic morphisms of $ 1 $-topoi
    \begin{equation*}
        \supperstar \colon X_{\et,\leq 0} \to \Sh(S)_{\leq 0}
    \end{equation*}
    that send constructible étale sheaves of sets to locally constant constructible sheaves of sets on $ S $.
    In particular, the global sections $ \Gal(X)(\pt) $ recovers the category of points $ \mathrm{Pt}(X_{\et}) $ of the étale topos of $ X $.
\end{recollection}

\begin{recollection}
    Let $ T $ be a spectral space (e.g., the underlying space of a qcqs scheme).
    Then the \topos $ \Sh(T) $ is coherent and by \cite[Lemma 9.5.3 \& Proposition 9.5.4]{Exodromy}, the truncated coherent objects of $ \Sh(T) $ are the constructible sheaves of anima on $ T $.
\end{recollection}

\begin{notation}\label{ntn:Gal_for_spectral_spaces}
    For a spectral space $ T $, we write $ \Gal(T_{\zar}) \colonequals \Gal(\Sh(T)) $.
\end{notation}

\begin{recollection}\label{rec:explicit_description_of_Gal_for_spectral_spaces}
    Let $ T $ be a spectral space.
    Since spectral spaces are sober, by \cite[Example 3.7.1]{Exodromy} and \HTT{Remark}{6.4.5.3}, for a profinite set $ S $, the value $ \Gal(T_{\zar})(S) $ is equivalent to the \textit{poset} of quasicompact maps $f \colon S \to T $ ordered by \textit{pointwise specialization:} $ f \leq g $ if and only if for all $ s \in S $, we have $ f(s) \in \overline{\{g(s)\}} $.
    In particular, $ \Gal(T_{\zar})(\pt) $ recovers the specialization poset of $ T $.
\end{recollection}

\begin{remark}
    Note that the condensed set underlying the condensed poset $\Gal(T_{\zar})$ is indeed a condensed set, i.e., is $ \kappa $-accessible for some $ \kappa $.
    In contrast, the condensed set represented by the topological space $ T $ is typically not $ \kappa $-accessible, see \cite[Warning~2.14]{Scholze:condensednotes}.
    The difference between the two is that $\Gal(T_{\zar})(S)$ is given by the set of \emph{quasicompact} maps $S \to T$, as opposed to all continuous maps.
\end{remark}

\begin{recollection}\label{rem:Gal_is_finite}
   For a qcqs scheme $ X $, the condensed \categories $ \Gal(X) $ and $ \Gal(X_{\zar}) $ are in the image of the fully faithful functor
    \begin{equation*}
        \iota \colon \ICat(\ProAnifin) \to \CondCat 
    \end{equation*}
    of \Cref{obs:categories_internal_to_profinite_anima_as_condensed_categories}.
    In fact, if we denote by $\Lay_{\uppi}$ the full subcategory of $\Catinfty$ spanned by \pifinite layered categories in the sense of \cite[Definition~2.3.7]{Exodromy}, then $\Gal(X)$ and $ \Gal(X_{\zar}) $ are even in the image of the fully faithful functor $\Pro(\Lay_{\uppi}) \to \CondCat$.
    See \cite[\S13.5]{Exodromy} for more details.
\end{recollection}

Now we fix some notation regarding condensed \categories and classifying anima.
\begin{definition}
	We define condensed \categories $ \ICond(\Ani) $ and $ \ICond(\Set) $ by the assignments
	\begin{equation*}
		S \mapsto \CondAni_{/S} \andeq S \mapsto \Cond(\Set)_{/S} \comma
	\end{equation*}
    respectively.
\end{definition}

\begin{notation}
	We denote the left adjoint to the inclusion $ \Ani \inclusion \Catinfty $ by $ \Bup \colon \fromto{\Catinfty}{\Ani} $.
	Given \acategory $ \Ccal $, we call $ \Bup\Ccal $ the \defn{classifying anima} of $ \Ccal $.
\end{notation}

\begin{nul}
	The functor $ \Bup $ preserves finite products.
	Hence post-composition with $ \Bup $ induces a functor
	\begin{equation*}
		\Bcond \colon \fromto{\CondCat}{\CondAni}
	\end{equation*}
	that is left adjoint to the inclusion $ \CondAni \inclusion \CondCat $.
\end{nul}

\begin{definition}
	Given a condensed \category $ \Ccal $, we call $ \Bcond(\Ccal) \in  \CondAni $ the \defn{condensed classifying anima} of $ \Ccal $.
\end{definition}

To see the desired comparison, the idea is that, by \cite[Corollary~1.2]{MR4574234}, we have a natural equivalence
\begin{equation*}
	\Functs(\Gal(X),\ICond(\Ani)) \simeq \Xproethyp \period
\end{equation*}
In other words, in the condensed world, \smash{$ \Xproethyp $} is a presheaf \category on $ \Gal(X)^{\op} $.
But the shape of a presheaf \topos is given by taking the classifying anima of the \category that it is presheaves on; the same holds in the condensed world.

\begin{remark}
    An independent and more direct proof of \cite[Corollary~1.2]{MR4574234} is going to appear in \cite{remy_sebastian_future}.
\end{remark}

\begin{proposition}\label{prop:Picond_is_BGal}
	Let $ X $ be a qcqs scheme.
	Then there is a natural equivalence of condensed anima
	\begin{equation*}
		\Picond(X) \equivalent \BcondGal(X)  \period
	\end{equation*}
\end{proposition}

\begin{proof}
    This follows immediately from combining \cite[Theorem~1.2]{MR4574234} and \cite[Proposition~4.4.1]{MR4752519}.
    For the reader not so familiar with the theory developed in \cite{MR4752519}, we spell out a more hands-on proof.
	Recall that for \categories $ \Ccal $ and $ \Dcal $, the functor
	\begin{equation*}
		 \Fun(\Bup\Ccal,\Dcal) \to \Fun(\Ccal,\Dcal) 
	\end{equation*}
	induced by precomposition along $ \Ccal \to \Bup \Ccal $ is fully faithful (since $ \Bup\Ccal \simeq \Ccal[\Ccal^{-1}] $ is the localization of $ \Ccal $ obtained by inverting all maps, this follows from the universal property of localization).
	Since limits of fully faithful functors are fully faithful \cites[Proposition 2.1]{arXiv:2503.03916}[Proposition A.1.3]{CatrinsThesis}, it follows that precomposition with $ b \colon \Gal(X) \to \BcondGal(X) $ defines a fully faithful functor
	\begin{equation*}
		\begin{tikzcd}
			\Functs(\BcondGal(X),\ICond(\Ani)) \arrow[r, "b^*"] & \Functs(\Gal(X),\ICond(\Ani)) \period
		\end{tikzcd}
	\end{equation*}
	Furthermore, by \cite[Lemma 4.3]{MR4574234} this functor admits a left adjoint $b_\sharp$.
	
    By \cite[Corollary~1.2]{MR4574234} we have a natural equivalence \smash{$ \Xproethyp \simeq \Functs(\Gal(X),\ICond(\Ani)) $}.
	Under this equivalence the functor 
	\begin{equation*}
		\pi^* \colon \CondAni \to \Xproethyp
	\end{equation*}
	 agrees with the functor given by precomposing with the unique morphism $ \Gal(X) \to \ast. $
	We write $ a \colon \BcondGal(X) \to \ast $ for the unique morphism, and obtain a commutative triangle
	\begin{equation*}
		\begin{tikzcd}
			\Functs(\BcondGal(X),\ICond(\Ani)) \arrow[r, "\bupperstar"] & \Xproethyp \\
			\CondAni \arrow[u, "\aupperstar"] \arrow[ur, "\piupperstar"'] & \phantom{\Xproethyp} \period
		\end{tikzcd}
	\end{equation*}
	But now since $ b^* $ is fully faithful and $ b^*(1) = 1 $, it follows that $ \blowersharp(1) = 1 $, 
    Thus,
	\begin{equation*}
		\pilowersharp (1) = \alowersharp \blowersharp(1) = \alowersharp(1) \period
	\end{equation*}
	Finally, by \cite[Corollary 3.20]{MR4574234} we have 
	\begin{equation*}
		\Functs(\BcondGal(X),\ICond(\Ani)) \simeq \CondAni_{/\BcondGal(X)}
	\end{equation*}
	and the functor $ \alowersharp $ identifies with the forgetful functor.
	In particular $ \alowersharp(1) \simeq \BcondGal(X) $.
\end{proof}

\begin{remark}
    In particular, \cref{prop:Picond_is_BGal} shows that if $X$ is a qcqs scheme with finitely many irreducible components, then the underlying group $\pionecond(X,\xbar)(\ast)$ coincides with Gabber's version of the proétale fundamental group, see \cite[Remark 7.4.12]{MR3379634}.
\end{remark}

\begin{corollary}\label{cor:Picond_of_0-dimensional_schemes}
	Let $ X $ be a qcqs scheme.
	If $ \dim(X) = 0 $, then $ \Picond(X) = \Gal(X) $ and this condensed anima is a $ 1 $-truncated profinite anima.
\end{corollary}

\begin{proof}
    This is immediate from \cite[Observation 1.25]{MR4686649} and \Cref{rem:Gal_is_finite}.
\end{proof}

\begin{example}[{($\Picond$ of a field)}]\label{ex:Picond-of-a-field}
  Let $ k $ be a field and choose a separable closure $ \kbar $ of $ k $.
  Write $ \Gal_k $ for the absolute Galois group of $ k $ with respect to $ \kbar $.
  Then the choice of separable closure induces an equivalence 
  \begin{equation*}
    \Picond(\Spec(k)) = \Gal(\Spec(k)) \simeq \BGal_{k} \period
  \end{equation*}
  The left-hand identification follows from \Cref{cor:Picond_of_0-dimensional_schemes}, and the right-hand identification follows from \cite[Examples 11.2.1 and 12.2.1]{Exodromy}.
\end{example}

We do not use the next corollary in the remainder of this article, but we include it for completeness:

\begin{corollary}
    Let $ X $ be a qcqs scheme. 
    If $\dim(X)=0$, then $\CondShape{X}= \ast$ if and only if the reduced scheme $X_{\red}$ is $\Spec(k)$ for $ k $ a separably closed field.
\end{corollary}

\begin{proof}
    As the étale \topos is invariant under universal homeomorphisms, the same holds for $\Gal$ and therefore $\Picond$. 
    As $X \to X_{\red}$ is a universal homeomorphism, the if direction follows by the \Cref{ex:Picond-of-a-field}. 
    For the reverse direction, note that $\Gal(X)(*)=\Pt(X_{\et})$ of a $ 0 $-dimensional affine scheme is contractible only if $X=\Spec(R)$ for $ R $ a local ring with separably closed residue field $ k $. 
    For such a scheme, it is $X_{\red}=\Spec(k)$.
\end{proof}

%-------------------------------------------------------------------%
%  Computation: Π_∞ᶜᵒⁿᵈ of henselian local rings                    %
%-------------------------------------------------------------------%

\subsection{Computation: \texorpdfstring{$\Picond$}{Πᶜᵒⁿᵈ} of henselian local rings}\label{subsec:Henselian_local_rings}

We conclude this section by explaining how to use the definitions to show that the condensed homotopy type of a w-strictly local scheme $ X $ (in the sense of \cref{def:w-strictly-local}) agrees with the profinite set $ \uppi_0(X) $ of connected components of $ X $.
This allows for a direct computation of the condensed homotopy type of a henselian local ring.

\begin{proposition}\label{prop:shape-of-w-strictly-local}
    Let $ X $ be a w-strictly local scheme. 
    Then $ \Picond(X) \simeq \uppi_{0}(X) $.
\end{proposition}

\begin{remark}\label{rem:comparison_with_light}
    Let $ X $ be a qcqs scheme that locally can be written as the spectrum of a countable colimit of finite type $\ZZ$-algebras.
    Then one can show that there is a hypercover $W_\bullet \to X$ consisting of w-strictly local $ X $-schemes with the property that $\uppi_0(X)$ is a light condensed set.
    Hence it follows from \cref{prop:shape-of-w-strictly-local} that in this case $ \Picond(X)$ is a light condensed anima in the sense that it is in the image of the fully faithful functor
    \begin{equation*}
        \Sh(\ProFin_{\aleph_1}) \hookrightarrow \CondAni \period
    \end{equation*}
    For a general scheme $ X $, the condensed homotopy type $\Picond(X)$ need not be light.
\end{remark}

Recall that the proétale site is ``tensored'' over profinite sets (cf.~\cite[Example 4.1.9]{BhattScholzeProetale}).

\begin{recollection}
    Let $ X $ be an affine scheme and $f_{\! 0} \colon S \to \uppi_0(X)$ a map from a profinite set.
    Recall that the pullback of topological spaces $ |X| \times_{\uppi_0(X)} S $ naturally has the structure of an affine scheme that we denote by $ X \otimes_{\uppi_0(X)} S $.
    This affine scheme comes equipped with a proétale map $f \colon X \otimes_{\uppi_0(X)} S \to X$ satisfying $\uppi_0(f) = f_{\! 0}$.
    Moreover, this construction is functorial in both $ X $ and $ S $.
    See \cite[Lemma 2.2.8]{BhattScholzeProetale} for details.
\end{recollection}

\begin{lemma}\label{lem:w-strictly-local-tensored}
    Let $ X $ be an affine scheme and $f_{\! 0} \colon S \to \uppi_0(X)$ a map from a profinite set.
    If $ X $ is w-strictly local, then so is $ X \otimes_{\uppi_0(X)} S $.
\end{lemma}

\begin{proof}
    Write $ X' \colonequals X \otimes_{\uppi_0(X)} S $.
    We can split the construction of $X'$ into two steps: first consider $ X'' = X \otimes S $ coming from ``tensoring'' by $ S $. 
    It satisfies $ \uppi_0(X'') = \uppi_0(X)\times S$. 
    Then realize $X'$ as a closed subscheme of $X''$ that is moreover an intersection of clopen subschemes, by looking at $S \subset \uppi_0(X)\times S = \uppi_0(X'')$ and writing $ S $ as an intersection of clopen subsets in this larger set.

    Let us first check it for $X''$. 
    By definition and \cite[Lemma 2.2.9]{BhattScholzeProetale}, an affine scheme is w-strictly local if it is w-local and all of its connected components are spectra of strictly henselian rings. 
    Here, we are using the following observation: the connected components of a w-local affine scheme are spectra of local rings. 
    Indeed, they are affine (being closed subschemes of an affine scheme) and have a single closed point (by definition of w-locality). 
    Thus, Zariski localizations at closed points of a w-local affine scheme match the corresponding connected components.

    One checks that  both of these conditions are satisfied for $X'' = X \otimes S$ by checking the following facts: 
    \begin{enumerate}
        \item We have $ \uppi_{0}(X \otimes S) = \uppi_{0}(X) \times S$.

        \item Every connected component of $X \otimes S$ is isomorphic (as a scheme) to some connected component of $ X $.

        \item We have $(X \otimes S)_{\cl} \simeq X_{\cl} \otimes S$. 
    \end{enumerate}
    Note that if $S = \lim_{i \in I} S_i$ for finite sets $S_i$, then $X \otimes S$ is defined as an inverse limit of the form $\lim_{i \in I} X^{S_i} = \lim_{i \in I} (X \sqcup \cdots \sqcup X)$ where the transition maps restricted to each copy of $ X $ appearing there are just identities onto another copy of $ X $.
    As a result, each of the above points is reasonably easy to check.
    
    The second step of passing from $X''$ to $X'$ by intersecting an inverse system of clopen subschemes follows similarly.
\end{proof}

\begin{proof}[{Proof of \Cref{prop:shape-of-w-strictly-local}}]
    By \Cref{cor:characterization_of_the_condensed_homotopy_type_as_a_proétale_cosheaf}, this statement holds when $ X $ is \wcontractible. 
    In general, pick a hypercover of the profinite set $\uppi_{0}(X)$ by extremally disconnected profinite sets. 
    By \cite[Lemma 2.2.8]{BhattScholzeProetale}, \Cref{rec:characterisation_weakly_contractible}, and \Cref{lem:w-strictly-local-tensored}, we obtain a proétale hypercover $X_\bullet \to X$ by \wcontractible affine schemes%
    \footnote{Here we have used that the functor in \emph{loc.\ cit.} commutes with limits and respects covers.}
    that recovers the original hypercover of $ \uppi_{0}(X) $ after applying $\uppi_{0}$. 
    We compute 
    \begin{align*}
        \Picond(X) &\equivalent \colim_{[n] \in \Deltaop} \Picond(X_n)  \\ 
        &\equivalent \colim_{[n] \in \Deltaop} \uppi_{0}(X_n) \equivalent \uppi_{0}(X) \comma
    \end{align*}
    as desired.
\end{proof}

We now move on to the promised applications.

\begin{corollary}\label{lem:tensoring-by-profinite-and-pi0}
    Let $ S $ be a profinite set and $ X $ a w-strictly local scheme. Then
    \begin{equation*}
        \Picond(X \otimes S) \simeq \uppi_0(X)\times S \period
    \end{equation*}
\end{corollary}

\begin{proof}
    This follows from \Cref{prop:shape-of-w-strictly-local} and \Cref{lem:w-strictly-local-tensored} with $f_0 = \pr_1 \colon \pi_0(X) \times S \to \pi_0(X)$ together with the equality $\uppi_0(X \otimes S) = \uppi_0(X)\times S$.
\end{proof}

\begin{corollary}\label{cor:Picond_of_henselian_local_rings}
    Let $ R $ be a henselian local ring with residue field $ \kappa $.
    Then the inclusion of the closed point $ \Spec(\kappa) \hookrightarrow \Spec(R)$ induces an equivalence
    \begin{equation*}
        \Picond(\Spec(\kappa)) \isomorphism \Picond(\Spec(R))
    \end{equation*}
    and both are equivalent to $\BGal_{\kappa}$.
\end{corollary}
    
\begin{proof}
    Write $ X = \Spec(R) $ and $ x = \Spec(\kappa) $.
    Fix a separable closure $\overline{\kappa}$ of $ \kappa $ and let $R^{\sh}$ be the corresponding strict henselization. 
    Writing $\overline{\kappa}$ as an increasing union of finite separable extensions (and using that $\FEt_x \simeq \FEt_X$) provides a presentation of $X' = \Spec(R^{\sh})$ as a pro-(finite étale) cover of $ X $, see \stacks{0BSL}. 
    Let $X_\bullet$ be the Čech nerve of this cover $ X' \to X $. 
    As the equivalence $\FEt_x \simeq \FEt_X$ extends to the categories of pro-objects, we compute that $X_\bullet$ writes as
    \begin{equation*}
        \begin{tikzcd}[sep=1.5em]
            \cdots \arrow[r, shift left=0.5ex] \arrow[r, shift right=0.5ex] \arrow[r, shift left=1.5ex] \arrow[r, shift right=1.5ex] & X' \otimes \Gal_\kappa \times \Gal_\kappa \arrow[r] \arrow[r, shift left=1ex] \arrow[r, shift right=1ex] & X' \otimes \Gal_\kappa \arrow[r, shift left=0.5ex] \arrow[r, shift right=0.5ex] & X'
        \end{tikzcd}
    \end{equation*}
    compatibly with the analogous presentation of the Čech nerve $x_\bullet$ of $\xbar = \Spec(\overline{\kappa})) \to \Spec(\kappa) = x$. 
    Applying $\Picond$ to the corresponding ``ladder'' diagram (coming from the map $x_\bullet \to X_\bullet$) and using that, for every $m \in \NN$,
    \begin{equation*}
        \Gal_\kappa^m \simeq \Picond(\xbar \otimes  \Gal_\kappa^m) \to \Picond(X' \otimes  \Gal_\kappa^m) \simeq  \Gal_\kappa^m 
    \end{equation*}
    is an isomorphism (where we are using \Cref{lem:tensoring-by-profinite-and-pi0} and the fact that both $\xbar$ and $X'$ are connected \wcontractible schemes), we conclude.
\end{proof}

%-------------------------------------------------------------------%
%-------------------------------------------------------------------%
%  Connected components of the condensed homotopy type              %
%-------------------------------------------------------------------%
%-------------------------------------------------------------------%

\section{Connected components of the condensed homotopy type}\label{sec:connected_componenets_of_the_condensed_homotopy_type}

Let $ X $ be a qcqs scheme.
In this section, we give an explicit description of the condensed set of connected components $\pizerocond(X) $ of the condensed homotopy type $\Picond(X)$.
To do so, we make use of the Galois category $\Gal(X_{\zar})$ of the Zariski \topos in the sense of \Cref{def:Gal}.
In \cref{subsec:pro-Zariski_sheaves}, we show that the condensed connected components of $ \Bcond\Gal(X_{\zar}) $ agree with \smash{$ \pizerocond(X) $}.
In \cref{subsec:explicit_description_of_picond_0}, we use this description to show that if $ X $ has finitely many irreducible components, then \smash{$ \picond_0(X) $} agrees with the profinite set $ \uppi_0(X) $ of connected components (\Cref{cor:pi0s_match_for_finitely_many_irr_comps}).
We also give examples of connected schemes whose \smash{$ \pizerocond(X) $} is nontrivial and show that \smash{$ \pizerocond(X) $} can be quite exotic in general.
Finally, in \cref{subsec:rings_of_continuous_functions}, we use our explicit description of \smash{$ \picond_0(X) $} to compute the condensed and étale homotopy types of the ring of continuous functions from a compact Hausdorff space to $ \CC $, see \Cref{cor:condensed_homotopy_type_of_rings_of_continuous_functions}.

%-------------------------------------------------------------------%
%  Prozariski sheaves                                               %
%-------------------------------------------------------------------%

\subsection{Prozariski sheaves}\label{subsec:pro-Zariski_sheaves}

Recall that for a scheme $ X $, we will write $X_{\zar}$ for the \topos of Zariski sheaves on $ X $.
In this subsection, we study a pro-version of the Zariski \topos.

\begin{definition}
    Let $ X $ be a qcqs scheme.
    Let us write $X_{\zar}^{\cons} \subset X $ for the full subcategory of Zariski sheaves, that is spanned by the \emph{constructible} sheaves on $ X $, i.e., those sheaves that are locally constant with \pifinite stalks on a finite constructible stratification of $ X $.
    We give $ \Pro(X_{\zar}^{\cons}) $ the \defn{effective epimorphism topology} where covers are generated by finite jointly effectively epimorphic families of maps.
    We call the \topos
    \begin{equation*}
        X^{\hyp}_{\prozar} \colonequals \Sh_{\eff}^{\hyp}(\Pro(X_{\zar}^{\cons}))
    \end{equation*}
    of hypersheaves for the effective epimorphism topology on $\Pro(X_{\zar}^{\cons})$, the hypercomplete \emph{pro\-zariski topos} of $ X $.
    Since pullbacks along qcqs morphisms of schemes preserve constructible sheaves, $X^{\hyp}_{\prozar}$ is functorial in $ X $.
\end{definition}

\begin{remark}
    This construction makes sense more generally for any \emph{bounded coherent} \topos (in the sense of \cite[\SAGapp{A}]{SAG}) and was called \emph{solidification} in \cite{pyknoticI} and \emph{pyknotification} in \cite{zbMATH07671238}.
\end{remark}

\begin{nul}
    Let $ X $ be a qcqs scheme.
    The pullback functor $X_{\zar} \to X_{\et}$ preserves constructible sheaves and thus defines a functor
    \begin{equation*}
        X_{\zar}^{\cons} \to X_{\et}^{\cons} \period
    \end{equation*}
    Extending to pro-objects we obtain a morphism of sites $\rhoupperstar \colon \Pro( X_{\zar}^{\cons}) \to \Pro(X_{\et}^{\cons}) $ and thus an algebraic morphism of \topoi
    \begin{equation*}
        X_{\prozar}^{\hyp} \to \Sh_{\eff}^{\hyp}(\Pro(X_{\et}^{\cons})) \period
    \end{equation*}
    Finally, \cite[Example~7.1.7]{Ultracategories} provides an equivalence $\Xproethyp \simeq \Sh_{\eff}^{\hyp}(\Pro(X_{\et}^{\cons}))$ so that we obtain an algebraic morphism
    \begin{equation*}
        \rhoupperstar \colon X_{\prozar}^{\hyp}  \to \Xproethyp \period
    \end{equation*}
\end{nul}

Recall that a map $ Y \to X$ is a \emph{Zariski localization} if $Y$ is isomorphic (over $ X $) to a finite disjoint union of open subschemes of $ X $.

\begin{nul}
    Let $ X $ be affine scheme.
    We write $\Zaraff{X} \subset \Sch_{/X} $ for the full subcategory spanned by the affine Zariski localizations of $ X $.
    Since open immersions between qcqs schemes are of finite presentation it follows from \stacks{01ZC} that the canonical functor
    \begin{equation*}
        \Pro(\Zaraff{X}) \to \Sch_{/X}
    \end{equation*}
    is fully faithful.
    Thus we may equip $\Pro(\Zaraff{X}) $ with the fpqc topology.
    Since the sheaf represented by a Zariski localization is constructible, we obtain a morphism of sites
    \begin{equation*}
        \mu \colon \Pro(\Zaraff{X}) \to \Pro(X^{\cons}_{\zar}) \period
    \end{equation*}
\end{nul}

\begin{lemma}\label{lem:Pro-Zariski_is_pyknotification}
    Let $ X $ be an affine scheme.
    Then the algebraic morphism of \topoi
    \begin{equation*}
        \muupperstar \colon \Shhyp_{\fpqc}(\Pro(\Zaraff{X})) \to X_{\prozar}^{\hyp}
    \end{equation*}
    is an equivalence.
\end{lemma}

\begin{proof}
    % The Yoneda embedding $\Zaraff{X} \to (X_{\zar}^{\cons})_{\leq 0}$ is fully faithful and thus after extending to pro-objects we get a fully faithful functor
    % \begin{equation*}
    %     \iota \colon \Pro(\Zaraff{X}) \to \Pro((X_{\zar}^{\cons})_{\leq 0}).
    % \end{equation*}

    % Note that any constructible Zariski sheaf admits an effective epimorphism from an affine Zariski localization of $ X $ 
    % Proceeding as in the proof of \cite[Proposition~3.3.8]{pyknoticI}  we see that any object in the domain admits an effective epimorphism from an object in the image of $\iota$.
    % Next we observe that by \cite[Proposition~6.1.24]{Ultracategories} the effective epimorphism topology agrees with the fpqc topology on $\Pro(\Zaraff{X})$, see also .
    % It follows that $ \iota$ induces an isomorphism on categories of hypersheaves
    % \begin{equation*}
    %     X_{\prozar}^{\hyp} \equivalence \Sh_{\eff}^{\hyp}(\Pro((X_{\zar}^{\cons})_{\leq 0})).
    % \end{equation*}
    % Thus the Lemma follows from \cite[Proposition~3.3.9]{pyknoticI} as the Zariski \topos is $1$-localic. 
    The proof is exactly the same as in \cite[Example~7.1.7]{Ultracategories}.
\end{proof}

\begin{remark}
    Let $ X $ be an affine scheme.
    Then under the equivalence of \Cref{lem:Pro-Zariski_is_pyknotification}, the functor $\rhoupperstar$ is induced by the morphism of sites
    \begin{equation*}
        \Pro(\Zaraff{X}) \to \Pro(\Etaff{X}) \comma
    \end{equation*}
    that comes from the inclusion $ \Zaraff{X} \hookrightarrow \Etaff{X}$.
    Here $\Etaff{X}$ denotes the category of affine étale $ X $-schemes.
\end{remark}

\begin{recollection}\label{rec:Galois_category_prozariski}
    For a qcqs scheme $ X $, we write $\Gal(X_{\zar})$ for the Galois category of the Zariski \topos in the sense of \cref{def:Gal}.
    Note that $ X_{\zar} $ is the \topos of sheaves on the spectral topological space $ |X| $.
    Hence by \Cref{rec:explicit_description_of_Gal_for_spectral_spaces}, for a profinite set $ S $, the category of sections $ \Gal(X_{\zar})(S) $ is the poset of continuous quasicompact maps $f \colon S \to \lvert X \rvert$ ordered by pointwise specialization: $f \leq g$ if and only if for all $ s \in S $, we have $ f(s) \in \overline{\{g(s)\}} $.
    In particular, $\Gal(X_{\zar})(\ast)$ is the \emph{specialization poset} of $ |X| $.
    To simplify notation, we denote the specialization poset of $ |X| $ by $ \Zpos{X} $.
\end{recollection}

\begin{lemma}\label{cor:prozar_is_presheaf_category}
    Let $ X $ be a qcqs scheme.
    Then there is a natural equivalence of \topoi
    \begin{equation*}
        X_{\prozar}^{\hyp} \equivalence \Functs(\Gal(X_{\zar}),\ICond(\Ani)) \period
    \end{equation*}
\end{lemma}

\begin{proof}
    Since $X_{\zar}$ is a spectral \topos in the sense of \cite[Definition~9.2.1]{Exodromy} and the profinite stratified shape of $X_{\zar}$ is given by $\Gal(X_{\zar})$, this follows from \cite[Theorem~1.1]{zbMATH07671238}.
\end{proof}

We are interested in \Cref{cor:prozar_is_presheaf_category} because it allows us to compute $\uppi_{0}$ of the relative shape of prozariski \topos over $ \CondAni $ via the condensed classifying anima of $\Gal(X_{\zar}) $.
The latter turns out to be a quotient of the condensed set underlying $\Gal(X_{\zar})$ by an explicit equivalence relation.
Furthermore, the next proposition readily implies that this actually computes $\uppi_{0}^{\cond}(X)$:

\begin{proposition}\label{prop:prozariski_fullyfaithful_proetale_topos}
    The functor $\rhoupperstar \colon X_{\prozar,\leq 0} \to X_{\proet,\leq 0}$ is fully faithful.
\end{proposition}

In order to prove \Cref{prop:prozariski_fullyfaithful_proetale_topos}, we make use of the following construction:

\begin{construction}\label{cons:Zariski-henselization}
    Let $ X $ be an affine scheme.
    Since the inclusion $\Zaraff{X} \hookrightarrow \Etaff{X}$ preserves finite limits, it admits a pro-left adjoint 
    \begin{equation*}
       \HenszarX \colon \Pro(\Etaff{X}) \to \Pro(\Zaraff{X}) \period
    \end{equation*}
    %Explicitly, $\HenszarS$ is the unique cofilered-limit-preserving functor $\Pro(S^{\et}_{\aff}) \to \Pro(S^{\zar}_{\aff})$ that carries an affine étale $ S $-scheme $ X $ to the prozariski localization $(T)_{X \to T}$, where $X \to T$ runs over all $T \in S^{\zar}_{\aff}$ together with a map of $ S $-schemes $X \to T$.
\end{construction}

\begin{definition}[(Zariski henselization)]\label{def:zariski-henselization}
    Let $ X $ be an affine scheme and $Y \in \Pro(\Etaff{X})$.
    We call $\HenszarX(Y)$ the \emph{Zariski henselization of $Y$ in $ X $}.
\end{definition}

\begin{lemma}\label{lem:counit_to_zariskihens_surj}
    Let $ X $ be an affine scheme and \smash{$ V \in \Pro(\Etaff{X}) $}.
    If $ V $ is \wcontractible, the unit morphism \smash{$ V \to \HenszarX(V) $} is surjective.
\end{lemma}

\begin{proof}
    Since $ V $ is \wcontractible, we can use the universal property of $\HenszarX(V)$ to show that any pro-Zariski cover of $\HenszarX(V)$ admits a section.
    This in particular shows that $\HenszarX(V)$ is w-local, see \cite[Lemma~2.4.2]{BhattScholzeProetale}.
    Since $ V \to \HenszarX(V)$ is flat and the image of a flat morphism is closed under generization \cite[Lemma~14.9]{AG1}, it suffices to show that all closed points are in the image.

    We now assume, for the sake of contradiction, that $\im(V) \subset \HenszarX(V)$ does not contain a closed point $x$.
    Since $\im(V)$ is quasicompact, there is some quasicompact open $H \subset \HenszarX(V)$ containing $\im(V)$ such that $x \notin H$.
    Since $H$ is quasicompact, there exists a covering $ (U_i)_{i\in I} $ of $ H $ by finitely many affine opens.
    Since $\im(V) \subset H$, it follows that the induced map
    \begin{equation*}
        \coprod_{i \in I} U_i \times_{\HenszarX(V)} V \to V
    \end{equation*}
    is surjective and thus admits a section \smash{$\alpha \from V \to \coprod_{i \in I} U_i \times_{\HenszarX(V)} V$}.
    By the universal property of Zariski henselization, the composition
    \begin{equation*}
        \begin{tikzcd}
            V \arrow[r, "\alpha"] & \coprod_{i \in I} U_i \times_{\HenszarX(V)} V \arrow[r] & \coprod_{i \in I} U_i
        \end{tikzcd}
    \end{equation*}
    factors uniquely through some $ \tilde{\alpha} \colon \HenszarX(V) \to \coprod_{i \in I} U_i$.
    Since the composite
    \begin{equation*}
        \begin{tikzcd}
            V \arrow[r, "\alpha"] & \coprod_{i \in I} U_i \times_{\HenszarX(V)} V \arrow[r] & \coprod_{i \in I} U_i \arrow[r] & \HenszarX(V)
        \end{tikzcd}
    \end{equation*}
    recovers the unit $V \to \HenszarX(V)$, it follows by uniqueness that the composite
    \begin{equation*}
        \begin{tikzcd}
            \HenszarX(V) \arrow[r, "\alphatilde"] & \coprod_{i \in I} U_i \arrow[r] & \HenszarX(V)
        \end{tikzcd}
    \end{equation*}
    is the identity.
    In particular the $U_i$ cover $\HenszarX(V)$ and thus $H = \HenszarX(V)$; this contradicts that $x \notin H$.
\end{proof}

\begin{lemma}\label{lem:pi^_in_terms_of_hens}
    Let $ X $ be an affine scheme, and \smash{$ F \in X_{\prozar}^{\hyp} $}.
    Then \smash{$ \rhoupperstar(F) \in \Xproethyp $} is the hypersheafification of the presheaf
    \begin{equation*}
        \Pro(\Etaff{X})^{\op} \to \Ani \comma \qquad W \mapsto F(\HenszarX(W)) \period
    \end{equation*}
    Moreover, if $W$ is \wcontractible, then $\rhoupperstar(F)(W) = F(\HenszarX(W))$.
\end{lemma}

\begin{proof}
    The functor $\rhoupperstar$ is given by the hypersheafification of the left Kan extension along the functor
    \begin{equation*}
        \iota\colon \Pro(\Zaraff{X})^{\op} \hookrightarrow \Pro(\Etaff{X})^{\op} \period 
    \end{equation*}
    Explicitly, for $F\in X_{\prozar}^{\hyp}$ the image is given by 
    \begin{equation}\label{eq:rhoupperstar_formula}
        \rhoupperstar(F) = \paren{W \mapsto \colim_{W \to \iota(V)} F(V)}^{\dagger} \comma
    \end{equation}
    where $ V \in \Pro(\Zaraff{X}) $, $ W \in \Pro(\Etaff{X}) $, and $ (-)^{\dagger} $ denotes hypersheafification.
    By the universal property of Zariski henselization, every map $W \to \iota(V)$ factors uniquely over $ \HenszarX(W) $, hence the colimit in \eqref{eq:rhoupperstar_formula} reduces to 
    \begin{equation*}
        \colim_{W \to \iota(V)} F(V) = F(\HenszarX(W)) \period
    \end{equation*}

    It remains to argue why hypersheafification does not change the value on a \wcontractible scheme $ W $.
    On the basis of \wcontractible schemes weakly étale over $ X $, the sheaf condition simplifies to sending finite coproducts to finite products.
    Moreover, every sheaf is a hypersheaf.
    Since $ \HenszarX $, being a left adjoint, preserves finite coproducts and $ F $ carries finite coproducts to finite products, the claim follows. 
\end{proof}

\begin{proof}[Proof of \Cref{prop:prozariski_fullyfaithful_proetale_topos}]
    We can immediately reduce to the case where $ X $ is affine.
    We want to show that for any $F \in X_{\prozar, \leq 0}$ and any $U \in \Pro(\Zaraff{X})$ the unit evaluated at $ U $
    \begin{equation*}
        F(U) \to \rhoupperstar(F)(U)
    \end{equation*}
    is an isomorphism.
    For this, pick a \wcontractible weakly étale $ X $-scheme $W$ with a surjection $W \twoheadrightarrow U$ and a further \wcontractible $ V $ with a surjection $V \twoheadrightarrow W \times_UW$.
    Using \Cref{lem:pi^_in_terms_of_hens}, it suffices to show that the natural map
    \begin{equation*}
        F(U) \to \lim\paren{ F(\HenszarX(W)) \rightrightarrows F(\HenszarX (V)) }
    \end{equation*}
    is an isomorphism.
    This is clear if we show that
    \begin{equation*}
        \HenszarX (V) \rightrightarrows \HenszarX(W) \to U
    \end{equation*}
    is the beginning of an augmented pro-Zariski hypercover.

    For this, first observe that since the surjection $W \twoheadrightarrow U$ factors through the canonical map $\HenszarX(W) \to U$, the rightmost morphism above is surjective.
    Note that we have a commutative diagram
    \begin{equation*}
        \begin{tikzcd}
            V \arrow[r, ->>] \arrow[d] & W \times_U W \arrow[d, ->>] \\
            \HenszarX(V) \arrow[r] & \HenszarX (W) \times_U \HenszarX (W) \period
        \end{tikzcd}
    \end{equation*}
    Here, the top horizontal morphism is surjective by definition and the right vertical morphism is surjective by \Cref{lem:counit_to_zariskihens_surj}.
    Thus the bottom horizontal morphism is also surjective, as desired.
\end{proof}

\begin{warning}
    \Cref{prop:prozariski_fullyfaithful_proetale_topos} is only true on the level of $ 0 $-truncated sheaves, i.e., sheaves of sets. 
    Full faithfulness on the level of sheaves of anima would imply an equivalence of the condensed homotopy type with the relative shape of the the prozariski \topos over $ \CondAni $.
    Therefore, it would also imply that the étale homotopy type of $ X $ agrees with the shape of the underlying topological space of $ X $, which is generally false.
    
    Note that if $ X $ is an everywhere strictly local scheme, by \cite[Corollary 2.5]{MR3649361} one has $ X_{\et} = X_{\zar} $.
    So, in this case $ \rhoupperstar $ is fully faithful for all sheaves of anima.
\end{warning}

%-------------------------------------------------------------------%
%  An explicit description of π₀ᶜᵒⁿᵈ                                %
%-------------------------------------------------------------------%

\subsection{An explicit description of \texorpdfstring{$\uppi_{0}^{\cond}$}{π₀ᶜᵒⁿᵈ}}\label{subsec:explicit_description_of_picond_0}

In this subsection, we give an explicit description of $ \picond_0(X) $.
To do this, we first observe that together the results from \cref{subsec:pro-Zariski_sheaves} show:

\begin{proposition}\label{prop:pi0proet_equals_pi0zar}
	Let $ X $ be a qcqs scheme.
	Then there is a natural isomorphism of condensed sets
	\begin{equation*}
		\picond_{0}(X) \equivalence \uppi_{0}(\BcondGal(X_{\zar})) \period
	\end{equation*}
\end{proposition}

\begin{proof}
    Consider the morphism of sites $ \pitilde \colon \ProFin \to \Pro(X_{\zar}^{\cons}) $ given by $ S \mapsto  S \times X $.
    We have a commutative triangle
    \begin{equation*}
        \begin{tikzcd}[row sep=3em, column sep=1.5em]
             & \CondAni \arrow[dl, "\pitilde^{*}"'] \arrow[ dr, "\piupperstar"] & 
            \\ 
            X_{\prozar}^{\hyp}  \arrow[rr, "\rhoupperstar"'] & & \Xproethyp
        \end{tikzcd}
    \end{equation*}
    Combining \Cref{cor:prozar_is_presheaf_category} and \cite[Lemma~4.3]{MR4574234}, it follows that $\pitilde^*$ has a left adjoint, that we denote $\pitilde_\sharp$.
    By \Cref{prop:prozariski_fullyfaithful_proetale_topos}, it follows that $ \picond_{0}(X) \equivalent \uppi_{0}(\pitilde_\sharp(1)) $.
    By \Cref{cor:prozar_is_presheaf_category}, the same argument as in \Cref{prop:Picond_is_BGal} shows that $\pitilde_\sharp(1) \equivalent \BcondGal(X_{\zar})$.
    Hence 
    \begin{equation*}
        \picond_{0}(X) \equivalent \uppi_{0}(\pitilde_\sharp(1)) \equivalent \uppi_0(\BcondGal(X_{\zar})) \comma  
    \end{equation*}
    as desired.
\end{proof}

\Cref{prop:pi0proet_equals_pi0zar} lets us explicitly describe $ \picond_{0}(X) $.

\begin{remark}\label{rmk:qc-maps-and-constructible}
    Let $ S $ be a profinite set and let $ T $ be a spectral space.
    The next theorem involves sets of continuous quasicompact maps $ \Map_{\qc}(S, T) $.
    Note that these are those maps such that the preimage of a quasicompact open is clopen. 
    It follows that these are precisely continuous maps in the constuctible topology, i.e.,
    \begin{equation*}
        \Map_{\qc}(S,T) = \Map(S, T^{\cons}) \period
    \end{equation*}
    Said differently, the inclusion of the full subcategory of profinite sets into the category of spectral spaces and quasicompact maps admits a right adjoint, given by sending a spectral space $ T $ to the underlying set of $ T $ equipped with the constructible topology.
\end{remark}

\begin{theorem}\label{thm:description_of_pi_0}
	Let $ X $ be a qcqs scheme.
	Then for every extremally disconnected profinite set $ S $, we have
	\begin{equation*}
		\picond_{0}(X)(S) \equivalent \Map_{\qc}(S,\lvert X \rvert)/\kern-0.2em\sim \comma
	\end{equation*}
	where $ f \sim g $ if and only if there is some $ n \in \NN $ and quasicompact maps $ s_1, t_1, \ldots , s_n, t_n \colon S \to \lvert X \rvert $ such that
	\begin{equation*}
		f \geq s_1 \leq t_1 \geq s_2 \leq t_2 \geq \cdots \geq s_n \leq t_n \geq g \period
	\end{equation*}
	Here, $ a \leq b $ if and only if for all $ s \in S $, we have $ a(s) \in \overline{\{b(s)\}}$.

    Moreover, if $S = \upbeta(M)$, restriction along the canonical map $M \to \upbeta(M)$ induces an isomorphism
    \begin{equation*}
        (\Map_{\qc}(S,\lvert X \rvert)/ \kern-0.2em\sim) \equivalence  \uppi_{0}((\Zpos{X})^M) \period
    \end{equation*}
    Here, $\uppi_{0}((\Zpos{X})^M)$ is the quotient of $(\Zpos{X})^M$ identifying two points $(x_m)_{m \in M}$ and $(y_m)_{m \in M}$ if and only if they can be connected by a finite zigzag of pointwise specializations.
\end{theorem}

\begin{proof}
    By \Cref{prop:pi0proet_equals_pi0zar}, the first statement reduces to showing that for every extremally disconnected profinite set $ S $, we have
    \begin{equation*}
        \uppi_{0}(\BcondGal(X_{\zar}))(S)=\Map_{\qc}(S,\lvert X \rvert)/\kern-0.2em\sim \period
    \end{equation*}
    This follows by the description of $\Gal(X_{\zar})$ in \cref{rec:Galois_category_prozariski} noticing that maps $f,g$ in the poset $\Map_{\qc}(S, \lvert X \rvert)$ are connected if and only if there exists a finite zig-zag of pointwise specializations as indicated in the statement. 

    For the second statement, by \Cref{prop:profinitcat_evaluated_at_Stone_chech}, we have a chain of canonical equivalences of partially ordered sets
    \begin{align*}
        \Map_{\qc}(\upbeta(M), \lvert X \rvert) &\equivalent \Gal(X_{\zar})(\upbeta(M)) \\
        &\equivalent \prod_M \Gal(X_{\zar})(*) = \prod_M \Zpos{X} \comma
    \end{align*}
    where the second equivalence is induced by $M \to \upbeta(M)$. 
    Under this identification, the equivalence relation generated by pointwise specialization corresponds to the equivalence relation defining $ \uppi_{0}((\Zpos{X})^M) $ explained in the final statement.
    This concludes the proof of the second claim.
\end{proof}    

\Cref{thm:description_of_pi_0} shows that $ \picond_{0}(X) $ gives the expected answer in many cases of interest:

\begin{corollary}\label{cor:pi0s_match_for_finitely_many_irr_comps}
	Let $ X $ be a qcqs scheme with finitely many irreducible components.
	Then the canonical map of condensed sets
	\begin{equation*}
		\picond_{0}(X) \to \uppi_{0}(X)
	\end{equation*}
	of \cref{nul:map-between-condensed-and-etale-homotopy-type} is an isomorphism.
\end{corollary}

\begin{proof}
	It suffices to check that the map is an isomorphism after evaluating at $ \upbeta(M) $ for any discrete set $ M $.
	By \Cref{thm:description_of_pi_0}, we need to see that the canonical map
	\begin{equation*}
		\uppi_{0}((\Zpos{X})^M) \to \uppi_{0}(X)^M
	\end{equation*}
	that sends a function $ M \to \lvert X \rvert $ to the composite with $ \lvert X \rvert \to \uppi_{0}(X) $ is an isomorphism (note that this is not immediate, since in general $\uppi_0$ does not commute with infinite products).
	It is surjective by surjectivity of $ \lvert X \rvert \to \uppi_{0}(X) $.
	For injectivity, suppose that we have maps $ f,g \colon M \to \lvert X \rvert $ that agree after composing with $ \uppi_{0} $.
	If the number of irreducible components of $ X $ is $ n $, it follows that we may connect any two points $ x,y \in X $ in the same connected component with a zig-zag of specializations involving at most $ 2n +1 $ other points.
	Thus we may also connect $ f $ and $ g $ with a zig-zag involving $ 2n +1 $ other maps and thus $ [f] = [g]$ in $ \uppi_{0}((\Zpos{X})^M) $, as desired.
\end{proof}

\begin{remark}
    For an alternative proof of \Cref{cor:pi0s_match_for_finitely_many_irr_comps}, see \cite[Proposition 2.2.25]{CatrinsThesis}.
\end{remark}

\begin{observation}\label{obs:dependence_of_homotopy_groups_on_basepoints}
    Let $ X $ be a qcqs scheme and let $ \xbar \to X $ and $ \xbar' \to X $ be geometric points.
    If $ X $ is connected and has finitely many irreducible components, then by \Cref{cor:pi0s_match_for_finitely_many_irr_comps}, $ \picond_0(X) = \pt $.
    Hence, for each $ n \geq 1 $, there exists an isomorphism \smash{$ \picond_n(X,\xbar) \equivalent \picond_n(X,\xbar') $}.
\end{observation}

In the remainder of this subsection, we provide some examples illustrating that $\pizerocond(X)$ can substantially differ from $\uppi_0(X)$ in general.
By \Cref{prop:pi0proet_equals_pi0zar}, $\pizerocond(X)$ only depends on the spectral space $ |X| $; so we formulate the following result only in terms of spectral spaces.

\begin{recollection}[{\cite[Chapter 0, \S 2.3]{MR3752648}}]
    A spectral space $ T $ is \defn{valuative} if, for each $ t \in T $, the set of generizations of $ t $ is totally ordered under the generization relation.
    Every point $ t $ of a valuative space $ T $ has a unique maximal generization, denoted $ t^{\max} $.

    The \defn{separated quotient} of a valuative spectral space $ T $ is the quotient $ T^{\sep} \colonequals T/\kern-0.3em\sim $ by the relation $ s \sim t $ if $ s^{\max} \sim t^{\max} $.
    By \cite[Chapter 0, Corollary~2.3.18]{MR3752648}, $T^{\sep} $ is a compact Hausdorff space.
\end{recollection}

For the next result, recall the Galois category of a spectral space from \Cref{ntn:Gal_for_spectral_spaces,rec:explicit_description_of_Gal_for_spectral_spaces}.

\begin{corollary}\label{cor:picond_0_of_valuative_spectral_spaces}
	Let $ T $ be a valuative spectral space.
	Then the natural map
	\begin{equation*}
		\uppi_{0}(\Gal(T_{\zar})) \to T^{\sep}
	\end{equation*}
	is an isomorphism of condensed sets.
\end{corollary}

\begin{proof}
	It again suffices to check this after evaluating at the Čech--Stone $ \upbeta(M) $ of any set $ M $.
	So let $ \alpha \colon \upbeta(M)  \to T^{\sep} $ be any continuous map.
	Since the quotient map $ \pi \colon T \to  T^{\sep}$ is surjective, we may pick a map $ a \from M \to T $ lifting $ \restrict{\alpha}{M} $.
	Using \Cref{prop:profinitcat_evaluated_at_Stone_chech} as in \Cref{thm:description_of_pi_0}, $ a $ extends to a quasicompact continuous map $ \abar \colon \upbeta(M) \to T $ and by construction we have $ \pi \circ \restrict{\abar}{M} = \restrict{\alpha}{M} $.
	By the universal property of Čech--Stone compactification, we thus get $ \pi \circ \abar = \alpha $, proving surjectivity.
	For injectivity, suppose that we are given maps $f,g \colon M \to T $ such that the composites with $ \pi $ agree. 
    By the valuative property, it follows that for any $ m \in  M $, $ f(m) $ and $ g(m) $ specialize to the same maximal element $ h(m) $.
	Thus we get a zig-zag
	\begin{equation*}
		f \leq  h \geq g
	\end{equation*}
	so that $ [f] = [g] $ in $ \uppi_{0}(\Gal(T_{\zar}))(\upbeta(M)) $, proving injectivity.
\end{proof}

\begin{example}\label{ex:adic_disk}
    \Cref{cor:picond_0_of_valuative_spectral_spaces} shows that even if $ X $ is a connected scheme, $ \picond_{0}(X) $ can be a nontrivial condensed set.
	Concretely, we may take $ T $ to be the underlying topological space of the adic unit disk.
	Then $ T $ is a connected spectral topological space, so there exists a ring $ R $ and a homeomorphism $ T \equivalent \lvert \Spec(R) \rvert$. 
    Thus $ \Spec(R) $ is connected but $ \picond_{0}(\Spec(R)) = T^{\sep} $ is a nontrivial compact Hausdorff space. 
    In fact, this space is homeomorphic to the underlying space of the corresponding Berkovich disk (cf.~\cite[Remark 8.3.2]{Huber:book}).
\end{example}

\begin{remark}
    Let $ X $ be a qcqs scheme. Note that $\pizerocond(X)$ is qs. 
    Indeed, this is clearly true for \wcontractible qcqs $ X $ and in general it follows by proétale covering by \wcontractibles and using the following observation: let $X' \to X$ be a proétale surjection. Then the induced map of condensed sets $\pizerocond(X') \to \pizerocond(X)$ is surjective. 
    Indeed, using \Cref{rec:homotopy_groups_of_condensed_anima}, this eventually boils down to the statement that for a map of simplicial sets that is surjective on vertices, the induced map on $\pi_0$ is surjective.
\end{remark}

\Cref{thm:description_of_pi_0} can also be used to show that for a general qcqs scheme $ X $, the condensed set $ \picond_0(X) $ can be quite exotic (in particular, $ \picond_0(X) $ is not generally quasiseparated in the sense of \Cref{rec:quasiseparated_condensed_set}). This is achieved in the following example.

\begin{example}[(schematic Warsaw circle)]\label{example:cond_pi0_and_warsaw_circle}
	Let $ X $ be a qcqs scheme with the property that any two points may be connected by a zig-zag of specializations but such that the minimal length of such a chain is not bounded by any natural number.
	Then we have
    \begin{equation*}
        \picond_0(X)(\ast) \equivalent \ast \period
    \end{equation*}
    However, for any function $ f \colon \NN \to |X| $ such that the minimal length of a zig-zag connecting $ f(n) $ and $ f(0) $ is at least $ n $, the function $ f $ and the constant function at $ f(0) $ yield different elements in $ \picond_0(X)(\upbeta(\NN)) $.
	Thus, $ \picond_0(X) $ is a nontrivial condensed set whose underlying set is the point and therefore not quasiseparated.
    Indeed, if it were quasiseparated it would be qcqs and thus representable by a compact Hausdorff space.
    
    Let us give a concrete example of a scheme satisfying these properties. 
    Fix an algebraically closed field $k$ and write $\ast = \Spec(k)$. 
    Let $X \in \ast_{\proet}$ be a scheme such that $\uppi_{0}(X) = \NN\cup \{\infty \} $, i.e., the converging sequence of points together with its limit. 
    Each connected component of $ X $ is just a copy of $\ast$. 
    Take two copies $X_1^+ = X_2^+ = \AA^1_k\times_{\ast} X$ of a scheme that, intuitively, is a sequence of affine lines converging to another affine line. 
    Fix two points, say $0,1$, on each copy of $\AA^1_k$ and glue $X_1^+$ and $X_2^+$ to obtain a \emph{zigzag} of \smash{$\AA_k^1$}'s intersecting at $ 0 $'s and $1$'s and converging to a copy of \smash{$\AA_k^1$}, as displayed in \Cref{fig:Warsaw}.
    \begin{figure}[!h]
        \centering
        \includegraphics[width=\linewidth]{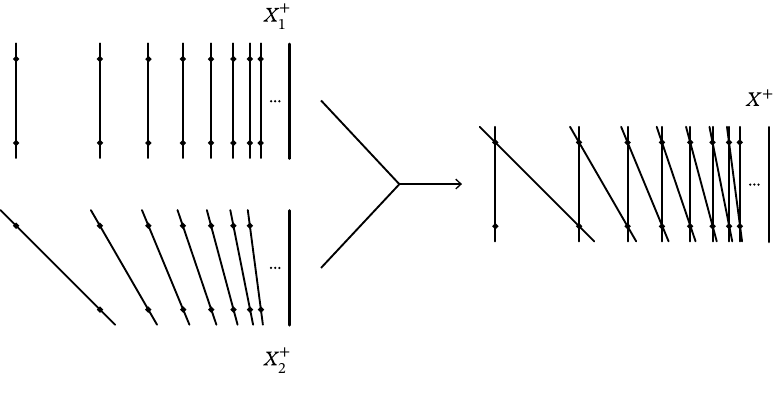}
        \caption{Constructing the scheme $ X^+ $.}
        \label{fig:Warsaw}
    \end{figure}
    Let us denote this scheme simply by $X^+$. 
    To formalize this gluing procedure, one notes that we are gluing affine schemes along closed subschemes, so by \cite[Theorem 3.4]{Schwede:Gluing} the pushout exists and is also affine.
    
    Now, this scheme satisfies the condition of having specialization-distances between points growing arbitrarily but it still needs a small correction: the points on the limit $\AA^1_k$ are not joinable by a specialization sequence with the points on the zigzag. 
    To amend it, add a further copy of $\AA^1_k$ joining an arbitrarily chosen pair of $ k $-points of the the leftmost line of the zigzag with the limit line of $X^+$. 
    Let us denote by $X^{++}$ this schematic `Warsaw circle'. 
    One can check that $X^{++}$ satisfies the desired properties.
\end{example}

%-------------------------------------------------------------------%
%  Computation: Πᶜᵒⁿᵈ of rings of continuous functions              %
%-------------------------------------------------------------------%

\subsection{Computation: \texorpdfstring{$\Picond$}{Πᶜᵒⁿᵈ} of rings of continuous functions}\label{subsec:rings_of_continuous_functions}

Let $ T $ be a compact Hausdorff space.
We conclude this section by using \Cref{thm:description_of_pi_0} to compute the condensed homotopy type of the ring of continuous functions $ \upC(T,\CC) $; we show that it is $ 0 $-truncated, and coincides with the condensed set represented by $ T $.
We accomplish this by proving a more general result.
To state it, recall that the ring $ \upC(T,\CC) $ has the property that every prime ideal is contained in a unique maximal ideal (see \Cref{thm:cont-functions-pm}).
Moreover, \cite[Chapitre VII, Proposition 4]{zbMATH03322176} shows that the local rings of $ \upC(T,\CC) $ at maximal ideals are strictly henselian.
We are able to compute the condensed homotopy types of rings satisfying these two properties.

To state our results, we first introduce some terminology.

\begin{notation}
    Given a ring $ R $, we write $ \MSpec(R) \subset |\Spec(R)| $ for the subset of maximal ideals, endowed with the subspace topology.
\end{notation}

\begin{recollection}[(see \Cref{appendix:rings_of_continuous_functions_and_Cech-Stone_compactification})]
    A ring $ R $ is a \defn{\pmring} if every prime ideal of $ R $ is contained in a unique maximal ideal.
    In this case, the space $ \MSpec(R) $ is compact Hausdorff.
\end{recollection}

First, we identify $ \picond_0 $ of an arbitrary \pmring.

\begin{proposition}\label{prop:pi0cond_of_a_pm-ring}
    Let $ R $ be a \pmring.
    Then there is a natural isomorphism of condensed sets
    \begin{equation*}
        \picond_0(\Spec(R)) \isomorphism \MSpec(R) \period
    \end{equation*}
    This isomorphism is constructed in the course of the proof.
\end{proposition}

\begin{proof}
    By \Cref{thm:continuous-retract}, the map of topological spaces $ \fromto{|\Spec(R)|}{\MSpec(R)} $ that sends a prime ideal $ \pfrak $ to the unique maximal ideal containing $ \pfrak $ is a continuous retraction of the inclusion.
    This retraction is also continuous for the constructible topology and therefore defines a map of condensed sets
    \begin{equation*}
        \Map_{\Top}(- ,|\Spec(R)|^{\cons}) \to \MSpec(R) \period
    \end{equation*}
    Furthermore it clearly respects the equivalence relation described in \Cref{thm:description_of_pi_0} and therefore induces a map
    \begin{equation*}
        \picond_0(\Spec(R)) \to \MSpec(R) \period
    \end{equation*}
    To check that this map is an isomorphism, it suffices to check this after evaluating at $\upbeta(M)$ for any set $M$. 
    Using the explicit description given in \Cref{thm:description_of_pi_0} and the fact that $\MSpec(R)$ is compact Hausdorff (\Cref{cor:max-spec-pm-ring-compact}), this is immediate.
\end{proof}

Under stronger hypotheses, we compute the whole condensed homotopy type:

\begin{theorem}\label{thm:proetale_homotopy_type_of_pm_rings}
    Let $ R $ be a \pmring with the property that all local rings at maximal ideals are strictly henselian.
    Then $ \Picond(\Spec(R)) $ is $ 0 $-truncated; hence there is a natural equivalence of condensed anima
    \begin{equation*}
        \equivto{\Picond(\Spec(R))}{\MSpec(R)} \period
    \end{equation*}
\end{theorem}

\noindent To show that $ \Picond(\Spec(R)) $ is $ 0 $-truncated, we use the description of the condensed homotopy type via exodromy.
We first prove some preparatory results about classifying anima of infinite products.

\begin{lemma}\label{lem:B_commutes_with_arbitrary_products_of_categories_that_admit_a_left_adjoint_from_an_anima}
    Let $ I $ be a set and let $ (\Ccal_i)_{i \in I} $ be \categories.
    Assume that for each $ i \in I $, there exists a left adjoint functor $ \lambda_i \colon \fromto{A_i}{\Ccal_i} $ where $ A_i $ is an anima.
    Then all of the maps in the commutative square 
    \begin{equation*}
        \begin{tikzcd}[sep=3em]
            \Bup(\textstyle \prod_{i \in I} A_i) \arrow[d, "\Bup(\prod_{i \in I} \lambda_i)"'] \arrow[r] & \prod_{i \in I} \Bup A_i \arrow[d, "\prod_{i \in I} \Bup\lambda_i"] \\ 
            \Bup(\textstyle \prod_{i \in I} \Ccal_i) \arrow[r] & \prod_{i \in I} \Bup\Ccal_i  \period
        \end{tikzcd}
    \end{equation*}
    are equivalences of anima.
\end{lemma}

\begin{proof}
    First observe that since each $ \lambda_i $ is a left adjoint, the induced functor on products
    \begin{equation*}
         \textstyle \prod_{i \in I} \lambda_i \colon \prod_{i \in I} A_i \to \prod_{i \in I} \Ccal_i 
    \end{equation*}
    is also a left adjoint.
    Since each $ A_i $ is an anima, the top horizontal map is an equivalence.
    Since $ \prod_{i \in I} \lambda_i $ and each $ \lambda_i $ is a left adjoint and the functor $ \Bup \colon \fromto{\Catinfty}{\Ani} $ sends left adjoints to equivalences \cite[Corollary 2.11]{MR4683160}, the vertical maps are also equivalences.
    Thus, by the 2-of-3 property, the bottom horizontal map is an equivalence, as desired.
\end{proof}

\begin{example}\label{ex:B_commutes_with_arbitrary_products_of_categories_that_are_coproducts_of_categories_with_initial_objects}
    Let $ I $ be a set and let $ (\Ccal_i)_{i \in I} $ be \categories.
    Assume that for each $ i \in I $, each connected component of the \category $ \Ccal_i $ admits an initial object.
    Then the hypotheses of \Cref{lem:B_commutes_with_arbitrary_products_of_categories_that_admit_a_left_adjoint_from_an_anima} are satisfied where each $ A_i $ is the set of initial objects of connected components of $ \Ccal_i $ and $ \lambda_i $ is the inclusion.
    In particular, 
    \begin{equation*}
         \Bup(\textstyle \prod_{i \in I} \Ccal_i) \equivalent \prod_{i \in I} \Bup\Ccal_i
    \end{equation*}
    is $ 0 $-truncated.
\end{example}

We also need the following criterion for detecting when a condensed anima is $ 0 $-truncated:

\begin{lemma}\label{lem:characterization_of_n-truncated_condensed_anima}
    Let $ n \geq 0 $ be an integer.
    Then a condensed anima $ A $ is $ n $-truncated if and only if for each set $ M $, the anima $ A(\upbeta(M)) $ is $ n $-truncated.
\end{lemma}

\begin{proof}
    Since every extremally disconnected profinite set is a retract of the Čech--Stone compactification of a set, this follows from the fact that every retract of an $ n $-truncated anima is $ n $-truncated.
\end{proof}

\begin{proof}[Proof of \Cref{thm:proetale_homotopy_type_of_pm_rings}]
    Note that, in light of \Cref{prop:pi0cond_of_a_pm-ring}, the final statement follows from the claim that $ \Picond(\Spec(R)) $ is $ 0 $-truncated; so we just show this.
    Let us write $ X = \Spec(R) $.
    By \Cref{lem:characterization_of_n-truncated_condensed_anima}, it suffices to show that for every set $ M $, the classifying anima of the category $\Gal(X)(\upbeta(M))$ is $ 0 $-truncated.
    Together, \Cref{rem:Gal_is_finite,prop:profinitcat_evaluated_at_Stone_chech} show that
    \begin{equation*}
        \Gal(X)(\upbeta(M)) \equivalent \prod_{m \in M} \Gal(X)(\{m\}) \equivalent \prod_{m \in M} \Pt(X_{\et}) \period 
    \end{equation*}
    So by \Cref{ex:B_commutes_with_arbitrary_products_of_categories_that_are_coproducts_of_categories_with_initial_objects}, it suffices to show that every connected component of $ \Pt(X_{\et}) $ has an initial object.
    This last statement is immediate from the assumption that $ R $ is a \pmring and all local rings at maximal ideals are strictly henselian.
\end{proof}

We now derive some consequences of \Cref{thm:proetale_homotopy_type_of_pm_rings}.
The first is a computation of the étale homotopy type of these \pmrings, which appears to be new.

\begin{corollary}\label{prop:etale_homotopy_type_of_pm_rings}
    Let $ R $ be a \pmring with the property that all local rings at maximal ideals are strictly henselian.
    Then there is a canonical equivalence of proanima
    \begin{equation*}
        \equivto{\Pietprotrun(\Spec(R))}{\Shapeprotrun(\MSpec(R))} \period
    \end{equation*}
    Here, $ \Shapeprotrun(\MSpec(R)) $ denotes the shape of the compact Hausdorff space $ \MSpec(R) $.
    See \Cref{ntn:shapes_of_topological_spaces}.
\end{corollary}

\begin{proof}
    We apply the functor $(-)\prodisccompl \colon \CondAni \to \ProAnitrun$ to the equivalence in \Cref{thm:proetale_homotopy_type_of_pm_rings}.
    To conclude, note that by \Cref{lem:Picond_recovers_Piet}, we have
    \begin{align*}
        \Picond(\Spec(R))\prodisccompl &\equivalent \Pietprotrun(\Spec(R)) \\ 
    \intertext{and by \Cref{lem:prodiscrete_completion_of_LCH_spaces} we have}
        \MSpec(R)\prodisccompl &\equivalent \Shapeprotrun(\MSpec(R)) \period \qedhere 
    \end{align*}
\end{proof}

Finally, we turn to the special case of rings of continuous functions.

\begin{corollary}\label{cor:condensed_homotopy_type_of_rings_of_continuous_functions}
    Let $ T $ be a topological space and let $ \Cb(T,\CC) $ denote the ring of bounded continuous functions to $ \CC $. 
    Then there are natural equivalences
    \begin{align*}
       \Picond(\Spec(\Cb(T,\CC))) &\equivalence \upbeta(T) \\ 
    \shortintertext{and}
        \Pietprotrun(\Spec(\Cb(T,\CC))) &\equivalence \Shapeprotrun(\upbeta(T)) \period
    \end{align*}
\end{corollary}

\begin{nul}
    Note that if $ T $ is compact Hausdorff, then $ \upbeta(T) = T $ and $ \Cb(T,\CC) = \upC(T,\CC) $. 
\end{nul}

\begin{proof}
    By the universal property of Čech--Stone compactification, the natural map $ T \to \upbeta(T) $ induces an isomorphism of rings
    \begin{equation*}
        \upC(\upbeta(T),\CC) \isomorphism \Cb(T,\CC) \period
    \end{equation*} 
    % This is stated in \cite[Lemma on p. 143]{MR861951}.
    By \Cref{thm:cont-functions-pm}, the ring $ \upC(\upbeta(T),\CC) $ is a \pmring and by \Cref{thm:compact=max-spec} there is a natural homeomorphism $ \isomto{\upbeta(T)}{\MSpec(\upC(\upbeta(T),\CC))} $.
    Furthermore, \cite[Chapitre VII, Proposition 4]{zbMATH03322176} shows that the local rings of $ \upC(\upbeta(T),\CC) $ at maximal ideals are strictly henselian. 
    Thus the claim follows from \Cref{thm:proetale_homotopy_type_of_pm_rings,prop:etale_homotopy_type_of_pm_rings} applied to $ R = \upC(\upbeta(T),\CC) $.
\end{proof}

\begin{remark}
    Let $ T $ be a compact Hausdorff space that admits a CW structure and $ t \in T $.
    Since $ T $ admits a CW structure, the shape $ \Shape(T) $ coincides with the underlying anima of $ T $.
    Hence \Cref{cor:condensed_homotopy_type_of_rings_of_continuous_functions} shows that, up to protruncation, the étale homotopy type of $ \Spec(\upC(T,\CC)) $ coincides with the underlying anima of $ T $.
    In particular, the SGA3 étale fundamental group of $ \Spec(\upC(T,\CC)) $ at the maximal ideal of functions vanishing at $ t $ coincides with the usual fundamental group $ \uppi_1(T,t) $. 
\end{remark}

%-------------------------------------------------------------------%
%-------------------------------------------------------------------%
%  Fiber sequences                                                  %
%-------------------------------------------------------------------%
%-------------------------------------------------------------------%

\section{Fiber sequences}\label{sec:fiber_sequences}

Let $ k $ be a field with separable closure $ \kbar \supset k $, and let $ X $ be a qcqs $ k $-scheme.
Write $ \Xkbar $ for the basechange of $ X $ to $ \kbar $.
Then the naturally null sequence of étale homotopy types
\begin{equation}\label{eq:etale_fundamental_fiber_sequence}
    \begin{tikzcd}[sep=1.5em]
        \Pietprotrun(\Xkbar) \arrow[r] & \Pietprotrun(X) \arrow[r] & \BGal_k
    \end{tikzcd}
\end{equation}
is a fiber sequence, see \cite[Theorem~0.2]{arXiv:2209.03476}.
The existence of this fiber sequence implies the usual fundamental exact sequence for étale fundamental groups \cites[\stackstag{0BTX}]{stacksproject}[Exposé IX, Théorème 6.1]{MR50:7129}.

The first goal of this section, accomplished in \cref{subsec:fundamental_fiber_sequence_for_the_condensed_homotopy_type}, is to prove the analogue of the fundamental fiber sequence \eqref{eq:etale_fundamental_fiber_sequence} for the condensed homotopy type.
The second goal of this section, accomplished in \cref{subsec:gometric_and_homotopy-theoretic_fibers}, is to show that given a smooth proper morphism of schemes $ \fromto{X}{S} $, up to suitable completion, the homotopy-theoretic fiber of the induced map $ \Picond(X) \to \Picond(S) $ agrees with the condensed homotopy type of the scheme-theoretic fiber.
See \Cref{thm:smooth_fiber_sequ}.

%-------------------------------------------------------------------%
%  The fundamental fiber sequence for the condensed homotopy type   %
%-------------------------------------------------------------------%

\subsection{The fundamental fiber sequence for the condensed homotopy type}\label{subsec:fundamental_fiber_sequence_for_the_condensed_homotopy_type}

Using the description of $ \Picond(X) $ as the condensed classifying anima $ \BcondGal(X) $, the same methods as in \cite{arXiv:2209.03476} allow us to prove the fundamental fiber sequence for the condensed homotopy type.
The key observation is that even though $ \Bcond $ does not preserve pullbacks, it preserves pullbacks along morphisms between condensed anima.
Let us now explain this point.

\begin{recollection}
	Let $ \Ccal $ be \acategory with pullbacks and $ \Dcal \subset \Ccal $ a full subcategory such that the inclusion admits a left adjoint $ L \colon \fromto{\Ccal}{\Dcal} $.
	We say that the localization $ L $ is \textit{locally cartesian} if for any cospan $ U \to W \ot V $ in $ \Ccal $ with $ U,W \in \Dcal $, the natural map
	\begin{equation*}
		\fromto{L(U \cross_W V)}{U \cross_W L(V)}
	\end{equation*}
	is an equivalence.
	See \cites[\S1.2]{MR3641669}[\S3.2]{MR3570135}.
\end{recollection}

\begin{nul}\label{nul:B_is_locally_cartesian}
    Importantly, the localization $ \Bup \colon \fromto{\Catinfty}{\Ani} $ is locally cartesian; see \cite[ Example 3.4]{arXiv:2209.03476}.
\end{nul}

\begin{corollary}\label{cor:Lcond_locally_cartesian}
	Let $ \Ccal $ be \acategory with finite limits and let $ L \colon \fromto{\Ccal}{\Dcal} $ be a locally cartesian localization that also perserves finite products.
	Then the localization $ L^{\cond} \colon \fromto{\Cond(\Ccal)}{\Cond(\Dcal)} $ is locally cartesian.
\end{corollary}

\begin{proof}
	By definition, the functor
	\begin{equation*}
		L^{\cond} \colon \fromto{\Funcross(\Extrop,\Ccal)}{\Funcross(\Extrop,\Dcal)}
	\end{equation*}
	is given by pointwise application of $ L \colon \fromto{\Ccal}{\Dcal} $.
	Since finite limits in $ \Cond(\Ccal) $ and $ \Cond(\Dcal) $ are computed pointwise, the claim follows from the assumption that the localization $ L $ is locally cartesian.
\end{proof}

\begin{example}\label{ex:Bcond_locally_cartesian}
	The localization $ \Bcond \colon \fromto{\CondCat}{\CondAni} $ is locally cartesian.
\end{example}

\begin{corollary}\label{cor:proetale_fundamental_fiber_sequence}
	Let $ f \colon \fromto{X}{S} $ be a morphism between qcqs schemes, and let $ \fromto{\sbar}{S} $ be a geometric point of $ S $.
	If $ \dim(S) = 0 $, then the naturally null sequence 
	\begin{equation*}
		\begin{tikzcd}[sep=1.5em]
			\Picond(\Xsbar) \arrow[r] & \Picond(X) \arrow[r] & \Picond(S) 
		\end{tikzcd}
	\end{equation*}
	is a fiber sequence in the \category $ \CondAni $.
    As a consequence, given a geometric point $ \xbar \to X_{\sbar} $, the induced sequence of pointed condensed sets
    \begin{equation*}
        \begin{tikzcd}[sep=1.25em]
            1 \arrow[r] & \picond_1(\Xsbar,\xbar) \arrow[r] & \picond_1(X,\xbar) \arrow[r] & \picond_1(S,\sbar) \arrow[r] & \picond_0(\Xsbar) \arrow[r] & \picond_0(X) \arrow[r] & \picond_0(S)
        \end{tikzcd}
    \end{equation*} 
    is exact.
\end{corollary}

\begin{proof}
	For the first claim, note that by \cite[Corollary 2.4]{arXiv:2209.03476} and the fact that the functor $\Pro(\Catinfty) \to \CondCat$ preserves limits, the natural square
	\begin{equation*}
		\begin{tikzcd}
			\Gal(\Xsbar) \arrow[r] \arrow[d] & \Gal(X) \arrow[d] \\ 
			\Gal(\sbar) \arrow[r] & \Gal(S)
		\end{tikzcd}
	\end{equation*}
	is a pullback square in $ \CondCat $.
	Moreover, since $ \sbar $ is a geometric point, $ \Gal(\sbar) \equivalent \pt $.
	Since $ \dim(S) = 0 $, by \Cref{cor:Picond_of_0-dimensional_schemes} the condensed \category $ \Gal(S) $ is a $ 1 $-truncated condensed anima.
	The claim now follows from \Cref{prop:Picond_is_BGal} and the fact that the localization $ \Bcond $ is locally cartesian.

    To conclude, note that since $ \Picond(S) \equivalent \Gal(S) $ is $ 1 $-truncated, the second claim follows from the first by taking homotopy condensed sets.
\end{proof}

\begin{corollary}\label{cor:fundamental_exact_sequence_for_condensed_homotopy_groups}
    Let $ k $ be a field with separable closure $ \kbar $, let $ X $ be a qcqs $ k $-scheme, and fix a geometric point $ \xbar \to X_{\kbar} $.
    If $ \picond_0(X_{\kbar}) = 1 $, then the sequence of condensed groups
    \begin{equation*}
        \begin{tikzcd}[sep=1.25em]
            1 \arrow[r] & \picond_1(X_{\kbar},\xbar) \arrow[r] & \picond_1(X,\xbar) \arrow[r] & \Gal_k \arrow[r] & 1
        \end{tikzcd}
    \end{equation*} 
    is exact.
    \hfill\qedhere
\end{corollary}

\begin{remark}\label{rmk:fundamental_exact_sequence_for_condensed_homotopy_groups_under_geometric_connectedness}
    By \Cref{cor:pi0s_match_for_finitely_many_irr_comps}, the hypotheses of \Cref{cor:fundamental_exact_sequence_for_condensed_homotopy_groups} are satisfied if $ X $ is geometrically connected and $ X_{\kbar} $ has finitely many irreducible components.
\end{remark}

As an application of the fundamental fiber sequence and \Cref{cor:condensed_homotopy_type_of_rings_of_continuous_functions}, we compute of the condensed homotopy type of rings of continuous functions to $ \RR $:

\begin{corollary}
    Let $ T $ be a compact Hausdorff space.
    Then there is a natural equivalence of condensed anima
    \begin{equation*}
        \Picond(\Spec(\upC(T,\RR))) \simeq T \times \BGal_{\RR} \period
    \end{equation*}
\end{corollary}

\begin{proof}
    As explained in \Cref{lemma:real-vs-complex-functions}, the natural ring homomorphism $ \upC(T,\RR) \tensor_{\RR} \CC \to \upC(T,\CC) $ is an isomorphism.
    Hence by the fundamental fiber sequence
    \begin{equation*}
        \Picond(\Spec(\upC(T,\CC))) \to  \Picond(\Spec(\upC(T,\RR))) \to \BGal_{\RR}
    \end{equation*}
    of \Cref{cor:proetale_fundamental_fiber_sequence}, we just have to show that action of $ \Gal_{\RR}$ on $ \Picond(\Spec(\upC(T,\CC))) $ is trivial.
    By \Cref{thm:proetale_homotopy_type_of_pm_rings}, we have natural identifications
    \begin{equation*}
        \Picond(\Spec(\upC(T,\CC))) \simeq \MSpec(\upC(T,\CC)) \simeq T \period
    \end{equation*}
    Thus it suffices to show that map on maximal spectra
    \begin{equation*}
        \MSpec(\upC(T,\CC)) \to \MSpec(\upC(T,\CC))
    \end{equation*}
    induced by complex conjugation is the identity.
    To see this, note that by \Cref{thm:compact=max-spec}, each maximal ideal is given by all functions $ T \to \CC $ that vanish at some fixed $ t \in T $, and a function vanishes at a point if and only if its conjugate does.
\end{proof}

%-------------------------------------------------------------------%
%  Geometric and homotopy-theoretic fibers                          %
%-------------------------------------------------------------------%

\subsection{Geometric and homotopy-theoretic fibers}\label{subsec:gometric_and_homotopy-theoretic_fibers}

Let $ f \colon X \to S $ be a smooth and proper morphism of schemes.
The goal of this subsection is is to show that, up to suitable completion, the homotopy-theoretic fiber of the induced map $ \Picond(f) \colon \Picond(X) \to \Picond(S) $ agrees with the condensed homotopy type of the scheme-theoretic fiber.

\begin{notation}
	For a morphism of schemes $ f \colon X \to  S $ and a geometric point $ \sbar \to S $, we denote by 
	\begin{equation*}
	    X_{(\sbar)} \colonequals X \times_S S_{(\sbar)}
	\end{equation*}
	the \defn{Milnor ball of $ f $ at $ \sbar $} .
    Here $S_{(\sbar)}$ denotes the strict localization at $\sbar$.
\end{notation}

\begin{recollection}[(\Sigmacompletion)]
    Let $\Sigma$ be a nonempty set of prime numbers.
    \begin{enumerate}
        \item We write $\Ani_{\Sigma} \subset \Anifin$ for the full subcategory spanned by those \pifinite anima all of whose homotopy groups are $\Sigma$-groups (i.e., their order is a product of elements of $\Sigma$).

        \item The inclusion $\Pro(\Ani_\Sigma) \inclusion \ProAnifin$ admits a left adjoint $(-)\Sigmacomp$ that we refer to as \emph{\Sigmacompletion}.

        \item We also write $(-)\Sigmacomp \from \CondAni \to \Pro(\Ani_\Sigma)$ for the left adjoint of the inclusion 
        \begin{equation*}
          \Pro(\Ani_\Sigma) \inclusion \ProAnifin \inclusion \CondAni \period
        \end{equation*}
    \end{enumerate}
\end{recollection}

As a consequence of the exodromy description of the condensed homotopy type, we can apply a profinite version of Quillen's Theorem B, see \cref{sec:profin_thm_B}, to prove:

\begin{theorem}\label{thm:smooth_fiber_sequ}
	Let $ f \colon X \to S $ be a smooth and proper morphism between qcqs schemes and let $\sbar \to S $ be a geometric point.
	Let $ \Sigma $ be a nonempty set of primes invertible on $ S $.
	Then the induced map
	\begin{equation*}
	    \Picond(X_{\sbar}) \to \fib_{\sbar}( \Picond(f))
	\end{equation*}
	becomes an equivalence after completion with respect to $ \Sigma $.
\end{theorem}

\begin{proof}
    We want to apply \Cref{thm:profinThmB} to the functor $\Gal(f) \colon \Gal(X) \to \Gal(S)$ induced by $ f $.
    To verify that the assumptions of \Cref{thm:profinThmB} are satisfied, we need to see that for any specialization $ \eta \colon \tbar' \to \tbar $ in $ S $, the induced map
    \begin{equation}\label{eq:paralell_transport_in_theorem_B}
        \Bcond(\Gal(X)_{\tbar/}) \to \Bcond(\Gal(X)_{\tbar'/})
    \end{equation}
    becomes an equivalence after \Sigmacompletion.

    Recall that by \cite[Corollary~12.4.5]{Exodromy}, we have a natural equivalence of underlying \categories
    \begin{equation}\label{eq:slices_of_Gal}
        \Gal(S_{(\tbar)}) \equivalence \Gal(S)_{\tbar/} \period
    \end{equation}
    Using \Cref{lem:lim_is_cons} below, one can show that this equivalence refines to an equivalence of condensed \categories, see \cite[Proposition~7.3.3.7]{Sebastian_Wolf-thesis} for more details.
    Furthermore, \cite[Proposition~2.4]{MR4686649} implies, that the natural functor
    \begin{equation*}
        \Gal(X_{(\tbar)}) \to \Gal(X)_{\tbar/} \comma
    \end{equation*}
    induced by the equivalence \eqref{eq:slices_of_Gal}, is an equivalence of condensed \categories as well.
    Thus by \Cref{lem:Picond_recovers_Piet}, the \Sigmacompletion of the map \eqref{eq:paralell_transport_in_theorem_B} identifies with the specialization map
    \begin{equation*}
        \Pietprofin(X_{(\tbar)})\Sigmacomp \to \Pietprofin(X_{(\tbar')})\Sigmacomp \period
    \end{equation*}
    By \cite[Proposition~2.49]{arXiv:2304.00938}, this specialization map is an equivalence.
    Thus, \Cref{thm:profinThmB} implies that the natural map $\Picond(X_{(\sbar)}) \to \fib_{\sbar}( \Picond(f))$ becomes an equivalence after \Sigmacompletion.
    Finally, note that by \Cref{lem:Picond_recovers_Piet} and \cite[Corollary~2.39]{MR4835288}, the natural map
    \begin{equation*}
        \Picond(X_{\sbar}) \to\Picond(X_{(\sbar)})
    \end{equation*}
    becomes an equivalence after \Sigmacompletion.
\end{proof}

\begin{remark}
    In the setting of \Cref{thm:smooth_fiber_sequ}, the canonical map $\Picond(X_{\sbar}) \to \fib_{\sbar}( \Picond(f))$ is not generally an equivalence before \Sigmacompletion.
    The reason why this fails is that the proper and smooth basechange theorems do not hold for arbitrary proétale sheaves; they only hold for constructible étale sheaves.
\end{remark}

\begin{remark}
    \Cref{thm:smooth_fiber_sequ} is an analogue of Friedlander's result \cite[Theorem~3.7]{MR352099}.
    Since we do not have to require that the base $ S $ be normal, at the cost of working with a more complicated homotopy type, our result holds in a more general setup.
    However, since the \Sigmacompletion functor does not preserve fiber sequences, it is also not immediate how to recover Friedlander's result from ours.
\end{remark}

%-------------------------------------------------------------------%
%-------------------------------------------------------------------%
%  Integral Descent                                                 %
%-------------------------------------------------------------------%
%-------------------------------------------------------------------%

\section{Integral Descent}\label{sec:integral-descent}

The goal of this section is to prove that the condensed homotopy type satisfies integral hyperdescent.
Let us start by formulating what we mean by this more precisely.

\begin{definition} 
    Let $ X $ be a scheme and $ \Ccal $ an \category.
    \begin{enumerate}
        \item We call an augmented simplical object $X_\bullet \to X$ an \defn{integral hypercover} if for each $n \geq 0 $, the morphism $X_n \to X$ is integral and $X_0 \to X$ and $X_n \to (\cosk_{n-1}(X_\bullet))_n$ are surjective.

        \item We call a functor $F \colon \Sch^{\qcqs} \to \Ccal$ a \defn{hypercomplete integral cosheaf} if $ F $ sends integral hypercovers to colimit diagrams.
    \end{enumerate}
\end{definition}

The main goal of \cref{subsec:integral_basechange_for_proetale_hypersheaves} is to show that $\Picond(-)$ is a hypercomplete integral cosheaf, which we achieve in \Cref{cor:integral-hyperdescent-main-result}.
In fact, our methods will show that already $\Gal(-)$ is a hypercomplete integral cosheaf of condensed categories.
In \cref{sec:strongly_künneth}, we use some of the results in this section to characterize those morphisms of schemes, for which the étale \topos is compatible with basechange; this included integral morphisms.

%-------------------------------------------------------------------%
%  Integral basechange for proétale hypersheaves                    %
%-------------------------------------------------------------------%

\subsection{Integral morphisms and right fibrations}\label{subsec:integral_basechange_for_proetale_hypersheaves}

In this subsection, we show that for an integral morphism of schemes, the induced functor on Galois categories is a right fibration of condensed categories.
We begin by recalling the notion of a right fibration of condensed \categories:

\begin{definition}\label{def:right_fibration_condensed}
   We say that a functor of condensed \categories $f\colon \Ccal \to \Dcal$ is a \defn{right fibration} if and only if the commutative square
    \begin{equation*}
        \begin{tikzcd}[column sep=3em]
            \Funcond([1],\Ccal) \arrow[r, "f \circ \blank"] \arrow[d, "\ev_1"'] & \Funcond([1],\Dcal) \arrow[d, "\ev_1"] \\
            \Ccal \arrow[r, "f"'] & \Dcal
        \end{tikzcd}
    \end{equation*}
    is a cartesian square in $ \CondCat $.
\end{definition}

\begin{remark}
    \Cref{def:right_fibration_condensed} is a special case of the notion of a right fibration of simplicial objects in a general \topos $ \Bcal $, as introduced in \cite[Definition~4.1.1]{arXiv:2103.17141}.
    In particular it follows from the discussion in \textit{loc. cit.} that right fibrations in $ \Fun(\Deltaop,\CondAni) $ are the right class in an orthogonal factorization system.
    The left class consists of the \emph{final} maps, i.e., the smallest saturated class which contains all maps of the form $\{n\} \times S \hookrightarrow [n] \times S$ for $n \in \NN$ and $S \in \ProFin$.
    See \cite[Lemma~4.1.2]{arXiv:2103.17141}.
\end{remark}

\begin{remark}\label{rem:right_fibrations_pointwise}
    A functor $f\colon \Ccal \to \Dcal$ of condensed \categories is a right fibration if and only if for every profinite set $ S $, the functor $f(S)\colon \Ccal(S) \to \Dcal(S)$ is a right fibration of \categories.
    Indeed, the square in \Cref{def:right_fibration_condensed} is cartesian if and only if this is true after evaluation at every profinite set $ S $.
    Under the equivalence $\Funcond([1],\Ccal)(S)\simeq\Fun([1], \Ccal(S))$, the claim then follows by the characterization of right fibrations via a corresponding cartesian square, see \cite[Proposition~3.4.5]{MR3931682}. 
\end{remark}

In the cases we care about, being a right fibration can often be detected on the level of underlying \categories, which we deduce from the following observation.

\begin{observation}\label{lem:lim_is_cons}
    Recall from \SAG{Theorem}{E.3.1.6} that the functor
    \begin{equation*}
        \lim \colon \ProAnifin \to \Ani
    \end{equation*}
    is conservative.
    It follows that the functor $ \lim_* \colon \ICat(\ProAnifin) \to \Catinfty $ given by postcomposition with $\lim$ is also conservative.
\end{observation}

\begin{lemma}\label{lem:left_fib_underlying}
	Let $ f \colon \Ccal \to \Dcal $ be a functor in $ \ICat(\ProAnifin) $ considered as a functor of condensed \categories.
	If the underlying functor of $ \infty $-categories is a right fibration, then $ f $ is a right fibration of condensed \categories.
\end{lemma}

\begin{proof}
    By definition, $ f $ is a right fibration if and only if the induced map
    \begin{equation}\label{eq:right_fibration_comparison}
        \Funcond([1],\Ccal) \to \Funcond([1],\Dcal) \times_{\Dcal} \Ccal
    \end{equation}
    is an equivalence of condensed \categories.
    Since $\Ccal $ and $\Dcal$ are in $\ICat(\ProAnifin) $, it follows that $\Funcond([1],\Ccal) $ and $ \Funcond([1],\Dcal)$ are also in $\ICat(\ProAnifin) $.
    Thus, by \Cref{lem:lim_is_cons}, the comparison map \eqref{eq:right_fibration_comparison} is an equivalence if and only if it an equivalence on underlying \categories.
    Since taking underlying \categories commutes with pullbacks, this proves the claim.
\end{proof}

By \Cref{rem:Gal_is_finite}, we immediately deduce the following.

\begin{corollary}\label{cor:right-fibration-iff-underlying}
    Let $f\from X \to Y$ be a morphism of qcqs schemes. 
    Then the induced functor 
    \begin{equation*}
      \Gal(f) \from \!\Gal(X) \to \Gal(Y)
    \end{equation*}
    is a right fibration of condensed categories if and only if this is true on the underlying categories.
\end{corollary}

\begin{proposition}\label{lem:integral-morphism-is-right-fibration-underlying}
    Let $f \from X \to Y$ be an integral morphism of qcqs schemes. 
    Then the induced functor
    \begin{equation*}
        \Gal(f) \from \Gal(X) \to \Gal(Y)
    \end{equation*}
    is a right fibration of condensed categories.
\end{proposition}

\begin{proof}
    By \Cref{cor:right-fibration-iff-underlying}, it suffices to check this on underlying categories.
    The statement about underlying categories appears in \cite[Proposition~14.1.6]{Exodromy}; for the convenience of the reader, we give a quick proof here.

    Throughout the proof, we simply write $\Gal(\blank)$ for the underlying category as well.
    By \stacks{09YZ}, any integral morphism $f \from X \to Y$ with $Y$ qcqs can be written as $f = \lim_{i} f_{i}$ for some cofiltered system of \emph{finite} morphisms $f_{i} \from X_{i} \to Y$.
    Since right fibrations are stable under limits, by the continuity of étale \topoi \cites[Éxpose~VII, Lemma~5.6]{SGA4ii}[Proposition~3.10]{MR4296353}, we may assume that $ f $ is finite.
    Since $\Gal(X)$ and $\Gal(Y)$ are $1$-categories, by \kerodon{015H} it suffices to show that any lifting problem of the form
    \begin{equation*}
        \begin{tikzcd}
            \{1\} \arrow[d, hook] \arrow[r] & \Gal(X) \arrow[d, "{\Gal(f)}"] \\
            {[1]} \arrow[r, "s"'] \arrow[ru, dotted, "{^{\exists!}?}"] & \Gal(Y).
        \end{tikzcd}
    \end{equation*}
    has a \emph{unique} solution.
    Writing $\ybar$ for the source of the map $s$, this diagram factors as 
    \begin{equation*}
        \begin{tikzcd}
            \{1\} \arrow[d, hook] \arrow[r] & \Gal(Y)_{\ybar/} \times_{\Gal(Y)} \Gal(X) \arrow[dr, phantom, very near start, "\lrcorner"{xshift=-3ex}] \arrow[d] \arrow[r] & \Gal(X) \arrow[d, "{\Gal(f)}"] \\
            {[1]} \arrow[ru, dotted, "{^{\exists!} ?}"] \arrow[r] \arrow[rr, bend right = 2em, "s"'] & \Gal(Y)_{\ybar/} \arrow[r] & \Gal(Y) \comma
        \end{tikzcd}
    \end{equation*}
    and it suffices to show that this induced lifting problem has a unique solution.

    By \cite[Corollary~12.4.5]{Exodromy} and \cite[Corollary~2.4]{MR4686649}, we can identify 
    \begin{equation*}
        \Gal(Y)_{\ybar/} \simeq \Gal(Y_{(\ybar)}) \andeq \Gal(X) \times_{\Gal(Y)} \Gal(Y_{(\ybar)}) \simeq \Gal(X \times_{Y} Y_{(\ybar)}) \period
    \end{equation*}
    Moreover, since $f \from X \to Y$ is finite, by \stacks{04GH} we have a coproduct decomposition $X \times_{Y} Y_{(\ybar)} = \coprod_{\xbar_i \in f^{-1}(\ybar)} X_{(\xbar_i)}$.
    Now the map 
    \begin{equation*}
        \{1\} \to \Gal(Y_{(\ybar)}) \times_{\Gal(Y)} \Gal(X) \simeq \coprod_{i} \Gal(X_{(\xbar_{i})})
    \end{equation*}
    factors through $\Gal(X_{(\xbar_{i_{0}})})$ for some $i_{0}$.
    Hence, writing $\xbar \colonequals \xbar_{i_{0}}$, we finally arrive at a lifting problem of the form
    \begin{equation*}
        \begin{tikzcd}
            \{1\} \arrow[d, hook] \arrow[r] & \Gal(X_{(\xbar)}) \arrow[d] \arrow[r] & \Gal(X) \arrow[d, "{\Gal(f)}"] \\
            {[1]} \arrow[r] \arrow[rr, bend right = 2em, "s"'] \arrow[ru, dotted, "{^{\exists!}?}"] & \Gal(Y_{(\ybar)}) \arrow[r] & \Gal(Y).
        \end{tikzcd}
    \end{equation*}
    Here, existence and uniqueness of a lift is clear.
    Let $ \ybar' $ be the target of the map $s$, determined by $\{1\} \to \Gal(X_{(\xbar)})$.
    Note that $\xbar$ is the initial object of $\Gal(X_{(\xbar)}) \simeq \Gal(X)_{\xbar/}$, and also the only object lifting $\ybar$.
    So if there exists a lift, it has to be the unique map from $\xbar \to \xbar'$ for $\xbar'$ the lift of $\ybar'$. 
    Since $\ybar$ is the initial object of $\Gal(Y_{(\ybar)}) \simeq \Gal(Y)_{\ybar/}$, it is clear that $\xbar \to \xbar'$ actually lifts the map $s \from \ybar \to \ybar'$ we started with.
\end{proof}

\begin{corollary}[(Künneth formula for integral morphisms)]\label{cor:Integral_maps_stable_under_base_change}
    Let $X \to Y$ be an integral morphism of qcqs schemes. 
    Then for any qcqs scheme $Y'$ and morphism $Y' \to Y$ the natural functor
    \begin{equation*}
        \Gal(X \times_Y Y') \to \Gal(X) \times_{\Gal(Y)} \Gal(Y')
    \end{equation*}
    is an equivalence.
\end{corollary}

\begin{proof}
    As integral morphisms and right fibrations are stable under pullbacks, by~\Cref{lem:integral-morphism-is-right-fibration-underlying} both functors
    \begin{equation*}
        \Gal(\pr_1) \colon \Gal(X \times_Y Y') \to \Gal(Y') \andeq \pr_1 \colon \Gal(X)\times_{\Gal(Y)} \Gal(Y') \to \Gal(Y')
    \end{equation*}
    are right fibrations.
    Therefore, by \kerodon{01VE} it suffices to see that the natural functor
    \begin{equation*}
        \Gal(X \times_Y Y') \to \Gal(X) \times_{\Gal(Y)} \Gal(Y')
    \end{equation*}
    becomes an equivalence after taking fibers over any $\ybar' \in \Gal(Y')$.
    This holds by \cite[Corollary~2.4]{MR4835288}.
\end{proof}

\begin{lemma}\label{lem:surjective_functors_of_profinite_categories}
    Let $ f \colon \Ccal \to \Dcal $ be a morphism in $ \ICat(\ProAnifin) $.
    Then $ f $ is surjective as a functor of condensed \categories (i.e., for all $S \in \Extr$, the functor $\Ccal(S) \to \Dcal(S)$ is surjective) if and only if the induced functor on underlying \categories $ f(\pt) \colon \Ccal(\pt) \to \Dcal(\pt) $ is surjective.
\end{lemma}

\begin{observation}
    The inclusion $\CondAni \to \CondCat$ also admits a right adjoint.
    We denote this right adjoint by $(-)^{\simeq}$.
\end{observation}

\begin{proof}[Proof of \Cref{lem:surjective_functors_of_profinite_categories}]
    First, by definition, if $ f $ is a surjective functor of condensed \categories, then $ f(\pt) \colon \Ccal(\pt) \to \Dcal(\pt) $ is surjective.
    Conversely, if $ f(\pt) \colon \Ccal(\pt) \to \Dcal(\pt) $ is surjective, then it follows from \SAG{Corollary}{E.4.6.3} that the induced map $ \Ccal^\simeq \to \Dcal^\simeq$ is an effective epimorphism in $\ProAnifin \subset \CondAni$.
    Now let $S \in \Extr$.
    Since any map $S \to \Dcal$ in $\CondCat$ factors through $\Dcal^\simeq$ and $ S $ is projective in $\CondAni$ it follows that we can find a lift in the diagram
    \begin{equation*}
        \begin{tikzcd}
            & \Ccal^{\simeq} \arrow[d, "f", two heads] \\
            S  \arrow[r] \arrow[ur, dotted] & \Dcal^{\simeq}
        \end{tikzcd}
    \end{equation*}
    which completes the proof.
\end{proof}

\begin{corollary}\label{lem:Gal_is_surj}
    Let $f \colon X \to Y$ be a surjective morphism of qcqs schemes.
    Then the functor of condensed categories $\Gal(f) \colon \Gal(X) \to \Gal(Y)$ is surjective.
\end{corollary}

\begin{proof}[Proof of \Cref{lem:Gal_is_surj}]
    By \Cref{lem:surjective_functors_of_profinite_categories}, we just need to see that the induced functor on categories of points $ \Gal(X)(\pt) \to \Gal(Y)(\pt) $ is surjective.
    Since any point of $X_{\et}$ is represented by a geometric point $\xbar \to X$, it is clear.
\end{proof}

Right fibrations automatically satisfy descent in the following sense:

\begin{definition}
    An augmented simplicial \category $\Ccal_\bullet \to \Ccal$ is a \defn{hypercover} if for each $n \in \NN$, the induced functor $\Ccal_n \to (\cosk_{n-1}(\Ccal_\bullet))_n$ is surjective.
\end{definition}

\begin{lemma}\label{lem:descent_for_right_fib}
    Let $\Ccal_\bullet \to \Ccal$ be a hypercover in $\Catinfty$, and assume that for each $n \in \NN$, the induced functor $ \Ccal_n \to \Ccal$ is a right fibration.
    Then $\colim_{\Deltaop} \Ccal_\bullet \equivalence \Ccal$.
\end{lemma}

\begin{proof}
    By straightening-unstraightening, our given hypercover translates to a hypercover of the terminal object in the \category $ \RFib(\Ccal) \equivalent \PSh(\Ccal) $ of right fibrations over $ \Ccal $.
    Furthermore, the inclusion $\RFib(\Ccal) \subset \Cat_{\infty,/\Ccal}$ preserves limits and colimits (the case of limits is clear as right fibrations are defined via a lifting property, for colimits see \cite[Corollary~A.5]{arXiv:2209.12569}).
    Since $ \RFib(\Ccal) $ is a presheaf \topos and therefore hypercomplete, the claim follows.
\end{proof}

We can now deduce the desired descent results.

\begin{corollary}\label{cor:integral-hyperdescent-main-result}
    \hfill
    \begin{enumerate}
        \item The functor $\Gal \colon \Sch^{\qcqs} \to \CondCat$ is a hypercomplete integral cosheaf.
        
        \item The functor $(-)\proethyp \colon (\Sch^{\qcqs})^{\op} \to \Catinfty$ with functoriality given by pullbacks is an integral hypersheaf.
        
        \item The functor $\Picond \colon \Sch^{\qcqs}  \to \CondAni$ is a hypercomplete integral cosheaf.
    \end{enumerate}
\end{corollary}
    
\begin{proof}
    By \cite[Theorem~1.2]{MR4574234}, we have a natural equivalence
    \begin{equation*}
        \Xproethyp \simeq \Functs(\Gal(X),\ICond(\Ani)) \comma
    \end{equation*}
    hence second assertion is an immediate consequence of the first. 
    By \Cref{prop:Picond_is_BGal}, the third assertion is also an immediate consequence of the first.
    Thus, we only need to prove the first assertion.

    Using \Cref{cor:Integral_maps_stable_under_base_change}, it follows that for any integral hypercover $X_\bullet \to X$ and $n \in \NN$, the canonical map
    \begin{equation*}
        \Gal(\cosk_{n-1}(X_\bullet)_n) \to \cosk_{n-1}(\Gal(X_\bullet))_n
    \end{equation*}
    is an equivalence.
    Thus, \Cref{lem:integral-morphism-is-right-fibration-underlying,lem:Gal_is_surj} imply that $\Gal(X_\bullet)$ is a hypercover of right fibrations of condensed categories.
    Since sifted colimits are computed pointwise in the \category $\CondCat = \Fun^{\times}(\Extr^{\op},\Catinfty)$, the claim follows by combining \cref{rem:right_fibrations_pointwise,lem:descent_for_right_fib}.
\end{proof}

We can also recover the schematic description of the over category $\Gal(X)_{/\xbar}$ given in \cite[Corollary~12.4.5]{Exodromy}:%
\footnote{The description of the under categories of $\Gal(X) $ in terms of strict henselizations in \textit{loc. cit.} is immediate from the definition. 
The description of over categories in terms of strict normalizations is less obvious, so we decided to include an argument here.}

\begin{corollary}
    Let $ X $ be a qcqs scheme, let $\xbar \to X$ be a geometric point, and let $X^{(\xbar)}$ denote the \emph{strict normalization} of $ X $ at $\xbar$ in the sense of \cite[Notation~12.4.2]{Exodromy}.
    Then the natural integral morphism $f\colon X^{(\xbar)}\to X$ induces an equivalence of condensed categories
    \begin{equation*}
        \Gal(X^{(\xbar)}) \equivalence \Gal(X)_{/\xbar} \period
    \end{equation*}
\end{corollary}

\begin{proof}
    Since the morphsism $ f $ is integral, by \Cref{lem:integral-morphism-is-right-fibration-underlying} the functor of condensed categories $ \Gal(f) $ is a right fibration.
    Hence for $\xbar \from \ast \to \Gal(X^{(\xbar)}) \to \Gal(X)$, the induced functor
    \begin{equation*}
        f_{/\xbar} \from \Gal(X^{(\xbar)})_{/\xbar} \to \Gal(X)_{/\xbar}
    \end{equation*}
    is an equivalence of condensed categories.
    The condensed category $\Gal(X^{(\xbar)})$ already has a terminal object induced by the generic point of $X^{(\xbar)}$, which is given by $\xbar \to X^{(\xbar)}$, cf. \cite[Theorem~2.4.21]{CatrinsThesis}.
    We conclude using that 
    \begin{equation*}
        \Gal(X^{(\xbar)})\simeq\Gal(X^{(\xbar)})_{/\xbar} \simeq \Gal(X)_{/\xbar} \period \qedhere
    \end{equation*}
\end{proof}

Finally, using some of the machinery developed in \cite{arXiv:2103.17141}, we can also deduce integral basechange for proétale hypersheaves.
We do not need this in the rest of this article, but it might be of independent interest.

\begin{proposition}\label{prop:integral-basechange}
    Let 
    \begin{equation*}
        \begin{tikzcd}
            X' \arrow[r, "q"] \arrow[d, "g"'] & X \arrow[d, "f"] \\
            Y' \arrow[r, "p"'] & Y 
        \end{tikzcd}
    \end{equation*}
    be a cartesian square of qcqs schemes where $ f $ is integral.
    Then the induced square
    \begin{equation*}
        \begin{tikzcd}
            (X')\proethyp \arrow[r, "\qlowerstar"] \arrow[d, "\glowerstar"'] & \Xproethyp \arrow[d, "\flowerstar"] \\
            (Y')\proethyp \arrow[r, "\plowerstar"'] & \Yproethyp 
        \end{tikzcd}
    \end{equation*}
    is horizontally left adjointable, i.e., the natural exchange transformation $p^{*}\flowerstar \to g_* q^*$ is an equivalence.
\end{proposition}

\begin{proof}
    By \cite[Corollary~1.2]{zbMATH07671238}, this square is identified with the square
    \begin{equation*}
        \begin{tikzcd}[row sep=3em, column sep=4em]
        	\Functs(\Gal(X'),\ICond(\Ani)) \arrow[r, "\Gal(q)_*"] \arrow[d, "\Gal(g)_*"'] & \Functs(\Gal(X),\ICond(\Ani)) \arrow[d, "\Gal(f)_*"] \\
        	\Functs(\Gal(Y'),\ICond(\Ani)) \arrow[r, "\Gal(p)_*"'] & \Functs(\Gal(Y),\ICond(\Ani)) \period
        \end{tikzcd}
    \end{equation*}
    Since $ f $ is integral, \Cref{lem:integral-morphism-is-right-fibration-underlying} shows that $ \Gal(f) $ is a right fibration, and \Cref{cor:Integral_maps_stable_under_base_change} shows that the natural map $\Gal(X') \to \Gal(X) \times_{\Gal(Y)} \Gal(Y')$ is an equivalence.
    Because right fibrations of condensed \categories are \textit{proper} functors \cite[Proposition~4.4.7]{arXiv:2103.17141}, the the above square is horizontally left adjointable.
\end{proof}

%-------------------------------------------------------------------%
%  Digression: strongly künnethable morphisms of schemes            %
%-------------------------------------------------------------------%

\subsection{Digression: strongly künnethable morphisms of schemes}
\label{sec:strongly_künneth}

We conclude this section by explaining at what level of generality the Künneth formula for étale \topoi (equivalently, \Cref{cor:Integral_maps_stable_under_base_change}) holds.

\begin{definition}\label{def:strongly-künnethable-morphism}
    We call a morphism of schemes $X \to Y$ \emph{strongly künnethable} if for any morphism $Y' \to Y$ the induced map
    \begin{equation*}
        (X \times_{Y} Y')_{\et} \to X_{\et} \times_{Y_{\! \et}} Y'_{\et}
    \end{equation*}
    is an equivalence.
\end{definition}

\begin{remark}
    Since all \topoi involved in \Cref{def:strongly-künnethable-morphism} are $1$-localic, being strongly künnethable is equivalent to the canonical geometric morphism
    \begin{equation*}
        (X \times_{Y} Y')_{\et,\leq 0} \to X_{\et,\leq 0} \times_{Y_{\! \et,\leq 0}} Y'_{\et,\leq 0}
    \end{equation*}
    of $1$-topoi being an equivalence.
\end{remark}

\begin{proposition}\label{prop:characterization-strongly-künnethable-quasi-finite}
    Let $f \colon X \to Y$ be a morphism of finite presentation.
    Then $ f $ is strongly künnethable if and only if it is quasi-finite.
\end{proposition}

\begin{proof}
    Let us first assume that $ f $ is quasi-finite.
    Since open immersions are strongly künnethable by \HTT{Remark}{6.3.5.8}, we may immediately reduce to the case where $ X $, $ Y $, and $ Y' $ are affine.
    Applying Zariski's main theorem, we can factor $ f $ as an open immersion followed by a finite morphism.
    Thus we may assume that $ f $ is finite.
    
    We have to check that the induced map
    \begin{equation}\label{eq:canonical_comp_map_topoi}
        (X \times_{Y} Y')_{\et,\leq 0} \to X_{\et,\leq 0} \times_{Y_{\et,\leq 0}} Y'_{\et,\leq 0}
    \end{equation}
    is an equivalence.
    By \Cref{cor:Integral_maps_stable_under_base_change}, it induces an equivalence of categories of points.
    Furthermore it follows from the site-theoretic description of the fiber product of topoi \cite[Exposé XI, \S 3]{MR3309086} that \eqref{eq:canonical_comp_map_topoi} is a coherent geometric morphism of coherent topoi.
    Thus, the Makkai--Reyes conceptual completeness theorem \SAG{Theorem}{A.9.0.6} implies that this geometric morphism is an equivalence.

    For the converse, assume that $ f $ is not quasi-finite.
    Then at least one geometric fiber of $ f $ is not quasi-finite.
    Since taking geometric fibers is compatible with taking étale \topoi \cite[Proposition 2.3]{MR4686649}, we may reduce to the case where $Y = \Spec(k)$ is the spectrum of a separably closed field $ k $.
    Furthermore, we may always modify $ X $ by quasi-finite maps to reduce to the case where $ X $ is integral of dimension at least $ 1 $.
    By Noether normalization, there exists a finite surjective map $h \colon X \to \AA_{k}^n$.
    Let $X_\bullet \to \AA_{k}^n$ denote the Čech nerve of $h$.
    Now if $ f $ were strongly künnethable, then since the maps $X_m \to \Spec(k)$ are the composite of a finite map $ d_0 \colon X_m \to X$ and $ f $, it would follow that also all maps $X_m \to \Spec(k)$ would be strongly künnethable as well.
    Thus for every $k$-scheme $Y'$ and every $ m \geq 0 $, the induced map
    \begin{equation*}
        \Gal(X_m \times Y') \to \Gal(X_m) \times \Gal(Y')
    \end{equation*}
    would be an equivalence.
    But by integral descent (\Cref{cor:integral-hyperdescent-main-result}), after passing to the colimit over $\Deltaop$, this would imply that the canonical map
    \begin{equation*}
        \Gal(\AA_{k}^n \times Y') \to \Gal(\AA_{k}^n) \times \Gal(Y')
    \end{equation*}
    is an equivalence.
    
    Thus we may assume that $X = \AA_{k}^n$ and therefore even that $X = \AA_{k}^1$.
    Now let $Z = \AA_{k}^1$ as well.
    This would imply that the canonical map
    \begin{equation*} 
        \Gal(\AA_{k}^{2}) \to \Gal(\AA_{k}^1) \times \Gal(\AA_{k}^1)
    \end{equation*}
    is an equivalence.
    In particular, it would induce an equivalence on underlying posets and thus an isomorphism of specialization posets
    \begin{equation*}
        \Zpos{(\AA_{k}^2)} \to \Zpos{(\AA_{k}^1)} \times \Zpos{(\AA_{k}^1)} \comma
    \end{equation*}
    which is a contradiction.
\end{proof}

%-------------------------------------------------------------------%
%-------------------------------------------------------------------%
%-------------------------------------------------------------------%
%  The condensed fundamental group                                  %
%-------------------------------------------------------------------%
%-------------------------------------------------------------------%
%-------------------------------------------------------------------%

\newpage
\part{The condensed fundamental group}\label{part:the_condensed_fundamental_group}

The purpose of this part is to analyze the fundamental group of the condensed homotopy type and its relationship to the étale and proétale fundamental groups.
We start by showing that, surprisingly, \smash{$ \pionecond(\AA_{\CC}^1) $} is nontrivial (see \Cref{cor:pionecond-of-A1-is-nontrivial}).
This can be viewed as saying that there exists a nontrivial proétale local system of \textit{condensed} rings on \smash{$ \AA_{\CC}^1 $}.
See \Cref{rem:nontrivial_local_system_on_AA1}.

In \cref{sec:quasiseparated_quotient_of_the_condensed_fundamental_group}, we show that a mild quotient of the condensed fundamental group of $ \AA_{\CC}^1 $ indeed becomes trivial.
Specifically, Clausen and Scholze introduced a localization $ \goesto{A}{A^{\qs}} $ of the category of condensed sets called the \textit{quasiseparated quotient} \cite[Lecture VI]{Scholze:analyticnotes}.
For topological groups, this is analogous to the Hausdorff quotient.
We show that if $ X $ is a topologically noetherian scheme that is geometrically unibranch, then there is a natural isomorphism of condensed groups
\begin{equation*}
    \pionecond(X, \xbar)^{\qs} \isomorphism \pioneet(X, \xbar) \period
\end{equation*}
See \Cref{thm:normal-implies-profinite}.
Under mild hypotheses on the scheme (e.g., being Nagata), we also prove a van Kampen formula for the quasiseparated quotient of the condensed fundamental group that only involves topological free products, topological quotients, and the étale fundamental group of the normalization, see \Cref{cor:vanKampen-for-cond-qs}.

In \cref{sec:Noohi_completion_of_the_condensed_fundamental_group}, we turn to the relationship between the condensed fundamental group and the proétale fundamental group introduced by Bhatt and Scholze \cite[\S 7]{BhattScholzeProetale}.
One of the special features of \smash{$ \piproet_1(X) $} is that it is a \textit{Noohi group}.
We show that if $ X $ is topologically noetherian, the \textit{Noohi completion} (suitably extended to condensed groups) of $ \picond_1(X) $ recovers \smash{$ \piproet_1(X) $}, see \Cref{thm:recovering_BS_fundamental_group}.

%-------------------------------------------------------------------%
%-------------------------------------------------------------------%
%  The quasiseparated quotient of the condensed fundamental group   %
%-------------------------------------------------------------------%
%-------------------------------------------------------------------%

\section{The quasiseparated quotient of the condensed fundamental group}\label{sec:quasiseparated_quotient_of_the_condensed_fundamental_group}

In \cref{subsec:picond_1_of_AA^1_is_nontrivial}, we begin by using the Galois category description of the condensed homotopy type to show that \smash{$ \picond_1(\AA_{\CC}^1) $} is nontrivial.
The rest of the section is dedicated to studying the quasiseparated quotient of \smash{$ \picond_1(\AA_{\CC}^1) $}.
In \cref{subsec:preliminaries_on_quasiseparated_quotients}, we recall the basics on quasiseparated quotients of condensed sets and prove some fundamental results about the quasiseparated quotient.
In \cref{subsec:picond_1_of_geometrically_unibranch_schemes}, we show that the quasiseparated quotient of \smash{$ \picond_1 $} of a geometrically unibranch and topologically noetherian scheme recovers \smash{$ \piet_1 $}.
In \cref{subsec:van_Kampen_for_picond_1}, we prove a van Kampen formula for the quasiseparated quotient of the condensed fundamental group, see \Cref{cor:vanKampen-for-cond-qs}.

%-------------------------------------------------------------------%
%  π₁ᶜᵒⁿᵈ(A¹) is nontrivial                                         %
%-------------------------------------------------------------------%

\subsection{\texorpdfstring{$\pionecond(\AA_{\CC}^1)$}{π₁ᶜᵒⁿᵈ(A¹)} is nontrivial}\label{subsec:picond_1_of_AA^1_is_nontrivial}

In this subsection, we show that \smash{$\pionecond$} can behave wildly, even in geometrically very simple situations.
For simplicity, we work over the complex numbers $\CC$.

\begin{notation}\label{ntn:normal_closure}
    For a topological group $ G $ and an (abstract) subgroup $H < G$, let $H^{\nc}$ denote the group-theoretic \textit{normal closure} of $H$ in $ G $.
    Let
    \begin{equation*}
        H^{\tnc} \colonequals \overline{H^{\nc}} 
    \end{equation*}
    be the \textit{topological normal closure} of $H$ in $ G $, i.e., the smallest \emph{closed} normal subgroup of $ G $ containing $H$ or, equivalently, the topological closure of $H^{\nc}$ in $ G $.
\end{notation}

\begin{proposition}\label{prop:counterexample-exact-sequence-generic-point}
    Let $ S \subset \CC$ be a subset.
    Let us write
    \begin{equation*}
        \AA_\CC^1\setminus{S} \colonequals \Spec(\CC[t][(t-a)^{-1} \mid a \in S]) \period
    \end{equation*}
    Let $\Freepf_\CC$ be the free profinite group on the underlying set of $\CC$.
    Let $N_S $ be the abstract normal subgroup of $\Freepf_\CC$ generated by $\ZZhat(a)$ for all $ a \in \CC \setminus{S}$.
    Write $\eta$ for the generic point of $\AA_\CC^1$ and $\etabar$ for the geometric generic point induced by choosing an algebraic closure of $\CC(T)$.
    There is a short exact sequence of (abstract) groups
    \begin{equation*}
        \begin{tikzcd}[cramped, sep=small]
            1 \arrow[r] & N_S \arrow[r] & \Freepf_\CC \arrow[r] & \picond_{1}(\AA_\CC^1\setminus{S},\etabar)(\ast) \arrow[r] & 1 \period
        \end{tikzcd}
    \end{equation*}
\end{proposition}

To prove \Cref{prop:counterexample-exact-sequence-generic-point}, we make use of an alternative description of $ \BGal(X)(\ast) $.
To explain this, we first recall that $ \Gal(X)(\pt) $ admits a conservative functor to a poset:

\begin{example}\label{ex:Gal(X)_is_a_stratified_anima}
    Let $ X $ be a qcqs scheme.
    Note that there is a natural functor
    \begin{equation*}
        s \colon \Gal(X)(\pt) \to \Zpos{X}
    \end{equation*}
    from the category of points of the étale topos to the specialization poset of $ |X| $.
    The functor $ s $ is the unique functor that sends a geometric point $ \xbar \to X $ to the underlying point $ x \in |X| $.
    Since the fiber of $ s $ over a point $ x \in \Zpos{X} $ is equivalent to the classifying anima of the discrete group $ \Gal_{\upkappa(x)} $, the functor $ s $ is conservative.
\end{example}

\noindent Our description thus relies on the following presentation of \categories with a conservative functor to a poset:

\begin{recollection}[{(\categories with a conservative functor to a poset)}]\label{rec:categories_with_a_conservative_functor_to_a_poset}
    Let $ P $ be a poset.
    Write $ \sd(P) $ for the poset of nonempty linearly ordered finite subsets of $ P $, ordered by inclusion. 
    The poset $ \sd(P) $ is referred to as the \defn{subdivision} of $ P $.
    Write $ \Cat_{\infty,/P}^{\cons} \subset \Cat_{\infty,/P} $ for the full subcategory spanned by those \categories over $ P $ such that the structure morphism $ \Ccal \to P $ is conservative. 
    Barwick--Glasman--Haine proved that the \textit{nerve} functor
    \begin{align*}
        \Nup_P \colon \Cat_{\infty,/P}^{\cons} &\longrightarrow \Fun(\sd(P)^{\op},\Ani) \\
            [\Ccal \to P] &\longmapsto \brackets{\{p_0 < \cdots < p_n\} \mapsto \Map_{\Cat_{\infty,/P}}(\{p_0 < \cdots < p_n\}, \Ccal)}
    \end{align*}
    is a fully faithful right adjoint.
    % Moreover, they identified the image of $ \Nup_P $ as the full subcategory spanned by those $ X \colon \sd(P)^{\op} \to \Ani $ such that for every nonempty linearly ordered finite subset $ \{p_0 < \cdots < p_n \} \subset P $, the natural map
    % \begin{equation*}
    %     X\{p_0 < \cdots < p_n \} \to X\{p_0 < p_1\} \crosslimits_{X\{p_1\}} X\{p_1 < p_2\} \crosslimits_{X\{p_2\}} \cdots \crosslimits_{X\{p_{n-1}\}} X\{p_{n-1} < p_n\}
    % \end{equation*}
    % is an equivalence.
    See \cite[Theorem~2.7.4]{Exodromy}.
\end{recollection}

The next result provides a convenient way of computing the classifying anima $ \Bup\Ccal $ in terms of the nerve $ \Nup_{P}(\Ccal) $.

\begin{proposition}
    Let $ P $ be a poset and $ \Ccal \to P $ a conservative functor.
    Then there is a natural equivalence
    \begin{equation*}
        \Bup \Ccal \equivalent \colim_{\sd(P)^{\op}} \Nup_{P}(\Ccal) \period
    \end{equation*}
\end{proposition}

\begin{proof}
    First, observe that the functor $ P \cross (-) \colon \Ani \to \Cat_{\infty,/P} $ factors through $ \Cat_{\infty,/P}^{\cons} $ and is right adjoint to the functor \smash{$ \Bup \colon \Cat_{\infty,/P}^{\cons} \to \Ani $} sending $ \Ccal \to P $ to the classifying anima $ \Bup \Ccal $.
    Since the colimit functor $ \Fun(\sd(P)^{\op},\Ani) \to \Ani $ is left adjoint to the constant functor, in light \Cref{rec:categories_with_a_conservative_functor_to_a_poset} and the diagram of adjunctions
    \begin{equation*}
        \begin{tikzcd}[sep=3.5em]
            \Fun(\sd(P)^{\op},\Ani) \arrow[r, shift left] & \Cat_{\infty,/P}^{\cons} \arrow[l, hooked', shift left, "\Nup_{P}"] \arrow[r, shift left, "\Bup"] & \Ani \comma \arrow[l, shift left, "{P \cross (-)}"]
        \end{tikzcd}
    \end{equation*} 
    it suffices to show that the composite right adjoint $ \Ani \to \Fun(\sd(P)^{\op},\Ani) $ is equivalent to the constant functor.
   
    To prove this, first note that for nonempty linearly ordered finite subset $ \{p_0 < \cdots < p_n\} \subset P $, the classifying anima $ \Bup \{p_0 < \cdots < p_n\} $ is contractible.
    Hence, for any anima $ A $ and nonempty linearly ordered finite subset $ \{p_0 < \cdots < p_n\} \subset P $, we have natural equivalences
    \begin{align*}
        \Nup_P(P \cross A)\{p_0 < \cdots < p_n\} &= \Map_{\Cat_{\infty,/P}}(\{p_0 < \cdots < p_n\},P \cross A) \\ 
        &\equivalent \Map_{\Ani}(\Bup\{p_0 < \cdots < p_n\},A) \\ 
        &\equivalent \Map_{\Ani}(\pt,A) \\ 
        &= A \period \qedhere
    \end{align*}
\end{proof}

\begin{example}\label{lem:BGal_via_decollages}
    In particular, if $ X $ is a qcqs scheme, then there is a natural equivalence
    \begin{equation*}
        \BGal(X)(\ast) \equivalent \colim_{\subdiv(\Zpos{X})^{\op}} \Nup_{\Zpos{X}}(\Gal(X)(\pt)) \period
    \end{equation*}
\end{example}

\begin{proof}[Proof of \Cref{prop:counterexample-exact-sequence-generic-point}]
    To simplify notation, write $ X = \AA_\CC^1\setminus{S}$, $ \Gal(X)$ for $\Gal(X)(\ast)$, and $ \Nup(\Gal(X)) $ for \smash{$ \Nup_{\Zpos{X}}(\Gal(X)) $}.
    We compute $ \BGal(X) $ using \Cref{lem:BGal_via_decollages}.
    Note that $\subdiv(\Zpos{X})$ consists of elements of the form
    \begin{equation*}
        \{a\} \comma \quad \{\eta\} \comma \andeq \{a  < \eta \}
    \end{equation*}
    for any $ a \in \CC\setminus{S} $, and the ordering is given by $ \{a\} < \{a  < \eta \}$ and $\{ \eta \}< \{a  < \eta \}$.
    Furthermore, the functor
    \begin{equation*}
        \Nup(\Gal(X)) \colon \subdiv(\Zpos{X})^{\op} \to \Ani
    \end{equation*}
    can be explicitly described by applying $\Pietprofin$ followed by $ \lim \colon \ProAnifin \to \Ani $ to the diagram $ \subdiv(\Zpos{X})^{\op} \to \Sch $ that sends $ \{a\} < \{a  < \eta \} > \{ \eta \} $
    to the span of schemes
    \begin{equation}\label{eq:DVR_decollage}
        \begin{tikzcd}[sep=3em]
        	\Spec(\CC[T]_{(a)}^{\h}) & \Spec(\CC[T]_{(a)}^{\h}) \setminus \{a\} \arrow[r] \arrow[l] & {\Spec(\CC(T)) \period}
        \end{tikzcd}
    \end{equation}
    See \cite[Example~12.2.2]{Exodromy}.
    
    For each $a \in \CC \setminus S $, we now choose a lift $\etabar_a $ of $ \etabar $ fitting into a commutative triangle
    \begin{equation*}
        \begin{tikzcd}
        	 & \Spec(\CC[T]_{(a)}^{\h}) \setminus \{a\} \arrow[d]  \\
        	\Spec(\overline{\CC(T)}) \arrow[r, "\etabar"'] \arrow[ur, dotted, "{\etabar_a}"] & \Spec(\CC(T)) \period
        \end{tikzcd}
    \end{equation*}
    In particular, we can lift the span \eqref{eq:DVR_decollage} to a span of pointed schemes; therefore, $ \Nup(\Gal(X))$ also lifts to a diagram of pointed anima $ \Nup(\Gal(X))_*$.
    Using that $\uppi_1$ is an equivalence between pointed, connected, $1$-truncated anima and the category of groups \HTT{Proposition}{7.2.12}, we may thus compute
    \begin{equation*}
        \uppi_{1}(\Bup\Gal(X),\etabar) \simeq \colim_{\subdiv(\Zpos{X})^{\op}} \uppi_{1}(\Nup(\Gal(X))_\ast) \period
    \end{equation*}
    
    Now for any $\{a\} < \{a  < \eta \} > \{ \eta \}$, the corresponding span in groups is given by
    \begin{equation*}
        \begin{tikzcd}[sep=3em]
        	\pt & \pioneet(\Spec(\CC[T]_{(a)}^{\h}) \setminus \{a\},\etabar_a) \arrow[r] \arrow[l] & \pioneet(\Spec(\CC(T)),\etabar) \period
        \end{tikzcd}
    \end{equation*}
    Moreover, the colimit of the diagram $ \uppi_{1}(\Nup(\Gal(X))_\ast) $ over $\subdiv(\Zpos{X})^{\op}$ is given by taking the quotient of $\pioneet(\Spec(\CC(T)),\etabar) = \Gal_{\CC(T)}$ by the (abstract) normal closure of the subgroup generated by the images of all the decomposition groups
    \begin{equation*}
        \Dup_a \colonequals \pioneet(\Spec(\CC[T]_{(a)}^{\sh}\setminus{\{a\}}) \period
    \end{equation*}
    By \cref{Thm:Gen_of_Douady}, there is an isomorphism
    \begin{equation*}
        \Freepf_{\CC} \equivalence \Gal_{\CC(T)} = \pioneet(\Spec(\CC(T)),\etabar) 
    \end{equation*}
    from the free profinite group on the set $ \CC $, under which the preimage of $\Dup_a$ is, up to conjugation, given by the profinite subgroup $\ZZhat(a)$ generated by $ a $.
    It follows that $\uppi_{1}(\BGal(X),\etabar) $ is isomorphic to the quotient of $\Freepf_{\CC}$ by the smallest (abstract) normal subgroup containing $\ZZhat(a)$ for all $a \in \CC \setminus S$, as desired.
\end{proof}

\begin{corollary}\label{cor:pionecond-of-A1-is-nontrivial}
    Let $ \xbar \to \AA^1_{\CC} $ be a geometric point.
    Then the abelianization of the underlying group \smash{$\picond_{1}(\AA_{\CC}^1,\xbar)(\ast)$} is nontrivial.
    As a consequence,
    \begin{equation*}
        \picond_{1}(\AA_{\CC}^1,\xbar) \neq 1 \andeq \picond_{1}(\AA_{\CC}^1,\xbar)^{\ab} \neq 1 \period
    \end{equation*}
\end{corollary}

\begin{proof}
    Since \smash{$ \AA_{\CC}^1 $} is irreducible, \Cref{obs:dependence_of_homotopy_groups_on_basepoints} implies that the condensed fundamental groups of \smash{$ \AA_{\CC}^1 $} with respect to all basepoints are isomorphic.
    So it suffices to treat the case where $ \xbar = \etabar $ is the geometric generic point.
    
    Consider the canonical continuous homomorphism \smash{$ \Freepf_\CC \to  \prod_{a \in \CC}\ZZhat $} that carries a generator $a$ to the unit vector at $a$.
    Note that since the image of this is homomomorphism dense, the source is profinite, and the target is Hausdorff, this homomorphism is surjective.
    Also notice that that the (abstract) normal subgroup $N_\emptyset$ lands in the subgroup $\bigoplus_{a \in \CC}\ZZhat$.
    Thus, by \Cref{prop:counterexample-exact-sequence-generic-point}, we obtain commutative diagram of abstract groups
    \begin{equation*}
        \begin{tikzcd}
        	1 \arrow[r] & N_{\emptyset} \arrow[r] \arrow[d] & \Freepf_\CC \arrow[r, ->>] \arrow[d, ->>] & \picond_{1}(\AA_{\CC}^1,\etabar)(\ast) \arrow[r] \arrow[d] & 1 \\
        	1 \arrow[r] & \bigoplus_{a \in \CC} \ZZhat \arrow[r] & \prod_{a \in \CC} \ZZhat \arrow[r, ->>] & Q \arrow[r] & 1 \comma
        \end{tikzcd}
    \end{equation*}
    where the rows are short exact sequences.
    Here, $ Q \neq 1 $ denotes the abstract quotient.
    Since the middle vertical map is surjective, the right vertical map is also surjective. 
    Since $ Q $ is abelian, we deduce that $ \picond_{1}(\AA_{\CC}^1,\etabar)(\ast)^{\ab} \neq 1 $.
\end{proof}

\begin{example}\label{ex:pionecond_of_P^1}
    The proof of \cref{cor:pionecond-of-A1-is-nontrivial} also shows that the abelianization of $\pionecond(\PP^1_{\CC},\xbar)(\ast)$ is nontrivial.
    Indeed, the argument of the proof of \cref{prop:counterexample-exact-sequence-generic-point} can be used to show that there is a pushout square of groups
    \begin{equation*}
        \begin{tikzcd}
        	\ZZhat \simeq \Dup_\infty \arrow[r] \arrow[d] & \pionecond(\AA^1_\CC,\xbar)(\ast) \arrow[d] \\
        	1 \arrow[r] & \pionecond(\PP^1_\CC,\xbar)(\ast) \period
        \end{tikzcd}
    \end{equation*}
    By the proof of \cref{cor:pionecond-of-A1-is-nontrivial}, $\pionecond(\AA^1_\CC,\xbar)(\ast)$ surjects onto $ Q = \big(\prod_{a \in \CC} \ZZhat \big) / \big( \bigoplus_{a \in \CC} \ZZhat \big) $.
    It follows that $\pionecond(\PP^1_\CC,\xbar)(\ast) $ surjects onto $Q/\im(\Dup_\infty)$ which has the same cardinality as $Q$ and is thus nontrivial.
\end{example}

One can also use \cref{cor:pionecond-of-A1-is-nontrivial} to show that for some exotic condensed rings, there are nontrivial lisse sheaves on $\AA_{\CC}^1$.

\begin{example}\label{rem:nontrivial_local_system_on_AA1}
    The forgetful functor $ \Cond(\Ring) \to \Cond(\Ab)$ admits a left adjoint given by applying the group ring functor pointwise and then sheafifying.
    Writing $ A = \pionecond(\AA^1_{\CC},\xbar)^{\ab} $, we thus obtain a nontrivial condensed ring $\ZZ[A]$.
    Furthermore, there is a canonical action of the condensed group $A$ on the free $\ZZ[A]$-module of rank $1$, given by multiplication.
    Using the monodromy equivalence of \cref{prop:Picond_classifies_lisse_sheaves}, this yields a lisse $\ZZ[A]$-module on $\AA^1_{\CC}$ that is \emph{not} constant, i.e., not in the image of the basechange functor
    \begin{equation*}
        \Dlis(*_{\proet};\ZZ[A]) \to \Dlis(\AA^1_{\CC,\proet};\ZZ[A]) \period
    \end{equation*}
\end{example}

While they fit best in this subsection, the following remark and example use the notion of a \textit{quasiseparated} condensed set.
We recall some background about quasiseparatedness and quasiseparated quotients in \cref{subsec:preliminaries_on_quasiseparated_quotients} below; hence the reader might prefer to return to these points after consulting \cref{subsec:preliminaries_on_quasiseparated_quotients}.

\begin{remark}
    The proof of \Cref{cor:pionecond-of-A1-is-nontrivial} can be adapted to show more generally that whenever $\CC \setminus{S}$ is infinite, the condensed group $ \picond_{1}(\AA_{\CC}^1\setminus{S},\etabar)$ is not profinite and therefore, by \Cref{thm:normal-implies-profinite}, also not quasiseparated.
    Indeed, if it were, it would follow from \Cref{prop:counterexample-exact-sequence-generic-point} that $N_S \subset \Freepf_{\CC}$ is a closed subgroup.
    Thus, the image of $N_S$ under the map $ \Freepf_\CC \to \prod_{a \in \CC } \ZZhat$ would also be closed in $\prod_{a \in \CC } \ZZhat$.
    But this image is exactly $\bigoplus_{a \in \CC \setminus{S}}\ZZhat$, which is not closed if $\CC \setminus{S}$ is infinite.
    Even more generally, the above arguments show that for any Dedekind scheme $ X $, if the abstract normal closure $N \subset \Gal_{\CC(X)}$ of the subgroup generated by all decomposition groups is not closed, then the condensed fundamental group of $ X $ is not quasiseparated.
\end{remark}

The next example shows that whenever $S \neq \emptyset$, even if $\CC \setminus S$ is finite, the condensed fundamental group on $\AA_{\CC}^1 \setminus{S}$ is not quasiseparated.
For example, this covers the case of the localization $\Spec(\CC[T]_{(T-a)})$ for $a \in \CC$.
To explain it, we need the following lemma about profinite groups.

\begin{lemma}\label{lem:nc-neq-tnc}
    Let $G = \Freepf_{\{a,b\}}$ be the free profinite group on two elements $a $ and $ b$, and let
    \begin{equation*}
        H \colonequals \ZZhat(b) \subset G
    \end{equation*}
    be the (necessarily free) profinite subgroup of $ G $ generated by $b$. 
    Then $ H^{\nc} \subsetneq H^{\tnc} $. 
\end{lemma}

\begin{proof}
    For each integer $ n \geq 1 $, let $g_n \colonequals  \prod_{i=1}^n (a^{i!}b^{i!}a^{-i!})$. 
    For each $ n $, we have $g_n \in H^{\nc}$. 
    Moreover, $ (g_n)_{n \geq 1} $ is a Cauchy sequence in $ G $.
    To prove that $ H^{\nc} \neq H^{\tnc} $, we show that $ (g_n)_{n \geq 1} $ converges 
    to an element outside of $ H^{\nc} $.

    We first claim that since $ G $ is Raĭkov-complete, the Cauchy sequence $ (g_n)_{n \geq 1} $ converges to some $g \in G$.
    Indeed, for a given $n_0 > 1 $ and $n>n_0$, we have
    \begin{equation*}
        g_{n_0}^{-1}g_n = \prod_{i=n_0+1}^n (a^{i!}b^{i!}a^{-i!}) \period
    \end{equation*}
    Let $N \vartriangleleft G$ be a normal open subgroup. 
    Then there exists $n_0$ such that for any $m \geq n_0$, we have $a^{m!}, b^{m!} \in N$. 
    This is because $a$ and $b$ are images of generators of $\ZZhat$ via (two different) continuous maps $\ZZhat \to G$, and the corresponding fact already holds in $\ZZhat$.
    It now follows that for any $n \geq n_0$, the element $g_{n_0}^{-1}g_n$ lies in $N$.
    By normality, $g_ng_{n_0}^{-1}$ also lies in $ N $.
    It follows that $g \in H^{\tnc}$. 
    
    We want to show that $g \notin H^{\nc}$. 
    Assume the contrary. 
    Then there exist some $r \in \NN $, $c_i \in G $, and $d_i \in H$ such that $g = \prod_{i = 1}^r c_id_ic_i^{-1}$.
    Now consider the following system of finite quotients of $ G $.
    For each $ m \geq 1 $, let $P_m \colonequals (\ZZ/m!)^{\cross m!}$ denote the $m!$-fold product of copies of $\ZZ/m!$, and write
    \begin{equation*}
        Q_m \colonequals P_m \rtimes \ZZ / m! \comma
    \end{equation*}
    where the action of $\ZZ / m!$ on $ P_m $ permutes the factors. 
    Define a homomorphism $G \twoheadrightarrow Q_m$ by
    \begin{equation*}
        b \mapsto (\overline{1}, 0, 0, \ldots) \in P_m = (\ZZ/m!)^{\cross m!} \andeq a \mapsto \bar{1} \in \ZZ / m! \period
    \end{equation*}
    Note that this map sends $g$ to $P_m$. 
    Now, for $m \gg r$, we get that, on the one hand, the image of $g$ in $P_m$ has an increasing (with $m$) number of nonzero entries and, on the other hand, the presentation $g = \prod_{i = 0}^r c_id_ic_i^{-1}$ implies that this number is bounded by $r$. 
    This is a contradiction.
\end{proof}

\begin{example}
    Let $S \subset \CC$ be a nonempty subset; we claim that $\picond_{1}(\AA^1\setminus{S},\etabar)$ is not quasiseparated.
    With the same notation as \Cref{lem:nc-neq-tnc}, we have a diagram of short exact sequences
    \begin{equation*}
        \begin{tikzcd}
            1 \arrow[r] & N_S \arrow[r] \arrow[d] & \Freepf_{\CC} \arrow[r] \arrow[d] & \picond_{1}(\AA_{\CC}^1\setminus{S},\etabar)(\ast) \arrow[r] \arrow[d] & 1 \\
            1 \arrow[r] & H^{\nc} \arrow[r] & \Freepf_{\{a,b\}} \arrow[r] & \Freepf_{\{a,b\}}/H^{\nc} \arrow[r] & 1 \comma
        \end{tikzcd}
    \end{equation*}
    where the middle vertical map sends $z \in \CC$ to $b$ if $z \in S$ and to $a$ otherwise.
    Then, by construction, $H^{\nc}$ is the image of $N_S$ under this map.
    Thus, if $\picond_{1}(\AA^1\setminus{S},\etabar)$ were quasiseparated, $N_S$ would be a closed subgroup (see \Cref{prop:quasiseparated-quotient-for-groups} below).
    Hence so would $H^{\nc}$, contradicting \Cref{lem:nc-neq-tnc}.
\end{example}

\begin{remark}[{(counterexample to ``$\pionecond$-properness'' of $ \PP^1_{\overline{\QQ}} $)}]\label{rem:counterexample_to_pioneproper}
    In this remark, we show that
    \begin{equation*}
        \pionecond(\PP^1_{\overline{\QQ}})(*) \notequivalent \pionecond(\PP^1_{\CC})(*)
    \end{equation*}
    by showing that the cardinality of the former is smaller than that of the latter. 
    This contrasts with the more classical story of $ \pioneet $; see \cite[Exposé X, Théorème 2.6]{MR50:7129} and the discussion in \cite[\S4.1, esp.\ Lemma 4.1.16]{kedlaya2017sheaves} and \cite[\S16]{MR4446467}.
    
    We have seen in  \Cref{ex:pionecond_of_P^1} that $ \pionecond(\PP^1_{\CC})(*) $ admits a quotient with the same cardinality as
    \begin{equation*}
        \textstyle Q = \big(\prod_{a \in \CC} \ZZhat \big) / \big( \bigoplus_{a \in \CC} \ZZhat \big) \comma
    \end{equation*}
    which will provide a lower bound for the cardinality. 
    On the other hand, as $ \PP^1_{\overline{\QQ}} $ is normal, we have seen before that the Galois group of the generic point $ \Gal_{\upkappa(\eta)} $ surjects onto \smash{$ \pionecond(\PP^1_{\overline{\QQ}})(*) $}. 
    By \cite[Theorem 2]{MR162796},
    \begin{equation*}
        \Gal_{\upkappa(\eta)} \simeq \Freepf_{\overline{\QQ}} \period
    \end{equation*}
    This will provide an upper bound for the cardinality.

    We now need to compute the cardinalities of some rather concrete profinite groups. 
    First, note that $|\ZZhat| = |\CC| = 2^{\aleph_0}$. 
    It follows that
    \begin{equation*}
        \big\lvert \textstyle \prod_{a \in \CC} \ZZhat \, \big\rvert = (2^{\aleph_0})^{2^{\aleph_0}} = 2^{\aleph_0 \cdot 2^{\aleph_0}} = 2^{2^{\aleph_0}} \period
    \end{equation*}
    We also have
    \begin{equation*}
        \big\lvert \textstyle \bigoplus_{a \in \CC} \ZZhat \, \big\rvert = \bigg\rvert \displaystyle \colim_{F \subset \CC \textup{ finite}} \textstyle \bigoplus_{a \in F} \ZZhat \,\bigg\rvert \leq |\CC|\cdot|\ZZhat| = 2^{\aleph_0}
    \end{equation*}
    Thus, $ |Q| = 2^{2^{\aleph_0}} $.

    Now, we want to bound $ |\Freepf_{M}| $, where $ M $ is a countable set (in our case $ M = \overline{\QQ} $). 
    From the universal property (or see \cite[Corollary 3.3.10]{MR2599132}) it follows that
    \begin{equation*}
        \Freepf_M \equivalent \lim_{F \subset M \textup{ finite}} \Freepf_{F} \period
    \end{equation*}
    Now (again from the universal property and thanks to the finiteness of the $F$'s), each of the groups $ \Freepf_F $ is just the profinite completion $ (\Free_F)^{\wedge} $ of the abstract free group $ \Free_F $ on $ F $. 
    In a finitely generated group, there are only finitely many normal subgroups of a given index. 
    This implies that the profinite completion of $ \Free_F $ can be written as a countably-indexed inverse limit of finite groups, so $ |\Freepf_F| = 2^{\aleph_0} $. 
    Thus, $ | \Freepf_M | \leq (2^{\aleph_0})^{\aleph_0} = 2^{\aleph_0}$. 
    Plugging in these bounds, we obtain the desired result.
\end{remark}

\begin{remark}[{(counterexample to proper base change for proétale sheaves)}]
        The results in this subsection can also be used to show that proper base change does not hold for proétale sheaves, even with torsion coefficients prime to the characteristic.
        Concretely, we claim that proper base change does not hold for the cartesian square
    \begin{equation*}
        \begin{tikzcd}
    		\PP^1_{\CC} \arrow[r, "q"] \arrow[d, "g"']& {\PP^1_{\overline{\QQ}}} \arrow[d, "f"] \\
    		\Spec(\CC) \arrow[r, "p"'] & \Spec(\overline{\QQ}) \period
    	\end{tikzcd}
    \end{equation*}
	That is, we claim that the natural transformation
	\begin{equation}\label{equ:beck_chav}
		p^* f_* \to g_* q^*
	\end{equation}
	of functors $\Dup_{\proet}(\PP^1_{\overline{\QQ}};\FF_p) \to \Dup_{\proet}(\Spec(\CC);\FF_p)$ is not an equivalence.
	By passing to left adjoints, this is equivalent to the natural transformation $ q_\sharp g^* \to f^* p_\sharp$ being an equivalence.
	Note that $ p_\sharp$ is an equivalence of \categories.
	After plugging in the unit and applying a further $f_\sharp$, \eqref{equ:beck_chav} being an equivalence would thus imply that there is an equivalence
	\begin{equation*}
	    g_\sharp( \mathbf{1}) \equivalence f_\sharp(\mathbf{1}) \period
	\end{equation*}
	Note that we may compute $ g_\sharp( \mathbf{1})$ (and similarly $f_\sharp(\mathbf{1})$) explicitly as the $\FF_p$-homology of the condensed homotopy type $ \Picond(\PP^1_{\CC}) $.
	The latter is computed by taking homology pointwise and then sheafifying.
	In particular, on global sections $ g_\sharp( \mathbf{1})(\ast)$ is simply the $ \FF_p $-homology of the anima $ \Picond(\PP^1_{\CC})(\ast) $.
	Since the anima $ \Picond(\PP^1_{\CC})(\ast)$ is connected, the universal coefficient theorem implies that
    \begin{equation*}
         \uppi_1( g_\sharp( \mathbf{1})(\ast)) \simeq \pionecond(\PP^1_{\CC},\xbar)(\ast)^{\ab} \otimes \FF_p \period
    \end{equation*}
    As in \cref{rem:counterexample_to_pioneproper}, the latter surjects onto a group with the same cardinality as
    \begin{equation*}
        \textstyle \big(\prod_{a \in \CC} \FF_p \big) / \big( \bigoplus_{a \in \CC} \FF_p \big) \comma
    \end{equation*}
    which is $2^{2^{\aleph_0}}$.
    
    On the other hand, we also see that $\uppi_1( f_\sharp( \mathbf{1})(\ast))$ is a quotient of $\Freepf_{\overline{\QQ}}$ an thus its cardinality is at most $2^{\aleph_0}$ by the computation in \Cref{rem:counterexample_to_pioneproper}.
    We conclude that $g_\sharp(\mathbf{1})$ and $f_\sharp(\mathbf{1})$ cannot be isomorphic, as desired.
\end{remark}

%-------------------------------------------------------------------%
%  Preliminaries on quasiseparated quotients                        %
%-------------------------------------------------------------------%

\subsection{Preliminaries on quasiseparated quotients}\label{subsec:preliminaries_on_quasiseparated_quotients}

\begin{recollection}\label{rec:quasiseparated_condensed_set}
    A condensed set $ A $ is \emph{quasiseparated} if for any maps $B\to A$ and $B'\to A$ in which $B$ and $B'$ are quasicompact, the pullback $B\times_A B'$ is quasicompact as well. 
    We denote by $\Cond(\Set)^{\qs}\subset \Cond(\Set)$ the full subcategory that is spanned by the quasiseparated condensed sets.
\end{recollection}

\begin{lemma}[{\cite[Lemma 4.14]{Scholze:analyticnotes}}]\label{lem:quasiseparated-quotient}
    The inclusion $\Cond(\Set)^{\qs}\subset \Cond(\Set)$ admits a left adjoint $(-)^{\qs}$ that preserves finite products. 
   
    Explicitly, if $ A $ is a condensed set, its \emph{quasiseparated quotient} $A^{\qs}$ can be computed by choosing a cover $U =\coprod_{i \in I} S_i\twoheadrightarrow A$ by profinite sets and by defining $A^{\qs}$ as the quotient of $ U $ by the closure of the equivalence relation $U \times_A U \subset U \times U$.
\end{lemma}

Since $(-)^{\qs}$ preserves finite products, it induces a functor $\Cond(\Grp)\to \Cond(\Grp)^{\qs}$ which is left adjoint to the inclusion.
Our next goal is to derive a more explicit description of the quasiseparated quotient of a condensed group.

\begin{definition}\label{def:closed-embedding}
    An inclusion $C \subset A$ of condensed sets is \emph{closed} if for every profinite set $ S $ and map $S \to A$, the pullback $C \times_A S \subset S$ is a closed subspace.
\end{definition}

\begin{proposition}\label{prop:quasiseparated-quotient-for-groups}
    Let $ G $ be a condensed group, and let $\overline{\{1\}}\subset G$ denote the intersection of all closed normal subgroups of $ G $.
    Then there is a natural isomorphism
    \begin{equation*}
        G^{\qs} \isomorphism G/\overline{\{1\}} \period
    \end{equation*}
\end{proposition}

\noindent For the proof, we need two auxiliary results.

\begin{lemma}\label{lem:quotient-by-closed-equivalence-relation-is-quasiseparated}
    Let $ A $ be a condensed set and let $R\subset A \times A$ be a closed equivalence relation. 
    Then the quotient $A/R$ is quasiseparated.
\end{lemma}

\begin{proof}
    First, let us choose a cover $U = \coprod_{i\in I} S_i\twoheadrightarrow A$ by profinite sets $S_i$.
    Set
    \begin{equation*}
        R_I \colonequals R \crosslimits_{A\times A} (U \times U)
    \end{equation*}
    and note that $R_I$ defines a closed equivalence relation on $ U $ with the property that the natural map $ U/R_I \to A/R $ is an isomorphism.
    Let $\Lambda$ be the filtered poset of finite subsets of $I$, and for each $J\in \Lambda$, let $U_J=\coprod_{j\in J} S_j$. 
    Then we can write $ U $ as the filtered union of the $U_J$, and for each $J \subset J'$ the inclusion $U_J \subset U_{J'}$ is a closed immersion of compact Hausdorff spaces.
    Moreover, for each $J\in \Lambda$, let us set
    \begin{equation*}
        R_J \colonequals R_I \crosslimits_{U \times U} (U_J \times U_J) \period
    \end{equation*}
    Then each $R_J$ defines a closed equivalence relation on $U_J$, and, since $\Lambda$ is filtered, we have $R = \colim_{J \in \Lambda} R_J$. 
    As a consequence, we may identify $ \colim_{J \in \Lambda} U_J/R_J \equivalent A/R $.
    Now since each $R_J$ is a closed equivalence relation on $U_J$, the condensed set $U_J/R_J$ is a compact Hausdorff space.
    Moreover, for every inclusion $U_J \subset U_{J'}$, the induced map $U_J/R_J \to U_{J'}/R_{J'}$ is injective by construction of $R_J$ and $R_{J'}$ and is therefore automatically a closed immersion.
    Hence the desired result follows from~\cite[Proposition~1.2~(4)]{Scholze:analyticnotes}.
\end{proof}

\begin{lemma}\label{lem:kernel-to-qs-group-is-closed}
    Let $ \phi \colon G \to H$ be a homomorphism of condensed groups.
    If $H$ is quasiseparated, then $\ker(\phi) $ is a closed subgroup of $ G $.
\end{lemma}

\begin{proof}
    Since $\ker(\phi)$ is the inverse image of $\{1\}\subset H$, it suffices to show that $ \{1\} $ is closed in $ H $.
    For this, pick any map from a profinite set $S \to H$. 
    Since $ S $ and $\{1\}$ are quasicompact and $H$ is quasiseparated, the fiber product $S \times_H \{1\} \subset S$ is quasicompact. 
    Since a subobject of a quasiseparated condensed set is quasiseparated, $ S \times_H \{1\} $ is also quasiseparated. 
    It follows that $ S \times_H \{1\} $ is compact, and hence a closed subset of $ S $, as desired. 
\end{proof}

\begin{proof}[Proof of \Cref{prop:quasiseparated-quotient-for-groups}]
    We begin by showing that the quotient $G/{\overline{\{1\}}}$ is quasiseparated. 
    To see this, first note that the map
    \begin{equation}\label{eq:equivalence-relation-quotient-group}
        (\pr_0, \mult)\colon G\times \overline{\{1\}}\to G\times G 
    \end{equation}
    is a closed immersion since when composing this map with the isomorphism $G\times G\to G\times G$ given by $(g,h) \mapsto (g, g^{-1}h)$, the resulting map can be identified with the product of the identity with the inclusion.
    Observe that the map in \eqref{eq:equivalence-relation-quotient-group} is precisely the equivalence relation defining the quotient group $G/{\overline{\{1\}}}$.
    Hence the quasiseparatedness of $G/{\overline{\{1\}}}$ follows from \Cref{lem:quotient-by-closed-equivalence-relation-is-quasiseparated}.

    To complete the proof, we need to show that for every map $\phi \from G\to H$ of condensed groups in which $H$ is quasiseparated, the kernel $\ker(\phi)$ contains $\overline{\{1\}}$. For this, it suffices to check that $\ker(\phi)$ is closed. This is \Cref{lem:kernel-to-qs-group-is-closed}.
\end{proof}   

In order to produce short exact sequences on the level of quasiseparated quotients, it is useful to know the following analogue of being a locally cartesian localization for the quasiseparated quotient.

\begin{proposition}\label{prop:qs-quotient-in-short-exact-sequences}
    Let $ \begin{tikzcd}[cramped, sep=small]
        1 \arrow[r] & N \arrow[r] & G \arrow[r] & H \arrow[r] & 1
    \end{tikzcd} $
    be a short exact sequence of condensed groups.
    If $ H $ is quasiseparated, the induced sequence $ \begin{tikzcd}[cramped, sep=small]
        1 \arrow[r] & N^{\qs} \arrow[r] & G^{\qs} \arrow[r] & H \arrow[r] & 1
    \end{tikzcd} $
    is again exact.
\end{proposition}

\begin{proof}
    Since $H = H^{\qs} $, we only need to show that $N^{\qs} \to G^{\qs}$ is injective. 
    Again since $H$ is quasiseparated, \Cref{lem:kernel-to-qs-group-is-closed} shows that $N \to G$ is closed.
    Therefore, $\overline{\{1\}}{}^{N} = \overline{\{1\}}{}^{G}$ (as subgroups of $ G $), and thus
    \begin{equation*}
        N^{\qs} = N/\overline{\{1\}}{}^{N} \longrightarrow G/\overline{\{1\}}{}^{G} = G^{\qs}
    \end{equation*}
    is injective.
\end{proof}

We now obtain a fundamental exact sequence of the quasiseparated quotient of the condensed fundamental group.

\begin{notation}
    Given a scheme $ X $ and geometric point $ \fromto{\xbar}{X} $, we write
    \begin{equation*}
        \pionecondqs(X,\xbar) \colonequals \pionecond(X,\xbar) ^{\qs}
    \end{equation*}
    for the quasiseparated quotient of the condensed fundamental group of $ X $.
\end{notation}

\begin{corollary}[(fundamental exact sequence on quasiseparated quotients)]\label{cor:fundamental-fiber-sequence-on-qs-quotients}
    Let $ k $ be a field with separable closure $ \kbar $, let $ X $ be a qcqs $ k $-scheme, and let $\xbar \to X_{\kbar}$ be a geometric point.
    If $ X $ is geometrically connected and $ X_{\kbar} $ has finitely many irreducible components, then the sequence of condensed groups
    \begin{equation*}
        \begin{tikzcd}[cramped, sep=small]
            1 \arrow[r] & \pionecondqs(X_{\kbar},\xbar) \arrow[r] & \pionecondqs(X,\xbar) \arrow[r] & \Gal_{k} \arrow[r] & 1
        \end{tikzcd}
    \end{equation*}
    is exact.
\end{corollary}

\begin{proof}
    Combine \Cref{cor:fundamental_exact_sequence_for_condensed_homotopy_groups,rmk:fundamental_exact_sequence_for_condensed_homotopy_groups_under_geometric_connectedness} with \Cref{prop:qs-quotient-in-short-exact-sequences}.
\end{proof}

%-------------------------------------------------------------------%
%  π₁ᶜᵒⁿᵈ of geometrically unibranch schemes                        %
%-------------------------------------------------------------------%

\subsection{\texorpdfstring{$\pionecondqs$}{π₁ᶜᵒⁿᵈ} of geometrically unibranch schemes}\label{subsec:picond_1_of_geometrically_unibranch_schemes}

It is a common theme in arithmetic geometry that various generalizations of $\pioneet$ are all equal (and profinite) for normal (more generally: geometrically unibranch) schemes. 
See \cite[Theorem 11.1]{MR0245577} and \cite[Lemma 7.4.10]{MR3379634} for instances of this phenomenon. 
As we saw before, this fails for \smash{$\pionecond$} and $X = \AA^1_{\CC}$. However, the expected behavior still holds for $\pionecondqs$. Proving this fact is the main goal of this subsection.

\begin{theorem}\label{thm:normal-implies-profinite}
    Let $ X $ be a qcqs geometrically unibranch scheme with finitely many irreducible components, and let $\xbar \to X $ be a geometric point. 
    Then the natural homomorphism $ \pionecond(X,\xbar) \to \pioneet(X,\xbar) $ induces an isomorphism
    \begin{equation*}
        \pionecondqs(X,\xbar) \isomorphism \pioneet(X,\xbar) \period
    \end{equation*}
    In particular, $ \pionecondqs(X,\xbar) $ is a profinite group.
\end{theorem}

For the proof, we need the following observation.

\begin{proposition}\label{prop:profinite_completion_pionecond}
    Let $ X $ be a qcqs scheme such that $ \pizerocond(X) $ is discrete. 
    Then for any geometric point $ \xbar \to X $, the natural comparison homomorphism 
    \begin{equation*}
      \pionecond(X, \xbar) \to \pioneet(X,\xbar)
    \end{equation*}
    of \cref{nul:map-between-condensed-and-etale-homotopy-type} exhibits $\pioneet(X, \xbar)$ as the profinite completion of $\pionecond(X, \xbar)$.
    The hypothesis on \smash{$ \pizerocond(X) $} is satisfied, for example, when $ X $ has locally finitely many irreducible components.
\end{proposition}

\begin{proof}
    Combine \Cref{lem:pi1-of-completion}, \Cref{lem:Picond_recovers_Piet}, and \Cref{cor:pi0s_match_for_finitely_many_irr_comps}.
\end{proof}

To prove the main result, we first want to show that this quasiseparated quotient is a compact topological group.
For this, we make use of the following simple consequence of the fact that the fundamental group of a simplicial set coincides with the fundamental group of its geometric realization:

\begin{lemma}\label{lem:surjection-on-pi1}
    Let $f \colon T_\bullet \to S_\bullet$ be a map of simplicial sets that is bijective on vertices and surjective on edges. 
    Then, for any choice of basepoint $ t \in T_0 $, the induced homomorphism 
    \begin{equation*}
        \flowerstar \from \uppi_1(T_\bullet, t) \to \uppi_1(S_\bullet, f(t))
    \end{equation*}
    is surjective.
    \hfill \qed
\end{lemma}

% \begin{proof}
%   The fundamental group of a pointed simplicial set $(T,t)$ is generated bysequences of consecutive edges (disregarding orientation) such that the first edge starts and the last edge ends in $t \in T_0$.
%   It thus suffices to prove that we can lift any such sequence in $ V $ starting and ending in $v$ to one in $W$ starting and ending in $w$.
%   So let us fix such a sequence in $(V,v)$. 
%   By surjectivity on edges, we can lift every edge appearing in this sequence to some edge in $W$. 
%   By bijectivity on vertices, any choice of such lifts assemble into a sequence starting and ending in $w$.
%   This finishes the proof.

%   See also Andy Putman's answer in \href{https://mathoverflow.net/questions/362500/morphism-with-connected-fibers-induce-surjection-on-fundamental-groups}{https://mathoverflow.net/questions/362500/morphism-with-connected-fibers-induce-surjection-on-fundamental-groups}.
% \end{proof}

\begin{lemma}\label{lem:surjection-on-pi1cond}
    Let $Y \to X$ be a morphism of qcqs schemes. 
    Assume that there exist proétale hypercovers $X_\bullet' \to X $ and $ Y_\bullet' \to Y $ by w-strictly local schemes and a morphism $Y'_\bullet \to X'_\bullet$ that fit into a commutative square
    \begin{equation*}
        \begin{tikzcd}
            Y_\bullet' \arrow[r] \arrow[d] & X_\bullet' \arrow[d] \\
            Y \arrow[r] & X
        \end{tikzcd}
    \end{equation*}
    such that:
    \begin{enumerate}
        \item The induced map of profinite sets $\uppi_{0}(Y_0') \to \uppi_{0}(X_{0}')$ is a bijection (and thus, a homeomorphism).

        \item The induced map of profinite sets $\uppi_{0}(Y_1') \to \uppi_{0}(X_1')$ is a surjection (and thus, a topological quotient map).    
    \end{enumerate}
    Then, for any choice of geometric points $\ybar \mapsto \xbar$, the induced homomorphism 
    \begin{equation*}
        \pionecond(Y,\ybar) \to \pionecond(X,\xbar)
    \end{equation*}
    is a surjection of condensed groups.
\end{lemma}

\begin{proof}
    By \cref{rec:homotopy_groups_of_condensed_anima,cor:characterization_of_the_condensed_homotopy_type_as_a_proétale_cosheaf,prop:shape-of-w-strictly-local}, the fundamental group $\pionecond(X,\xbar)$  can be computed as
    \begin{equation*}
        \Extr^{\op} \ni S \mapsto \uppi_{1}\bigg(\colim_{[m] \in \Deltaop}\Map_{\Top}(S,\uppi_{0}(X_m')),\xbar\bigg) \period
    \end{equation*}
    In other words, for each extremally disconnected profinite set $ S $, we have to compute the fundamental group of the simplicial set $\Map_{\Top}(S,\uppi_{0}(X'_\bullet))$ given by $[m] \mapsto \Map_{\Top}(S, \uppi_{0}(X'_{m}))$.
    Analogous statements hold for $Y'_\bullet$ and $Y$.
    
    The assumptions on the maps $\uppi_{0}(Y'_0) \to \uppi_{0}(X'_{0})$ and $\uppi_{0}(Y'_1) \to \uppi_{0}(X'_1)$ imply that, for each $S \in \Extr$, the induced map
    \begin{equation*}
        \Map_{\Top}(S, \uppi_{0}(Y'_{\! \bullet})) \to \Map_{\Top}(S, \uppi_{0}(X'_{\bullet}))
    \end{equation*}
    of simplicial sets satisfies the assumptions of \Cref{lem:surjection-on-pi1}. 
    It follows that, for each $ S $, the map
    \begin{equation*}
        \pionecond(Y,\ybar)(S) \to \pionecond(X,\xbar)(S)
    \end{equation*}
    is a surjection, as desired.
\end{proof}

\begin{lemma}\label{lem:generic-surjective-hypercover-geometrically-unibranch}
  Let $ X $ be a quasiseparated, geometrically unibranch, irreducible scheme and let $\eta \in X$ be its generic point. 
  Let $X_\bullet$ be any proétale hypercover by \wcontractible qcqs schemes of $ X $.  
  Then there exists a proétale hypercover $Y_{\!\bullet}$ of $\eta$ satisfying the conditions of \Cref{lem:surjection-on-pi1cond} (with respect to $X_\bullet$ and the map $\eta  \to X$).    
\end{lemma}

\begin{proof}
    Let $X_{\bullet,\eta}$ be the basechange of $X_\bullet$ to $\eta$.
    Note that by geometrical unibranchness and the fact that each connected component of a \wcontractible proétale $X'$ over $ X $ is the strict localization at some geometric point of $ X $ (see, e.g., \cite[Lemma 3.15]{LaraHESpaper}), the map $\uppi_{0}(X_{\bullet,\eta}) \to \uppi_{0}(X_\bullet)$ is a levelwise homeomorphism.
    In particular, the profinite sets $\uppi_{0}(X_{i,\eta})$ are still extremally disconnected. 
    Being w-strictly local, however, will usually be lost after base-changing to $\eta$. 
    We want to define a w-strictly local hypercover $Y_{\!\bullet}$ of $\eta$ with a map to $X_{\bullet,\eta}$ that still has the desired properties on $\uppi_{0}$ in low degrees.

    To do that, fix a geometric point $\etabar$ lying over $\eta$ and write $X_{0, \etabar} \colonequals X_{0, \eta} \times_{\eta} \etabar$.
    The projection induces a surjective map of profinite sets $\uppi_{0}(X_{0, \etabar}) \to \uppi_{0}(X_{0, \eta})$. 
    As the target is extremally disconnected, this map admits a section.
    Let $T \subset \uppi_{0}(X_{0, \etabar})$ be the image of one such section. 
    By \cite[Lemma 2.2.8]{BhattScholzeProetale}, there exists a pro-(Zariski localization) $W_0 \to X_{0, \etabar}$ that realizes the map $T \subset  \uppi_{0}(X_{0, \etabar})$ on connected components.
    Such $W_0$ is, in particular, weakly étale over $\etabar$; by \cref{ex:over_sep_fields_everything_is_wstrl} we deduce that $ W_0 $ is w-strictly local.
    By construction, the map $\uppi_{0}(W_0) \to \uppi_{0}(X_{0,\eta})$ induced by $W_0 \to X_{0, \etabar} \to X_{0, \eta}$ is a homeomorphism.
    
    We can extend this to a map of hypercovers
    \begin{equation*}
        Y_{\bullet} \colonequals \cosk_0(W_0) \crosslimits_{\cosk_0(X_{\bullet,\eta})} X_{\bullet,\eta} \longrightarrow X_{\bullet,\eta}
    \end{equation*}
    that induces a bijection on $ 0 $-simplices.
    The map on $1$-simplices is explicitly given by
    \begin{equation}\label{eq:map_on_1-simplices}
        (W_0 \times_\eta W_0) \crosslimits_{X_{0,\eta}\times_\eta X_{0,\eta}} X_{1,\eta} \longrightarrow X_{1,\eta} \period
    \end{equation}
    Since $W_0 \to X_{0,\eta}$ is surjective, we deduce that \eqref{eq:map_on_1-simplices} is surjective.
    Furthermore, all terms of $Y_\bullet$ are weakly étale over $\etabar$, hence, by \cref{ex:over_sep_fields_everything_is_wstrl}, they are w-strictly local.
    This completes the proof.
\end{proof}

\begin{corollary}\label{cor:pi1cond-quotient-of-Galois-group-of-function-field-normal-case}
    Let $ X $ be a quasiseparated, geometrically unibranch, irreducible scheme with generic point $\eta \in X$.
    Choose a geometric point $ \etabar $ lying over $ \eta $.
    Then the natural map 
    \begin{equation*}
        \Gal_{\upkappa(\eta)} = \pionecond(\Spec(\upkappa(\eta)), \etabar) \longrightarrow \pionecond(X, \etabar)
    \end{equation*}
    is a surjection of condensed groups.
\end{corollary}

\begin{proof}
    Combine \Cref{lem:surjection-on-pi1cond,lem:generic-surjective-hypercover-geometrically-unibranch,ex:Picond-of-a-field}.
\end{proof}

\begin{lemma}\label{lem:qs-quotient-of-profinite-group-is-profinite}
    Let $G' \twoheadrightarrow G$ be a surjection of condensed groups. 
    Assume that $G'$ is a profinite group. 
    Then the quasiseparated quotient $G^{\qs}$ is a profinite group.
\end{lemma}

\begin{proof}
    Since the quotient of a quasicompact condensed set is quasicompact, the quotient $G^{\qs}$ is qcqs.
    By \cite[Proposition 2.8]{ClausenScholze:Complex}, its underlying condensed set is a compact Hausdorff space. 
    Since the embedding of compact Hausdorff spaces into condensed sets is fully faithful and commutes with products finite products, it follows that $G^{\qs}$ is a compact Hausdorff group.
    Since $ G^{\qs} $ also admits a surjection from the profinite group $ G' $, we deduce that the compact Hausdorff group $G^{\qs}$ is itself profinite.
\end{proof}

Finally, we are ready to prove the main result of this subsection.

\begin{proof}[Proof of \Cref{thm:normal-implies-profinite}]
    Note that, since $\ProFinGrp \subset \CondGrp^{\qs} \subset \CondGrp$, the profinite completion $G^{\wedge}$ of a condensed group $ G $ factors over the quasiseparated quotient $G^{\qs}$ of $ G $.
    Our assumptions guarantee that every connected component of $ X $ is irreducible.
    By the preceding preparatory results \Cref{cor:pi1cond-quotient-of-Galois-group-of-function-field-normal-case,lem:qs-quotient-of-profinite-group-is-profinite}, we thus have that $\pionecondqs(X, \xbar)$ is already profinite, hence agrees with the profinite completion \smash{$\pionecond(X, \xbar)^{\wedge}$}.
    By \Cref{prop:profinite_completion_pionecond}, this latter profinite completion recovers \smash{$\pioneet(X, \xbar)$}.
    This completes the proof.
\end{proof}

\begin{warning}\label{rem:a_naive_qs_quotient_doesnt_exist}
    It seems like a natural idea to try to extend the notion of quasiseparatedness and quasiseparated quotients to all \emph{condensed anima}, and also extend \Cref{thm:normal-implies-profinite} from fundamental groups to homotopy types.
    However, a sufficiently nicely behaved quasiseparated quotient of condensed anima can \emph{not} exist.
    More precisely, there is \emph{no} full subcategory $ \Ccal \subset \CondAni$ with the following properties:
    \begin{enumerate}
        \item The inclusion $\Ccal \subset \CondAni$ admits a left adjoint $(-)^{\qs}$.

        \item A condensed set is in $ \Ccal $ if and only if its is quasiseparated.
        
        \item For any quasiseparated condensed group $ G $, the condensed anima $\Bup G$ is contained in $ \Ccal $.
    \end{enumerate}
    Indeed,  both $\Bup \ZZ$ and $\Bup \ZZhat$ would be contained in $ \Ccal $.
    Since $\ZZhat/\ZZ$ is the fiber of the canonical map $\Bup \ZZ \to \Bup \ZZhat $, the condensed set $\ZZhat/\ZZ$ would also be contained in $ \Ccal $.
    But $\ZZhat/\ZZ$ is not quasiseparated.
\end{warning}

%-------------------------------------------------------------------%
%  The van Kampen and Künneth formulas for π₁ᶜᵒⁿᵈ                   %
%-------------------------------------------------------------------%

\subsection{The van Kampen and Künneth formulas for \texorpdfstring{$\pionecondqs$}{π₁ᶜᵒⁿᵈ}}\label{subsec:van_Kampen_for_picond_1}

The goal of this subsection is to prove a van Kampen formula for the quasiseparated quotient of the condensed fundamental group (\Cref{cor:vanKampen-for-cond-qs}).
We then use this to prove a Künneth formula for this quasiseparated quotient (\Cref{cor:Kunneth_formula_for_quasiseparated_fundamental_groups}).
To do this, we start by analyzing the relationship between free topological groups and free condensed groups as well as free products of topological groups and condensed groups.

\begin{notation}
    The forgetful functor $ \CondGrp \to \CondSet$ has a left adjoint
    \begin{equation*}
        \Free^\cond_{(-)} \colon \CondSet \to \CondGrp \period
    \end{equation*}
    For a condensed set $ M $, the condensed group $\Free^\cond_M$ is given more explicitly as the sheafification of the functor 
    \begin{align*}
        \Free^{\pre}_M \colon \ProSetfin^{\op} &\to \Grp \\ 
        S &\mapsto \Free_{M(S)} \period
    \end{align*}
    The free group on $M$ comes with a canonical map $M \to \Free^\cond_M$ in $\CondSet$.
\end{notation}

\begin{nul}
    For a profinite set $ T $, we want to compare \smash{$ \Free^\cond_T $} with \smash{$\Freetop_T$}, i.e., the free topological group on $ T $ (see \cite[Chapter 7]{ArhangelskiiTkachenko}).
    Note that, by the universal property of \smash{$ \Free^\cond_T $}, there is a canonical homomorphism 
    \begin{equation*}
      \Freecond_T \to \underline{\Freetop_T}
    \end{equation*}
    in $\CondGrp$.
    To do this, we recall some important facts about free topological groups adn free products of topological groups.
\end{nul}

\begin{recollection}[(on free topological groups and products)]\label{rmk:recap-free-top-gps}
    In this recollection, $ T $ always denotes a topological space and $G_i$ denote topogical groups.
    \begin{enumerate}
        \item\label{rmk:item:recap-free-top-gps:compactly-gen} Markov showed that for every Tychonoff (=completely regular) space $ T $, the free topological group $\Freetop_T$ on $ T $ exists and the unit $\eta \colon T \to \Freetop_T$ is a topological embedding. 
        In addition, the image $\eta(T)$ is a free algebraic basis for $ G $.
        See \cite[Theorems 7.1.2 \& 7.1.5]{ArhangelskiiTkachenko}.
        
        \item When $ T $ is compact (more generally, $ \mathrm{k}_{\upomega}$), Graev--Mack--Morris--Ordman showed that $ \Freetop_T $ is the topological colimit of subspaces
        \begin{equation*}
          (\Free_T)_{\leq n} = \{\textrm{words of reduced length } \leq n \} \period
        \end{equation*}
        See \cite[Theorem 7.4.1]{ArhangelskiiTkachenko}.
        
        \item By \cite{Graev}, the underlying set of $\freeprodtop_i G_i$ is the abstract free product and if the groups are Hausdorff, their free product is Hausdorff too.
        
        Moreover, when each $G_i$ is either compact or finitely generated discrete (e.g., $\ZZ^{\freeprod r}$), by looking at the surjection from a suitable free product (see \Cref{lm:freegp-surjects-onto-freeprod} below) and using \eqref{rmk:item:recap-free-top-gps:compactly-gen}, it follows that $\freeprodtop_i G_i$ is a topological colimit of compact subsets of \textit{bounded words}.
        Here, by bounded words we in particular mean that all ``letters'' from one of the copies of $\ZZ$ sit inside of some interval $[-n,n]$. 
        See \cite[Remark 4.27]{LaraFESpaper}.
    \end{enumerate}
\end{recollection}

\begin{recollection}
    In the context of (abstract) free groups on a set $M$ (resp., free products of groups $G_1, \ldots, G_n$) we say that $g_{m_1}^{r_1} \cdots g_{m_n}^{r_n}$ (resp., $g_1 \cdots g_n$), where $g_{m_i}$ is the generator corresponding to $m_i \in M$ (resp., where $g_i$ is a nontrivial element of one of the groups $G_{j(i)}$) is a \emph{reduced word} if for $1 \leq i < n$, we have $m_i \neq m_{i+1}$ (resp., $j(i) \neq j(i+1)$). 
\end{recollection}

The following result is a nonabelian analogue of \cites[Proposition 2.1]{Scholze:analyticnotes}.
The proof essentially follows the one of \emph{loc. cit.}

\begin{proposition}\label{prop:freecondgp-vs-freetopgp}
    Let $ T $ be a compact Hausdorff topological space.
    Then the natural map
    \begin{equation}\label{eq:map_from_free_condensed_group_to_free_topological_group}
        \Freecond_T \to \underline{\Freetop_T}
    \end{equation}
    is an isomorphism.
\end{proposition}

\begin{nul}
    In the proof, we use the following convention: for a profinite set $ S $ and $t \in T(S)$, we denote by $g_t \in \Freecond_T$ the element given by the composite
    \begin{equation*}
        \begin{tikzcd}
            S \arrow[r, "t"] & T \arrow[r] & \Freecond_T \comma
        \end{tikzcd}
    \end{equation*}
    where $T \to \Freecond_T$ is the unit map.
\end{nul}

\begin{proof}
    First, we want to check that the map \eqref{eq:map_from_free_condensed_group_to_free_topological_group} is injective. 
    Note that this boils down to checking that any section of $\Free^{\pre}_T$ that maps to \smash{$1 \in \underline{\Freetop_T}$}, trivializes after passing to a cover in $\ProSetfin$.

    Observe that this is the case for the underlying groups.
    Indeed, it is enough to check that the map $\Free_{T(*)} \to \Freetop_{T}(*)$ is injective.\footnote{We are using here that evaluating $\Free^\cond_W$ on $*$ as a sheaf is the same as evaluating its defining presheaf.} 
    This follows directly from \Cref{rmk:recap-free-top-gps} \eqref{rmk:item:recap-free-top-gps:compactly-gen}.

    We now treat the injectivity for a general $S \in \ProSetfin$.
    Assume that $1 \neq g \in \Free_{T(S)}$ maps to $1 \in \Freetop_T(S)$. 
    By the previous point, for any $s \in S$, the restriction $g(s) \in \Free_{T(*)}$ is trivial.
    Write $g$ as a reduced word $g = g_{t_1}^{r_1}g_{t_2}^{r_2}\cdots g_{t_m}^{r_m}$, where now $t_j \in T(S)$. 
    All $g_{t_j}$ are nonzero and, if $m > 1$, we have $g_{t_i} \neq g_{t_{i+1}}$ for $1 \leq i \leq m-1$. 

    If $m=1$, then we plug in any $s \in S$ to see that $1 = g(s) = g_{t_1(s)}^{r_1}$. But the right hand side cannot be trivial being a generator in the free group raised to a nonzero power -- a contradiction.

    Assume now that $m>1$.
    Let $S_{j}$ denote the closed subset of $ S $ where $t_j = t_{j+1}$. 
    First, note that the $S_j$'s (where $1 \leq j < m$) jointly cover $ S $. 
    Indeed, if that would not be the case, then any point $s$ in the complement would have the property that
    \begin{equation*}
        1 = g(s) =  g_{t_1(s)}^{r_1}g_{t_2(s)}^{r_2}\cdots g_{t_m(s)}^{r_m} 
    \end{equation*}
    is a nontrivial reduced word, a contradiction.

    Thus, passing to a finite closed cover of $ S $, we can assume that $t_j = t_{j+1}$ for some $j$, effectively decreasing the ``$m$'' in the shortest word that $g$ can be written as.
    By induction, this implies that $g$ has to be trivial -- a contradiction.

    As the proof of injectivity is finished, we now move on to surjectivity. 
    Consider the map of compact topological spaces
    \begin{equation*}
        T^n \times \{-1, 0, 1 \}^n \to (\Freetop_T)_{\leq n}
    \end{equation*}
    given by $(t_1,\ldots,t_n,\epsilon_1,\ldots,\epsilon_n) \mapsto g_{{t_1}}^{\epsilon_1} \cdots g_{{t_n}}^{\epsilon_n}$.
    This map is clearly surjective.
    It fits into a commutative square
    \begin{equation*}
        \begin{tikzcd}
            T^n \times \{-1, 0, 1 \}^n \arrow[r] \arrow[d] & (\Freetop_T)_{\leq n} \arrow[d] \\
            \Free^\cond_T \arrow[r] & \bigcup_m (\Freetop_T)_{\leq m} = \Freetop_T \period
        \end{tikzcd}
    \end{equation*}
    Evaluating at any $S \in \Extr$, and using \cite[Lemma~4.3.7]{BhattScholzeProetale}, this shows the surjectivity of the lower horizontal map (by varying $ n $).
\end{proof}

\begin{remark}
    Assume that $S = \lim_i S_i$ is a profinite set with $S_i$ finite. 
    Essentially, the same proof strategy (but without having to use the results of \Cref{rmk:recap-free-top-gps} \eqref{rmk:item:recap-free-top-gps:compactly-gen}) shows further that $\Free_S^{\cond}$ and $\underline{\Free_S^{\top}}$ are isomorphic to the group $\underline{\bigcup_m \lim_i \left((\Free_{S_i})_{\leq m}\right)}$. 
    This is analogous to the presentation in \cite[Proposition 2.1]{Scholze:analyticnotes}.
\end{remark}

Now we turn to analyzing free products of condensed and topological group.

\begin{notation}
    We denote the the coproduct in the category of condensed groups by $*^\cond $.
    It can be explicitly described as the sheafification of the presheaf $*^{\pre}_iG_i$ given by
    \begin{align*}
        \ProSetfin^{\op} &\to \Grp \\
        S &\mapsto \freeprod_i G_i(S) \period
    \end{align*}
\end{notation}

\begin{nul}
    Free products of topological groups $\freeprodtop$ exist as well.
    For $G_i \in \TopGrp$ there is a canonical homomorphism $*^{\cond}_i \underline{G_i} \to \underline{*^{\top}_i G_i}$.
\end{nul}

In order to compare condensed and topological free products, we first prove an auxiliary lemma.

\begin{lemma}\label{lm:freegp-surjects-onto-freeprod}
    Let $G_1,\ldots,G_m$ be compact Hausdorff topological groups and $r \in \NN$. 
    Denote by $T = G_1 \sqcup \cdots \sqcup G_m \sqcup \{1,\ldots,r\} $ the topological space that is the disjoint union of the the topological groups $G_1,\ldots,G_m $ and $r$ singletons. 
    Then the canonical homomorphism
    \begin{equation*}
        \Free^\cond_T \to G_1 \freeprodcond \cdots \freeprodcond G_m \freeprodcond \ZZ^{\freeprodcond r}.
    \end{equation*}
    is surjective.
    An analogous fact holds for topological free products.
\end{lemma}

\begin{proof}
    The universal properties of these groups give a homomorphism as above (here, we are mapping each of the $r$ points in $ T $ to $1 \in \ZZ$ via one of the $r$ canonical maps $\ZZ \to \ZZ^{\freeprodcond r}$).
    This map already exists on the level of the defining presheaves and is surjective there, so the map of sheaves is surjective as well.
    
    We omit the details for the topological counterpart (it uses \Cref{rmk:recap-free-top-gps}).
\end{proof}

\begin{proposition}\label{prop:condprod-vs-topprod}
    Let $G_1,\ldots,G_m$ be compact Hausdorff topological groups and $r \in \NN$. 
    Then the natural map
    \begin{equation*}
        G_1 \freeprodcond \cdots \freeprodcond G_m \freeprodcond \ZZ^{\freeprodcond r} \longrightarrow \underline{G_1 \freeprodtop \cdots \freeprodtop G_m \freeprodtop \ZZ^{\freeprodtop r}}
    \end{equation*}
    is an isomorphism in $\CondGrp$.
\end{proposition}

\begin{proof}
    To see the surjectivity, one can either redo the argument in the proof of \Cref{prop:freecondgp-vs-freetopgp} or use its statement together with \Cref{lm:freegp-surjects-onto-freeprod} and the square (with $T = G_1 \sqcup \cdots \sqcup G_m \sqcup * \sqcup \cdots \sqcup *$)
    \begin{equation*}
        \begin{tikzcd}
            \Freecond_T \arrow[r] \arrow[d] & \freeprodcond_i G_i \arrow[d] \\
            \underline{\Freetop_T} \arrow[r] & \underline{\freeprodtop_i G_i} \period
        \end{tikzcd}
    \end{equation*}

    Now, for the injectivity, the argument is very similar to the proof of \Cref{prop:freecondgp-vs-freetopgp}. 
    We can work with $\freeprodpre_i G_i$. The homomorphism of underlying groups
    \begin{equation*}
        \freeprod_i G_i(*) \to \big(\freeprodtop_i G_i\big)(*)
    \end{equation*}
    is a bijection (see \Cref{rmk:recap-free-top-gps}).

    Now, fix $S \in \ProSetfin$ and let $g = g_1g_2 \cdots g_n \in \freeprod_i G_i(S)$ be mapping to $1 \in \big(\freeprodtop_i G_i\big)(S)$. 
    Here, each $g_j$ is in some $G_{\alpha(j)}(S)$ and we can assume this presentation of $g$ is a reduced word (we assume $m>1$ as the case when $m=1$ is again easy). 
    We know that $g(s) \in *_j G_j(*)$ is trivial for any $s \in S$. 

    Let $S_{j}$ denote the closed subsets of $ S $ where $g_j$ vanishes. 
    First, note that the $S_j$'s (where $1\leq j \leq n$) jointly cover $ S $. 
    Indeed, if that's not the case, then any point $s$ in the complement would have the property that $g(s) =  g_1(s)g_2(s)\cdots g_n(s)$ is a nontrivial reduced word -- a contradiction.

    But now, passing to the this cover, we have again reduced the length of the presentation of $g$ as a word. 
    We are done by induction.
\end{proof}

\begin{lemma}\label{lem:closure-in-topological}
    Let $ T $ be a compactly generated topological space.
    Sending a closed subspace $Z \subset T$ to $\underline{Z} \to \underline{T}$ induces an order-preserving bijection between closed subspaces of $ T $ and closed condensed subsets of $\underline{T}$.
    The inverse is given by sending a closed condensed subset $ Z \subset \underline{T}$ to $Z(\ast) \subset \underline{T}(\ast) = T$ equipped with the subspace topology.
\end{lemma}

\begin{proof}
    In order to avoid confusion during the proof, we will write $\underline{S}$ for the condensed set represented by a profinite set $ S $.
    We at first check that the inverse defined above is well-defined, that is, that $Z(\ast)$ is a closed subset of $ T $.
    We may check this after pulling back along any continuous map $f \colon S \to T$ for $ S $ a profinite set.
    Then the pullback $S \times_T Z(\ast)^{} \subset S$ is the subspace given by those $s \in S$ such that $f(s) \in Z(\ast)$.
    If we alternatively compute the pullback $Z \times_{\underline{T}} \underline{S}$ in $\CondSet$, then $Z \times_{\underline{T}} \underline{S} \subset \underline{S}$ is a closed condensed subset by definition.
    In particular, $(Z \times_{\underline{T}} \underline{S}) (\ast)$ is a closed subset of $ S $.
    But $(Z \times_{\underline{T}} \underline{S}) (\ast) = Z(\ast) \times_T S$, as subsets of $ S $, and thus $Z(\ast)$ is closed.

    Furthermore, for a closed subspace $ Z \subset T $, we have $Z = \underline{Z}(\ast)$.
    So, conversely, let us start with a closed condensed subset $Z \subset \underline{T}$.
    Then for any $S \in \ProFin$ we claim that the subset $ Z (S) \subseteq T(S)$ is given by those $f \colon \underline{S} \to \underline{T}$ such that for all $s \in S$, $f(s) \in Z(\ast)$.
    Indeed, since $Z$ is a subobject, $ f $ is in $Z(S)$, if and only if the monomorphism $j \colon Z \times_{\underline{T}}\underline{S} \to \underline{S}$ is an isomorphism.
    But since $j$ is a closed immersion, it follows that $j$ is an isomorphism if and only if $j(\ast)$ is.
    But this is the case if and only $f(s) \in Z(\ast) $ for all $s \in S$, as claimed. 
    Since the same description applies to the condensed subset represented by the subspace $Z(\ast)$ equipped with the closed subspace structure, the claim follows.
\end{proof}

\begin{corollary}\label{cor:qs-quotient-of-cgwh-group}
    Let $ G $ be a topological group and $H \lhd \underline{G}$ a normal condensed subgroup.
    Assume that $G^{\qs}$ is represented by a compactly generated topological group $ G_0 $.
    Let
    \begin{equation*}
        H_0 \colonequals \im(H \to G \to G^{\qs} \equivalent \underline{G_0}) \period
    \end{equation*}
    Then the canonical homomorphism of condensed groups
    \begin{equation*}
        \left(G/H\right)^{\qs} \to \underline{G_0 \big/{\overline{H_0(\ast)}}}
    \end{equation*}
    is an isomorphism.
    Here, $ \overline{H_0(\ast)} $ denotes the topological closure in $ G $.
\end{corollary}

\begin{proof}
    Comparing universal properties, we see that the natural map $\left(G/H\right)^{\qs} \to (G^{\qs}/H_0)^{\qs}$ is an isomorphism.
    By \Cref{prop:quasiseparated-quotient-for-groups}, it follows further that the natural map
    \begin{equation*}
        (G^{\qs}/H_0)^{\qs} \to G^{\qs}/ \overline{H_0}
    \end{equation*}
    is an isomorphism. 
    Now since $G^{\qs} \equivalent \underline{G_0}$, \Cref{lem:closure-in-topological} shows that $\overline{H_0} \equivalent \overline{H_0(\ast)}$, completing the proof.
\end{proof}

We now turn to the van Kampen formula.
To do so, we fix some notation.

\begin{notation}\label{ntn:normalization_for_van_Kampen}
    Let $ X $ be a scheme.
    \begin{enumerate}
        \item Assume $ X $ is connected and has finitely many irreducible components.
        Write $\nu \colon X^{\nu} \to X$ for the normalization and write
        \begin{equation*}
            X^{2\nu} \colonequals X^{\nu} \cross_X X^{\nu} \andeq X^{3\nu} \colonequals X^{\nu} \cross_X X^{\nu} \cross_{X} X^{\nu} \period
        \end{equation*}
        Assume that $ X^{2\nu} $ and $ X^{3\nu} $ also have finitely many irreducible components (this is true, for example, if $ X $ is Nagata).
        Decompose $X^{\nu} = \coprod_i X^{\nu}_i $ into connected components.
        Write $ \Gamma $ for the ``dual graph'' with vertices $V = \uppi_{0}(X^{\nu})$ and edges $E = \uppi_{0}(X^{2\nu})$, and fix a maximal tree $ T $ of $ \Gamma$.
       
        \item We write
        \begin{equation*}
            \Pionecond(X) \colonequals \trun_{\leq 1}\! \Picond(X) \andeq \Pioneetprofin(X) \colonequals \trun_{\leq 1}\!\Pietprofin(X)
        \end{equation*}
        for the \defn{condensed fundamental groupoid} of $ X $ and \defn{profinite étale fundamental groupoid} of $ X $, respectively.
        Here, $\trun_{\leq 1}$ denotes $1$-truncation of condensed (resp., profinite) anima.
    \end{enumerate}
\end{notation}

\begin{theorem}[(van Kampen formula for the quasiseparated fundamental group)]\label{cor:vanKampen-for-cond-qs}
    In the notation of \Cref{ntn:normalization_for_van_Kampen}, after making choices of geometric base points and étale paths (as in \cite[Corollary 5.3]{Stix:vanKampen}), there is a natural isomorphism
    \begin{equation*}
        \pionecondqs(X,\xbar) \simeq \underline{\big(\freeprodtop_i\piet_1(X^{\nu}_i,\xbar_i)\freeprodtop\uppi_{1}(\Gamma,T)\big)/H^{\tnc}} \comma
    \end{equation*}
    where $H$ is the subgroup generated by the following relations:
    \begin{enumerate}
        \item\label{item_in:cor:first_rel:vanKampen} For all $ e \in E $ and $ g \in \piet_1(e,\xbar(e)) $ we have $ \piet_1(\partial_1)(g)\vec{e} = \vec{e}\piet_1(\partial_0)(g) $.
        
        \item For all $ f \in \uppi_0(X^{3\nu}) $, we have 
        \begin{equation*}
            \overrightarrow{(\partial_2f)}\alpha^{(f)}_{102}(\alpha^{(f)}_{120})^{-1}\overrightarrow{(\partial_0f)}\alpha^{(f)}_{210}(\alpha^{(f)}_{201})^{-1}\Big(\overrightarrow{(\partial_1f)}\Big)^{-1} \alpha^{(f)}_{021} (\alpha^{(f)}_{012})^{-1} = 1 \period
        \end{equation*}
    \end{enumerate}
    Here, each $ \alpha^{(f)}_{ijk} $ is an element of some $ \piet_1(X^{\nu}_{\ell},\xbar_{\ell}) $ and $\overrightarrow{e}, \overrightarrow{(\partial_if}) \in \uppi_1(\Gamma, T)$.
\end{theorem}

\begin{proof}
    Combining \Cref{cor:integral-hyperdescent-main-result}, the fact that $1$-truncation is a left adjoint, and \cite[Proposition A.1]{arXiv:2207.09256}, we obtain an equivalence of condensed groupoids
    \begin{equation*}
        \colim_{[k] \in \Deltaop_{\leq 2}}\Pionecond(X^{k\nu}) \isomorphism \Pionecond(X) \period
    \end{equation*}
    The fixed geometric points and étale paths fix points and paths in $\Pionecond(X)(*) $, $ \Pionecond(X_i^{\nu})(*) $, \textellipsis, so also in any $\Pionecond(X)(S) $, $\Pionecond(X_i^{\nu})(S) $, \textellipsis for $S \in \Extr$.
    By \Cref{cor:pi0s_match_for_finitely_many_irr_comps}, these groupoids are connected. 
    We now want to pass from a statement about fundamental groupoids to a statement involving fundamental groups.
    For a fixed $S \in \Extr$, we can apply the usual ``discrete'' van Kampen formula: see \cite[Theorem 3.7]{LaraFESpaper} for a version for $2$-complexes of Noohi (and so also discrete) groups or \cite[Chapter IV, \S 5]{Bourbaki:TopologieAlg}, cf.\ also \cite{Stix:vanKampen}. 
    It implies that
    \begin{equation*}
    \pionecond(X,\xbar) \simeq \big(\freeprodcond_i\pionecond(X^{\nu}_i,\xbar_i)\freeprodcond\uppi_{1}(\Gamma,T)\big)/H'
    \end{equation*}
    where $H'$ is the normal condensed subgroup that for each $ S $ is generated by relations analogous relations as in the statement, but where $g \in \pionecond(e,\xbar(e))(S)$, etc.
    
    Now, passing to quasiseparated quotients and using $\pionecond(X^{\nu}_i,\xbar_i)^{\qs} = \piet_1(X^{\nu}_i,\xbar_i)$ (this is \Cref{thm:normal-implies-profinite}) together with \Cref{prop:condprod-vs-topprod} and \Cref{cor:qs-quotient-of-cgwh-group} yields the result. 
    
    We have used the following observation to get $ g \in \piet_1(e,\xbar(e)) $ as opposed to $ g $ being an element of $\pionecondqs(e,\xbar(e)) $ or $ \pionecond(e,\xbar(e)) $ in relation (\ref{item_in:cor:first_rel:vanKampen}): although $ X^{2\nu} $ might not be normal, so $ \pionecondqs(e,\xbar(e)) $ might differ from $ \pioneet(e,\xbar(e)) $, the maps $ \pionecondqs(\partial_1), \pionecondqs(\partial_0) $ have profinite groups as the targets and thus, factorize through the profinite completion of $ \pionecondqs(e,\xbar(e)) $, which is $ \pioneet(e,\xbar(e)) $ (cf.\ \Cref{prop:profinite_completion_pionecond}). 
    As the topological normal closure of the image of $ \pionecondqs(e,\xbar(e))(\ast) $ inside $ \pioneet(e,\xbar(e) $ is the whole group (one uses the universal property of the profinite completion to check this), the set of relations 
    \begin{equation*}
        \setbar{ \piet_1(\partial_1)(g)\vec{e}\piet_1(\partial_0)(g)^{-1} \vec{e}^{-1} }{  e \in E, g \in \piet_1(e,\xbar(e)) }
    \end{equation*}
    is still in $ H^{\tnc} $ and contains the original set of relations (i.e., a similarly-defined one where $g \in \pionecondqs(e,\xbar(e))$), as desired.
\end{proof}

\begin{example}
    Let $ k $ be a separably closed field.
    \begin{enumerate}
        \item Let $C_1 $ and $ C_2$ be normal curves over $ k $ with fixed closed points $c_i \in C_i$. Let $C = C \sqcup_{c_1 = c_2} C_2$ be the gluing of these curves along these closed points. 
        Then
        \begin{equation*}
            \picondqs_1(C,c) \equivalent \piet_1(C_1,c_1) \freeprodtop \piet_1(X_2,c_2) \period 
        \end{equation*}
        
        \item Let $ C $ be the nodal curve over $ k $ obtained from $ \PP_k^1 $ by identifying $ 0 $ and $ 1 $. 
        Then
        \begin{equation*}
            \picondqs_1(C,c) \equivalent \ZZ \period
        \end{equation*}
    \end{enumerate}
    For more computations involving the van Kampen formula (but for Noohi groups), see \cite{LaraFESpaper}.
\end{example}

\begin{corollary}[(Künneth formula for the quasiseparated fundamental groups)]\label{cor:Kunneth_formula_for_quasiseparated_fundamental_groups}
    Let $ k $ be a separably closed field and let $ X $ and $ Y $ be $ k $-schemes such that $ X $, $ Y $, and $ X \times_k Y $ satisfy the hypotheses of \Cref{ntn:normalization_for_van_Kampen}.
    Let $ \zbar \to X \cross_k Y $ be a geometric point lying over geometric points $ \xbar \to X $ and $ \ybar \to Y $. 
    If $ Y $ is proper or $ \characteristic(k) = 0 $, then the natural homomorphism of condensed groups
    \begin{equation*}
        \picondqs_1(X\times_{k} Y,\zbar) \to \picondqs_1(X,\xbar) \times \picondqs_1(Y,\ybar)
    \end{equation*}
    is an isomorphism.
\end{corollary}

To prove this result, one can combine the van Kampen formula for \smash{$\pionecondqs$} and the classical Künneth formula for $\piet_1$ as in the proof of \cite[Proposition 3.29]{LaraFESpaper}, but this would require one to argue using the explicit relations appearing in the van Kampen theorem. 
To avoid it, it is beneficial to first apply the classical van Kampen in the groupoid form and only compute the fundamental groups at the very end.
This is how we structure the proof below.

\begin{proof}[Proof of \Cref{cor:Kunneth_formula_for_quasiseparated_fundamental_groups}]
    Fix integral hypercovers $\nu_{X,\bullet}, \nu_{Y,\bullet}$ by normal schemes of $ X $ and $Y$.
    Their product is again an integral hypercover of $X \times_k Y$ by 
    normal schemes. 
    Apply $\Pietprofin(-)$ to these diagrams and pass to colimits in $\CondAni$. The fixed geometric point $\bar{z}$ points them. 
    Then $1$-truncate and apply $\pionecondqs(-)$ to both sides. We get a homomorphism of condensed groups
    \begin{equation*}
        \pione\biggl( \colim_{[m] \in \Deltaop}\Pioneetprofin(X_{m}\times Y_{m}), \pt \biggr)^{\qs} 
        \to 
        \pione\biggl(\colim_{[m] \in \Deltaop}\Pioneetprofin(X_{m})\times \Pioneetprofin(Y_{m}), \pt\biggr)^{\qs}
    \end{equation*}
    Using \cite[Proposition A.1]{arXiv:2207.09256}, we can compute the colimits as colimits over the full subcategory $ \Deltaop_{\leq 2} \subset \Deltaop $. 
    Apply the usual Künneth formula for $\pioneet$ (c.f.\ \cite[Exposé~X, Corollaire 1.7 \& Exposé~XII, Proposition 4.6]{MR50:7129} or \cite[\S 4]{MR4835288}), which implies that
    \begin{equation*}
        \Pioneetprofin(X_{m} \times Y_{m}) = \Pioneetprofin(X_{m})\times \Pioneetprofin(Y_{m}) \comma
    \end{equation*}
    to get an isomorphism
    \begin{equation*}
        \pione\biggl(\colim_{[m] \in \Deltaop_{\leq 2}}\Pioneetprofin(X_{m}\times Y_{m}), \pt\biggr)^{\qs} \isomorphism \pione\biggl(\colim_{[m] \in \Deltaop_{\leq 2}}\Pioneetprofin(X_{m}),\pt\biggr)^{\qs}\times \pione\biggl(\colim_{[m] \in \Deltaop_{\leq 2}}\Pioneetprofin(Y_{m}), \pt\biggr)^{\qs} \period
    \end{equation*}
    
    Now, using the equality $\pionecondqs = \pioneet$ on normal schemes and arguing via the van Kampen formula as in \Cref{cor:vanKampen-for-cond-qs} to replace the fundamental groupoids by groups, we get that, e.g., 
    \begin{equation*}
        \pione\biggl(\colim_{[m] \in \Deltaop_{\leq 2}}\Pioneetprofin(X_{m}),\pt\biggr)^{\qs} = \pionecondqs(X,\xbar)
    \end{equation*}
    and similarly for $Y$ and $X \times Y$. Note that $ X^{2\nu}, X^{3\nu} $ (and similarly for $Y^{\ldots}$) might not be normal, but in the van Kampen formula all maps from \smash{$\pionecondqs$} of (connected components) of those schemes will always factor though a profinite group (by normality of $ X^{\nu}, Y^{\nu} $ and $ X^{\nu} \times Y^{\nu}$), so we were allowed to replace \smash{$ \pionecond $} by $ \Pioneetprofin $ even for those non-normal schemes in the above computation (cf.\ similar argument appears in the proof of \Cref{cor:vanKampen-for-cond-qs}). This completes the proof.
\end{proof}

\begin{corollary}\label{cor:pi1qs-properness}
    Let $ K \supset k $ be an extension of separably closed fields, and let $ X $ be a $ k $-scheme satisfying the hypotheses of \Cref{ntn:normalization_for_van_Kampen}.  
    If $\characteristic(k) = 0$ or $ X $ is proper, then the projection $X_K \to X$ induces an isomorphism
    \begin{equation*}
        \pionecondqs(X_K) \isomorphism \pionecondqs(X) \period
    \end{equation*}
\end{corollary}

\begin{remark}
    In the parlance of \cite{kedlaya2017sheaves}, the property of schemes established in \Cref{cor:pi1qs-properness} could be called \textit{$\pionecondqs$-properness}.
    As explained in \Cref{rem:counterexample_to_pioneproper}, before passing to quasiseparated quotients, this is already false for $ X = \PP_k^1 $.
\end{remark}

\begin{remark}\label{rem:kurosh}
    In the context of anabelian geometry, it is sometimes beneficial to have a version of the Kurosh subgroup theorem available in the category of groups where our fundamental groups live, or at least its corollary: the characterization of maximal finite/compact/\textellipsis\ subgroups of a free product as a ``vertex subgroup'' (i.e., one of the free summands up to conjugation). 
    See, e.g., \cite{Mochizuki:Semi-graphs}. 
    Proving such a result for the proétale fundamental group seems rather tricky due to the presence of Noohi completions. 
    For $\pionecondqs$, however, this can be done: see \Cref{prop:vertex-subgroups-for-pi1qs}.
\end{remark}

\begin{proposition}\label{prop:vertex-subgroups-for-pi1qs}
    Let $ X $ be a scheme and $\xbar$ a geometric point.
    Assume that there are profinite groups $ (G_i)_{i \in I} $ and an integer $ r \in \NN $ such that
    \begin{equation*}
        \pionecondqs(X,\xbar) \ \simeq \ \freeprodtop_iG_i\freeprodtop \ZZ^{*r} \period
    \end{equation*}
    Let $H$ be a compact topological group and $\phi \colon H \to \pionecondqs(X,\xbar)$ a continuous homomorphism. 
    Then there exists an index $ i $ and an element  $g \in \pionecondqs(X,\xbar )$ such that 
    \begin{equation*}
        \im(\phi) \subset gG_ig^{-1} \period
    \end{equation*}
\end{proposition}

\begin{proof}
    This follows follows from \cite[Theorem 1]{MorrisNickolas:locally-compact-free-products}.
\end{proof}

\begin{remark}
    We expect the assumptions of \Cref{prop:vertex-subgroups-for-pi1qs} to be satisfied, e.g., when $ X $ is a (semistable) curve over a separably closed field $ k $, with $G_i = \pioneet(X^{\nu}_i,\xbar_i)$, where $X = \coprod_i X^{\nu}_i$ is the the normalization of $ X $.

     For \smash{$\pioneet$} (or even \smash{$\pioneproet$}), this is a classical computation using the van Kampen theorem when $ X $ is semistable.
     See \cite[Example 5.5]{Stix:vanKampen} in the case of \smash{$\pioneet$} or \cite[Theorem 1.17]{Lavanda} for \smash{$\pioneproet$}.
     With some care, this can be done for arbitrary curves, see \cite[Theorem 2.27]{larayuzhang2022theoremmeromorphicdescentspecialization}. 
     A similar computation (using \Cref{cor:vanKampen-for-cond-qs}) should extend this to \smash{$\pionecondqs$}.
\end{remark}

%-------------------------------------------------------------------%
%-------------------------------------------------------------------%
%  Noohi completion of the condensed fundamental group              %
%-------------------------------------------------------------------%
%-------------------------------------------------------------------%

\section{Noohi completion of the condensed fundamental group}\label{sec:Noohi_completion_of_the_condensed_fundamental_group}

Let $ X $ be a topologically noetherian scheme.
The goal of this section is to recover the proétale fundamental group \smash{$\pioneproet(X,\xbar)$} of \cite[\S7]{BhattScholzeProEtale} from the condensed fundamental group \smash{$\pionecond(X, \xbar)$}.
The main input needed for this is the observation that all \textit{weakly locally constant sheaves} in the sense of \cite[Definition 7.3.1]{BhattScholzeProetale} can be recovered from $\pionecond(X,\xbar)$.
We prove a stronger derived version of that result in \cref{sec:recovering_locally_constant_sheaves}.
In \cref{subsec:recovering_the_proetale_fundamental_group}, we explain how to Noohi complete condensed groups and show that the Noohi completion of $\pionecond(X,\xbar)$ is indeed the proétale fundamental group.
See \Cref{thm:recovering_BS_fundamental_group}.

%-------------------------------------------------------------------%
%  Recovering weakly locally constant sheaves                       %
%-------------------------------------------------------------------%

\subsection{Recovering weakly locally constant sheaves}\label{sec:recovering_locally_constant_sheaves}

In this subsection, we explain how to recover weakly locally constant proétale sheaves on a scheme $ X $ as representations of the condensed homotopy type.
The following is a generalization of \cite[Definition 7.3.1]{BhattScholzeProetale} to sheaves of anima:

\begin{recollection}
    Recall that for a qcqs scheme $ X $ there is a canonical algebraic morphism $\Sh(\uppi_{0}(X)) \to X_{\et} $ induced by sending a clopen subset of $ \uppi_{0}(X) $ to its preimage in $ X $.
	Furthermore, we say that \smash{$ F \in \Xproethyp $} is \defn{locally weakly constant} if there is a proétale cover $ \{U_i \to X\}_{i\in I} $ by qcqs schemes such that each $ \restrict{F}{U_i} $ is in the image of the canonical algebraic morphism
	\begin{equation*}
        \begin{tikzcd}
            \Sh(\uppi_{0}(U_i)) \arrow[r] & U_{i,\et}^{\hyp}  \arrow[r, "\nuupperstar"] & U^{\hyp}_{i,\proet} \period
        \end{tikzcd}
	\end{equation*}
	We write \smash{$ \wLoc(X) \subset \Xproethyp $} for the full subcategory spanned by the locally weakly constant sheaves.
\end{recollection}

We want to show that $ \wLoc(X) $ is equivalent to the \category of continuous functors from $ \Picond(X) $ into the following condensed subcategory of $ \ICond(\Ani) $.

\begin{definition}
	We define the condensed \category $ \Aniult $ by the assignment
	\begin{equation*}
		S \mapsto \Sh(S)
	\end{equation*}
	for every profinite set $ S $.%
    \footnote{The fact that $ \Aniult $ satisfies descent for surjections of profinite sets follows from the proper basechange theorem.
    See \cite[Theorem 0.5 \& Example 1.28]{arXiv:2210.00186}.}
    Similarly, we refer to the $ 0 $-truncated version of this condensed \category by $\Setult$.
\end{definition}

\begin{recollection}
    Let $ S $ be a profinite set, and write $ \cupperstar_S \colon \PSh(S) \to \CondAni_{/S} $ for the left Kan extension of the natural functor
    \begin{equation*}
        \Open(S) \inclusion \CondAni_{/S}
    \end{equation*}
    along the Yoneda emebdding.
    Then the restriction
    \begin{equation*}
       c_S^*\colon \Sh(S) \to \CondAni_{/S}
    \end{equation*}
    is a fully faithful left exact left adjoint.
    See \cite[\S3.2 \& Corollary 4.9]{arXiv:2210.00186}.
    Moreover, this comparison functor is natural in $ S $ \cite[Lemma 3.16]{arXiv:2210.00186}, hence induces a fully faithful functor of condensed \categories 
    \begin{equation*}
      \Aniult \hookrightarrow \ICond(\Ani) \period 
    \end{equation*}
\end{recollection}

\begin{remark}
    The superscript `ult' comes from the word \textit{ultrastructure}.
    Any category with filtered colimits and infinite products can be canonically upgraded to an ultracategory by equipping it with the \emph{categorical ultrastructure}, see \cite[Example~1.3.8]{Ultracategories}.
    In \cite[Construction~4.1.1]{Ultracategories} Lurie explains how to regard ultracategories as condensed categories.
    Furthermore it follows from  \cite[Theorem~3.4.4]{Ultracategories} that the image of $\Set$ equipped with the categorical ultrastructure is precisely $\Setult$.
\end{remark}

\begin{recollection}
    By \cite[Corollary~1.2]{MR4574234}, precomposition with the localization functor $b \colon \Gal(X) \to \BcondGal(X)=\CondShape{X}$ induces a fully faithful functor 
    \begin{equation*}
        \begin{tikzcd}
            b^* \colon \Functs\big(\CondShape{X}, \ICond(\Ani) \big) \arrow[r, hooked] & \Functs\big \lparen\Gal(X),  \ICond(\Ani) \big\rparen \simeq \Xproethyp \period
        \end{tikzcd}
    \end{equation*}
    Cf. the proof of \Cref{prop:Picond_is_BGal}.
\end{recollection}

\begin{theorem}\label{thm:monodromy_for_locally_weakly_constant_sheaves}
	Let $ X $ be a qcqs scheme. 
    The composite fully faithful functor
    \begin{equation}\label{eq:functorlocconst}
        \begin{tikzcd}
            \Functs\big(\CondShape{X}, \Aniult \big) \arrow[r, hooked] & \Functs\big(\CondShape{X}, \ICond(\Ani) \big) \arrow[r, "b^*", hooked] & \Xproethyp
        \end{tikzcd}
    \end{equation} 
    has image the full subcategory $\wLoc(X)$ of locally weakly constant sheaves.
\end{theorem}

The idea of the proof is to show it first in the case of \wcontractible schemes, then conclude by proétale hyperdescent.

\begin{lemma}\label{lem:thm_holds_on_wcontr}
	Let $W$ be a \wcontractible scheme. 
    Then the fully faithful functor 
    \begin{equation*}
        \Functs(\uppi_0(W), \Aniult)\to W^{\hyp}_{\proet}
    \end{equation*} 
    has image $\wLoc(W)$.
\end{lemma}

\begin{proof}
	Recall from \Cref{ex:Picond_on_w-contractibles} that since $ W $ is \wcontractible, $ \Picond(W) \equivalent \uppi_{0}(W) $.
    Moreover, since $ \uppi_0(W) $ is a profinite set, the Yoneda lemma implies that
	\begin{equation*}
		\Functs(\uppi_{0}(W),\Aniult) \simeq \Aniult(\uppi_{0}(W)) \simeq \Sh(\uppi_{0}(W)) 
	\end{equation*}
    and the given functor is identified with the functor
    \begin{equation*}
        \Sh(\uppi_0(W)) \hookrightarrow W^{\hyp}_{\proet}
    \end{equation*}
    given by pullback along $W \to \uppi_0(W)$.
    Therefore it lands in $\wLoc(W)$ by definition; it remains to show surjectivity.
    
    To show surjectivity, let $ F \in \wLoc(W) $.
    Then there is a proétale cover $p \colon U \to W $ such that $ \pupperstar(F) $ is in the image of $ \Sh(\uppi_{0}(U)) \to U\proethyp $.
	Since $ W $ is \wcontractible, we can pick a section $ s \colon W \to U $ of $ p $.
    Since the square 
	\begin{equation*}
		\begin{tikzcd}
			W \arrow[r, "\nu"] \arrow[d, "s"'] & \uppi_{0}(W) \arrow[d, "\uppi_{0}(s)"] \\ 
			U \arrow[r] & \uppi_{0}(U)
		\end{tikzcd}
	\end{equation*}
	commutes, we see that $ F=  \supperstar \pupperstar(F) $ is in the image of $ \nu^* $.
\end{proof}

\begin{proof}[Proof of \Cref{thm:monodromy_for_locally_weakly_constant_sheaves}]
	As we have a chain of fully faithful functors \eqref{eq:functorlocconst}, we regard
    \begin{equation*}
        \Functs\big(\Picond(X),\Aniult \big)
    \end{equation*}
    as a full subcategory of \smash{$ \Xproethyp $}.
	It remains to show that this full subcategory agrees with the full subcategory $ \wLoc(X) $.
	Since the assignment $ \goesto{Y}{\Picond(Y)} $ is a hypercomplete proétale cosheaf, the assignment
	\begin{equation*}
		Y \mapsto \Fun(\CondShape{Y},\Aniult)
	\end{equation*}
	is in a fact a subsheaf of the proétale hypersheaf $ Y \mapsto Y\proethyp $.
	Furthermore, by definition, the assignment
	\begin{equation*}
		Y \mapsto \wLoc(Y)
	\end{equation*}
	is subsheaf of the proétale hypersheaf \smash{$ Y \mapsto Y\proethyp $}. 
	Since \wcontractible schemes form a basis for the proétale topology, it suffices to see that they agree on \wcontractibles, which is the content of \Cref{lem:thm_holds_on_wcontr}.
\end{proof}

%-------------------------------------------------------------------%
%  Recovering the proétale fundamental group                        %
%-------------------------------------------------------------------%

\subsection{Recovering the proétale fundamental group}\label{subsec:recovering_the_proetale_fundamental_group}

The goal of this subsection is to show that the \textit{Noohi completion} of the condensed fundamental group  recovers the proétale fundamental group.
Since the proétale fundamental group is a topological group, we first need to explain some technical points about the relationship between topological groups and condensed groups.

\begin{recollection}
	The canonical functor $\TopGrp \to \Cond(\Grp) $ from topological groups to condensed groups admits a left adjoint
	\begin{equation*}
	   (-)^{\top} \colon \Cond(\Grp) \to \TopGrp \period
	\end{equation*}
    Note, however, that in general it is not the restriction of the left adjoint ``underlying topological space'' functor 
    \begin{equation*}
        (-)(\ast)_{\top} \colon \Cond(\Set) \to \Top
    \end{equation*}
    to condensed groups, as the latter functor does not preserve products.
\end{recollection}

It turns out that $ (-)^{\top} $ can be described as the composite of $(-)(\ast)_{\top}$ with the left adjoint of the inclusion of topological groups into \emph{quasitopological groups}.

\begin{recollection}\label{rmk:condensed_group_adjoint_and_quasitopological_groups}
    A \emph{quasitopological group} is a topological space $ G $ with an abstract group structure such that:
    \begin{enumerate}
        \item The inversion operation $G\to G $ given by $ g\mapsto g^{-1}$ is continuous.

        \item For each $h \in G$, the translation maps $ G \to G$ given by $g \mapsto gh$ and $g \mapsto hg$  are continuous.
    \end{enumerate}
    The embedding $\TopGrp \subset \qTopGrp$ of topological groups into quasitopological groups admits a left adjoint
    \begin{equation*}
        \tau \colon \qTopGrp \to \TopGrp
    \end{equation*}
    that moreover preserves the underlying abstract group and only affects the topology \cite[Lemma 3.2 \& Theorem 3.8]{Brazas:Fundamental_group_as_topological2013}.
\end{recollection}

While the functor $(-)(\ast)_\top$ does not provide an adjoint between $\CondGrp$ and $\TopGrp$, its image still lands in $\qTopGrp$. 
This is essentially because the condition of continuity of the inversion and translation maps does not involve forming a product. 
That is, we have a functor
\begin{equation*}
    (-)(\ast)_\top \colon \CondGrp \to \qTopGrp \period
\end{equation*}
Postcomposing with $\tau$, we get a functor
\begin{equation*}
    \tau \circ (-)(\ast)_\top \colon \CondGrp \to \TopGrp \period
\end{equation*}
One can then quite directly verify the following: 

\begin{lemma}[{(see \cite[Proposition 1.3.16]{CatrinsThesis} for details)}]\label{lem:underlying_topological_group_of_a_condensed_group}
    The composite $\tau \circ (-)(\ast)_\top$ is left adjoint to the ``associated condensed group'' functor. 
    Visually,
    \begin{equation*}
        \tau \circ (-)(\ast)_{\top} \colon \CondGrp \rightleftarrows \TopGrp \colon \underline{(-)} \period
    \end{equation*}
    Said differently, $ (-)^{\top} \equivalent \tau \circ (-)(\ast)_{\top} $.
\end{lemma}

\begin{nul}\label{nul:Gtop_doesn't_change_the_underlying_group}
    It follows from this discussion that for $G \in \CondGrp$, the abstract group $G(\ast)$ and the underlying group of $G^{\top}$ coincide.
\end{nul}

Before proceeding further, we provide a description of the category of $ G^{\top} $-sets purely in terms of condensed mathematics.

\begin{lemma}\label{lem:condensed_group_representations}
    Let $ G $ be a condensed group with condensed classifying anima $\Bup G$, i.e., the condensed groupoid that sends an extremally disconnected set $S$ to the one object groupoid with automorphisms $G(S)$.
    There is a natural equivalence of categories 
    \begin{equation*}
        \Functs(\Bup G,\Setult) \equivalence \Action{G^{\top}}
    \end{equation*}
    that is compatible with the forgetful functors to $ \Set $.
\end{lemma}

\begin{proof}
    We first prove the following: the category $\Functs(\Bup G,\Setult) $ is equivalent to the category of pairs $(M,\alpha)$ where $M\in \Set$ and $ \alpha \colon G\to \underline{\Aut(M)}$ is a map of condensed groups.
    Here, $\Aut(M) $ is the group of automorphisms of $ M $ equipped with the compact-open topology.
    A map $(M,\alpha) \to (N,\beta)$ is given by a map of sets $f \colon M \to N$ such that the square
    \begin{equation*}
        \begin{tikzcd}
        	G \arrow[r, "\alpha"] \arrow[d, "\beta"'] & \underline{\Aut(M)} \arrow[d, "\underline{\flowerstar}"] \\
        	\underline{\Aut(N)} \arrow[r, "\underline{\fupperstar}"'] & \underline{\Hom_{\Top}(M,N)}
        \end{tikzcd}
    \end{equation*}
    commutes (here $\Hom_{\Top}(M,N)$ is again given the compact-open topology).
    If this description holds, the claim follows:
    by the adjunction between condensed sets and topological spaces and \Cref{rmk:condensed_group_adjoint_and_quasitopological_groups}, the homomorphisms $ \alpha $ and $ \beta $ correspond to unique homomorphisms of quasitopological groups $ \alpha' \colon G(\pt)_{\top} \to \Aut(M) $ and $ \beta' \colon G(\pt)_{\top} \to \Aut(N) $ making the square
    \begin{equation*}
        \begin{tikzcd}
            G(\pt)_{\top} \arrow[r, "\alpha'"] \arrow[d, "\beta'"'] & \Aut(M) \arrow[d, "\flowerstar"] \\
            \Aut(N) \arrow[r, "\fupperstar"'] & \Hom_{\Top}(M,N)
        \end{tikzcd}
    \end{equation*}
    commute.
    Again, by adjunction, \Cref{lem:underlying_topological_group_of_a_condensed_group}, and \Cref{nul:Gtop_doesn't_change_the_underlying_group}, the homomorphisms $ \alpha' $ and $ \beta' $ correspond to unique homomorphisms of topological groups $ \alpha'' \colon G^{\top} \to \Aut(M) $ and $ \beta'' \colon G^{\top} \to \Aut(N) $ making the square 
    \begin{equation*}
        \begin{tikzcd}
            G^{\top} \arrow[r, "\alpha''"] \arrow[d, "\beta''"'] & \Aut(M) \arrow[d, "\flowerstar"] \\
            \Aut(N) \arrow[r, "\fupperstar"'] & \Hom_{\Top}(M,N)
        \end{tikzcd}
    \end{equation*}
    commute. 
    Thus the assignment
    \begin{align*}
        \paren{ M,\alpha \colon G \to \underline{\Aut(M)} } &\mapsto \paren{M,\alpha'' \colon G^{\top} \to \Aut(M)} \\
    \shortintertext{defines an equivalence of categories}
        \Functs(\Bup G,\Setult) &\equivalence \Action{G^\top}
    \end{align*}   
    as desired.
    
    Now we prove that $\Functs(\Bup G,\Setult) $ admits the above description. 
    The fully faithful functor $\Functs(\Bup G,\Setult) \hookrightarrow \Functs(\Bup G,\ICond(\Set) )$ fits into a cartesian square 
    \begin{equation*}
        \begin{tikzcd}
        	\Functs(\Bup G,\Setult) \arrow[r, "\ev_*"] \arrow[d, hooked] & \Set \arrow[d, hooked] \\
        	\Functs(\Bup G,\ICond(\Set)) \arrow[r, "\ev_*"'] & \CondSet \comma
        \end{tikzcd}
    \end{equation*}
    where the horizontal arrows are given by pullback along $\ast \to \Bup G$.
    Indeed, this follows from the fact that the functors
    \begin{equation*}
        \Functs(-, \Setult), \Functs(-, \ICond(\Set))\colon \CondAni^{\op} \to \Cat_1
    \end{equation*}
    are sheaves and $\ast \to \Bup G$ is a cover in $\CondAni$.
    Now recall that by \cite[Corollary~3.20]{MR4574234}, for a condensed set $A$, there is a natural equivalence of categories
    \begin{equation*}
        \Functs(A,\ICond(\Set))\equivalent \CondSet_{/A} \period
    \end{equation*} 
    Using this combined with \cite[Proposition A.1]{arXiv:2207.09256} and applying $\Functs(-, \ICond(\Set))$ to the Čech nerve of $\ast \to \Bup G$, we obtain an equivalence
    \begin{equation*}
        \Functs(\Bup G, \ICond(\Set)) \simeq \lim \left( \CondSet \rightrightarrows \CondSet_{/G} \triplerightarrow \CondSet_{/G \cross G} \right) \period
    \end{equation*}

    Explicitly unwinding the descent data, we see that $ \Functs(\Bup G, \ICond(\Set))$ is equivalent to the usual category of condensed sets with an action by the condensed group $ G $.
    In other words, its objects are condensed sets $ A $ together with a map $G \to \underline{\Aut}(A)$ of condensed groups and the maps are defined as above.
    Here $\underline{\Aut}(A)$ is the maximal condensed subgroup of the condensed monoid $\underline{\Hom}(A,A)$ given by the internal hom in $\CondSet$.
    Thus, the proof will be complete if for a set $M$, we can show that there is a canonical isomorphism
    \begin{equation*}
        \underline{\Aut}(M) \equivalence \underline{\Aut(M)} \period
    \end{equation*}
    For this, we observe that we have a canonical isomorphism
    \begin{equation*}
        \underline{\Hom}(M,M) \equivalent \underline{\Hom_{\Top}(M,M)} \comma
    \end{equation*}
    under which the corresponding condensed subgroups of automorphisms agree.
    This completes the proof.
\end{proof}

In order to prove the main result of this section, we recall a bit about Noohi groups.

\begin{recollection}[{\cite[\S7.1]{MR3379634}}]
    For a topological group $ G $, let $F_G \colon \Action{G} \to \Set$ denote the forgetful functor from the category of sets equipped with a continuous $ G $-action to the category of sets. 
    We say $ G $ is \emph{Noohi} if the canonical continuous map
    \begin{equation*}
        G \to \Aut(F_G)
    \end{equation*}
    is a homeomorphism of groups. 
    Here, $\Aut(F_G)$ is topologized using the compact-open topology on groups $\Aut(F_G(M))$ for $M \in \Action{G}$.
    We write $\NoohiGrp \subset \TopGrp$ for the full subcategory spanned by the Noohi groups.

    Noohi groups are useful when one wants to generalize Grothendieck's Galois theory to allow infinite fibers (cf.\ the ``infinite Galois theory'' of \cite[\S7.2]{MR3379634}).
    This formalism was used to define the proétale fundamental group of a scheme in \S7.4 of \emph{loc.\ cit.}. 
    For any scheme $ X $ with locally finitely many irreducible components (this assumption suffices by \cite[Remark 7.3.11]{MR3379634}) and geometric point $ \xbar \to X $, the group \smash{$\pioneproet(X,\xbar)$} is Noohi. 
    Similarly, the fundamental group of de Jong in rigid geometry \cite{deJong:Etale_fund_grps_analytic_1995} and its later generalizations \cites{ALY-variants}{ALY-arcs} are all Noohi.

    Noohi groups can also be characterized in purely topological terms as Hausdorff, Raĭkov complete groups such that open subgroups form a fundamental system of neighborhoods of $1$.

    The inclusion $\NoohiGrp \subset \TopGrp$ admits a left adjoint $(-)^\Noohi$, called \textit{Noohi completion}, given by
    \begin{equation*}
        G \mapsto \Aut(F_G) \period
    \end{equation*}
    See \cite[\S2]{LaraFESpaper} for this and some other properties of Noohi groups and Noohi completion.
\end{recollection}

We now extend Noohi completion to condensed groups.

\begin{definition}
    Let $G \in \CondGrp$. 
    The \defn{Noohi completion} of $ G $ is the Noohi group 
    \begin{equation*}
        G^\Noohi \colonequals \left(G^{\top}\right)^{\Noohi} \period
    \end{equation*}
\end{definition}

\begin{remark}
    For a condensed group $ G $, one can also define a version of Noohi completion directly as a condensed group without passing through $(-)^{\top}$.
    More precisely one can show that $ \underline{G^{\Noohi}} $ coincides with the condensed group defined by the assignment
    \begin{equation*}
        S \mapsto \Aut\bigg\lparen\!\!
        \begin{tikzcd}
            \Functs(\Bup G, \Setult) \arrow[r] & \Set \arrow[r, "\Gammaupperstar_S"] & \Sh(S)
        \end{tikzcd}
        \!\!\bigg\rparen \period
    \end{equation*}
    We do not need this observation in this article.
\end{remark}

We conclude by proving the main result of this section.

\begin{theorem}\label{thm:recovering_BS_fundamental_group}
    Let $ X $ be a qcqs scheme with finitely many irreducible components\footnote{This is equivalent to being qcqs and having \emph{locally} finitely many irreducible components.}  and $\xbar \to X$ a geometric point.
    Then there is a natural isomorphism
    \begin{equation*}
        \pionecond(X,\xbar)^{\Noohi} \equivalent \pioneproet(X,\xbar) \period
    \end{equation*}
\end{theorem}

\begin{proof}
    Since $ X $ has finitely many irreducible components, by \Cref{cor:pi0s_match_for_finitely_many_irr_comps} we may assume that $ X $, and therefore $ \Picond(X) $, is connected.
    It follows from \Cref{thm:monodromy_for_locally_weakly_constant_sheaves} that we have a chain of natural equivalences
    \begin{align*}
        \Functs(\Bup \pionecond(X, \xbar), \Setult) &\equivalent \Functs(\Pionecond(X),\Setult) \\
        & \equivalent \Functs(\Picond(X),\Setult) \\
        &\equivalent \wLoc(X)_{\leq 0} \\ & \equivalent \Action{\piproet_1(X,\xbar)}
    \end{align*}
    that are compatible with the forgetful functors to $\Set$.
    Here, the last equivalence follows from the definition of $\pioneproet(X,\xbar)$ in \cite[Definition 7.4.2]{BhattScholzeProetale} combined with Lemmas 7.3.9 and 7.4.1 in \emph{loc.\ cit.}.
    Thus \Cref{lem:condensed_group_representations} shows that there is a natural equivalence
    \begin{equation*}
        \Action{\pionecond(X,\xbar)^{\top}} \simeq \Action{\piproet_1(X,\xbar)} \period
    \end{equation*}
    In particular, both groups have the same Noohi completion.
    Since $\piproet_1(X,\xbar)$ is Noohi complete \cite[Theorem~7.2.5]{BhattScholzeProetale}, the claim follows.
\end{proof}

%-------------------------------------------------------------------%
%-------------------------------------------------------------------%
%-------------------------------------------------------------------%
%  Appendices                                                       %
%-------------------------------------------------------------------%
%-------------------------------------------------------------------%
%-------------------------------------------------------------------%

\newpage
\appendix

\part*{Appendices}\addcontentsline{toc}{part}{Appendices}

%-------------------------------------------------------------------%
%-------------------------------------------------------------------%
%  Rings of continuous functions & Čech–Stone compactification      %
%-------------------------------------------------------------------%
%-------------------------------------------------------------------%

\section{Rings of continuous functions \& Čech--Stone compactification}\label{appendix:rings_of_continuous_functions_and_Cech-Stone_compactification}
\textit{by Bogdan Zavyalov}

The main goal of this section is to provide the crucial input for the computation of the condensed shape of rings of continuous functions in \cref{subsec:rings_of_continuous_functions}. 
 Namely, we give a self-contained account for the identification (see \cref{thm:compact=max-spec} below) of the Čech--Stone compactification of a topological space $ X $ with the maximal spectrum of the ring of continuous functions on $ X $.

This identification has already been established in \cite{MR282962} using the notion of \pmring. In this appendix, we follow the ideas already present in \cite{MR282962}. We do not claim originality of any results in this appendix. Instead, we hope that this appendix gives a self-contained and reader-friendly exposition of some ideas from \cite{MR282962} and \cite{MR407579}. See also \cites{MR1211828}{MR1296009}{MR1384344}.

Throughout this appendix, we denote by $\RR$ (resp., $\CC$) the topological ring of real numbers (resp. complex numbers) with the Euclidean topology. 
For a topological space $ X $, we denote by $\upC(X, \RR)$ (resp., $\upC(X, \CC)$) the ring of real-valued (resp., complex-valued) continuous functions on $ X $. 

\subsection{Main constructions}

The main goal of this subsection is to introduce some constructions that will be used in the rest of this appendix. We also study their basic properties.

\begin{construction}\label{construction:evaluation}
    Let $ X $ be a topological space. 
    \begin{enumerate} 
        \item For each point $x\in X$, we define the \textit{evaluation functional} $\ev_x\colon \upC(X, \RR)\to \RR$ by the formula
        \begin{equation*}
            \ev_x(f) \colonequals f(x) \period
        \end{equation*}

        \item We define the map 
        \begin{equation*}
            \iota_X\colon X\to \Spec\big(\upC(X,\RR)\big)
        \end{equation*}
        to be the unique map that sends each point $x\in X$ to $\ker(\ev_x)$. 
    \end{enumerate}
\end{construction}

\begin{remark}
    The map $\iota_X$ is clearly natural in $ X $.
\end{remark}

For our later convenience, we record some basic properties of $\iota_X$. 

\begin{lemma}\label{lem:dense_image_etc} Let $ X $ be a topological space.
\begin{enumerate}
    \item\label{lem:dense_image_etc-1} The natural map $\iota_X \colon X \to \Spec\big(\upC(X,\RR)\big)$ is continuous;
    \item\label{lem:dense_image_etc-2} the image of $\iota_X(X)  \subset \Spec\big(\upC(X,\RR)\big)$ is a dense subset;
    \item\label{lem:dense_image_etc-3} the map $\iota_X$ factors through $\MSpec\big(\upC(X,\RR)\big)$. 
\end{enumerate}
\end{lemma}
\begin{proof}
    In order to see the first claim, it suffices to show that $\iota_X^{-1}\big(\Dup(f)\big)$ is an open subset of $ X $ for every $f\in \upC(X, \RR)$. This follows immediately from the formula $\iota_X^{-1}\big(\Dup(f)\big) = \{x\in X \ | \ f(x)\neq 0\}$ and the assumption that $ f $ is continuous.

    Now we prove the second claim. Let $ Z \colonequals  \Vup(I) \subset \Spec\big(\upC(X, \RR)\big)$ be a closed subset containing $\iota_X(X)$. 
    Then the construction of $\iota_X$ implies that, for every $f\in I$, we have $0=\ev_x(f)=f(x)$ for all $x\in X$. Thus $I=0$, and so we conclude that $Z = \Vup(0) = \Spec\big(\upC(X, \RR) \big) $. 

    To justify the last claim, it is enough to prove that $\ker(\ev_x)$ is a maximal ideal for every $x\in X$. For this, it suffices to show that $\ev_x$ is surjective. Fix a constant $c\in \RR$ and denote by $\underline{c}$ the corresponding constant function on $ X $. Then the surjectivity of $\ev_x$ follows immediately from the observation that $\ev_x(\underline{c})=c$. 
\end{proof}

\begin{remark} 
In what follows, we also denote by $\iota_X$ the restriction $\iota_X \colon X \to \MSpec\big(\upC(X,\RR)\big)$.
\end{remark}

Later in this appendix we show that if $ X $ is a compact Hausdorff space, then $\iota_X$ is a homeomorphism.
See \Cref{thm:compact=max-spec}.

\begin{warning}     
    The map $\iota_X$ is neither injective nor surjective for a general topological space $ X $.
\end{warning}

%-------------------------------------------------------------------%
%  pm-rings                                                         %
%-------------------------------------------------------------------%

\subsection{\pmrings}\label{appendix_subsec:pm-rings}

In this subsection, we introduce the notion of \emph{pm-rings} following \cite{MR282962}. Then we show that the natural inclusion $\MSpec(A) \hookrightarrow \Spec(A)$ admits a continuous retraction for a \pmring $ A $. As a consequence, we deduce that $\MSpec(A)$ is a compact Hausdorff space for any \pmring $ A $. 
We use the results of this subsection to relate the Čech--Stone compactification of an arbitrary topological space $ X $ to the maximal spectrum of the ring of continuous functions on $ X $. 

\begin{definition}[{\cite{MR282962}}]
    A ring $ A $ is a \textit{\pmring} if every prime ideal $\pfrak \subset A$ is contained in a unique maximal ideal $\pfrak \subset \mfrak_{\pfrak}\subset A$.
\end{definition}

\begin{definition}
    For a \pmring $ A $, we define the \emph{retract map} $r_A\colon \Spec(A) \to \MSpec(A)$ as the unique map that sends a point $x$ to its unique closed specialization (equivalently, it sends each prime ideal $\pfrak$ to the unique maximal ideal $\mfrak_{\pfrak}$ containing $\pfrak$). When there is no possibility of confusion, we will denote the map $r_A$ simply by $r$. 
\end{definition}

\begin{remark} 
    Below, we present a proof that $r_A$ is always continuous for a \pmring $ A $. This beautiful proof is due to de Marco and Orsatti. 
    However, we want to emphasize that, a priori, it is absolutely not clear whether the map $r_A$ has to be continuous or not. In fact, the author finds it quite surprising and is not aware of any one-line proof of this fact. 
\end{remark}

\begin{theorem}[{\cite[Theorem 1.2]{MR282962}}]\label{thm:continuous-retract} 
    Let $ A $ be a \pmring. 
    Then $r\colon \Spec(A) \to \MSpec(A)$ is a continuous retraction of the natural embedding $\iota \colon \MSpec(A)\hookrightarrow \Spec(A)$.
\end{theorem}
In fact, \cite[Theorem 1.2]{MR282962} shows that $ A $ is a \pmring if and only if $\iota$ admits a continuous retract (and $r$ is the unique continuous retract in this case). However, since we never need the other direction and it is significantly easier, we decided not to include it in this exposition. 
\begin{proof}
    Throughout this proof, we denote by $ \Vup_{\Spec}(I) \subset \Spec(A) $  the vanishing locus of an ideal $ I $ inside $\Spec(A)$, and by $\Vup_{\Max}(I) \colonequals \Vup_{\Spec}(I) \intersect \MSpec(A)$ the vanishing locus of $I$ inside $\MSpec(A)$.

    By construction, we know that $r\circ \iota = \id{}$. 
    So the only thing we really need to show is that the map $r$ is continuous. 
    We fix a closed subset $Z\subset \MSpec(A)$ and define
    \begin{equation*}
        I\colonequals \Intersection_{\mfrak\in Z} \mfrak \andeq J\colonequals \Intersection_{\substack{\pfrak \in r^{-1}(Z)}} \pfrak \period
    \end{equation*}

    For the purpose of proving continuity of $r$, it is enough to show that $r^{-1}(Z)=\Vup_{\Spec}(J)$. Clearly, $r^{-1}(Z) \subset \Vup_{\Spec}(J)$. Therefore, after unravelling all the definitions, we see that it suffices to show that, for any prime ideal $\pfrak \subset A$ such that $J\subset \pfrak $, we have $r(\pfrak)\in Z$. 
    
    \textit{Step 1: We show $Z=\Vup_{\Max}(I)$.} 
    Since $Z$ is closed, we know that $Z=\Vup_{\Max}(K)$ for some ideal $K\subset A$. 
    By construction, for any $\mfrak \in Z$, we have $K\subset \mfrak $.
    In particular, $K\subset I = \Intersection_{\mfrak \in Z}\mfrak $. 
    Thus, $\Vup_{\Max}(I)\subset \Vup_{\Max}(K) = Z$. On the other hand, the definition of $I$ implies that $Z\subset \Vup_{\Max}(I)$. Therefore, we conclude that 
    \begin{equation*}
    \Vup_{\Max}(I) \subset \Vup_{\Max}(K) = Z \subset \Vup_{\Max}(I) \period
    \end{equation*}
    This implies that $\Vup_{\Max}(I) = Z$. 
    
    Now we set $M \colonequals \Union_{\mfrak \in Z} \mfrak $. We note that $ 1 \notin M $, so $ M \neq A $. We warn the reader that the set $ M $ is not generally an ideal in $ A $.
    
    \textit{Step 2: Let $\pfrak \subset M$ be a prime ideal in $ A $. Then $r(\pfrak)\in Z$.} 
    Since $\pfrak \subset M$ and $I= \Intersection_{\mfrak\in Z} \mfrak$, we conclude that $\pfrak+I \subset M \neq A$. 
    Thus, we can find a maximal ideal $\nfrak \subset A$ such that
    \begin{equation*}
        \pfrak\subset \pfrak+I \subset \nfrak \period
    \end{equation*}
    Therefore, $r(\pfrak) = \nfrak $.
    Since $ I \subset \nfrak $, Step 1 ensures that $ \nfrak \in Z $. 
    This shows that $ r(\pfrak)\in Z$.
    
    \textit{Step 3: Let $J \subset \pfrak$ be a prime ideal in $ A $. Then $r(\pfrak)\in Z$.} Since each prime ideal is contained in a unique maximal ideal, it suffices to find a prime ideal $\qfrak \subset \pfrak$ such that $\qfrak \subset M$; then Step 2 implies that $ r(\pfrak) = r(\qfrak) \in Z $.
    
    Now we choose any $t\in A\sminus \pfrak$ and $s\in A\sminus M$. 
    Then $ts\neq 0$ since otherwise it would imply that
    \begin{equation*}
        t \in \Intersection_{\mfrak \in Z} \mfrak = J \subset \pfrak \period
    \end{equation*} 
    Hence, the multiplicative system
    \begin{equation*}
        S = \setbar{ts}{ t\in A\sminus \pfrak \text{ and } s\in A\sminus M}
    \end{equation*}
    does not contain $ 0 $. 
    Therefore, the localization $A[S^{-1}]$ is nonzero. 
    Thus, any maximal ideal in $A[S^{-1}]$ defines a prime ideal $\qfrak \subset A$ disjoint from $ S $. 
    Since $1\in A\sminus \pfrak$ and $1\in A\sminus M$, we conclude that $\qfrak \subset \pfrak \cap M$, finishing the proof. 
\end{proof}

\begin{corollary}\label{cor:max-spec-pm-ring-compact} 
    Let $ A $ be a \pmring. 
    Then $\MSpec(A)$ is a compact Hausdorff space.
\end{corollary}

\begin{proof}
    \Cref{thm:continuous-retract} implies that $r\colon \Spec(A) \to \MSpec(A)$ is a continuous surjection. 
    Since $ \Spec(A) $ is quasicompact and images of quasicompact spaces are quasicompact, $\MSpec(A)$ is seen to be quasicompact.
    
    Now we show that $ \MSpec(A) $ is Hausdorff. First, \stacks{0904} implies that it suffices to show that, for any two closed points $x,y\in \Spec(A)$, there does not exist a point $z\in \Spec(A)$ which specializes to both $x$ and $y$. This follows immediately from the fact that every point of $\Spec(A)$ specializes to a unique closed point. 
\end{proof}

\begin{definition}\label{defn:max-spec-functorial} 
    Let $f\colon A \to B$ be a homomorphism between \pmrings. We define the \textit{induced map of maximal spectra} $\MSpec(f) \colon \MSpec(B) \to \MSpec(A)$ as the composition
    \begin{equation*}
        \begin{tikzcd}[sep=3.5em]
            \MSpec(B) \arrow[r, "\iota_B"] & \Spec(B) \arrow[r, "\Spec(f)"] & \Spec(A) \arrow[r, "r_A"] & \MSpec(A) \period
        \end{tikzcd}
    \end{equation*}
\end{definition}

\begin{warning} 
    In general, for a ring homomorphism $A \to B$, the induced map of spectra $\Spec(f) \colon \Spec(B) \to \Spec(A)$ does not send $\MSpec(B)$ to $\MSpec(A)$. 
    This does not even hold for a general homomorphism of \pmrings. 
    Indeed, consider a rank $2$ valuation ring $ V $ with fraction field $K$ and a rank-$1$ localization $\Ocal$. Then the map $\Spec(\Ocal) \to \Spec(V)$, induced by the inclusion $V \subset \Ocal$, sends the closed point of $\Spec(\Ocal) $ to a non-closed point of $\Spec(V)$. 
\end{warning}

%-------------------------------------------------------------------%
%  Rings of continuous functions                                    %
%-------------------------------------------------------------------%

\subsection{Rings of continuous functions}\label{appendix_subsec:rings_of_continuous_functions}

The main goal of this section is to show that the rings of continuous functions $\upC(X, \RR)$ and $\upC(X, \CC)$ are \pmrings for any topological space $ X $. This will be the crucial ingredient in showing that the Čech--Stone compactification $\upbeta(X) $ is homeomorphic to $\MSpec\big(\upC(X, \RR)\big)$. 

We do not claim originality of any results of this subsection. In fact, our presentation that $\upC(X, \RR)$ is a \pmring follows \cite[Theorem 2.11]{MR407579} quite closely. The case of $\upC(X, \CC)$ seems to be missing in \cite{MR407579}. 

Throughout the section, we fix a topological space $ X $. 

\begin{definition} 
    Let $f\in \upC(X, \RR)$ be a continuous function.
    Its \emph{vanishing locus} is the set
    \begin{equation*}
        \Vup_X(f) \colonequals \setbar{x \in X}{f(x) = 0 } \period
    \end{equation*}
\end{definition}

\begin{definition}
    For a subset $S\subset \upC(X, \RR)$, the \emph{collection of its zero sets} is the subset 
    \begin{equation*}
        \Vup_X[S] \colonequals \{\Vup_X(f) \ | \ f\in S\} \subset \Sub(X)
    \end{equation*}
    of the set of all vanishing loci of elements in $ S $.\footnote{We denote by $\Sub(X)$ the set of all subsets of $ X $.}  
    For brevity, we put $\Vup_X[X]\colonequals \Vup_X[\upC(X, \RR)]$ for the set of all vanishing loci of continuous functions on $ X $. 
\end{definition}

\begin{lemma}[{\cite[Theorem 2.3]{MR407579}}]\label{lemma:properties-V} 
    Let $I\subset \upC(X, \RR)$ be an ideal and let $Z_1, Z_2 \in \Vup_X[I]$. Then 
    \begin{enumerate}
        \item $Z_1\cap Z_2 \in \Vup_X[I]$;

        \item if $Z\in \Vup_X[X]$ and $Z_1\subset Z$, then $Z\in \Vup_X[I]$.
    \end{enumerate}
\end{lemma}

\begin{proof}
    Let $Z_1=\Vup_X(f_1)$, $Z_2 = \Vup_X(f_2)$, and $Z=\Vup_X(f)$ for $f_1, f_2\in I$ and $f\in \upC(X, \RR)$. 
    For the first claim, note that
    \begin{equation*}
        Z_1\cap Z_2 = \Vup_X(f_1) \cap \Vup_X(f_2) =\Vup_X(f_1^2 + f_2^2)\in \Vup_X[I]. 
    \end{equation*}
    The second claim follows immediately from the observation that 
    \begin{equation*}
        Z = Z_1\cup Z = \Vup_X(f_1) \cup \Vup_X(f) = \Vup_X(f_1f)\in \Vup_X[I].\qedhere
    \end{equation*}
\end{proof}

\begin{definition} 
    An ideal $I\subset \upC(X, \RR)$ is a \textit{$zs$-ideal} if $\Vup_X(f)\in \Vup_X[I]$ implies $f\in I$.
\end{definition}

\begin{remark} 
    Often, $zs$-ideals are called $z$-ideals.
\end{remark}

\begin{theorem}[{\cite[Theorem 2.5]{MR407579}}]\label{theorem:maximal-z} 
    Let $\mfrak \subset \upC(X, \RR)$ be a maximal ideal. 
    Then $\mfrak $ is a $zs$-ideal.
\end{theorem}

\begin{proof}
    We denote by $I_\mfrak \subset \upC(X, \RR)$ the subset of continuous functions whose vanishing locus is equal to a vanishing locus of a function in $\mfrak$, i.e.,  
    \begin{equation}\label{equation:defn-I-m}
        I_\mfrak \colonequals \{f\in \upC(X, \RR) \ | \ \Vup_X(f) \in \Vup_X[\mfrak]\} \period
    \end{equation}
    
    Now \Cref{lemma:properties-V} implies that $I_\mfrak$ is an ideal. We pick continuous functions $f, g\in I_\mfrak$ and $h\in \upC(X, \RR)$ and wish to show that $f+g\in I_\mfrak$ and $fh\in I_\mfrak$.
    The former claim follows from the observation $\Vup_X(f+g) \supset \Vup_X(f)\cap \Vup_X(g)$ and \Cref{lemma:properties-V}, while the latter claim follows from the observation $\Vup_X(fh) \supset \Vup_X(f)$ and \Cref{lemma:properties-V}. 

    Now \cref{equation:defn-I-m} implies that, for the purpose of showing that $\mfrak$ is a $zs$-ideal, it suffices to show that $\mfrak=I_\mfrak$. Clearly, we have $\mfrak \subset I_\mfrak$. Therefore, the fact that $\mfrak$ is a maximal ideal implies that, in order to show that $\mfrak=I_\mfrak$, it suffices to show that $1\notin I_\mfrak$. This is equivalent to showing that $\varnothing \notin \Vup_X[\mfrak]$. For this note that any $f\in \mfrak$ is not invertible, therefore $\varnothing \neq \Vup_X(f)$. This finishes the proof. 
\end{proof}

\begin{lemma}\label{lemma:properties-z-ideals} Let $I, J\subset \upC(X, \RR)$ be two $zs$-ideals. Then $I$ is a radical ideal and $I\cap J$ is a $zs$-ideal.
\end{lemma}
\begin{proof}
    We start with the first claim. Suppose $f\in \mathrm{rad}(I)$, so $f^n\in I$ for some $ n $. Then we note that $\Vup_X(f)=\Vup_X(f^n)$. So the definition of a $zs$-ideal implies that $f\in I$. In other words, $I$ is radical. 

    Now we deal with the second claim. We first claim that $\Vup_X[I\cap J]=\Vup_X[I]\cap \Vup_X[J]$. We always have an inclusion $\Vup_X[I\cap J] \subset \Vup_X[I]\cap \Vup_X[J]$, so it suffices to show that $\Vup_X[I]\cap \Vup_X[J] \subset \Vup_X[I\cap J]$. Pick $Z\in \Vup_X[I]\cap \Vup_X[J]$. By definition, this means that there are elements $f\in I$ and $g\in J$ such that $Z=\Vup_X(f) = \Vup_X(g)$. Since $J$ is a $zs$-ideal, it implies that $f\in J$. Therefore, $f\in I\cap J$ and, hence, $Z\in \Vup_X[I\cap J]$.
    
    Now let $f\in \upC(X, \RR)$ be a continuous function such that $\Vup_X(f)\in \Vup_X[I\cap J] = \Vup_X[I]\cap \Vup_X[J]$. Then we use the fact that both $I$ and $J$ are $zs$-ideals to conclude that $f\in I\cap J$, i.e., $I\cap J$ is a $zs$-ideal. 
\end{proof}

\begin{remark} 
    \Cref{lemma:properties-z-ideals} implies that the ideal $(\id{\RR}) \in \upC(\RR, \RR)$ is \emph{not} a $zs$-ideal.
\end{remark}

\begin{lemma}[{\cite[Theorem 2.9]{MR407579}}]\label{lemma:prime-z-ideals} 
    Let $I\subset \upC(X, \RR)$ be a $zs$-ideal. 
    Then the following are equivalent:
    \begin{enumerate}
        \item\label{lemma:prime-z-ideals.1} The ideal $I$ is prime.

        \item\label{lemma:prime-z-ideals.2} The ideal $I$ contains a prime ideal.

        \item\label{lemma:prime-z-ideals.3} For any $f, g\in \upC(X, \RR)$ such that $fg=0$, we have $f\in I$ or $g\in I$.

        \item\label{lemma:prime-z-ideals.4} For every $f\in \upC(X, \RR)$, there is a subset $Z\subset X$ such that $Z\in \Vup_X[I]$ and $ f|_Z $ does not change its sign. 
    \end{enumerate}
\end{lemma}

\begin{proof}
    The implications \eqref{lemma:prime-z-ideals.1} $\Rightarrow$ \eqref{lemma:prime-z-ideals.2} and \eqref{lemma:prime-z-ideals.2} $\Rightarrow$ \eqref{lemma:prime-z-ideals.3} are trivial. 
    
    Now we show \eqref{lemma:prime-z-ideals.3} $\Rightarrow$ \eqref{lemma:prime-z-ideals.4}. We start by considering the continuous functions $f^+\colonequals \max(f, 0)$ and $f^-\colonequals \min(f, 0)$. 
    Then clearly we have
    \begin{equation*}
        f^+\cdot f^- = 0 \comma
    \end{equation*}
    so we have $f^+\in I$ or $f^-\in I$. 
    Suppose $f^+\in I$ (the other case is similar), then we can choose 
    \begin{equation*}
        Z \colonequals \{x\in X \ | \ f(x)\leq 0\} = \Vup_X(f^+)\in \Vup_X[I] \period
    \end{equation*}
    
    Now we show \eqref{lemma:prime-z-ideals.4} $\Rightarrow$ \eqref{lemma:prime-z-ideals.1}. We pick two continuous functions $f,g\in \upC(X, \RR)$ such that $fg\in I$ and wish to show that $f\in I$ or $g\in I$. For this, we consider the continuous function $h=|f|-|g|$. 
    Our assumption implies that there is a zero set $Z\in \Vup_X[I]$ such that $h|_{Z}$ is, say, nonnegative (the other case is similar). Note that if $f(x)=0$ and $x\in Z$, then $h(x) = -|g(x)| \geq 0$. Hence, $h(x)=g(x)=0$ for such $x\in X$. So we conclude that $Z\cap \Vup_X(fg) = Z\cap (\Vup_X(f)\cup \Vup_X(g)) = Z\cap \Vup_X(g)$. Therefore, we see that $\Vup_X(g)\in \Vup_X[I]$ by virtue of \Cref{lemma:properties-V} and the following sequence of inclusions:
    \begin{equation*}
        \Vup_X(g) \supset Z\cap \Vup_X(g) =Z\cap \Vup_X(fg)
    \end{equation*}
    Therefore, we conclude that $g\in I$ since $I$ is a $zs$-ideal. 
\end{proof}

We are almost ready to show that $\upC(X, \RR)$ is a \pmring. 
For the proof, we need the following result from commutative algebra.

\begin{lemma}\label{lemma:intersection-not-prime} 
    Let $ R $ be a ring and let $ \pfrak_1, \pfrak_2\subset R $ be prime ideals such that neither of them is contained in the other. 
    Then $\pfrak_1\cap \pfrak_2$ is not a prime ideal.
\end{lemma}

\begin{proof}
    Choose $t\in \pfrak_1\sminus \pfrak_2$ and $s\in \pfrak_2\sminus \pfrak_1$. 
    Then $st\in \pfrak_1\cap \pfrak_2$ but $s\notin \pfrak_1\cap \pfrak_2$ and $t\notin \pfrak_1\cap \pfrak_2$.
\end{proof}

\begin{theorem}[{\cite[Theorem 2.11]{MR407579}}]\label{thm:cont-functions-pm} 
    For any topological space $ X $, the ring $\upC(X,\RR) $ is a \pmring.
    Hence $\MSpec\big(\upC(X,\RR)\big)$ is a compact Hausdorff topological space.
\end{theorem}

\begin{proof}
    Note that the second claim follows from the first and \Cref{cor:max-spec-pm-ring-compact}.
    For the first, since every prime ideal $\pfrak \subset \upC(X,\RR) $ is contained in some maximal ideal, so it suffices to show that $ \pfrak $ cannot be contained in two different maximal ideals $\mfrak_1 $ and $\mfrak_2$. 
    We set $ I\colonequals \mfrak_1\cap \mfrak_2 $. Then \Cref{theorem:maximal-z} and \Cref{lemma:properties-z-ideals} imply that $I$ is a $zs$-ideal. By construction, we have an inclusion $\pfrak \subset I$. Therefore, \Cref{lemma:prime-z-ideals} ensures that $I$ is a prime ideal. However, this contradicts \Cref{lemma:intersection-not-prime}. Hence, there is only one maximal ideal containing $\pfrak$. 
\end{proof}

We now prove that that $\upC(X, \CC)$ is a \pmring.
We need some preparatory lemmas. 

% \begin{lemma}\label{lemma:no-roots-of-negative-1} 
%     Let $\mfrak \subset \upC(X, \RR)$ be a maximal ideal and let $\upkappa(\mfrak) = \upC(X, \RR)/\mfrak$ be the corresponding residue field. Then $\upkappa(\mfrak) \otimes_{\RR} \CC$ is a field. 
% \end{lemma}

% \begin{proof}
%     It suffices to show that the equation $X^2 + 1=0$ has no solutions in $\upkappa(\mfrak)$. In other words, we need to show that there are no continuous functions $f\in \upC(X, \RR)$  and $g\in \mfrak$ such that $f^2 = -1 + g$. 
    
%     Suppose that such functions exist. Then we note that $g$ is not an invertible function since it lies in a maximal ideal. Therefore, there is a point $x\in X$ such that $g(x)=0$. Thus, we see that $f(x)^2 = -1 + g(x) = -1$. Contradiction, so no such functions exist. 
% \end{proof}

\begin{lemma}\label{lemma:real-vs-complex-functions} 
    The canonical map $\upC(X, \RR) \otimes_{\RR} \CC \to \upC(X, \CC)$ is an isomorphism.
\end{lemma}

\begin{proof}
    First, we note that the assertion is equivalent to showing that the canonical map $\upC(X, \RR)  \oplus i \cdot \upC(X, \RR)  \to \upC(X, \CC)$ is an isomorphism. In other words, we need to show that any continuous function $f\in \upC(X, \CC)$ can be uniquely written as $f = g + i\cdot h$ with $g, h\in  \upC(X, \RR)$. Uniqueness is clear. To see existence, we note that $f = \mathrm{Re}(f) + i\cdot \mathrm{Im}(f)$. 
\end{proof}

\begin{lemma}\label{lemma:MSpec_bijeciton}
    The canonical map $\Spec(\upC(X, \CC)) \to \Spec(\upC(X, \RR))$ restricts to a bijection
    \begin{equation*}
        c \from \MSpec(\upC(X, \CC)) \to \MSpec(\upC(X, \RR)) \period
    \end{equation*}
\end{lemma}

\begin{proof}
    By \cref{lemma:real-vs-complex-functions}, $\upC(X, \RR) \to \upC(X, \CC)$ is a finite ring extension and thus $\Spec(\upC(X, \CC)) \to \Spec(\upC(X, \RR))$ maps closed points to closed points.
    To show that it restricts to a bijection on closed points, it suffices to see that for every maximal ideal $ \mfrak \subset \upC(X, \RR)$ with residue field $\upkappa(\mfrak)$, the tensor product $\upkappa(\mfrak) \otimes_{\upC(X, \RR)} \upC(X, \CC)$ is a field.
    By \cref{lemma:real-vs-complex-functions}, this is equivalent to showing that $\upkappa(\mfrak) \otimes_{\RR} \CC$ is a field.
    For this it suffices to show that the equation $X^2 + 1=0$ has no solutions in $\upkappa(\mfrak)$. In other words, we need to show that there are no continuous functions $f\in \upC(X, \RR)$  and $g\in \mfrak$ such that $f^2 = -1 + g$. 
    Suppose that such functions exist. Then we note that $g$ is not an invertible function since it lies in a maximal ideal. Therefore, there is a point $x\in X$ such that $g(x)=0$. Thus, we see that $f(x)^2 = -1 + g(x) = -1$. Contradiction, so no such functions exist. 
\end{proof}

\begin{corollary}\label{ref:complex_cts_fcts_pm_ring}
    For any topological space $ X $, the ring $\upC(X,\CC) $ is a \pmring.
    Hence $\MSpec\big(\upC(X,\CC)\big)$ is a compact Hausdorff topological space.
\end{corollary}

\begin{proof}
    Note that the second claim follows form the first and \Cref{cor:max-spec-pm-ring-compact}.
    For the first, let $\mathfrak{P} \subset \upC(X, \CC)$ be a prime ideal and let $\mathfrak{M}\subset \upC(X, \CC)$ be a maximal ideal containing $\mathfrak{P}$. We put $\pfrak \colonequals \mathfrak{P} \cap \upC(X, \RR)$ and we set $\mfrak \subset \upC(X, \RR)$ to be the unique maximal ideal containing $\pfrak$. 
    Since $\Spec\big(\upC(X, \CC)\big) \to \Spec\big(\upC(X, \RR)\big)$ is a finite morphism (see \Cref{lemma:real-vs-complex-functions}), it sends closed points to closed points. So we conclude that $\mathfrak{M} \cap \upC(X, \RR) = \mfrak$.
    Thus the claim follows from \cref{lemma:MSpec_bijeciton}.
\end{proof}

\begin{corollary}\label{cor:Mspec(C)=Mspec(R)}
    The canonical map $c \from \MSpec(\upC(X, \CC)) \to \MSpec(\upC(X, \RR))$ is a homeomorphism.
\end{corollary}

\begin{proof}
    By \cref{thm:cont-functions-pm,ref:complex_cts_fcts_pm_ring} the source and target are both compact Hausdorff spaces, so the claim follows from \cref{lemma:MSpec_bijeciton}.
\end{proof}

%-------------------------------------------------------------------%
%  Čech–Stone compactification                                      %
%-------------------------------------------------------------------%

\subsection{Čech--Stone compactification via algebraic geometry}\label{appendix_subsec:Cech-Stone_compactification}

In this subsection, we show that, for any topological space $ X $, the compact Hausdorff space $\MSpec(\upC(X, \RR))$ satisfies the universal property of the Čech--Stone compactification of $ X $.

\begin{definition} 
    The \textit{Čech--Stone compactification} of a topological space $ X $ is a pair $(\upbeta(X), i_X)$ of a compact Hausdorff space $\upbeta(X)$ and a continuous map $i_X\colon X \to \upbeta(X)$ such that, for every other compact Hausdorff space $Y$ with a continuous map $f\colon X \to Y$, there is a unique continuous map $\upbeta(f)\colon \upbeta(X) \to Y$ satisfying $f = \upbeta(f) \circ i_X$. 
\end{definition}

We recall (see \cref{construction:evaluation}) that, for every topological space $ X $, we have the natural map $\iota_X \colon X \to \MSpec(\upC(X, \RR))$. 
We also write $\iota_{X \otimes_\RR \CC} \from X \to \MSpec(\upC(X, \CC))$ for the composition of $\iota_X$ followed by the inverse of $c \from \MSpec(\upC(X, \RR)) \to \MSpec(\upC(X, \CC))$.
Our goal is to show that both $\Big(\MSpec\big(\upC(X, \RR)\big), \iota_X\Big)$ and $\Big(\MSpec\big(\upC(X, \CC)\big), \iota_{X \otimes_\RR \CC}\Big)$ are Čech--Stone compactifications of $X$.  

\begin{theorem}\label{thm:compact=max-spec} 
    Let $ X $ be a compact Hausdorff space. 
    Then all maps in the commutative triangle
    \begin{equation*}
        \begin{tikzcd}
            X \arrow[r, "{\iota_X}"] \arrow[d, "{\iota_{X \otimes_\RR \CC}}"'] & \MSpec(\upC(X, \RR)) \\
            \MSpec(\upC(X, \CC)) \arrow[ru, "c"']
        \end{tikzcd}
    \end{equation*}
    are homeomorphisms.
\end{theorem}

\begin{proof}
    The diagram commutes by construction and $c$ is a homeomorphism by \Cref{cor:Mspec(C)=Mspec(R)}.
    It thus suffices to show that $\iota_X$ is a homeomorphism.
    
    \textit{Step 1: $\iota_X$ is injective.} 
    To show injectivity of $ \iota_X $, it suffices to show that any two different points $x,y \in X $ can be separated by a continuous function $f\colon X \to \RR$. 
    More precisely, we need to find a continuous function $f\colon X \to \RR$ such that $f(x)=0$ and $f(y)\neq 0$.
    Such a function exists by Urysohn's Lemma \cite[Theorem 33.1]{MR3728284}.

    \textit{Step 2: $\iota_X$ has dense image.} This follows directly from \cref{lem:dense_image_etc}. 
    
    \textit{Step 3: $\iota_X$ is a homeomorphism.} 
    Since $ X $ is quasi-compact, we conclude that its image $\iota_X(X)$ is also quasi-compact. 
    Since $\MSpec(\upC(X,\RR))$ is Hausdorff (see \Cref{thm:cont-functions-pm}), we conclude that $\iota_X(X)$ is closed. 
    Since $\iota_X(X)\subset \MSpec(\upC(X,\RR))$ is dense, we conclude that $\iota_X$ must be surjective. 
    Therefore, $\iota_X$ is a bijective continuous map between compact Hausdorff spaces (see \Cref{thm:cont-functions-pm}), so it is a homeomorphism in virtue of \stacks{08YE}. 
\end{proof}

\begin{lemma}\label{lemma:st-functorial} 
    Let $f\colon X \to Y$ be a continuous map of topological spaces. 
    Then there is a unique continuous map $\widetilde{f} \colon \MSpec\big(\upC(X, \RR)\big) \to \MSpec\big(\upC(Y, \RR)\big)$ that makes the square
    \begin{equation*}
        \begin{tikzcd}[sep=2.5em]
            X \arrow[r, "f"] \arrow[d, "\iota_X"'] & Y \arrow[d, "\iota_Y"] \\
            \MSpec\big(\upC(X, \RR)\big) \arrow[r, dotted, "\widetilde{f}"'] & \MSpec\big(\upC(Y, \RR)\big)
        \end{tikzcd}
    \end{equation*}
    commute.
\end{lemma}

\begin{proof}
    First, we note that $\iota_X(X) \subset \MSpec\big(\upC(X, \RR)\big)$ is dense by \Cref{lem:dense_image_etc}. 
    Therefore, $\widetilde{f}$ is unique if it exists. 
    For the existence, we denote by $\fupperstar\colon \upC(Y, \RR) \to \upC(X, \RR)$ the natural pullback homomorphism. 
    Then $\widetilde{f} = \MSpec(\fupperstar)$ does the job (see \Cref{defn:max-spec-functorial,thm:cont-functions-pm}). 
\end{proof}

\begin{theorem}\label{thm:st-universal-property} 
    Let $ X $ be a topological space, $Y$ a compact Hausdorff space, and $f\colon X \to Y$ a continuous map. 
    Then there is a unique continuous map $\widetilde{f} \colon  \MSpec\big(\upC(X, \RR)\big) \to Y$ that makes the triangle
    \begin{equation*}
        \begin{tikzcd}[sep=2.5em]
            X \arrow[r, "f"] \arrow[d, "\iota_X"'] & Y  \\
             \MSpec\big(\upC(X, \RR)\big) \arrow[ur, dotted, "\widetilde{f}"'] &
        \end{tikzcd}
    \end{equation*}
    commute. 
\end{theorem}

\begin{proof}
    This follows immediately from \Cref{thm:compact=max-spec,lemma:st-functorial}. 
\end{proof}

\begin{corollary} 
   Let $ X $ be a topological space. 
   Then both
   \begin{equation*}
        \big(\MSpec(\upC(X, \RR)), \iota_X\big) \andeq \big(\MSpec(\upC(X, \CC)), \iota_{X \otimes_\RR \CC}\big)
   \end{equation*}
    are Čech--Stone compactifications of $ X $.
\end{corollary}

\begin{proof}
  Combine \Cref{cor:Mspec(C)=Mspec(R),thm:st-universal-property}. 
\end{proof}

%-------------------------------------------------------------------%
%-------------------------------------------------------------------%
%  A profinite analogue of Quillen's Theorem B                      %
%-------------------------------------------------------------------%
%-------------------------------------------------------------------%

\section{A profinite analogue of Quillen's Theorem B}\label{appendix:a_profinite_analogue_of_Quillens_theorem_B}

The goal of this appendix is to prove \Cref{thm:profinThmB}, an analogue of Quillen's Theorem B after completion at a set of primes. 
Most of the material here is a part of the sixth author's thesis \cite[\S7.3]{Sebastian_Wolf-thesis}.
Nevertheless, here the main result is formulated slightly more generally and the exposition was changed to make it more readable for those less familiar with the theory of internal higher categories developed by the fifth and sixth authors.

%-------------------------------------------------------------------%
%  Quillen's Theorem B                                              %
%-------------------------------------------------------------------%

\subsection{Quillen's Theorem B}\label{sec:classical_thm_B}

Given a functor of \categories $f \colon \Ccal \to \Dcal$, Quillen's Theorem B \cite[Theorem~B]{MR0338129} gives a way of calculating the homotopy fiber of the induced map of classifying anima $\Bup f \colon \Bup \Ccal \to \Bup \Dcal $.
We begin this appendix by giving a short and model-independent proof of Theorem B that is easier to generalize than Quillen's original argument.

\begin{theorem}[(Quillen's Theorem B)]\label{thm:Quillen_theorem_B}
	Let $f \colon \Ccal \to \Dcal$ be a functor of \categories such that for any $d \to d' \in \Dcal$ the induced map
	\begin{equation*}
	\Bup \Ccal_{/d} \to \Bup \Ccal_{/d'}
	\end{equation*}
	is an equivalence.
	Then for  any $d \in \Dcal$, the induced commutative square of anima
	\begin{equation*}
        \begin{tikzcd}
    		\Bup \Ccal_{/d} \arrow[r] \arrow[d] & \Bup \Ccal \arrow[d, "\Bup f"] \\
    		\ast \simeq \Bup \Dcal_{/d} \arrow[r] & \Bup \Dcal
    	\end{tikzcd}
    \end{equation*}
	is cartesian.
\end{theorem}

The proof rests on the following observation:

\begin{proposition}
	\label{prop:realization_of_kan_fib}
	Let $p \colon \Fcal \to \Dcal$ be a left fibration with corresponding straightened functor $\ptilde \colon \Dcal \to \Ani $.
    If for each map $s \colon d \to d'$ in $ \Dcal $, the induced map $\ptilde(s)$ is an equivalence, then for each $ d \in \Dcal$, the square
	\begin{equation*}
        \begin{tikzcd}
    		\Fcal_d \arrow[r] \arrow[d] & \Bup \Fcal \arrow[d, "\Bup p"] \\
    		\ast \arrow[r, "d"'] & \Bup \Dcal
    	\end{tikzcd}
    \end{equation*}
	is cartesian.
\end{proposition}

\begin{proof}
	By assumption, $\ptilde \colon \Dcal \to \Ani$ factors through the unit map $\Dcal \to \Bup \Dcal$.
	Pulling back the universal left fibration, we thus get a diagram
	\begin{equation*}
        \begin{tikzcd}
    		\Fcal_d \arrow[r] \arrow[d] \arrow[dr, phantom, "\lrcorner"{description, very near start}] & \Fcal \arrow[r] \arrow[d, "p"'] \arrow[dr, phantom, "\lrcorner"{description, very near start}] & \Fcal' \arrow[r] \arrow[d] \arrow[dr, phantom, "\lrcorner"{description, very near start}] & \Ani_{\ast/} \arrow[d] \\
    		\ast \arrow[r, "d"'] & \Dcal \arrow[r] \arrow[rr, bend right = 2em, "\ptilde"'] & \Bup \Dcal \arrow[r] & \Ani
    	\end{tikzcd}
    \end{equation*}
	in which all squares are cartesian.
    Note that since left fibrations are conservative and $ \Bup \Dcal $ is an anima, $\Fcal'$ is an anima.
	Since $\Bup \colon \Catinfty \to \Ani$ is locally cartesian (see \cref{nul:B_is_locally_cartesian}), by applying $ \Bup $ to the middle and left-hand squares, we get another diagram
    \begin{equation*}
        \begin{tikzcd}
            \Fcal_d \arrow[r] \arrow[d] \arrow[dr, phantom, "\lrcorner"{description, very near start}] & \Bup\Fcal \arrow[r, "\sim"{yshift=-0.25ex}] \arrow[d, "\Bup p"'] \arrow[dr, phantom, "\lrcorner"{description, very near start}] & \Fcal' \arrow[d] \\
            \ast \arrow[r, "d"'] & \Bup \Dcal \arrow[r, "\id{}"'] & \Bup \Dcal
        \end{tikzcd}
    \end{equation*}
	in which all squares are cartesian, completing the proof.
\end{proof}

\begin{remark}
	\label{rem:real_kan_fib}
	The assumptions of \Cref{prop:realization_of_kan_fib} are satisfied whenever the left fibration $p$ is additionally a right fibration, i.e., a Kan fibration.
\end{remark}

We now need to build the correct left fibration to which we can apply \Cref{prop:realization_of_kan_fib}.
For this we need the following.

\begin{notation}
    Let $ \Dcal $ be \acategory.
    We write $ \Cocart(\Dcal) \subset \Cat_{\infty,/\Dcal} $ for the subcategory with objects cocartesian fibrations $ p \colon \Fcal \to \Dcal $ and morphisms the cocartesian functors.
    We write 
    \begin{equation*}
        \LFib(\Dcal) \subset \Cocart(\Dcal)
    \end{equation*}
    for the full subcategory spanned by the left fibrations. 
    Note that $ \LFib(\Dcal) $ is also a full subcategory of $ \Cat_{\infty,/\Dcal} $.
\end{notation}

\begin{recollection}\label{rem:fiberwise_B}
	For \acategory $\Dcal, $ the inclusion $\Fun(\Dcal,\Ani) \hookrightarrow \Fun(\Dcal,\Catinfty)$ admits a left adjoint given by postcomposition with $\Bup \colon \Catinfty \to \Ani$.
	Under the straightening-unstraightening equivalence, this corresponds to a left adjoint of the inclusion
    \begin{equation*}
        \LFib(\Dcal) \hookrightarrow \Cocart(\Dcal) \period
    \end{equation*}
	Explicitly, this adjoint sends a cocartesian fibration $ p \colon \Pcal \to \Dcal$ to the unique left fibration $L(p) \colon \Fcal \to \Dcal$ that fits in a commutative triangle
	\begin{equation*}
        \begin{tikzcd}
    		\Pcal \arrow[rr, "\iota"] \arrow[dr, "p"'] & & \Fcal \arrow[dl, "L(p)"] \\
    		& \Dcal & \phantom{\Fcal} \comma
    	\end{tikzcd}
    \end{equation*}
	where the functor $\iota$ is initial.
	Indeed, such a factorization exists because left fibrations are the right class in the initial-left fibration factorization system, see, e.g., \cite[\S~4.1]{arXiv:2103.17141}.
	This also implies that for any left fibration $ q \colon \Gcal \to \Dcal$, there is a natural equivalence 
	\begin{equation*}
		\Map_{\Cocart(\Dcal)}(p,q) \simeq  \Map_{\Cat_{\infty,/\Dcal}}(p,q) \simeq  \Map_{\LFib(\Dcal)}(L(p),q) \period
	\end{equation*}
	Here, left-hand equivalence holds since for left fibrations every edge is cocartesian.
    The right-hand equivalence follows from the fact that the left fibrations are the right class of a  factorization system \HTT{Lemma}{5.2.8.19}.
\end{recollection}

In order to prove \Cref{thm:Quillen_theorem_B}, we fix some notation regarding oriented fiber products of \categories.

\begin{recollection}\label{rec:oriented_fiber_products}
    Let $ f \colon \Ccal \to \Dcal $ be a functor of \categories.
    We consider the oriented fiber product (also called comma \category) $\Ccal \orientedtimes_\Dcal \Dcal$ defined via the pullback
    \begin{equation*}
        \begin{tikzcd}
            \Ccal \orientedtimes_\Dcal \Dcal \arrow[r] \arrow[d] \arrow[dr, phantom, "\lrcorner"{description, very near start, xshift=-1.5ex}] & \Fun([1],\Dcal) \arrow[d, "{(\ev_0,\ev_1)}"] \\
            \Ccal \times \Dcal \arrow[r, "f \cross \id{\Dcal}"'] & \Dcal \times \Dcal
        \end{tikzcd}
    \end{equation*}
    in $ \Catinfty $.
    Note that by the universal property of the pullback, the functors $ (\id{\Ccal},f) \colon \Ccal \to \Ccal \cross \Dcal $ and
    \begin{equation*}
        \begin{tikzcd}
            \Ccal \arrow[r, "f"] & \Dcal \arrow[r, "\id{(-)}"] & \Fun([1],\Dcal)
        \end{tikzcd}
    \end{equation*}
    induce a functor $ j\colon \Ccal \to \Ccal \orientedtimes_\Dcal \Dcal $.
    By \HTT{Corollary}{2.4.7.12}, the projection $ \pr_2 \colon \Ccal \orientedtimes_\Dcal \Dcal \to \Dcal $ is a cocartesian fibration.
    The cocartesian fibration $ \pr_2 $ classifies the functor
    \begin{equation*}
        \Dcal \to \Catinfty \comma \quad d \mapsto \Ccal_{/d } \period
    \end{equation*}
    Furthermore, $ f $ factors as 
    \begin{equation*}
        \begin{tikzcd}
            \Ccal \arrow[r, "j"] & \Ccal \orientedtimes_\Dcal \Dcal \arrow[r, "\pr_2"] & \Dcal \comma
        \end{tikzcd}
    \end{equation*}
    and $j$ admits a right adjoint given by projecting to the first factor.
\end{recollection}

\begin{proof}[Proof of \Cref{thm:Quillen_theorem_B}]
	We apply the left adjoint $L$ of \cref{rem:fiberwise_B} to the cocartesian fibration $ \pr_2 \colon \Ccal \orientedtimes_\Dcal \Dcal \to \Dcal $.
	Our assumptions precisely say that the resulting left fibration $L(\pr_2) \colon \Fcal \to \Dcal$ satisfies the assumptions of \Cref{prop:realization_of_kan_fib}.
	Thus we get a commutative diagram
    \begin{equation*}
        \begin{tikzcd}
            \Bup \Ccal_{/d } \arrow[r] \arrow[d] & \Bup \Ccal \arrow[r, "\Bup j"] \arrow[d, "\Bup f"] & \Bup (\Ccal \orientedtimes_{\Dcal} \Dcal) \arrow[r, "\Bup \iota"] & \Bup \Fcal \arrow[d, "\Bup L (\pr_2)"] \\
            \ast \simeq \Bup \Dcal_{/d } \arrow[r, "d"'] & \Bup \Dcal \arrow[rr, equals] & & \Bup \Dcal \comma
        \end{tikzcd}
    \end{equation*}
	where the outer square is cartesian.
	Furthermore, since $ \Bup $ inverts adjoints and initial functors (see, e.g., \cite[Corollary 2.11(4) \& Remark 2.20]{MR4683160}), the right square is cartesian.
	Thus the left square is cartesian, as desired.
\end{proof}

%-------------------------------------------------------------------%
%  Profinite Theorem B                                              %
%-------------------------------------------------------------------%

\subsection{Profinite Theorem B}\label{sec:profin_thm_B}

The goal of this subsection is to prove a variant of Quillen's Theorem B for profinite categories following the general strategy of \cref{sec:classical_thm_B}.
The main ingredient of the proof of \Cref{thm:Quillen_theorem_B} was the straightening-unstraightening equivalence.
However profinite categories are not well-behaved enough to admit a full straightening-unstraightening equivalence.
The solution is to embed profinite categories into condensed categories, where we have a straightening-unstraightening equivalence thanks to \cite[Theorem~6.3.1]{arXiv:2209.05103}.
The precise theorem we aim to prove in this subsection is the following:

\begin{theorem}\label{thm:profinThmB}
	Let $ \Sigma $ be a nonempty set of prime numbers.\\
	Let $ f\colon \Ccal \to \Dcal $ be a map in $\ICat(\ProAnifin)$ such that for any map $ d \to d' $ in $ \Dcal $ the map of condensed anima
	\begin{equation*}
	   \Bcond (\Ccal_{/d}) \to  \Bcond (\Ccal_{/d'})
	\end{equation*}
	becomes an equivalence after \Sigmacompletion.
	Then, for all $ d \in \Dcal $, the induced map
    \begin{equation*}
        \Bcond(\Ccal_{/d}) \to \fib_d(\Bcond f)
    \end{equation*}
    becomes an equivalence after \Sigmacompletion.
\end{theorem}

As mentioned above, straightening-unstraightening plays a crucial role in our proof. 
Thus, we begin by defining cocartesian fibrations of condensed \categories.

\begin{definition}\label{def:cocartesian-fibrations-of-condensed-categories}
    Let $ \Ccal $ be a condensed \category.
    \begin{enumerate}
        \item A functor $p \colon \Pcal \to \Ccal$ of condensed \categories is a \defn{cocartesian fibration} if for each $S \in \ProFin$, the induced functor $p(S) \colon \Pcal(S) \to \Ccal(S)$ is a cocartesian fibration and, furthermore, for each map $\alpha \colon  T \to S$ in $ \ProFin$, the functor $\alpha^* \colon \Pcal(S) \to \Pcal(T)$ sends $p(S)$-cocartesian morphisms to $p(T)$-cocartesian morphisms.

        \item A cocartesian fibration $ p \colon \Pcal \to \Ccal $ is a \defn{left fibration} if for each $S \in \ProFin$, the induced functor $p(S) \colon \Pcal(S) \to \Ccal(S)$ is a left fibration.

        \item We write $\Cocart^{\cts}(\Ccal)$ for the subcategory of $ \CondCat_{/\Ccal}$ with objects the cocartesian fibrations and morphisms the functors $f \colon \Pcal \to \Qcal $ over $ \Ccal $ such that for every $S \in \ProFin$, the functor $f(S)$ preserves cocartesian morphisms.
        We write $ \LFib^{\cts}(\Ccal) \subset \Cocart^{\cts}(\Ccal) $ for the full subcategory spanned by the cocartesian fibrations.
    \end{enumerate}
\end{definition}

\begin{remark}
    Let us denote by $\Funcocart([1],\Catinfty)$ the subcategory of $\Fun([1],\Catinfty)$ with objects cocartesian fibrations and a morphism from $ p \colon \Pcal \to \Ccal $ to $ p' \colon \Pcal' \to \Ccal' $ is a square squares
    \begin{equation*}
        \begin{tikzcd}
            \Pcal \arrow[r, "f"] \arrow[d, "p"'] & \Pcal' \arrow[d, "p'"] \\
            \Ccal \arrow[r] & \Ccal'
        \end{tikzcd}
    \end{equation*}
    such that $ f $ sends $p$-cocartesian morphisms to $p'$-cocartesian morphisms.
    Then combining \cite[Theorem~4.5]{MR3690268} and \HA{Proposition}{7.3.2.6} shows that the inclusion
    \begin{equation*}
        \Funcocart([1],\Catinfty) \hookrightarrow \Fun([1],\Catinfty)
    \end{equation*}
    is a right adjoint.
    In particular, the inclusion preserves limits.

    Let $p \colon \Pcal \to \Ccal$ be a functor of condensed \categories.
    The closure of $ \Funcocart([1],\Catinfty) $ under limits in $ \Fun([1],\Catinfty) $ shows that if $p$ is a cocartesian fibration, then any map of condensed anima $ s \colon B \to A $, the functor $\supperstar$ in the square
    \begin{equation*}
        \begin{tikzcd}
        	\Functs(A,\Pcal) \arrow[r, "\supperstar"] \arrow[d, "\plowerstar"'] & \Functs(B,\Pcal) \arrow[d, "\plowerstar "] \\
        	\Functs(A,\Ccal) \arrow[r, "\supperstar"'] & \Functs(B,\Ccal)
        \end{tikzcd}
    \end{equation*}
    sends $p(A)$-cocartesian morphisms to $p(B)$-cocartesian morphisms.
    Thus, using \cite[Proposition~3.17]{arXiv:2209.05103}, it follows that our definition of cocartesian fibration agrees with the definition given in \cite{arXiv:2209.05103} in the case $\Bcal = \CondAni$.
\end{remark}

\begin{remark}
	By \cref{rem:right_fibrations_pointwise}, a functor of condensed \categories $ p \colon \Fcal \to \Ccal $ is a left fibration in the sense of \Cref{def:cocartesian-fibrations-of-condensed-categories} if and only if $ p^{\op} $ is a right fibration in the sense of \Cref{def:right_fibration_condensed}.
	Furthermore, if $ \Fcal \to \Ccal $ is a left fibration and $ \Pcal \to \Ccal $ is a cocartesian fibration, then every functor $ f \colon \Pcal \to \Fcal $ of condensed \categories over $ \Ccal $ is a map in $\Cocart^{\cts}(\Ccal)$.
\end{remark}

For the condensed version of straighetning-unstraightening, we need to consider the condensed \category of condensed \categories:

\begin{definition}
	We write $\ICond(\Catinfty)$ for the condensed \category given by the assignment
	\begin{equation*}
	    \ProFin^{\op} \ni S \mapsto \ICat(\CondAni_{/S}) \period
	\end{equation*}
\end{definition}

\begin{theorem}[({\cite[Theorem~6.3.1]{arXiv:2209.05103} and \cite[Theorem~4.5.1]{arXiv:2103.17141}})]\label{thm:internal_str}
	There is an natural equivalence of \categories
	\begin{equation*}
	   \Cocart^{\cts}(\Ccal) \simeq \Functs(\Ccal, \ICond(\Catinfty))
	\end{equation*}
	Moreover, this equivalence restricts to a natural equivalence
	\begin{equation*}
	   \LFib^{\cts}(\Ccal) \simeq \Functs(\Ccal,\ICond(\Ani)) \period
	\end{equation*}
\end{theorem}

We also have the following analogue of \cref{rem:fiberwise_B} for condensed \categories:

\begin{observation}\label{rem:condensned_fiberwise_B}
	Recall that the inclusion $\CondAni \inclusion \CondCat$ admits a left adjoint $\Bcond \colon \CondCat \to \CondAni$.
	It is easy to see that both of these functors are compatible with basechange and therefore lift to an adjunction of condensed \categories
	\begin{equation*}
		\adjto{\iota}{\ICond(\Ani)}{\ICond(\Catinfty)}{\Bcond} \comma
	\end{equation*}
	i.e., an adjunction in the $(\infty,2)$-category of condensed \categories.
    See also \cite[Definition~3.1.1 and Proposition~3.2.14]{MR4752519}.
	Thus the induced functor
    \begin{equation*}
        \Functs(\Ccal,\ICond(\Ani)) \to \Functs(\Ccal, \ICond(\Catinfty))
    \end{equation*}
    admits a left adjoint given by postcomposition with $ \Bcond $.\\
	Under the straightening-unstraightening equivalence of \Cref{thm:internal_str}, this corresponds to a left adjoint $ L $ of the inclusion
    \begin{equation*}
        \LFib^{\cts}(\Ccal) \hookrightarrow \Cocart^{\cts}(\Ccal) \period
    \end{equation*}
	
	Since left fibrations of condensed categories are the right class in the initial-left fibration factorization systems, as in \cref{rem:fiberwise_B}, it follows from \HTT{Lemma}{5.2.8.19} that the left adjoint is given by factoring $ \Pcal \to \Ccal$ into  an initial functor followed by a left fibration.
\end{observation}

To follow the strategy outlined in \cref{sec:classical_thm_B}, we need a version of \Cref{prop:realization_of_kan_fib}.
Now another complication enters. 
Unlike in \cref{sec:classical_thm_B}, the maps $\Bcond (\Ccal_{/d}) \to  \Bcond (\Ccal_{/d'})$ are not assumed to be equivalences on the nose, but only after \Sigmacompletion.
Thus, we also need an analogue of \Cref{prop:realization_of_kan_fib} that works up to completion.
We prove the following statement, which is a variant of \cite[Corollary~5.4]{zbMATH07226705}:

\begin{proposition}\label{prop:technical_moerdijk_nuiten_thing}
	Let $ \Xcal $ be \acategory with colimits and let $L \colon \CondAni \to \Xcal$ be a colimit-preserving functor.
	Let $ \Ccal$ be a condensed \category and $p \colon \Fcal \to \Ccal$ a left fibration of condensed \categories corresponding via \Cref{thm:internal_str} to a functor of condensed \categories $\ptilde \colon \Ccal \to \ICond(\Ani)$.
	Assume that for each profinite set $ S $, the functor
	\begin{equation*}
	    \begin{tikzcd}
            \Ccal(S) \arrow[r, "\ptilde(S)"] & \CondAni_{/S} \arrow[r] & \CondAni \arrow[r, "L"] & \Xcal
        \end{tikzcd}
	\end{equation*}
	sends all morphisms to equivalences.
	Then for every $d \colon S \to \Ccal$, the induced map
	\begin{equation*}
	   \ptilde(d) \colon S \times_{\Ccal} \Fcal \to S \times_{\Bcond \Ccal} \Bcond \Fcal
	\end{equation*}
	becomes an equivalence after applying $L$.
\end{proposition}

\begin{recollection}\label{nul:Kan_fibs}
    For the proof of the \Cref{prop:technical_moerdijk_nuiten_thing}, we recall that a functor of condensed \categories $f \from \Fcal \to \Ccal$ is a \defn{Kan fibration} if it is both a left and right fibration.
    Equivalently, $ f $ is Kan fibration if any of the following equivalent conditions is satisfied:
    \begin{enumerate}
    	\item For any $S \in \ProFin$, the functor $f(S)$ is a Kan fibration.

    	\item The functor $ f $ is right orthogonal to all maps of the form $ S \times \{ \varepsilon\} \to S \times [n]$, where $S \in \ProFin$, $n \in \NN$, and $\varepsilon \in \{ 0 ,n\}$.
    \end{enumerate}
    Indeed, this follows immediately from \cref{rem:right_fibrations_pointwise} and \cite[Lemma~4.1.2]{arXiv:2103.17141}.
\end{recollection}

\begin{proof}[Proof of \Cref{prop:technical_moerdijk_nuiten_thing}]
    We work in the \category 
    \begin{equation*}
        \CondAni_{\DDelta} \colonequals \Fun(\Deltaop,\CondAni)
    \end{equation*}
    of simplicial objects in $ \CondAni$.
    We factor $ S \to \Ccal $ as $S \xrightarrow{i} T \xrightarrow{f} \Ccal$
    where $ i $ is contained in the smallest saturated class in $ (\CondAni_{\DDelta})_{/ \Ccal} $ containing all maps of the form
    \begin{equation*}
        \begin{tikzcd}
            \{\varepsilon\} \times S \arrow[rr] \arrow[dr] & & {[n] \times S} \arrow[dl] \\
            & \Ccal &
        \end{tikzcd}
    \end{equation*}
    where $n \in \NN$, $ \varepsilon \in \{0,n\} $, and $ S \in \ProFin $, and $ f $ is right orthogonal to these maps.
    It follows from \cref{nul:Kan_fibs} that $ f $ is a Kan fibration.
    Since Kan fibrations are levelwise Kan fibrations, it follows from \cref{rem:real_kan_fib} that the natural map
    \begin{equation*}
        \Bcond(S \times_{\Ccal} \Fcal) \to S \times_{ \Bcond \Ccal}  \Bcond \Fcal
    \end{equation*}
    is an equivalence
    Thus it suffices to see that the induced map $ S \times_{\Ccal} \Fcal \to T \times_{\Ccal} \Fcal $ becomes an equivalence after applying $ L \circ \Bcond $.

    We note that, by the universality of colimits in $ \CondAni_{\DDelta} $, the class $ \Mcal $ of all maps $ s \colon A \to B $ in $ (\CondAni_{\DDelta})_{/ \Ccal} $, that have the property that
    \begin{equation*}
       \textstyle L  \colim_{\Deltaop} (A \times_{\Ccal} \Fcal) \to L \colim_{\Deltaop} (B \times_{\Ccal} \Fcal)
    \end{equation*} 
    is an equivalence is a saturated class in the sense of \cite[Definition~2.5.5]{arXiv:2103.17141}. 
    To see that $ i $ is contained in $ \Mcal $, it therefore suffices to check this for the maps $ \{\varepsilon\} \times S \to [n] \times S $, where $ S \in \ProFin $ and $ \varepsilon \in \{0,n\} $.
    Note that since the pulled back functor $ ([n] \times S) \times_{\Ccal} \Fcal \to [n] \times S$ is again a left fibration and the pullback of a final functor along a left fibration is final \cite[Proof of Proposition 4.4.7]{arXiv:2103.17141}, the induced funtor $ (\{n\} \times S) \times_{\Ccal} \Fcal  \to ([n] \times S) \times_{\Ccal} \Fcal $ is final.
    In particular,
    \begin{equation*}
        \Bcond((\{n\} \times S) \times_{\Ccal} \Fcal) \to \Bcond(([n] \times S) \times_{\Ccal} \Fcal)
    \end{equation*}
    is an equivalence, so $ \{n \} \times S \to [n] \times S $ is in $ \Mcal $.
    Furthermore, under this equivalence, the induced map
    \begin{equation*}
        (\{0\} \times S) \times_{\Ccal} \Fcal \to  \Bcond(([n]\times S) \times_{\Ccal} \Fcal)
    \end{equation*}
    is identified with the map $ (\{0\} \times S) \times_{\Ccal} \Fcal  \to (\{n\} \times S) \times_{\Ccal} \Fcal $ induced by $ 0 \to n $ in $[n]$ (see \Cref{lem:Identifying_straightened_functor} and \Cref{rem:small_fix} below).
    This map is an $ L $-equivalence by assumption. 
    Therefore, $ i $ is contained in $ \Mcal $, which completes the proof.
\end{proof}

\begin{lemma}\label{lem:Identifying_straightened_functor}
	Let $ p \colon \Fcal \to \Ccal$ be a left fibration of condensed \categories with straightened functor $ \ptilde \colon \Ccal \to \ICond(\Ani) $.
	Then for any morphism $ \alpha $ in $ \Ccal(S) $ for some $S \in \ProFin$, given by $ \alpha \colon [1] \times S \to \Ccal $,  the map $ \ptilde(\alpha) $ in $ \CondAni_{/S} $ is given by composing
	\begin{equation*}
	   (\{0\} \times S) \times_{\Ccal} \Fcal \to \Bcond(([1] \times S) \times_{\Ccal} \Fcal)
	\end{equation*}
	with the inverse of the equivalence $(\{1\} \times S) \times_{\Ccal} \Fcal \equivalence  \Bcond(([1] \times S) \times_{\Ccal} \Fcal) $.
\end{lemma}

\begin{proof}
	By pulling back along $ \alpha $ we may assume that $ \alpha $ is the identity.
	Also we have an equivalence
	\begin{equation*}
	\LFib^{\cts}([1] \times S) \simeq \Functs([1] \times S, \ICond(\Ani)) \simeq \Fun([1], \CondAni_{/S})\period
	\end{equation*}
	Now observe that $ \ptilde(\alpha) $ can be computed as $\ev_1(\epsilon \colon \const \ev_0 \ptilde \to \ptilde)$, where $ \epsilon $ denotes the counit of the adjunction $ \adjto{\const}{\CondAni_{/S}}{\Fun([1], \CondAni_{/S})}{\ev_0} $.
	Translating to the fibrational perspective via \Cref{thm:internal_str}, we obtain a rectangle
    \begin{equation*}
        \begin{tikzcd}
            \{1\} \times F_{\{0\}} \arrow[r] \arrow[d, "\ptilde(\alpha)"'] \arrow[dr, phantom, "\lrcorner"{description, very near start, xshift=-1.5ex}] & {F_{\{0\}} \times_{\{0\} \times S} ([1] \times S) \simeq [1] \times F_{\{0\}} } \arrow[d, "\epsilon"] \\
            F_{\{1\}} \arrow[r] \arrow[d] \arrow[dr, phantom, "\lrcorner"{description, very near start, xshift=-1.5ex}] & F \arrow[d] \\
            \{1\} \times S \arrow[r] & {[1] \times S}
        \end{tikzcd}
    \end{equation*}
	and we are done once we see that the composite $ F_{\{0\}} \to  F_{\{0\}} \times_{\{0\} \times S} ([1] \times S) \to F$ is identified with the inclusion $ F_{\{0\}} \to F $ after applying $ \Bcond $.
	But this is clear, since the two inclusions $ \{i\}  \times F_{\{0\}}  \hookrightarrow [1] \times  F_{\{0\}}, i = 0, 1, $ are identified after applying $ \Bcond$ and the composite
	\begin{equation*}
	\{0 \} \times  F_{\{0\}} \hookrightarrow [1] \times  F_{\{0\}} \to F
	\end{equation*}
	yields the inclusion $ F_{\{0\}} \to F $ by construction.
\end{proof}

\begin{remark}\label{rem:small_fix}
	In the situation of \Cref{lem:Identifying_straightened_functor}, we may more generally consider a map
    \begin{equation*}
        \alpha \colon [n] \times S \to \Ccal 
    \end{equation*}
    corresponding to a composable sequence of $ n $ arrows in $ \Ccal(S) $.
	Let us denote by $ j \colon [1] \to [n] $ the map that sends $ 0 $ to $ 0 $ and $ 1 $ to $ n $.
	We then get a commutative diagram
	\begin{equation*}
        \begin{tikzcd}
    		{(\{0\} \times S) \times_{\Ccal} \Fcal} \arrow[r] \arrow[d, equals] & {\Bcond(([1] \times S) \times_{\Ccal} \Fcal)} \arrow[d] & {( \{1\} \times S) \times_{\Ccal} \Fcal} \arrow[l] \arrow[d, "\wr"{xshift=-0.15em}] \\
    		{(\{0\} \times S) \times_{\Ccal} \Fcal} \arrow[r] & {\Bcond(([n] \times S) \times_{\Ccal} \Fcal)} & {( \{n\} \times S) \times_{\Ccal} \Fcal} \arrow[l]
    	\end{tikzcd}
    \end{equation*}
	where the middle vertical map is induced by $ j $.
	Since left fibrations are \emph{smooth} \cite[Proposition~4.4.7]{arXiv:2103.17141}, the right horizontal maps are equivalences and thus also the vertical map in the middle is an equivalence.
	It follows that the composite of the lower left map with the inverse of the lower right map is equivalent to $ \ptilde $ applied to the composite of the $ n $ arrows determined by $ \alpha $.
\end{remark}

One difference between \Cref{prop:technical_moerdijk_nuiten_thing} and \Cref{thm:profinThmB} is that in the former we consider fibers over general profinite sets $ S $, while in the latter we only look at fibers over points.
To reduce from profinite sets to points, we use the following observation:

\begin{lemma}\label{lem:Sigmacomp_locally_cart}
	Let $ \Sigma $ be a nonempty set of prime numbers.
    Consider a cartesian square
	\begin{equation*}
        \begin{tikzcd}
    		B \arrow[dr, phantom, "\lrcorner"{description, very near start}] \arrow[r] \arrow[d] & A \arrow[d] \\
    		T \arrow[r] & S
    	\end{tikzcd}
    \end{equation*}
	in $\CondAni$ such that $ A $ is the colimit of a diagram $\Deltaop \to \ProAnifin \to \CondAni$ and $S,T \in \Pro(\Ani_\Sigma)$.
	Then this square remains cartesian after \Sigmacompletion.
\end{lemma}

\begin{proof}
    Since $ \CondAni $ is \atopos, geometric realizations are universal in $ \CondAni $.
    By \cite[Example~1.9 and Corollary~1.13]{Haine:profinite_completions_of_products}, geometric realizations are also universal in $\Pro(\Ani_\Sigma)$.
    Thus we may assume that $ A \in \ProAnifin $.
    Since the functor $ \ProAnifin \to \CondAni$ is fully faithful, the composite
    \begin{equation*}
        \begin{tikzcd}
            \ProAnifin \arrow[r] & \CondAni \arrow[r, "{(-)\Sigmacomp}"] & \Pro(\Ani_{\Sigma})
        \end{tikzcd}
    \end{equation*}
    agrees with the \Sigmacompletion functor $(-)\Sigmacomp \colon  \ProAnifin \to \Pro(\Ani_{\Sigma})$.
    The claim now follows from the fact that \Sigmacompletion is locally cartesian \cite[Proposition~3.18]{MR4686649}.
\end{proof}

\begin{nul}
    Let $ f \colon \Ccal \to \Dcal $ be a functor of condensed \categories.
    We now consider the condensed \category $\Ccal \orientedtimes_{\Dcal} \Dcal$ defined via the pullback square
    \begin{equation*}
        \begin{tikzcd}
            \Ccal \orientedtimes_\Dcal \Dcal \arrow[r] \arrow[d] \arrow[dr, phantom, "\lrcorner"{description, very near start, xshift=-1.5ex}] & \Funcond([1],\Dcal) \arrow[d, "{(\ev_0,\ev_1)}"] \\
            \Ccal \times \Dcal \arrow[r, "f \cross \id{\Dcal}"'] & \Dcal \times \Dcal
        \end{tikzcd}
    \end{equation*}
    as in \Cref{rec:oriented_fiber_products}.
    By by \HTT{Corollary}{2.4.7.12}, the projection $\pr_2 \colon \Ccal \orientedtimes_{\Dcal} \Dcal \to \Dcal$ is a cocartesian fibration of condensed \categories .
\end{nul}

For sake of completeness we verify the following two facts which we have already used for ordinary \categories in the proof of \Cref{thm:Quillen_theorem_B}.
First recall that by unstraightening the cocartesian fibration of condensed \categories $ \ev_1 \colon \Funcond([1],\Ccal)\to \Ccal$, one sees that overcategories of condensed \categories are functorial.

\begin{proposition}\label{prop:functor_that_comma_gives}
	Let $ f \colon \Ccal \to \Dcal $ be a functor of condensed \categories and consider the natural cocartesian fibration $\pr_2 \colon \Ccal \orientedtimes_{\Dcal} \Dcal \to \Dcal $.
	Then for every profinite set $ S $ and morphism $ d \to d' $ in $ \Dcal(S) $, the induced functor on fibers is the canonical functor
	\begin{equation*}
	   \Ccal_{/d} = \Ccal \times_{ \Dcal} \Dcal_{/d} \longrightarrow \Ccal \times_{\Dcal} \Dcal_{/d'} = \Ccal_{/d'}
	\end{equation*}
	in $ \CondCat_{/S} $ induced by the slice functoriality $ \Dcal_{/d} \to \Dcal_{/d'} $.
\end{proposition}

\begin{proof}
	We observe that the pullback square
	 \begin{equation*}
        \begin{tikzcd}
            \Ccal \orientedtimes_\Dcal \Dcal \arrow[r] \arrow[d] \arrow[dr, phantom, "\lrcorner"{description, very near start, xshift=-1.5ex}] & \Funcond([1],\Dcal) \arrow[d, "{(\ev_0,\ev_1)}"] \\
            \Ccal \times \Dcal \arrow[r, "f \cross \id{\Dcal}"'] & \Dcal \times \Dcal
        \end{tikzcd}
    \end{equation*}
	is in fact a pullback square in $ \Cocart^{\cts}(\Dcal) $.
	Under the equivalence of \Cref{thm:internal_str}, it therefore corresponds to a cartesian square of functors $ \Dcal \to \ICond(\Catinfty) $
	\begin{equation*}
        \begin{tikzcd}
            \Ccal \orientedtimes_{\Dcal} \Dcal \arrow[r] \arrow[d] & \Dcal_{/(-)} \arrow[d] \\
            {\const (\Ccal)} \arrow[r, "f"'] & {\const (\Dcal)}
        \end{tikzcd}
    \end{equation*}
	which proves the claim.
\end{proof}

\begin{lemma}\label{lem:functor_into_comma_final}
	For any functor of condensed \categories $ f \colon \Ccal \to \Dcal $, the functor $ j \colon \Ccal \to \Ccal \orientedtimes_{\Dcal} \Dcal $ is a fully faithful left adjoint.
\end{lemma}

\begin{proof}
	The functor $ j $ sits inside the commutative diagram
	\begin{equation*}
        \begin{tikzcd}
    		\Ccal \arrow[r, "f"] \arrow[d, "j"'] \arrow[dr, phantom, very near start, "\lrcorner"{xshift=-0.5ex, yshift=-0.5ex}] & \Dcal \arrow[d, "\const"] \\
    		\Ccal \orientedtimes_{\Dcal} \Dcal \arrow[r] \arrow[d] \arrow[dr, phantom, very near start, "\lrcorner"{xshift=-3ex}] & \Funcond([1],\Dcal) \arrow[d, "\ev_0"] \\
    		\Ccal \arrow[r, "f"'] & \Dcal 
    	\end{tikzcd}
    \end{equation*}
	in which all squares are cartesian.
	Since $ \const $ is the fully faithful left adjoint of $ \ev_0 $, the proof of \cite[Lemma~6.3.9]{MR4752519} shows that $ j $ is also a fully faithful left adjoint.
\end{proof}

\begin{proof}[Proof of \Cref{thm:profinThmB}]
	We factor $ f $ as
    \begin{equation*}
        \begin{tikzcd}
            \Ccal \arrow[r, "j"] & \Ccal \orientedtimes_{\Dcal} \Dcal \arrow[r, "\pr_2"] & \Dcal
        \end{tikzcd}
    \end{equation*}
    and apply the left adjoint of \cref{rem:condensned_fiberwise_B} to the cocartesian fibration $ \pr_2$.
	The resulting left fibration $p \colon \Fcal \to \Ccal$ classifies the functor
	\begin{equation*}
	   \Bcond \circ \widetilde{\pr}_2 \colon \Ccal \to \ICond(\Ani)
	\end{equation*}
	and is given by factoring
    \begin{equation*}
        \begin{tikzcd}
            \Ccal \orientedtimes_{\Dcal} \Dcal \arrow[r, "\iota"] & \Fcal \arrow[r, "p"] & \Ccal \comma
        \end{tikzcd}
    \end{equation*}
	where $\iota$ is initial and $p$ is a left fibration.
	Here, $\widetilde{\pr}_2$ is the unstraightened functor of $ \pr_2$.

	We now apply \Cref{prop:technical_moerdijk_nuiten_thing} to the left fibration $p$, with $L$ the \Sigmacompletion functor
	\begin{equation*}
	   (-)\Sigmacomp \colon \CondAni \to \Pro(\Ani_\Sigma) \period
	\end{equation*}
	Thus we have to verify that for any $S \in \ProFin$ and any map $\alpha \colon d \to d' \in \Ccal(S)$, the induced map $\Bcond \widetilde{\pr}_2(\alpha) $ becomes an equivalence after \Sigmacompletion.
	By construction $\widetilde{\pr}_2(d)$ is defined via a cartesian square
	\begin{equation*}
        \begin{tikzcd}
    		\widetilde{\pr}_2(d) \arrow[r] \arrow[d] \arrow[dr, phantom, very near start, "\lrcorner"{xshift=-1.5ex, yshift=0.5ex}] & \Ccal \orientedtimes_{\Dcal} \Dcal \arrow[d] \\
    		S \arrow[r, "d"'] & {\Dcal}
    	\end{tikzcd}
    \end{equation*}
	and similarly for $\widetilde{\pr}_2(d')$.
	It follows that both $\widetilde{\pr}_2(d)$ and $\widetilde{\pr}_2(d')$ are in $\ICat(\Pro(\Ani_\pi))$ since the latter is closed under limits in $\CondCat$.
	It follows that for any point $s \colon \ast \to S $ the cartesian square
	\begin{equation*}
        \begin{tikzcd}
    		\Bcond \widetilde{\pr}_2(d \circ s) \arrow[r] \arrow[d] \arrow[dr, phantom, very near start, "\lrcorner"{xshift=-2ex}] & \Bcond \widetilde{\pr}_2(d) \arrow[d] \\
    		\ast \arrow[r, "s"'] & S
    	\end{tikzcd}
    \end{equation*}
	satisfies the assumptions of \Cref{lem:Sigmacomp_locally_cart}, since $ \Bcond $ is the geometric realization of the corresponding simplicial object.
    Thus it remains cartesian after \Sigmacompletion (also the same holds for $d'$ instead of $d$).
	By \SAG{Theorem}{E.3.6.1}, equivalences in $\Pro(\Ani_\Sigma)$ can be checked fiberwise.
    Thus we may thus reduce to the case where $ S = \ast$.
	But in this case $\Bcond \widetilde{\pr}_2(\alpha)$ is by construction the map
	\begin{equation*}
	\Bcond (\Ccal_{/d}) \to  \Bcond (\Ccal_{/d'}) \comma
	\end{equation*}
	which becomes an equivalence after \Sigmacompletion by assumption.
	Thus, \Cref{prop:technical_moerdijk_nuiten_thing} shows that in the commutative diagram
    \begin{equation*}
        \begin{tikzcd}[column sep=4em, row sep=3em]
            \Bcond \Ccal_{/d } \arrow[r] \arrow[d] & \Bcond \Ccal \arrow[r, "\Bcond j"] \arrow[d, "\Bcond f"] & \Bcond (\Ccal \orientedtimes_{\Dcal} \Dcal) \arrow[r, "\Bcond \iota"] & \Bcond \Fcal \arrow[d, "\Bcond L (\pr_2)"] \\
            \ast \simeq \Bcond \Dcal_{/d } \arrow[r, "d"'] & \Bcond \Dcal \arrow[rr, "{\id{}}"'] & & \Bcond \Dcal \comma
        \end{tikzcd}
    \end{equation*}
	the outer square is cartesian.
	Since $ \Bcond$ inverts left adjoints and initial functors of condensed \categories, the claim follows.
\end{proof}

%-------------------------------------------------------------------%
%-------------------------------------------------------------------%
%  Galois groups of function fields                                 %
%-------------------------------------------------------------------%
%-------------------------------------------------------------------%

\section{Galois groups of function fields}\label{appendix:Galois_groups_of_function_fields}

It is well-known that there is an isomorphism of profinite groups
\begin{equation*}
    \Freepf_\CC \equivalent \Gal_{\CC(T)}
\end{equation*}
between the free profinite group on the underlying set of $ \CC $ and the absolute Galois group of the function field $ \CC(T) $.
See \cites{MR162796}{MR1800587}.
Moreover, it seems to be folklore that this isomorphism can be chosen so that the free profinite group generated by an element $ a \in \CC $ corresponds to a decomposition group of the prime $ (T-a) $.
See \cite[\S1.8]{MR1352283}.
The purpose of this appendix is to record a proof of this folklore statement.
This was also implicitly shown in \cite{MR289471}, and we do not claim originality of any of the results in this appendix.

\begin{notation}
    Throughout this section we fix an algebraic closure $K$ of the function field $\CC(T)$.
    We write $ \Gal_{\CC(T)} \colonequals \Gal(K/\CC(T))$.
\end{notation}

\begin{recollection}
    Write $ \overline{\CC[T]} \subset K $ for the integral closure of $\CC[T]$ in $K$.
    For any $a \in \CC$ a choice of prime ideal $\abar$ in $\overline{\CC[T]}$ lying over $(T-a)$ then determines a decomposition group $ \Dup_{\abar} \subset \Gal_{\CC(T)}$.
    Moreover, if $\abar'$ is another choice of prime above $(T-a)$, then $\Dup_{\abar'}$ is conjugate to $\Dup_{\abar}$.
\end{recollection}

Our goal is to prove the following result, which is a slight refinement of \cite[Theorem 2]{MR162796} for $C = \CC$.

\begin{theorem}\label{Thm:Gen_of_Douady}
    There is an isomorphism of profinite groups
    \begin{equation*}
        \Freepf_\CC \to \Gal_{\CC(t)}
    \end{equation*}
    such that for each $a \in \CC$ the image of $\ZZhat(a)$ under this isomorphism is the decomposition group $\Dup_{\abar|a}$ of a prime $ \abar $ lying over $ (T-a) $.
\end{theorem}

\begin{definition}\label{def:adapted}
    Let $M$ be a set.
    Write $\Sigma$ for the system of finite subsets $S \subset M$ partially ordered by inclusion.
    Let $((G_S)_{S \in \Sigma}, (\rho^{T}_{S})_{S \subset T})$ be an inverse system of profinite groups with limit $G_{M} \colonequals \lim_{S \in \Sigma} G_S$ and write $\rho^{M}_{S} \from G_M \to G_S$ for the canonical projection.
    Let $N$ be either the whole of $M$, or an element of $\Sigma$.
    \begin{enumerate}
        \item \label{def:adapted.1} We say that a function $\phi \from N \to G_N$ is \emph{adapted} if $\rho^{N}_{S}(\phi(n)) = 1$ for all finite subsets $S \subset N$ and all $n \not\in S$.
        
        \item \label{def:adapted.2} We say that a function $\phi \from N \to G_N$ is an \emph{adapted basis} if $\phi$ is adapted and if the map $\Freepf_N \to G_N$ induced by $\phi$ is an isomorphism.
        
        \item \label{def:adapted.3} We say that a system $\Bcal = (\Bcal_S)_{S \in \Sigma}$ of sets of functions $\Bcal_{S} \subset \Hom(S, G_S)$ is a \emph{system of adapted bases} if the following conditions hold.
        \begin{enumerate}[label = {\upshape (\alph*)},ref  = \alph*]
            \item \label{def:first_cond_adapted_basis} For each $S \in \Sigma$, $\Bcal_S \subset \Hom(S, G_S) = \prod_S G_S$ is a nonempty closed subset consisting of adapted bases.
            
            \item \label{def:second_cond_adapted_basis} For each $S \subset T \in \Sigma$, and each $\phi \in \Bcal_T$, the restriction $S \subset T \xrightarrow{\phi} G_T \xrightarrow{\rho^T_S} G_S$ is an element of $\Bcal_S$.
        \end{enumerate}
    \end{enumerate}
\end{definition}

\begin{proposition}\label{prop:generalized-basis-theorem}
    Let $M$ be a set.
    Write $\Sigma$ for the poset of finite subsets $ S \subset M $ partially ordered by inclusion.
    Let $((G_S)_{S \in \Sigma}, (\rho^{T}_{S})_{S \subset T})$ be an inverse system of profinite groups with limit $G_{M} \colonequals \lim_{S \in \Sigma} G_S$ and write $\rho^{M}_{S} \from G_M \to G_S$ for the canonical projection.
    Let $ \Bcal $ be a system of adapted bases.
    If all the transition maps $\rho^T_S \from G_T \to G_S$ are surjective, then there exists an adapted basis $M \to G_M$ such that for each $S \in \Sigma$, the restriction 
    \begin{equation*}
      S \subset M \to G_M \xrightarrow{\rho^M_S} G_S
    \end{equation*}
    is a basis contained in $\Bcal_S$.
\end{proposition}

\begin{proof}
    In \cite[Theorem 1]{MR162796}, Douady proved the above claim in the case where $ \Bcal $ is the system of adapted bases consisting of $\Bcal_S$ the set of all adapted bases $S \to G_S$.
    However, the argument he gives actually only uses the axiomatic of a general system of adapted bases in the above sense.
\end{proof}

We will use the following lemma:

\begin{lemma}\label{lem:conjugator_set_closed}
    Let $G $ be a profinite group and let $H, H' \subset G $ be closed subgroups.
    Let $ \alpha \colon G  \to G'$ be a homomorphism of profinite groups.
    Let
    \begin{equation*}
        M \colonequals \setbar{g \in G}{\alpha (g^{-1}) \alpha(H) \alpha(g) = \alpha(H')} \period
    \end{equation*}
    Then $M$ is closed in $ G $.
\end{lemma}

\begin{proof}
    We first consider the set
    \begin{equation*}
        M' \colonequals \setbar{g \in G}{\alpha (g^{-1}) \alpha(H) \alpha(g) \subset \alpha(H')} \period
    \end{equation*}
    For $ h \in H $, write
    \begin{equation*}
        M_h' \colonequals \setbar{g \in G}{\alpha(g^{-1} h g) \in \alpha(H')} \period
    \end{equation*}
    This is preimage of $\alpha(H') \subset G'$ under the continuous map $G \to G'$ that sends $g$ to $\alpha(g^{-1} h g)$.
    Since $\alpha(H') \subset G'$ is closed it follows that $M'_h$ is closed.
    Since
    \begin{equation*}
        M' = \bigcap_{h \in H} M'_h
    \end{equation*}
    it follows that $M'$ is closed.
    Now note that the same argument shows that
    \begin{equation*}
        M'' \colonequals \setbar{ g \in G }{\alpha(g) \alpha(H') \alpha(g)^{-1} \subset \alpha(H)}
    \end{equation*}
    is closed.
    Thus $M = M' \cap M''$ is closed.
\end{proof}

\begin{proof}[Proof of \cref{Thm:Gen_of_Douady}]
    Our choice of algebraic closure yields an isomorphism
    \begin{equation*}
        \Gal_{\CC(T)} \isomorphism \lim_{S \subset \CC \textup{ finite}} \pioneet(\AA^1 \setminus S,\etabar) \period
    \end{equation*}
    Let us write $G_S = \pioneet(\AA^1 \setminus{S},\etabar )$.
    We want to apply \Cref{prop:generalized-basis-theorem} to this inverse systems of groups and the system of adapted bases $\Bcal_S$ that consists of those maps $\phi \colon S \to G_S$ that are adapted bases and for any $s \in S$, the subgroup $\ZZhat(\phi(s))$ is (conjugate to) a decomposition group at $s$.
    To see that $(\Bcal_S)_S$ is a system of adapted bases, we need to show that the conditions \hyperref[def:first_cond_adapted_basis]{\Cref*{def:adapted}-(\ref*{def:adapted.3}.\ref*{def:first_cond_adapted_basis})} and \hyperref[def:first_cond_adapted_basis]{\Cref*{def:adapted}-(\ref*{def:adapted.3}.\ref*{def:second_cond_adapted_basis})} are satisfied.
    It is clear that (\ref*{def:adapted.3}.\ref*{def:second_cond_adapted_basis}) is satisfied, so we only check (\ref*{def:adapted.3}.\ref*{def:first_cond_adapted_basis}).
    We start by verifying that $\Bcal_S \subset \Hom(S, G_S)$ is closed.
    To this end, note that the larger subset $ \Bcal^\text{all}_S \subset  \Hom(S,G_S)$, consisting of all adapted bases is closed, see the beginning of the proof of \cite[Proposition 3.4.9]{MR2548205}.
    To conclude, it suffices to see that for all $s \in S$ the subset $\Sigma_s \subset G_S$, consisting of those $\sigma \in G_S$ with the property that $\ZZhat(\sigma)$ is a decomposition group at $s$, is closed.
    Indeed, in this case
    \begin{equation*}
        \Bcal_S = \Bcal^\text{all}_S \cap \prod_{s \in S} \Sigma_s \subset \Hom(S, G_S) = \prod_S G_S \period
    \end{equation*}
    is seen to be an intersection of closed subsets, hence itself closed.
    Fix one decomposition group $\Dup_s$ at $s$. 
    Since $\Dup_s \equivalent \ZZhat$, the subset $N \subset \Dup_s$ of elements that topologically generate $\Dup_s$ is closed.
    Now observe that $\Sigma_s$ agrees with the image of the continuous map
    \begin{equation*}
        N \times G_S \to G_S; \; \; (n,g) \mapsto g^{-1} n g
    \end{equation*}
    and is therefore closed, since the domain is compact.
    Finally, we need to check that $\Bcal_S \neq \emptyset$.
    Choose a point $x \in \CC \setminus S$ and an étale path $\alpha \from \etabar \rightsquigarrow x$ and consider the isomorphism
    \begin{equation*}
      \psi  \colon \uppi_{1}^{\top} (\CC \setminus S, x )^{\wedge} \isomto{} \pioneet(\AA_{\CC}^1 \setminus{S},x) \isomto{} \pioneet(\AA_{\CC}^1 \setminus{S}, \etabar)
    \end{equation*}
    obtained from the Riemann existence theorem and conjugation with $\alpha^{-1}$.
    Recall that $\uppi_{1}^{\top} (\CC \setminus S, x )$ is freely generated by simple loops $\gamma_s$ at $x$ around $s$, that do not loop around other points in $ S $.
    Then $ (s \mapsto \psi(\gamma_s))$ is clearly an adapted basis and furthermore $\psi(\gamma_s)$ generates a decomposition group at $s$.
    Thus $(s \mapsto \psi(\gamma_s)) \in \Bcal_S$.

    By applying \Cref{prop:generalized-basis-theorem}, we obtain an isomorphism $\phi \from \Freepf_\CC  \isomto{}  \Gal_{\CC(T)}$ with the property that for all finite subsets $S \subset \CC$ and $a \in S$, $(\rho^{\CC}_{S} \circ \phi)(a)$ generates a decomposition group at $a$ in $G_S$.
    We now show that $\phi(a)$ generates a decomposition group at $a$ in $\Gal_{\CC(T)}$ for any $a \in \CC$.
    To this end, fix one decomposition group $\Dup_a \subset \Gal_{\CC(T)}$ of $a$.
    By the above, for every finite subset $S \subset \CC$ there exists some $g \in \Gal_{\CC(T)}$ such that $\ZZhat(\phi(a)) = g^{-1} \Dup_a g$ in $G_S$.
    Now by \Cref{lem:conjugator_set_closed} the set $C_S$ of all such $g$ is closed.
    Therefore $\bigcap_{S} C_S = \lim_{S} C_S$ is nonempty as a cofiltered limit of nonempty compact Hausdorff spaces.
    By construction, any element $g \in \bigcap_S C_S$ has the property that $\ZZhat(\phi(a)) = g^{-1} \Dup_a g $ holds after projecting to $G_S$ simultaneously for all $S \subset \CC$ finite.
    Since both $\Dup_a$ and $\ZZhat(\phi(a))$ are closed subgroups of $\Gal_{\CC(T)} = \lim_{S \subset \CC \text{ finite}} G_S$, this shows that indeed $g^{-1} \Dup_a g = \ZZhat(\phi(a))$.
    In particular, $\phi(a)$ generates a decomposition group as desired.
\end{proof}

%-------------------------------------------------------------------%
%-------------------------------------------------------------------%
%  References                                                       %
%-------------------------------------------------------------------%
%-------------------------------------------------------------------%

\DeclareFieldFormat{labelalphawidth}{#1}
\DeclareFieldFormat{shorthandwidth}{#1}
\printbibliography[heading=references]

%-------------------------------------------------------------------%
%-------------------------------------------------------------------%
%  Addresses                                                       %
%-------------------------------------------------------------------%
%-------------------------------------------------------------------%

\bigskip
\small{
\noindent \textsc{Peter J. Haine, University of Southern California, Kaprielian Hall, 248C, Los Angeles, CA 90089, USA}
\medskip

\noindent \textsc{Tim Holzschuh, Universität Heidelberg, Institut für Mathematik, Im Neuenheimer Feld 205, 69120 Heidelberg, Germany}
\medskip

\noindent \textsc{Marcin Lara, Instytut Matematyczny PAN, Śniadeckich 8, 00-656 Warsaw, Poland}
\medskip

\noindent \textsc{Catrin Mair, Technische Universität Darmstadt, Schlossgartenstraße 7, 64289 Darmstadt, Germany}
\medskip

\noindent \textsc{Louis Martini, Norwegian University of Science and Technology (NTNU), Alfred Getz’ vei 1, 7034 Trondheim, Norway}
\medskip

\noindent \textsc{Sebastian Wolf, Universität Regensburg, Universitätsstraße 31, 93053 Regensburg, Germany}
}

\end{document}